\documentclass[reqno]{amsart}
\usepackage{graphicx}
\usepackage{amsmath,amssymb,amsfonts}
\usepackage{color,fullpage}
\usepackage{stix,stmaryrd}
\usepackage{multirow,array}
\usepackage{makecell}

\usepackage[colorlinks=true, pdfstartview=FitV, linkcolor=blue, citecolor=blue, urlcolor=blue]{hyperref}

\begin{document}
\newcommand{\dd}{\mathrm{d}}
\newcommand{\ii}{\mathrm{i}}
\newcommand{\ee}{\mathrm{e}}
\newcommand{\LZeroBlue}{{L^0_\squarellblack}}
\newcommand{\LZeroBlueOne}{{L^{0,1}_\squarellblack}}
\newcommand{\LZeroBlueTwo}{{L^{0,2}_\squarellblack}}
\newcommand{\LZeroBlueThree}{{L^{0,3}_\squarellblack}}
\newcommand{\LInftyBlue}{{L^\infty_\squarellblack}}
\newcommand{\LInftyBlueOne}{{L^{\infty,1}_\squarellblack}}
\newcommand{\LInftyBlueTwo}{{L^{\infty,2}_\squarellblack}}
\newcommand{\LZeroRed}{{L^0_\squareurblack}}
\newcommand{\LZeroRedOne}{{L^{0,1}_\squareurblack}}
\newcommand{\LZeroRedTwo}{{L^{0,2}_\squareurblack}}
\newcommand{\LInftyRed}{{L^\infty_\squareurblack}}
\newcommand{\LInftyRedOne}{{L^{\infty,1}_\squareurblack}}
\newcommand{\LInftyRedTwo}{{L^{\infty,2}_\squareurblack}}
\newcommand{\LInftyRedThree}{{L^{\infty,3}_\squareurblack}}
\newcommand{\VZeroBlue}{{\mathbf{V}^0_\squarellblack}}
\newcommand{\VInftyBlue}{{\mathbf{V}^\infty_\squarellblack}}
\newcommand{\VZeroRed}{{\mathbf{V}^0_\squareurblack}}
\newcommand{\VInftyRed}{{\mathbf{V}^\infty_\squareurblack}}
\newcommand{\OurPower}[2]{{#1^{#2}_\squarellblack}}
\newcommand{\OurArg}[1]{{\arg_\squarellblack(#1)}}
\newcommand{\OurLog}[1]{{\log_\squarellblack(#1)}}
\newcommand{\OurLogAvg}[1]{{\langle\log_\squarellblack(#1)\rangle}}
\newcommand{\pEInftyRed}{{\partial E^\infty_\squareurblack}}
\newcommand{\pEZeroBlue}{{\partial E^0_\squarellblack}}
\newcommand{\lNaught}{{p}}
\newcommand{\lNaughtZeroBlue}{{p^0_\squarellblack}}
\newcommand{\lNaughtInftyRed}{{p^\infty_\squareurblack}}
\newcommand{\LensRedPM}{{\Lambda^\pm_\squareurblack}}
\newcommand{\LensBluePM}{{\Lambda^\pm_\squarellblack}}
\newtheorem{lem}{Lemma}
\newtheorem{prop}{Proposition}
\newtheorem{conjecture}{Conjecture}
\newtheorem{theorem}{Theorem}
\newtheorem{corollary}{Corollary}
\newtheorem{rhp}{Riemann-Hilbert Problem}
\newtheorem{rem}{Remark}

\title{Rational Solutions of the Painlev\'e-III Equation:  Large Parameter Asymptotics}
\author{Thomas Bothner}
\address{Department of Mathematics, University of Michigan, 2074 East Hall, 530 Church Street, Ann Arbor, MI 48109-1043, United States}
\email{bothner@umich.edu}
\author{Peter D. Miller}
\address{Department of Mathematics, University of Michigan, 2074 East Hall, 530 Church Street, Ann Arbor, MI 48109-1043, United States}
\email{millerpd@umich.edu}

\thanks{T. B. acknowledges support of the AMS and the Simons Foundation through a travel grant and P. D. M. is supported by the National Science Foundation under grants DMS-1513054 and DMS-1812625. The authors are grateful to Y. Sheng for useful conversations.}

\date{\today}

\begin{abstract}
The Painlev\'e-III equation with parameters $\Theta_0=n+m$ and $\Theta_\infty=m-n+1$ has a unique rational solution $u(x)=u_n(x;m)$ with $u_n(\infty;m)=1$ whenever $n\in\mathbb{Z}$.  Using a Riemann-Hilbert representation proposed in \cite{BothnerMS18}, we study the asymptotic behavior of $u_n(x;m)$ in the limit $n\to+\infty$ with $m\in\mathbb{C}$ held fixed.  We isolate an eye-shaped domain $E$ in the $y=n^{-1}x$ plane that asymptotically confines the poles and zeros of $u_n(x;m)$ for all values of the second parameter $m$.  We then show that unless $m$ is a half-integer, the interior of $E$ is filled with a locally uniform lattice of poles and zeros, and the density of the poles and zeros is small near the boundary of $E$ but blows up near the origin, which is the only fixed singularity of the Painlev\'e-III equation.  In both the interior and exterior domains we provide accurate asymptotic formul\ae\ for $u_n(x;m)$ that we compare with $u_n(x;m)$ itself for finite values of $n$ to illustrate their accuracy.  We also consider the exceptional cases where $m$ is a half-integer, showing 
that the poles and zeros of $u_n(x;m)$ now accumulate along only one or the other of two ``eyebrows'', i.e., exterior boundary arcs of $E$.  
\end{abstract}
\maketitle

%

\tableofcontents

\section{Introduction}
Generic solutions of the six Painlev\'e equations cannot be expressed in terms of elementary functions, hence the common terminology of \emph{Painlev\'e transcendents} for the general solutions of these famous equations.  However, it is also known that all of the Painlev\'e equations except for the Painlev\'e-I equation admit solutions expressible in terms of classical special functions (e.g., Airy solutions for Painlev\'e-II, or Bessel solutions for Painlev\'e-III) as well as rational solutions, both of which occur for certain isolated values of the auxiliary parameters (each Painlev\'e equation except Painlev\'e-I is actually a family of differential equations indexed by one or more complex parameters).  Rational solutions of Painlev\'e equations have attracted interest partly because they are known to occur in several diverse applications such as the description of equilibrium configurations of fluid vortices \cite{Clarkson09} and of particular solutions of soliton equations \cite{Clarkson06}, electrochemistry \cite{BassNRS10}, parametrization of string theories \cite{Johnson06}, spectral theory of quasi-exactly solvable potentials \cite{ShapiroT14}, and the description of universal wave patterns \cite{BuckinghamM12}.  In several of these applications it is interesting to consider the behavior of the rational Painlev\'e solutions when the parameters in the equation become large (possibly along with the independent variable); as the degree of the rational function is tied to the parameters via B\"acklund transformations, in this limit algebraic representations of rational solutions become unwieldy and hence less attractive than analytical ones as a means for extracting asymptotic behaviors.  Recent progress on the analytical study of large-degree rational Painlev\'e solutions includes \cite{BertolaB15,BuckinghamM14,BuckinghamM15,MillerS17} for Painlev\'e-II and \cite{Buckingham17,MasoeroR18} for Painlev\'e-IV.  Both of these equations have the property that there is no fixed singular point except the point at infinity.  On the other hand, the Painlev\'e-III equation is the simplest of the Painlev\'e equations having a finite fixed singular point (at the origin).
This paper is the second in a series beginning with \cite{BothnerMS18} concerning the large-degree asymptotic behavior of rational solutions to the Painlev\'e-III equation, which we take in the generic form
\begin{equation}
\frac{\dd^2u}{\dd x^2}=\frac{1}{u}\left(\frac{\dd u}{\dd x}\right)^2-\frac{1}{x}\frac{\dd u}{\dd x} + 
\frac{4\Theta_0 u^2 + 4(1-\Theta_\infty)}{x} + 4u^3-\frac{4}{u},\ \ \ \ x\in\mathbb{C}.
\label{eq:PIII}
\end{equation}
It is convenient to represent the constant parameters $\Theta_0$ and $\Theta_\infty$ in the form
\begin{equation}
\Theta_0=n+m\quad\text{and}\quad\Theta_\infty=m-n+1.
\label{eq:Thetas-m-n}
\end{equation}
It is known that if $n\in\mathbb{Z}$, there exists a unique rational solution $u(x)=u_n(x;m)$ of \eqref{eq:PIII} that tends to $1$ as $x\to\infty$.  The odd reflection $u(x)=-u_n(-x;m)$ provides a second distinct rational solution.  Similarly, if $m\in\mathbb{Z}$, there are two rational solutions tending to $\pm\ii$ as $x\to\infty$, namely $u(x)=\pm\ii u_m(\pm\ii x;n)$, while if neither $m$ nor $n$ is an integer, \eqref{eq:PIII} has no rational solutions at all.  If only one of $m$ and $n$ is an integer, then there are exactly two rational solutions; however if both $m\in\mathbb{Z}$ and $n\in\mathbb{Z}$ there are exactly four distinct rational solutions: $u_n(x;m)$, $-u_n(-x;m)$, $\ii u_m(\ii x;n)$, and $-\ii u_m(-\ii x;n)$.  

\subsection{Representations of $u_n(x;m)$}
\subsubsection{Algebraic representation}
\label{sec:algebraic-representation}
It has been shown \cite{Clarkson03,ClarksonLL16,Umemura99} that $u_n(x;m)$ admits the representation
\begin{equation}
u_n(x;m)=\frac{s_n(x;m-1)s_{n-1}(x;m)}{s_n(x;m)s_{n-1}(x;m-1)};\quad u_{-n}(x;m)=\frac{1}{u_n(x;m)},\quad n\in\mathbb{Z}_{\ge 0},
\label{eq:un-exact-fraction}
\end{equation}
where $\{s_n(x;m)\}_{n=-1}^{\infty}$ are polynomials in $x$ with coefficients polynomial in $m$ that are defined by the recurrence formula
\begin{equation}
s_{n+1}(x;m):=\frac{(4x+2m+1)s_n(x;m)^2-s_n(x;m)s_n'(x;m)-x\left(s_n(x;m)s_n''(x;m)-s_n'(x;m)^2\right)}{2s_{n-1}(x;m)}
\label{eq:sn-recurrence}
\end{equation}
and the initial conditions $s_{-1}(x;m)=s_0(x;m)=1$.  The polynomials $\{s_n(x;m)\}_{n=-1}^\infty$ are frequently called the \emph{Umemura polynomials}, although in \cite{Umemura99} Umemura originally considered instead related functions that are polynomials in $1/x$.  For $n$ not too large, the recurrence relation \eqref{eq:sn-recurrence} provides an effective computational strategy to obtain the poles and zeros of $u_n(x;m)$.  The rational function $u_n(x;m)$ has the following symmetry:
\begin{equation}
u_n(-x;m)=\frac{1}{u_n(x;-m)}.
\label{eq:u-n-exact-symmetry}
\end{equation}
This follows from the fact that $u(x)\mapsto u(-x)^{-1}$ is a symmetry of \eqref{eq:PIII}--\eqref{eq:Thetas-m-n} corresponding to the parameter mapping $(m,n)\mapsto (-m,n)$.  Since this symmetry preserves rationality and asymptotics $u\to 1$ as $x\to\infty$, it descends from general solutions to the particular solution $u_n(x;m)$ as written in \eqref{eq:u-n-exact-symmetry}.

\subsubsection{Analytic representation}
\label{sec:analytic-representation}
The goal of this paper is to study $u_n(x;m)$ when $n$ is a large positive integer and $m$ is a fixed complex number.  The representation \eqref{eq:un-exact-fraction} is useful to determine numerous properties of the rational Painlev\'e-III solutions, however when $n$ is large another representation becomes more preferable.  To explain this alternate representation, we first define some $x$-dependent arcs in an auxiliary complex $\lambda$-plane as follows.
Given $x\in\mathbb{C}$ with $x\neq 0$ and $|\mathrm{Arg}(x)|<\pi$, there is an intersection point $\lNaught$ and four oriented arcs $\LInftyRed$, $\LZeroRed$, $\LInftyBlue$, and $\LZeroBlue$ such that:
\begin{itemize}
\item The arc $\LInftyRed$ originates from $\lambda=\infty$ in such a direction that $\ii x\lambda$ is negative real and terminates at $\lambda=\lNaught$, 
the arc $\LZeroRed$ begins at $\lambda=\lNaught$ and terminates at $\lambda=0$ in a direction such that $-\ii x\lambda^{-1}$ is negative real, and the net increment of the argument of $\lambda$ along $\LInftyRed\cup\LZeroRed$ is
\begin{equation}
\Delta\arg(\squareurblack)=2\mathrm{Arg}(x)
\pm 2\pi.
\label{eq:increment-argument-red}
\end{equation}
The ambiguity of the sign in \eqref{eq:increment-argument-red} will be explained below (see Remark~\ref{rmk:surgery}).
\item The arc $\LInftyBlue$ originates from $\lambda=\infty$ in such a direction that $-\ii x\lambda$ is negative real and terminates at $\lambda=\lNaught$,
the arc $\LZeroBlue$ begins at $\lambda=\lNaught$ and terminates at $\lambda=0$ in a direction such that $\ii x\lambda^{-1}$ is negative real, and the net increment of the argument of $\lambda$ along $\LInftyBlue\cup\LZeroBlue$ is
\begin{equation}
\Delta\arg(\squarellblack)=2\mathrm{Arg}(x).
\label{eq:increment-argument-blue}
\end{equation}
\item The arcs $\LInftyRed$, $\LZeroRed$, $\LInftyBlue$, and $\LZeroBlue$ do not otherwise intersect.
\end{itemize}
%
See Figure~\ref{fig:y-5-Exp-i-pi-over-4} below for an illustration in the case of $\mathrm{Arg}(x)=\tfrac{1}{4}\pi$.
We define a single-valued branch of the argument function $\lambda\mapsto\arg(\lambda)$, henceforth denoted $\OurArg{\lambda}$, by first selecting $\LInftyBlue\cup\LZeroBlue$
as the branch cut, and then defining $\OurArg{\lambda}=0$ for sufficiently large positive $\lambda$ when $\mathrm{Im}(x)>0$ and $\OurArg{\lambda}=\pi$ for sufficiently large negative $\lambda$ when $\mathrm{Im}(x)<0$.  It is easy to see that this definition is consistent for $x>0$ but there is a jump across the negative real $x$-axis.  We define an associated branch of the complex logarithm $\log(\lambda)$ by setting $\OurLog{\lambda}:=\ln|\lambda|+\ii\,\OurArg{\lambda}$.  Then,
given $q\in\mathbb{C}$, the corresponding branch of the power function $\lambda^q$ will be denoted by 
$\OurPower{\lambda}{q}:=\ee^{q\,\OurLog{\lambda}}$.
Finally, we denote by $L$ the union of the four oriented arcs $\LInftyRed$, $\LZeroRed$, $\LInftyBlue$, and $\LZeroBlue$, and define the function 
\begin{equation}
\varphi(\lambda):=\lambda-\lambda^{-1},\ \ \ \lambda\in\mathbb{C}\setminus\{0\}.
\label{eq:varphi}
\end{equation}
The following Riemann-Hilbert problem  was formulated in \cite[Sec.\@ 1.2]{BothnerMS18}.  Here and below we follow the convention that subscripts $+$/$-$ refer to boundary values taken on a jump contour from the left/right,
and $\sigma_3:=\mathrm{diag}[1,-1]$ denotes a standard Pauli spin matrix.
\begin{rhp}  Given parameters $m\in\mathbb{C}$ and $n=0,1,2,3,\dots$, as well as $x\in\mathbb{C}\setminus\{0\}$ with $|\mathrm{Arg}(x)|<\pi$, seek a $2\times 2$ matrix function $\lambda\mapsto\mathbf{Y}(\lambda)=\mathbf{Y}_n(\lambda;x,m)$ with the following properties.
\begin{itemize}
\item[1.]\textbf{Analyticity:}  $\lambda\mapsto\mathbf{Y}(\lambda)$ is analytic in the domain $\lambda\in\mathbb{C}\setminus L$.  It takes continuous boundary values on $L\setminus\{0\}$ from each maximal domain of analyticity.  
\item[2.]\textbf{Jump conditions:}  The boundary values $\mathbf{Y}_\pm(\lambda)$ are related on each arc of $L$ by the following formul\ae:
\begin{equation}
\mathbf{Y}_+(\lambda)
=\mathbf{Y}_-(\lambda)\begin{bmatrix}
1 & \displaystyle -\frac{\sqrt{2\pi}\OurPower{\lambda}{-(m+1)}}{\Gamma(\tfrac{1}{2}-m)}\lambda^n\ee^{\ii x\varphi(\lambda)}\\
0 & 1
\end{bmatrix},\quad \lambda\in \LZeroRed,
\label{eq:Yjump-1}
\end{equation}
\begin{equation}
\mathbf{Y}_+(\lambda)
=\mathbf{Y}_-(\lambda)\begin{bmatrix}
1 & \displaystyle \frac{\sqrt{2\pi}\OurPower{\lambda}{-(m+1)}}{\Gamma(\tfrac{1}{2}-m)}\lambda^n\ee^{\ii x\varphi(\lambda)}\\
0 & 1
\end{bmatrix},\quad \lambda\in \LInftyRed,
\label{eq:Yjump-2}
\end{equation}
\begin{equation}
\mathbf{Y}_+(\lambda)
=\mathbf{Y}_-(\lambda)\begin{bmatrix}
1 & 0\\
\displaystyle \frac{\sqrt{2\pi}(\OurPower{\lambda}{(m+1)/2})_+(\OurPower{\lambda}{(m+1)/2})_-}{\Gamma(\tfrac{1}{2}+m)}\lambda^{-n}\ee^{-\ii x\varphi(\lambda)} & 1
\end{bmatrix},\quad \lambda\in \LInftyBlue,
\label{eq:Yjump-3}
\end{equation}
\begin{equation}
\mathbf{Y}_+(\lambda)
=\mathbf{Y}_-(\lambda)\begin{bmatrix}
-\ee^{2\pi\ii m} & 0\\
\displaystyle \frac{\sqrt{2\pi}(\OurPower{\lambda}{(m+1)/2})_+(\OurPower{\lambda}{(m+1)/2})_-}{\Gamma(\tfrac{1}{2}+m)}\lambda^{-n}\ee^{-\ii x\varphi(\lambda)} & -\ee^{-2\pi \ii m}\end{bmatrix}, \quad \lambda\in \LZeroBlue.
\label{eq:Yjump-4}
\end{equation}
\item[3.]\textbf{Asymptotics:}  $\mathbf{Y}(\lambda)\to\mathbb{I}$ as $\lambda\to\infty$.  Also,
the matrix function $\mathbf{Y}(\lambda)\OurPower{\lambda}{-(\Theta_0+\Theta_\infty)\sigma_3/2}=\mathbf{Y}(\lambda)\OurPower{\lambda}{-(m+\tfrac{1}{2})\sigma_3}$ has a well-defined limit as $\lambda\to 0$ (the same limit from each side of $L$).
\end{itemize}
\label{rhp:renormalized}
\end{rhp}

\begin{rem}
Given any choice of sign in \eqref{eq:increment-argument-red}, the sign may be reversed by a surgery performed on $\LInftyRed\cup\LZeroRed$ for any given value of $x\neq 0$, $|\mathrm{Arg}(x)|<\pi$ which leaves the conditions of Riemann-Hilbert Problem~\ref{rhp:renormalized} invariant.
The surgery consists of bringing $\LInftyRed$ together (with the same orientation) with $\LZeroRed$ in some small arc. The jump for $\mathbf{Y}$ cancels on this small arc because the jump matrices in \eqref{eq:Yjump-1}--\eqref{eq:Yjump-2} are inverses of each other; thus, up to some relabeling, one has effectively changed the sign in \eqref{eq:increment-argument-red}.  In \cite{BothnerMS18} the choice of sign in \eqref{eq:increment-argument-red} was tied to the sign of $\mathrm{Im}(x)$ due to the derivation of Riemann-Hilbert Problem~\ref{rhp:renormalized} from direct/inverse monodromy theory, however the above surgery argument shows that the sign is in fact arbitrary.  The freedom to choose this sign will be important later when the solution of Riemann-Hilbert Problem~\ref{rhp:renormalized} is constructed for large $n$.
\label{rmk:surgery}
\end{rem}

It turns out that if Riemann-Hilbert Problem~\ref{rhp:renormalized} is solvable for some $x \in\mathbb{C}\setminus\{0\}$, then we may define corresponding matrices $\mathbf{Y}_1^\infty(x)$ and $\mathbf{Y}_0^0(x)$ by expanding $\mathbf{Y}(\lambda)=\mathbf{Y}_n(\lambda;x,m)$ for large and small $\lambda$, respectively:
\begin{equation}
\mathbf{Y}(\lambda)=\mathbb{I}+\mathbf{Y}_1^\infty(x)\lambda^{-1}+\mathcal{O}(\lambda^{-2}),\quad\lambda\to\infty;\quad\mathbf{Y}_1^\infty(x)=[Y_{1,jk}^\infty(x)]_{j,k=1}^2
\end{equation}
and
\begin{equation}
\mathbf{Y}(\lambda)\OurPower{\lambda}{-(m+\tfrac{1}{2})\sigma_3}=\mathbf{Y}_0^0(x)+\mathcal{O}(\lambda),\quad\lambda\to 0;\quad\mathbf{Y}_0^0(x)=[Y_{0,jk}^0(x)]_{j,k=1}^2.
\end{equation}
Then, according to \cite[Theorem 1]{BothnerMS18}, an alternate formula for the rational solution $u_n(x;m)$ of the Painlev\'e-III equation \eqref{eq:PIII} is
\begin{equation}
u_n(x;m)=\frac{-\ii Y^\infty_{1,12}(x)}{Y^0_{0,11}(x)Y^0_{0,12}(x)},
\label{eq:u-n-from-Y-formula}
\end{equation}
where we have suppressed the parametric dependence on $n\in\mathbb{Z}$ and $m\in\mathbb{C}$ on the right-hand side.

\subsection{Results and outline of paper}
\label{sec:results}
A good way to introduce our results is to first explain a simple formal asymptotic calculation.  Since we are interested in solutions $u=u_n(x;m)$ of \eqref{eq:PIII} with parameters written in the form \eqref{eq:Thetas-m-n} when $n$ is large, and since numerical experiments such as those in \cite[Sec.\@ 2]{BothnerMS18} suggest that the largest poles and zeros of $u_n(x;m)$ lie at a distance $|x|$ from the origin proportional to $n$ with a local spacing that neither grows nor shrinks with $n$, it is natural to introduce a complex parameter $y\neq 0$ and a new independent variable $w\in\mathbb{C}$ by setting $x=ny+w$.  It follows that if $u(x)$ solves \eqref{eq:PIII}--\eqref{eq:Thetas-m-n}, then $\lNaught(w):=-\ii u(x=ny+w)$ satisfies
\begin{equation}
\frac{\dd^2\lNaught}{\dd w^2}=\frac{1}{\lNaught}\left(\frac{\dd \lNaught}{\dd w}\right)^2 + \frac{4\ii}{y}(\lNaught^2-1)-4\lNaught^3+\frac{4}{\lNaught}+\mathcal{O}(n^{-1})
\end{equation}
in which the $\mathcal{O}(n^{-1})$ symbol absorbs several terms each of which is explicitly proportional to $n^{-1}\ll 1$.  Dropping these formally small terms leads to an autonomous  second-order equation which is amenable to classical analysis:
\begin{equation}
\frac{\dd^2\dot{\lNaught}}{\dd w^2}=\frac{1}{\dot{\lNaught}}\left(\frac{\dd\dot{\lNaught}}{\dd w}\right)^2+\frac{4\ii}{y}(\dot{\lNaught}^2-1)-4\dot{\lNaught}^3+\frac{4}{\dot{\lNaught}},
\label{eq:autonomous}
\end{equation}
where $\dot{\lNaught}$ denotes a formal approximation to $\lNaught$.  Solutions of the equation\footnote{More properly, it is a family of equations parametrized by $y\in\mathbb{C}\setminus\{0\}$.} \eqref{eq:autonomous} can be classified as follows:
\begin{itemize}
\item 
Equilibrium solutions $\dot{\lNaught}\equiv\text{constant}$.  Generically with respect to $y$ there are four such equilibria:  $\dot{\lNaught}\equiv \pm 1$ and 
\begin{equation}
\dot{\lNaught}\equiv \lNaught_0^\pm(y):=\frac{\ii}{2y}\mp\ii\sqrt{\frac{1}{4y^2}+1},
\end{equation}
where to be precise we take the square roots to be equal to $1$ at $y=\infty$ and to be analytic in $y$ except on a line segment branch cut connecting the branch points $y=\pm\tfrac{1}{2}\ii$ in the $y$ parameter plane.  Note that of these four, the unique equilibrium that tends to $-\ii$ as $y\to\infty$ (as would be consistent with $u=u_n(x;m)\to 1$ as $x\to\infty$) is $\dot{\lNaught}\equiv \lNaught_0^+(y)$.
\item
Non-equilibrium solutions.  These can be obtained by integrating \eqref{eq:autonomous} to find a first integral.  Thus, provided $\dot{\lNaught}(w)$ is non-constant, we may write \eqref{eq:autonomous} in the equivalent form
\begin{equation}
\left(\frac{\dd\dot{\lNaught}}{\dd w}\right)^2 = \frac{16}{y^2}P(\dot{\lNaught};y,C),\quad P(\dot{\lNaught};y,C):=-\frac{y^2}{4}\dot{\lNaught}^4+\frac{\ii y}{2}\dot{\lNaught}^3 + C\dot{\lNaught}^2 + \frac{\ii y}{2}\dot{\lNaught}-\frac{y^2}{4},
\label{eq:dotV-ODE}
\end{equation}
in which $C$ is a constant of integration.  There are two types of non-equilibrium solutions:
\begin{itemize}
\item If $C$ is generic given $y$ such that $P$ has $4$ distinct roots, then all non-constant solutions of \eqref{eq:dotV-ODE} are (doubly-periodic) elliptic functions of $w$ with elliptic modulus depending on $C$ and $y$.
\item If $C=C(y)$ is such that the quartic $P$ has fewer than $4$ distinct roots, then the higher-order roots are necessarily equilibrium solutions of \eqref{eq:autonomous} and all non-constant solutions of \eqref{eq:dotV-ODE} are (singly-periodic) trigonometric functions of $w$.
\end{itemize}
\end{itemize}
Our rigorous analysis of $u_n(x;m)$ in the large-$n$ limit shows that all of the above types of solutions of the approximating equation \eqref{eq:autonomous} play a role. In order to begin to explain our results, first observe that if $x$ is replaced with $ny+w$, then for large $n$, the dominant factors in the off-diagonal elements of the jump matrices in 
Riemann-Hilbert Problem~\ref{rhp:renormalized} are the exponentials $\ee^{\pm nV(\lambda;y)}$, where
\begin{equation}
V(\lambda;y):= -\log(\lambda)-\ii y \varphi(\lambda).
\label{eq:V-define}
\end{equation}
The fact that $V$ is multi-valued is not important because $\ee^{\pm nV(\lambda;y)}$ is single-valued whenever $n\in\mathbb{Z}$.  However,  $\mathrm{Re}(V(\lambda;y))$ is certainly single-valued for $\lambda\in\mathbb{C}\setminus\{0\}$ and $y\in\mathbb{C}$.
For simplicity, in the rest of the paper we write $\lNaught(y):=\lNaught_0^+(y)$.  Since $\lNaught(y)$ is analytic and non-vanishing in its domain of definition, the left-hand side of the equation
\begin{equation}
\mathrm{Re}(V(\lNaught(y);y))=0
\label{eq:critical}
\end{equation}
defines a harmonic function in the complex $y$-plane omitting the vertical branch cut of $\lNaught(y)$ connecting the branch points $\pm\tfrac{1}{2}\ii$.  Therefore, \eqref{eq:critical} determines a curve in the latter domain that turns out to be the union of four analytic arcs:  two rays on the imaginary axis connecting the branch points $y=\pm\tfrac{1}{2}\ii$ to $y=\pm\ii\infty$ respectively, an arc in the right half-plane joining the two branch points, and its image under reflection through the imaginary axis.  The union of the latter two arcs is the boundary of a compact and simply-connected eye-shaped set denoted $E$ containing the origin $y=0$.  The eye $E$ is symmetric with respect to reflection through the origin as well as both the real and imaginary axes.  See Figure~\ref{fig:partialE} below.  Our first result is then the following.
\begin{theorem}[Equilibrium asymptotics of $u_n(x;m)$]
Fix $m\in\mathbb{C}$ and let $K\subset\mathbb{C}\setminus E$ be bounded away from $E$, i.e., $\mathrm{dist}(y,E)>0$.  Then 
\begin{equation}
u_n(ny;m)=\ii\lNaught(y) +\mathcal{O}(n^{-1}),\quad n\to+\infty,\quad y\in K,
\label{eq:u-outside}
\end{equation}
where the error estimate is uniform for $y\in K$.  
\label{theorem:outside}
\end{theorem}
Thus, $u_n$ is approximated by the unique equilibrium solution of \eqref{eq:autonomous} that tends to $-\ii$ as $y\to\infty$, provided that $y$ lies outside the eye $E$.
Since $\lNaught(y)$ is analytic and non-vanishing as a function of $y$ bounded away from $E$, the uniform convergence immediately implies the following.
\begin{corollary}
Fix $m\in\mathbb{C}$ and let $K$ be as in the statement of Theorem~\ref{theorem:outside}.  Then
$u_n(\cdot;m)$ has no zeros or poles
in the set $nK$ for $n$ sufficiently large.
\label{corollary:outside-no-poles-or-zeros}
\end{corollary}
As an application of these results, let $y\in\mathbb{C}\setminus E$ and let $C_y$ denote a positively-oriented loop surrounding the point $y$.  Then, 
from Cauchy's integral formula it follows that, as $n\rightarrow+\infty$,
\begin{equation}
\frac{\dd^ju_n}{\dd x^j}(ny;m)=\frac{1}{n^j}\frac{\dd^ju_n}{\dd y^j}(ny;m)=\frac{j!}{2\pi\ii n^j}\oint_{C_y}\frac{u_n(ny';m)\,\dd y'}{(y'-y)^{j+1}}=\ii n^{-j}\frac{\dd^j\lNaught}{\dd y^j}(y) + \mathcal{O}(n^{-j-1}), \quad j=1,2,3,\dots,
\label{eq:u-prime-outside}
\end{equation}
where to evaluate the integral we used \eqref{eq:u-outside}.  It is easy to see that the error term enjoys similar uniformity properties as in Theorem~\ref{theorem:outside}.\bigskip
   
Next, we let $E_\mathrm{L}$ (resp., $E_\mathrm{R}$) denote the part of the interior of $E$ lying in the open left (resp., right) half-plane, compare again Figure~\ref{fig:partialE}.  We now develop an asymptotic formula for $u_n(x;m)$ when $n^{-1}x\in E_\mathrm{R}$ and $m\in\mathbb{C}\setminus(\mathbb{Z}+\tfrac{1}{2})$.  Since $E_\mathrm{L}$ and $E_\mathrm{R}$ are related by reflection through the origin, by the symmetry \eqref{eq:u-n-exact-symmetry} this formula will also be sufficient to describe $u_n(x;m)$ for large $n$ when $n^{-1}x\in E_\mathrm{L}$, because $-m\in\mathbb{C}\setminus(\mathbb{Z}+\tfrac{1}{2})$ whenever $m\in\mathbb{C}\setminus(\mathbb{Z}+\tfrac{1}{2})$.  Given $m\in\mathbb{C}\setminus (\mathbb{Z}+\tfrac{1}{2})$, in \eqref{eq:four-factors}--\eqref{eq:N-of-y} below we define complex-valued functions $\mathcal{Z}_n^\bullet(y,w;m)$, $\mathcal{Z}_n^\circ (y,w;m)$, $\mathcal{P}_n^\bullet(y,w;m)$, $\mathcal{P}_n^\circ(y,w;m)$, and $N(y)$,  whose real and imaginary parts are smooth but non-analytic functions of the real and imaginary parts of $y\in E_\mathrm{R}$ and which are entire functions of $w\in\mathbb{C}$, with $N:E_\mathrm{R}\to\mathbb{C}$ non-vanishing.  These functions depend crucially on a smooth but non-analytic function $C=C(y)$ defined on $E_\mathrm{R}$ by a procedure described in Sections~\ref{sec:Boutroux-y-small} and \ref{sec:Boutroux-degenerate}, and also on a related smooth function $B=B(y)$ with $\mathrm{Re}(B(y))<0$ defined by \eqref{eq:B-of-y}. In detail, compare \eqref{eq:four-factors},
\begin{equation}
\begin{split}
\mathcal{Z}_n^\bullet(y,w;m)&=\Theta(\mathscr{z}_n^\bullet(y,w;m)-\ii\pi-\tfrac{1}{2}B(y),B(y)),\ \ \ \ \ 
\mathcal{Z}_n^\circ(y,w;m)=\Theta(\mathscr{z}_n^\circ(y,w;m)+\ii\pi+\tfrac{1}{2}B(y),B(y)),\\
\mathcal{P}_n^\bullet(y,w;m)&=\Theta(\mathscr{p}_n^\bullet(y,w;m) +\ii\pi +\tfrac{1}{2}B(y),B(y)),\ \ \ \ \ 
\mathcal{P}_n^\circ(y,w;m)=\Theta(\mathscr{p}_n^\circ(y,w;m)-\ii\pi -\tfrac{1}{2}B(y),B(y)),
\end{split}
\end{equation}
in which $\Theta(z,B)$ denotes the Riemann theta function defined by \eqref{eq:Riemann-Theta},
and in which the complex-valued phases $\mathscr{z}_n^\bullet$, $\mathscr{z}_n^\circ$, $\mathscr{p}_n^\bullet$, and $\mathscr{p}_n^\circ$ are well-defined affine linear functions of $w\in\mathbb{C}$ and $n\in\mathbb{Z}_{\ge 0}$ with coefficients and constant terms that are smooth functions of $y\in\mathbb{R}$ depending parametrically on $m\in\mathbb{C}\setminus(\mathbb{Z}+\tfrac{1}{2})$.
We then define
\begin{equation}
\dot{u}_n(y,w;m):=N(y)\frac{\mathcal{Z}_n^\bullet(y,w;m)\mathcal{Z}_n^\circ(y,w;m)}{\mathcal{P}_n^\bullet(y,w;m)\mathcal{P}_n^\circ(y,w;m)},\quad y\in E_\mathrm{R},\quad w\in\mathbb{C},\quad m\in\mathbb{C}\setminus(\mathbb{Z}+\tfrac{1}{2}),
\label{eq:udot-elliptic}
\end{equation}
excluding isolated exceptional values of $(y,w)\in E_\mathrm{R}\times\mathbb{C}$ for which the denominator vanishes.
\begin{theorem}[Elliptic asymptotics of $u_n(x;m)$]
Fix $m\in\mathbb{C}\setminus(\mathbb{Z}+\tfrac{1}{2})$.  For each $n\in\mathbb{Z}_{\ge 0}$ and each $y\in E_\mathrm{R}$, the function $\dot{\lNaught}(w):=-\ii\dot{u}_n(y,w;m)$ is a non-equilibrium elliptic function solution of \eqref{eq:autonomous} in the form \eqref{eq:dotV-ODE} with integration constant $C=C(y)$.
If $\epsilon>0$ is an arbitrarily small fixed number and $K_y\subset E_\mathrm{R}$ and $K_w\subset\mathbb{C}$ are compact sets, then
\begin{equation}
u_n(ny+w;m)=\dot{u}_n(y,w;m)+\mathcal{O}(n^{-1}),\quad n\to+\infty,
\label{eq:elliptic-asymptotics-ER}
\end{equation}
holds uniformly on the set of $(y,w,n)$ defined by the conditions $y\in K_y$, $w\in K_w$ such that
\begin{equation}
\begin{split}
\mathrm{dist}(\mathscr{z}_n^\bullet(y,w;m),2\pi\ii\mathbb{Z}+B(y)\mathbb{Z})&\ge\epsilon,\ \ \ \ \ \ \ 
\mathrm{dist}(\mathscr{z}_n^\circ(y,w;m),2\pi\ii\mathbb{Z}+B(y)\mathbb{Z})\ge\epsilon,\\
\mathrm{dist}(\mathscr{p}_n^\bullet(y,w;m),2\pi\ii\mathbb{Z}+B(y)\mathbb{Z})&\ge\epsilon,\ \ \ \ \ \ \ 
\mathrm{dist}(\mathscr{p}_n^\circ(y,w;m),2\pi\ii\mathbb{Z}+B(y)\mathbb{Z})\ge\epsilon.
\end{split}
\label{eq:cheese-carve-outs}
\end{equation}
Under the same conditions and with the same sense of convergence,
\begin{equation}
u_n(-(ny+w);-m)=\frac{1}{\dot{u}_n(y,w;m)}+\mathcal{O}(n^{-1}),\quad n\to+\infty,
\label{eq:elliptic-asymptotics-EL}
\end{equation}
which provides asymptotics of $u_n(ny;m)$ when $y\in E_\mathrm{L}$.
\label{theorem:eye}
\end{theorem}
The formula \eqref{eq:elliptic-asymptotics-EL} follows from \eqref{eq:elliptic-asymptotics-ER} with the use of the symmetry \eqref{eq:u-n-exact-symmetry} (and that $\dot{u}_n(y,w;m)$ is bounded and bounded away from zero on the indicated set, as it happens).
Thus, provided that $n^{-1}x$ lies in either domain $E_\mathrm{L}$ or $E_\mathrm{R}$ and $m$ is \emph{not} a half-integer, the rational Painlev\'e-III function $u_n(x;m)$ is locally approximated by a non-equilibrium elliptic function solution of the differential equation \eqref{eq:autonomous}.  Note that the fact that the leading term on the right-hand side of \eqref{eq:elliptic-asymptotics-EL} is an elliptic function follows from the first statement of Theorem~\ref{theorem:eye} and the fact that the integrated form \eqref{eq:dotV-ODE} admits the symmetry $(\dot{p},C,y,w)\mapsto (-\dot{p}^{-1},C,-y,-w)$.\smallskip

\begin{rem}  
The fact that in \eqref{eq:elliptic-asymptotics-ER} and \eqref{eq:elliptic-asymptotics-EL} we are approximating a function of a single complex variable $x=ny+w$ with a function of two independent complex variables $(y,w)$ deserves some explanation.  Indeed, given $x$ there are many different choices of parameters $(y,w)$ for which $x=ny+w$, so the form of $\dot{u}_n(y,w;m)$ actually gives a family of approximations for the same quantity.  The variable $w$ captures the local properties of the rational function $u_n(x;m)$; it is the scale on which $u_n(x;m)$ resembles a fixed elliptic function.  On the other hand the variable $y$ captures the way that the elliptic modulus depends on the point of observation within the eye $E$ and unlike the meromorphic dependence on $w$, $\dot{u}_n(y,w;m)$ is a decidedly non-analytic function of $y$.  If we approximate $u_n(x;m)$ by setting $w=0$ and letting $y$ vary, we obtain a globally accurate (on $K_y$) approximation that is unfortunately not analytic in $y$.  However if we fix $y\in E_\mathrm{R}$ and let $w$ vary, we obtain a locally accurate ($w\in K_w$, so $x-ny=w=\mathcal{O}(1)$ as $n\to+\infty$) approximation that is an exact elliptic function depending only parametrically on $y$.  
\label{remark:two-vars}
\end{rem}

If in any of the conditions \eqref{eq:cheese-carve-outs} we put $\epsilon=0$, then the corresponding phase agrees with a point of the lattice $2\pi\ii\mathbb{Z}+B(y)\mathbb{Z}$ and the associated factor in the definition of $\dot{u}_n(y,w;m)$ vanishes.  For $\epsilon>0$, each condition in \eqref{eq:cheese-carve-outs} defines a ``swiss-cheese''-like region in the variables $(y,w)$ given $n\in\mathbb{Z}_{\ge 0}$ and $m\in\mathbb{C}\setminus(\mathbb{Z}+\tfrac{1}{2})$ with holes centered at points corresponding to lattice points.  In fact, if $y\in E_\mathrm{R}$ is also fixed, then the lattice $2\pi\ii\mathbb{Z}+B(y)\mathbb{Z}$ is a uniform lattice and each of the conditions in \eqref{eq:cheese-carve-outs} omits from the complex $w$-plane the union of disks of radius $\epsilon$ centered at the lattice points.  On the other hand, if instead it is $w\in\mathbb{C}$ that is fixed, then each of the conditions \eqref{eq:cheese-carve-outs} omits from the complex $y$-plane neighborhoods of diameter proportional to $\epsilon n^{-1}$ containing the points in a set that can be roughly characterized as a curvilinear grid of spacing proportional to $n^{-1}$.  

\begin{corollary}
Fix $m\in\mathbb{C}\setminus(\mathbb{Z}+\tfrac{1}{2})$ and a compact set $K_y\subset E_\mathrm{R}$.  If $\{y_n\}_{n=N}^\infty\subset K_y$ is a sequence such that $\mathcal{Z}_n^\bullet(y_n,0;m)=0$ for $n=N, N+1,\dots$ (or such that $\mathcal{Z}_n^\circ(y_n,0;m)=0$ for $n=N,N+1,\dots$), then for each sufficiently small $\epsilon>0$ there is exactly one simple zero, and possibly a group of an equal number of additional zeros and poles, of $u_n(ny;m)$ within $|y-y_n|<\epsilon n^{-1}$ for $n$ sufficiently large.  Likewise, if $\{y_n\}_{n=N}^\infty\subset K_y$ is a sequence such that $\mathcal{P}_n^\bullet(y_n,0;m)=0$ for $n=N, N+1,\dots$ (or such that $\mathcal{P}_n^\circ(y_n,0;m)=0$ for $n=N,N+1,\dots$), then for each sufficiently small $\epsilon>0$ there is exactly one simple pole, and possibly a group of an equal number of additional zeros and poles, of $u_n(ny;m)$ within $|y-y_n|<\epsilon n^{-1}$ for $n$ sufficiently large.
\label{corollary:eye-zeros-and-poles:better}
\end{corollary}
The proof of this result depends on Theorem~\ref{theorem:eye} and some additional technical properties of the zeros of the factors in the formula \eqref{eq:udot-elliptic} and will be given in Section~\ref{sec:properties-of-udot-elliptic}.  The proof is based on an index argument, which computes the net number of zeros over poles within a small disk.  For this reason, we cannot rule out the possible attraction of one or more pole-zero pairs of the rational function $u_n(x;m)$, in excess of a simple zero (or pole), toward a given zero (or singularity) of the approximating function.  However, we do not observe any such ``excess pairing'' in practice.  One approach to ruling out any excess pairing would be to compare against  precise counts of the zeros and poles of $u_n(x;m)$ as documented in \cite{ClarksonLL16}.  However, such a comparison would require accurate approximations in domains that completely cover the eye $E$ without overlaps.  In this paper we avoid analyzing $u_n(x;m)$ near the origin, the corners $y=\pm\tfrac{1}{2}\ii$, and the ``eyebrows'' (except in the special case $m\in\mathbb{Z}+\tfrac{1}{2}$; see below).  These are projects for the future.  Although for these reasons there remains some ambiguity about the distribution of poles and zeros of the rational function $u_n(x;m)$, our analysis gives very detailed information about the distribution of singularities and zeros of the approximation $\dot{u}_n(y,w;m)$.  In particular, we have the following.
\begin{theorem}
Let $m\in\mathbb{C}\setminus(\mathbb{Z}+\tfrac{1}{2})$.  There is a continuous function $\rho:E_\mathrm{R}\to\mathbb{R}_+$, $\rho\in L^1_\mathrm{loc}(E_\mathrm{R})$, such that for any compact set $K\subset E_\mathrm{R}$, 
\begin{equation}
\lim_{n\to+\infty}\frac{1}{n^2}\#\{y\in K,\;\dot{u}_n(y,0;m)=0\}=\lim_{n\to+\infty}\frac{1}{n^2}\#\{y\in K,\;\dot{u}_n(y,0;m)=\infty\}=\int_K\rho(y)\,\dd A(y),
\label{eq:density-integral}
\end{equation}
where $\dd A(y)$ denotes Lebesgue measure in the $y$-plane.  The density $\rho$ is independent of $m\in\mathbb{C}\setminus(\mathbb{Z}+\tfrac{1}{2})$ and satisfies $\rho(y)\to 0$ as $y\to\partial E_\mathrm{R}\setminus\{0\}$ and $\rho(r\ee^{\ii\theta})=h(\theta)r^{-1}+o(r^{-1})$ as $r\downarrow 0$ for some function $h:(-\pi/2,\pi/2)\to\mathbb{R}_+$.
\label{theorem:density}
\end{theorem}
We would expect that the same statement holds with $\dot{u}_n(y,0;m)$ replaced by $u_n(ny;m)$, but this would require ruling out the excess pairing phenomenon mentioned above.  The density function $\rho(y)$ is defined in \eqref{eq:zero-pole-density} below, and the proof of Theorem~\ref{theorem:density} is given in Section~\ref{sec:properties-of-udot-elliptic}.  Although the proof of Theorem~\ref{theorem:density} does not allow us to consider sets $K$ that depend on $n$ in any serious way, the assumtion that  \eqref{eq:density-integral} holds when $K$ is the disk of radius $n^{-2}$ centered at the origin leads to the prediction that this disk contains $\mathcal{O}(1)$ zeros/singularities of $\dot{u}_n(y,0;m)$ consistent with the empirical observation that the smallest zeros and poles of $u_n(x;m)$ scale like $n^{-1}$ in the $x$-plane \cite{BothnerMS18}.\bigskip 

While the asymptotic approximations of the rational Painlev\'e-III function $u_n(x;m)$ for $n^{-1}x\in E_\mathrm{L}\cup E_\mathrm{R}$ are much more complicated than the simple formula $\ii\lNaught(n^{-1}x)$ valid for $n^{-1}x\in\mathbb{C}\setminus E$, they are easily implemented numerically, once the necessary ingredients developed as part of the proof of Theorem~\ref{theorem:eye} are incorporated.  To quantitatively illustrate the accuracy of the approximations described in Theorems~\ref{theorem:outside} and \ref{theorem:eye}, we compare $u_n(x;m)$ with its approximations for $x$ restricted to a real interval that bisects $E$ in Figures~\ref{fig:RealPlots-m0}--\ref{fig:RealPlots-miOver5}.
\begin{figure}[h]
\begin{center}
\includegraphics[width=0.49\textwidth]{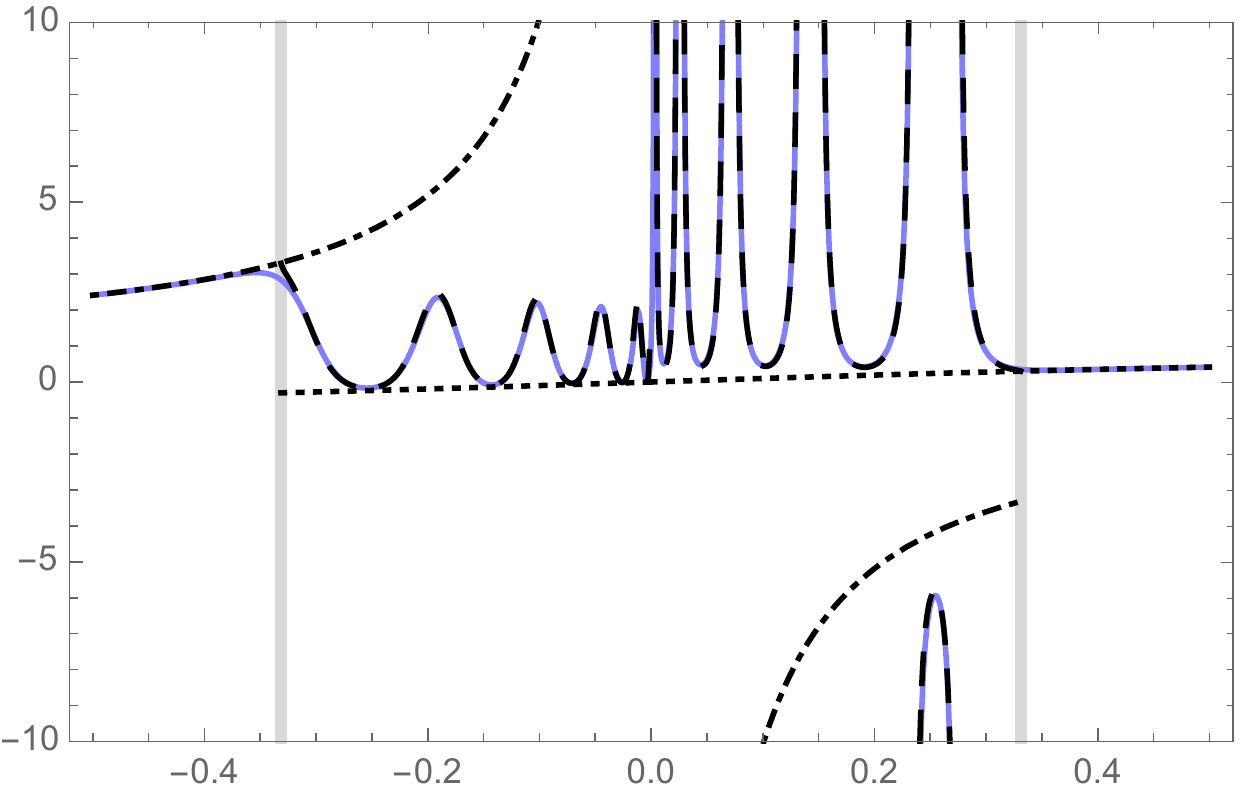}\hfill%
\includegraphics[width=0.49\textwidth]{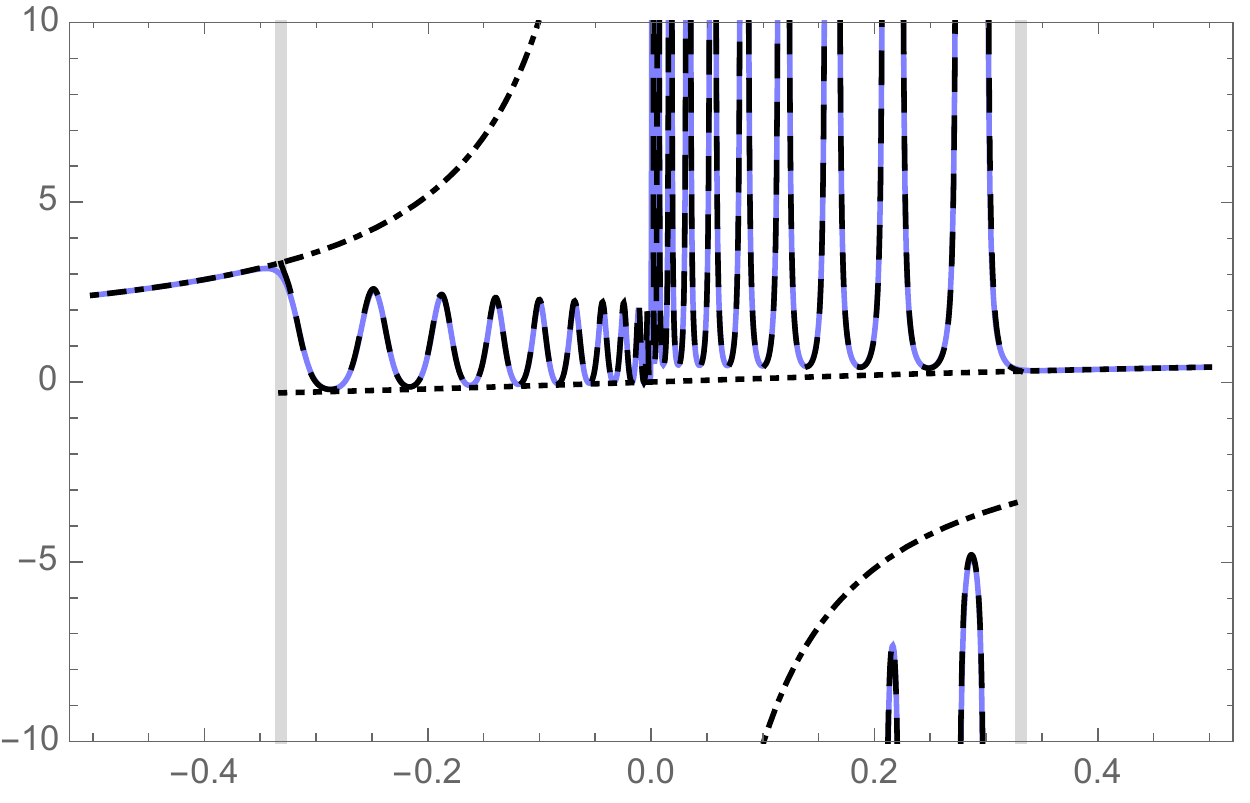}
\end{center}
\caption{Comparison of $u_n(ny;m)$ (blue curve) with its approximations over the interval $-0.5<y<0.5$ for $m=0$ with $n=10$ (left) and $n=20$ (right).  The points where this interval intersects $\partial E$ are shown with vertical gray lines.  The approximation $\dot{u}_n(y,0;m)$ of Theorem~\ref{theorem:eye} is plotted in between the gray lines with black broken curves.  The dotted curve is the analytic continuation into $E$ from the right of the outer approximation $\ii\lNaught(y)$ described in Theorem~\ref{theorem:outside}.  Likewise, the dash/dotted curve is the meromorphic continuation into $E$ from the left of the same outer approximation.}
\label{fig:RealPlots-m0}
\end{figure}
\begin{figure}[h]
\begin{center}
\includegraphics[width=0.49\textwidth]{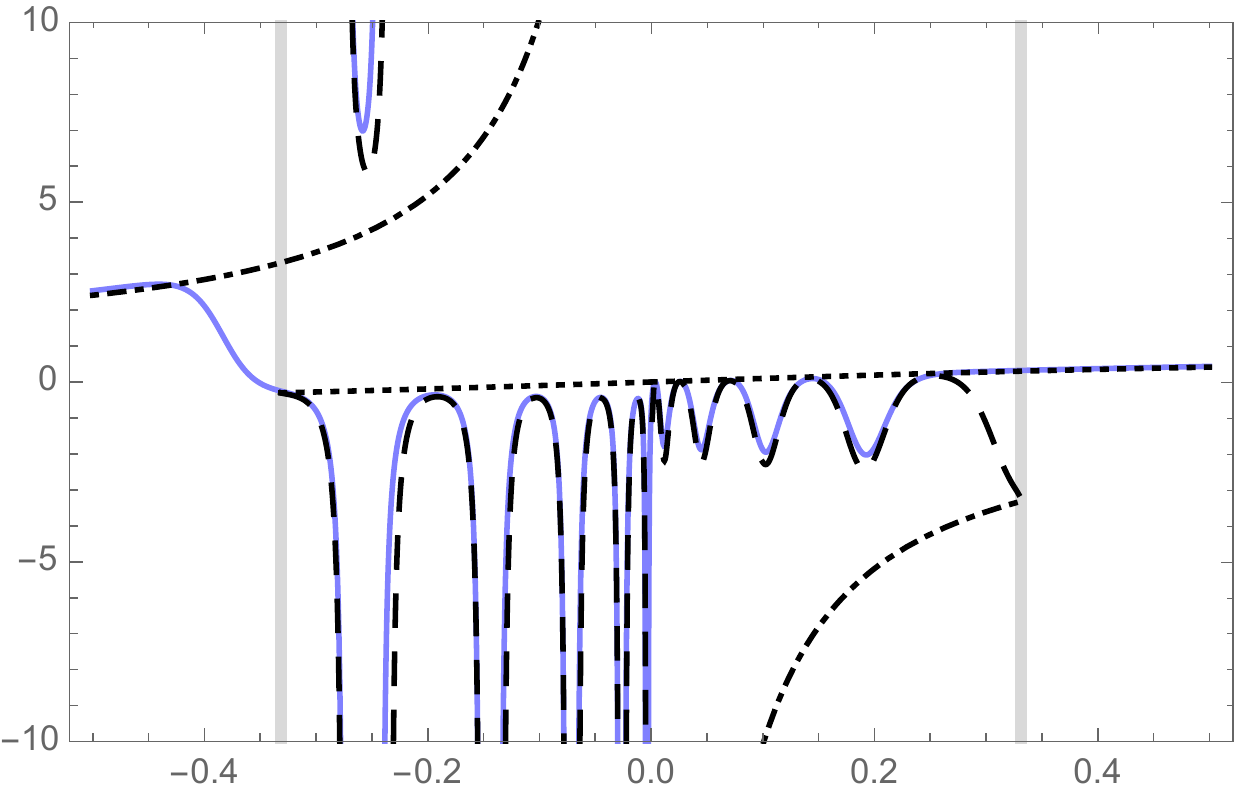}\hfill%
\includegraphics[width=0.49\textwidth]{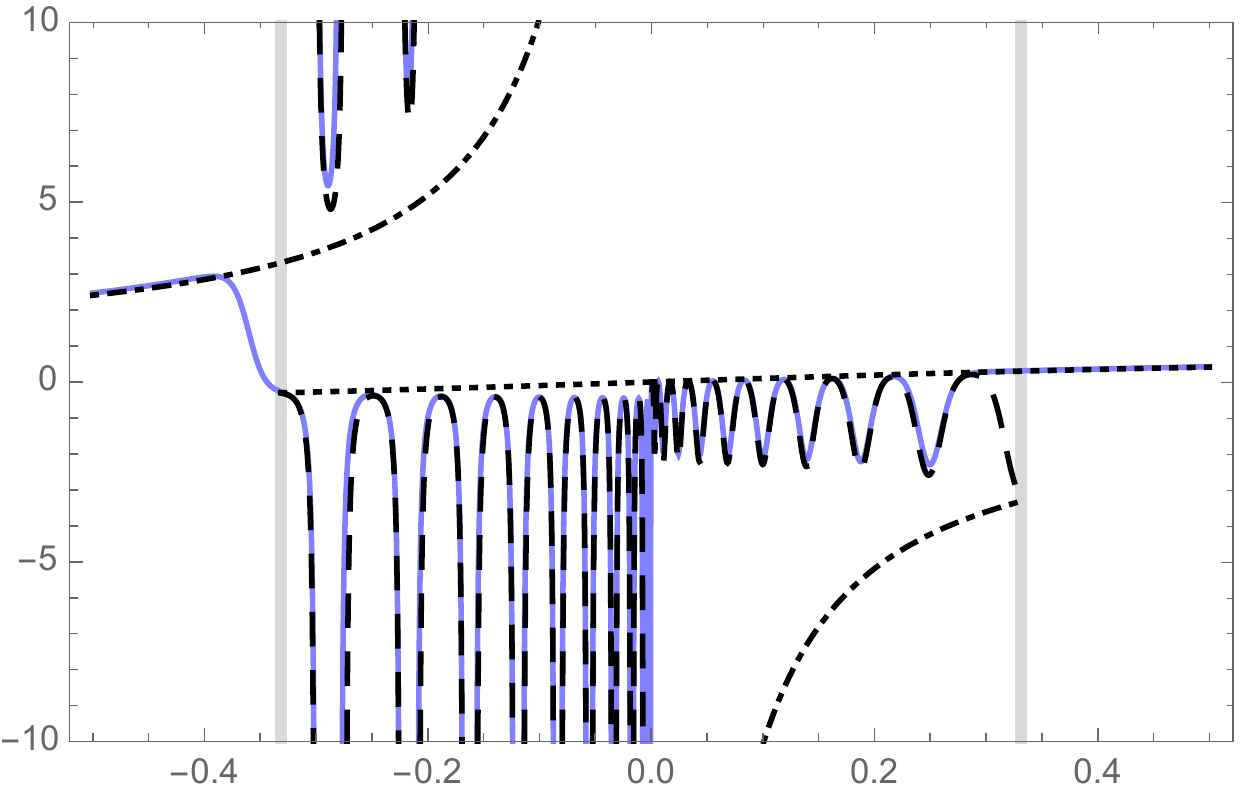}
\end{center}
\caption{As in Figure~\ref{fig:RealPlots-m0} but for $m=1$.}
\label{fig:RealPlots-m1}
\end{figure}
\begin{figure}[h]
\begin{center}
\includegraphics[width=0.49\textwidth]{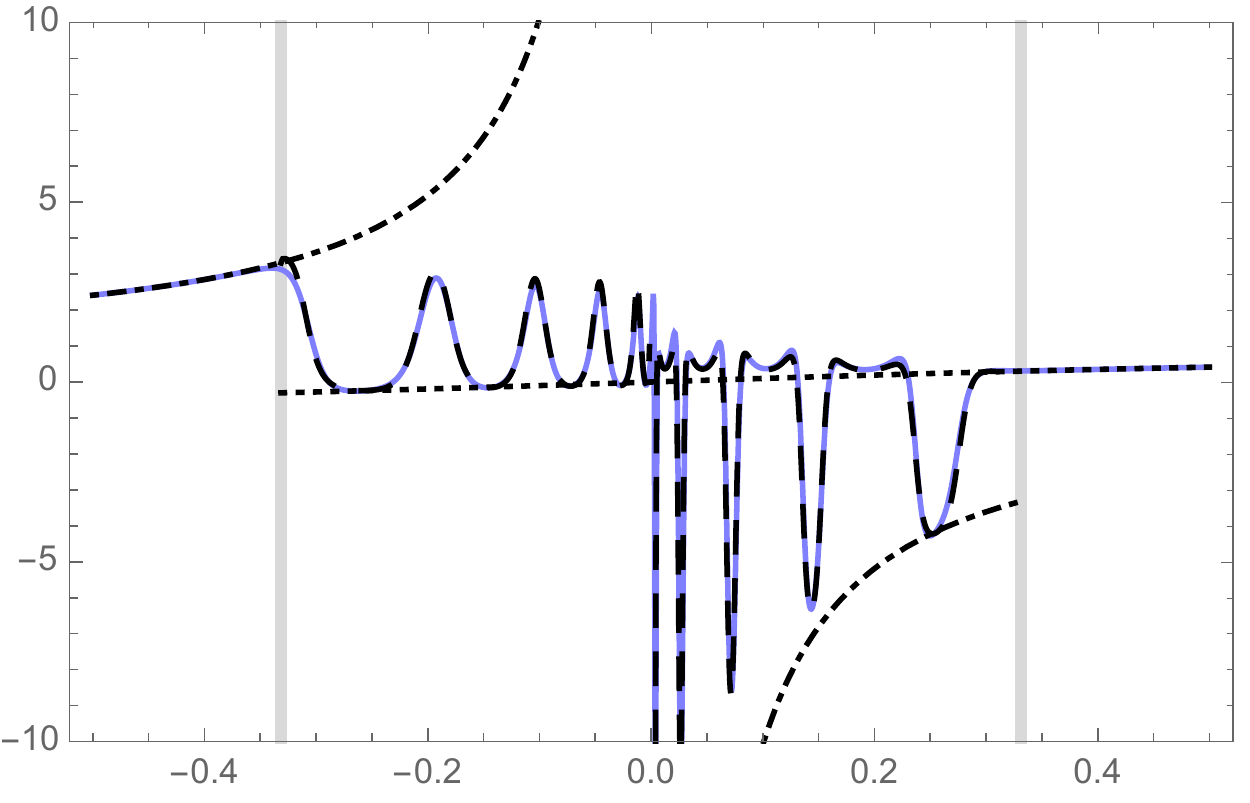}\hfill%
\includegraphics[width=0.49\textwidth]{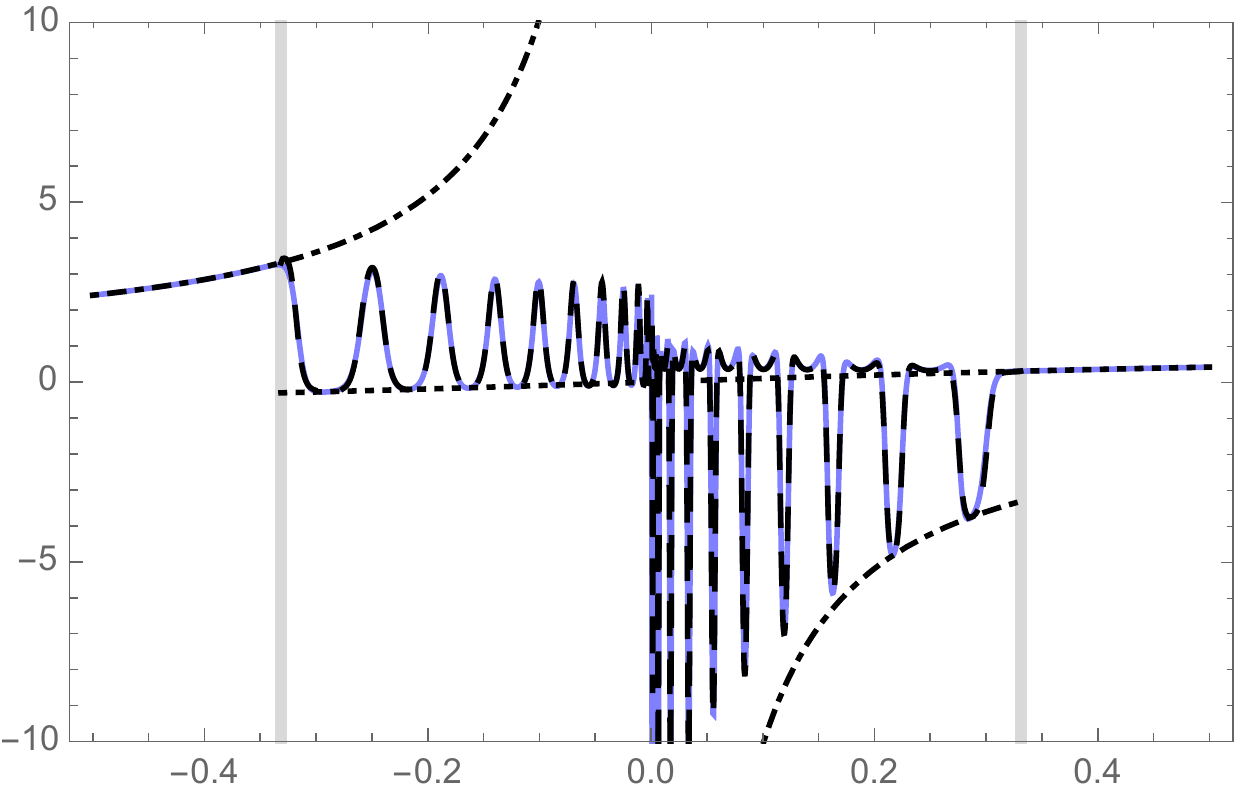}\\
\includegraphics[width=0.49\textwidth]{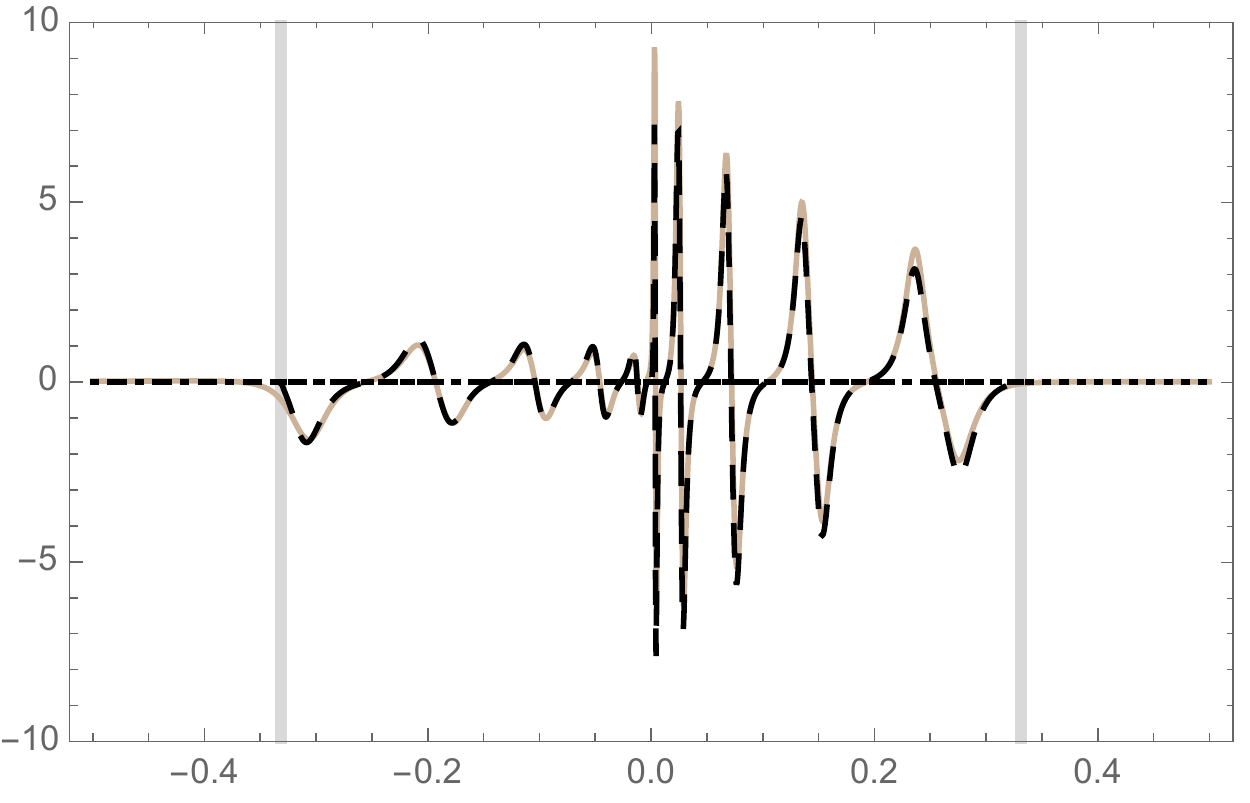}\hfill%
\includegraphics[width=0.49\textwidth]{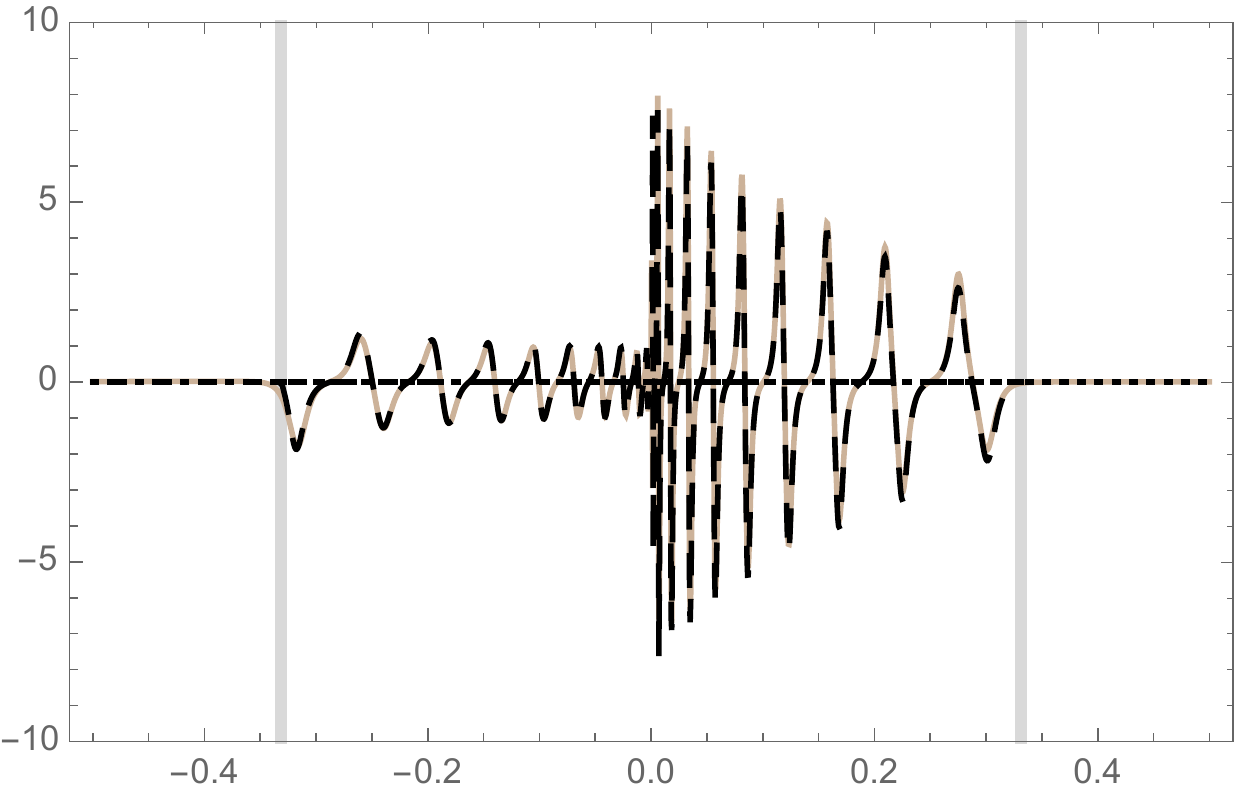}
\end{center}
\caption{As in Figures~\ref{fig:RealPlots-m0}--\ref{fig:RealPlots-m1} but for $m=\tfrac{1}{5}\ii$.  Here the top row compares the real parts and the bottom row compares the imaginary parts (the graph of $\mathrm{Im}(u_n(ny;m))$ is shown with a brown curve).}
\label{fig:RealPlots-miOver5}
\end{figure}
In these figures, we found it compelling to plot the approximate formula $\ii\lNaught(y)$ of Theorem~\ref{theorem:outside} continued into the eye $E$ from the left and right, even though we have no basis for comparing the graphs of these (reciprocal) continuations with that of $u_n(ny;m)$ when $y\in E$.  Indeed, in some situations these graphs appear to form quite accurate upper or lower envelopes of the wild modulated elliptic oscillations of $u_n(ny;m)$ that occur when $y\in E$ and that are captured with locally uniform accuracy by $\dot{u}_n(y,0;m)$.  We have no explanation for these somewhat imprecise observations, but we find them interesting and note that similar phenomena occur for the rational solutions of the Painlev\'e-II equation (also without explanation) as was noted in \cite{BuckinghamM14}.\bigskip

Now, we go into the complex $y$-plane where we can illustrate both the shape of the eye $E$ and the phenomenon of attraction of poles and zeros of $u_n(ny;m)$ to the left ($E_\mathrm{L}$) and right ($E_\mathrm{R}$) halves.  In these figures, the zeros and poles of the rational Painlev\'e function $u_n(ny;m)$ are plotted with the following convention (as in our earlier paper \cite{BothnerMS18}):
\begin{itemize}
\item Zeros of $u_n(x;m)$ that are also zeros of $s_n(x;m-1)$:  blue filled dots.
\item Zeros of $u_n(x;m)$ that are also zeros of $s_{n-1}(x;m)$:  blue unfilled dots.
\item Poles of $u_n(x;m)$ that are also zeros of $s_n(x;m)$:  red filled dots.
\item Poles of $u_n(x;m)$ that are also zeros of $s_{n-1}(x;m-1)$:  red unfilled dots.
\end{itemize}
In addition to displaying the overall attraction of the poles and zeros to the eye domain $E$, the plots in Figures~\ref{fig:m0-ZerosPlus}--\ref{fig:m1Over4-PolesMinus} are also intended to demonstrate the remarkable accuracy of the approximation of Theorem~\ref{theorem:eye} in capturing the locations of individual poles and zeros as described in Corollary~\ref{corollary:eye-zeros-and-poles:better}.  As described in Section~\ref{sec:properties-of-udot-elliptic} below, each of the four factors in the fraction on the right-hand side of \eqref{eq:udot-elliptic} has zeros that may be characterized as the intersection points of integral level curves of two different functions (see \eqref{eq:Zero-Quantum} and \eqref{eq:Pole-Quantum} below) defined on $E_\mathrm{R}$ (and via the symmetry \eqref{eq:u-n-exact-symmetry}, $E_\mathrm{L}$).  We plot the families of level curves for each of the four factors in separate figures in order to demonstrate another phenomenon that is evident but for which we have no good explanation:  the zeros of the separate factors in the approximation $\dot{u}_n(y,0;m)$ as defined by \eqref{eq:udot-elliptic} appear to correspond precisely to the actual zeros of the four polynomial factors in the formula \eqref{eq:un-exact-fraction} for the rational Painlev\'e-III function $u_n(ny;m)$.  This coincidence is what motivates the superscript notation ($\bullet$ versus $\circ$) on the four factors in \eqref{eq:udot-elliptic}; the zeros of the factors with superscript $\bullet$ (resp., $\circ$) apparently correspond in the limit $n\to+\infty$ to filled (resp., unfilled) dots.%
\begin{figure}[h]
\begin{center}
\includegraphics[width=0.3\textwidth]{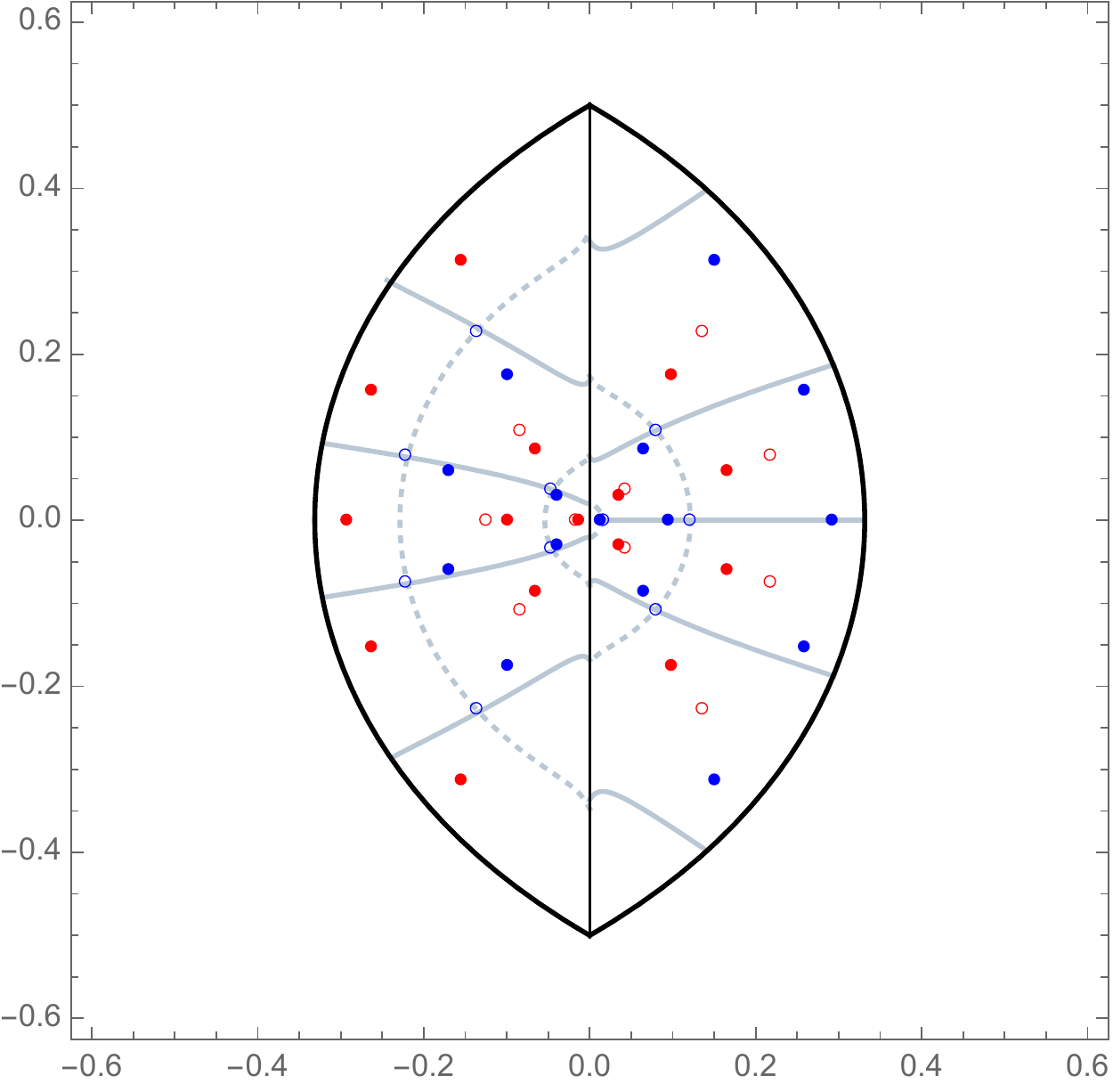}\hfill%
\includegraphics[width=0.3\textwidth]{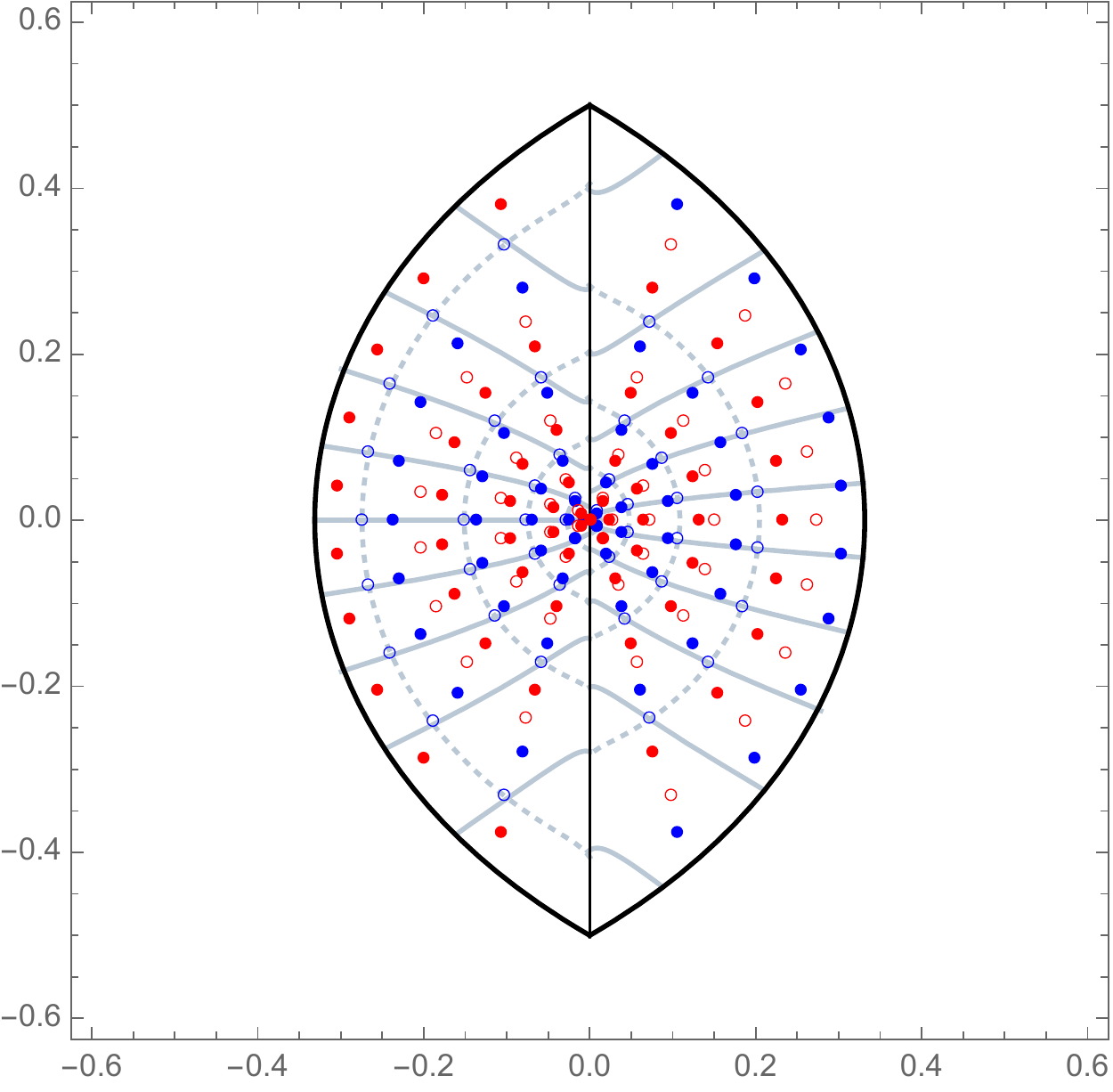}\hfill%
\includegraphics[width=0.3\textwidth]{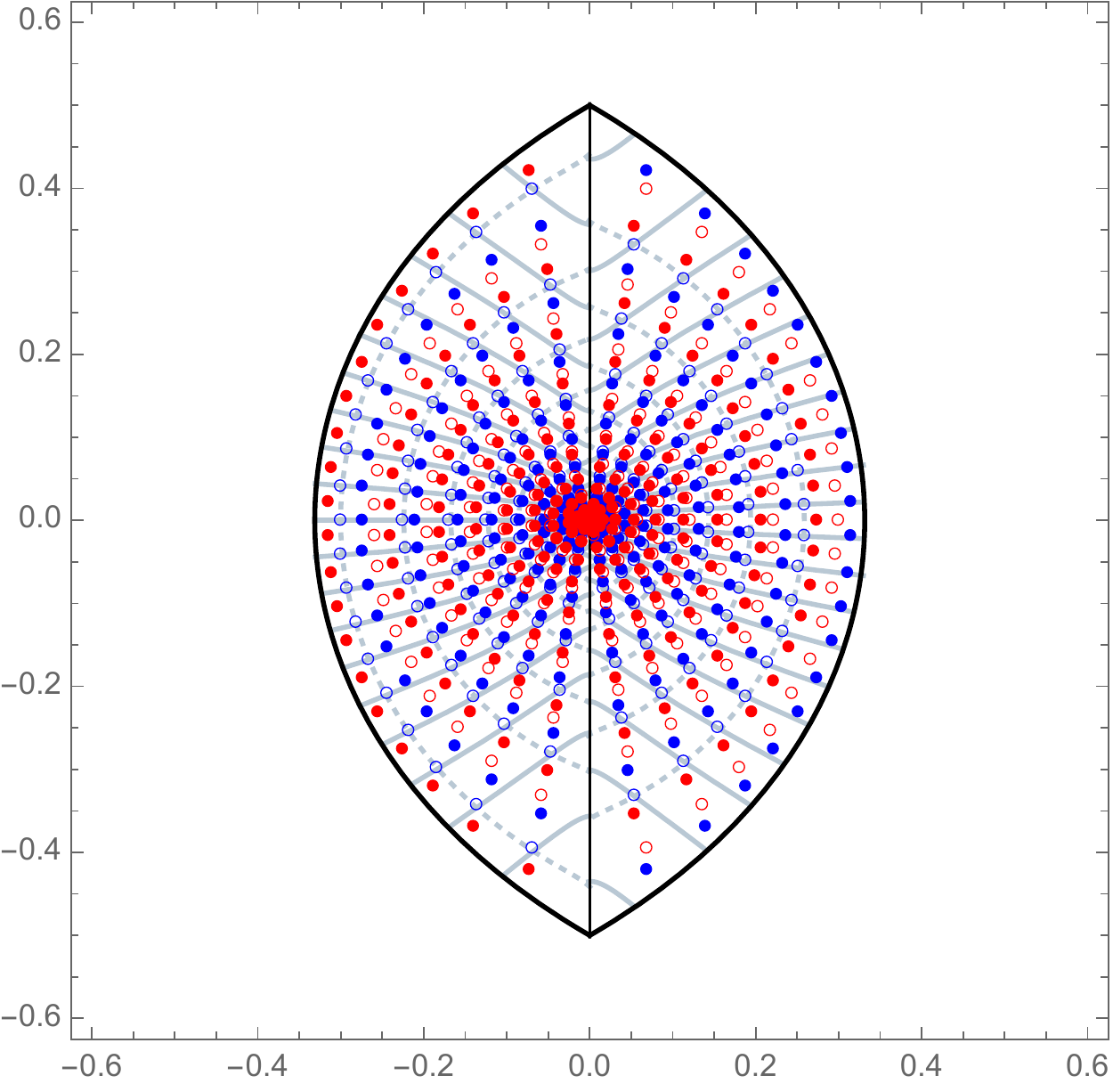}%
\end{center}
\caption{The black curves including the vertical segment form the boundary of the left ($E_\mathrm{L}$) and right ($E_\mathrm{R}$) halves of the eye $E$.  The light blue curves are $\alpha_n^{0,+}(y,0,m)\in\mathbb{Z}$ (solid) and $\beta_n^{0,+}(y,0,m)\in\mathbb{Z}$ (dotted) plotted in the $y$-plane; see \eqref{eq:Zero-Quantum} for definitions of these functions.  These plots are for $m=0$ and $n=5$ (left), $n=10$ (center), and $n=20$ (right).  The blue/red dots are the actual zeros/poles of $u_n(ny;m)$ (filled for zeros of $s_n$ and unfilled for zeros of $s_{n-1}$).  Note how the unfilled blue dots are attracted toward the intersections of the curves, which are the zeros of $\dot{u}_n(y,0;m)$ arising from roots of $\mathcal{Z}_n^\circ(y,0;m)$.}
\label{fig:m0-ZerosPlus}
\end{figure}
\begin{figure}[h]
\begin{center}
\includegraphics[width=0.3\textwidth]{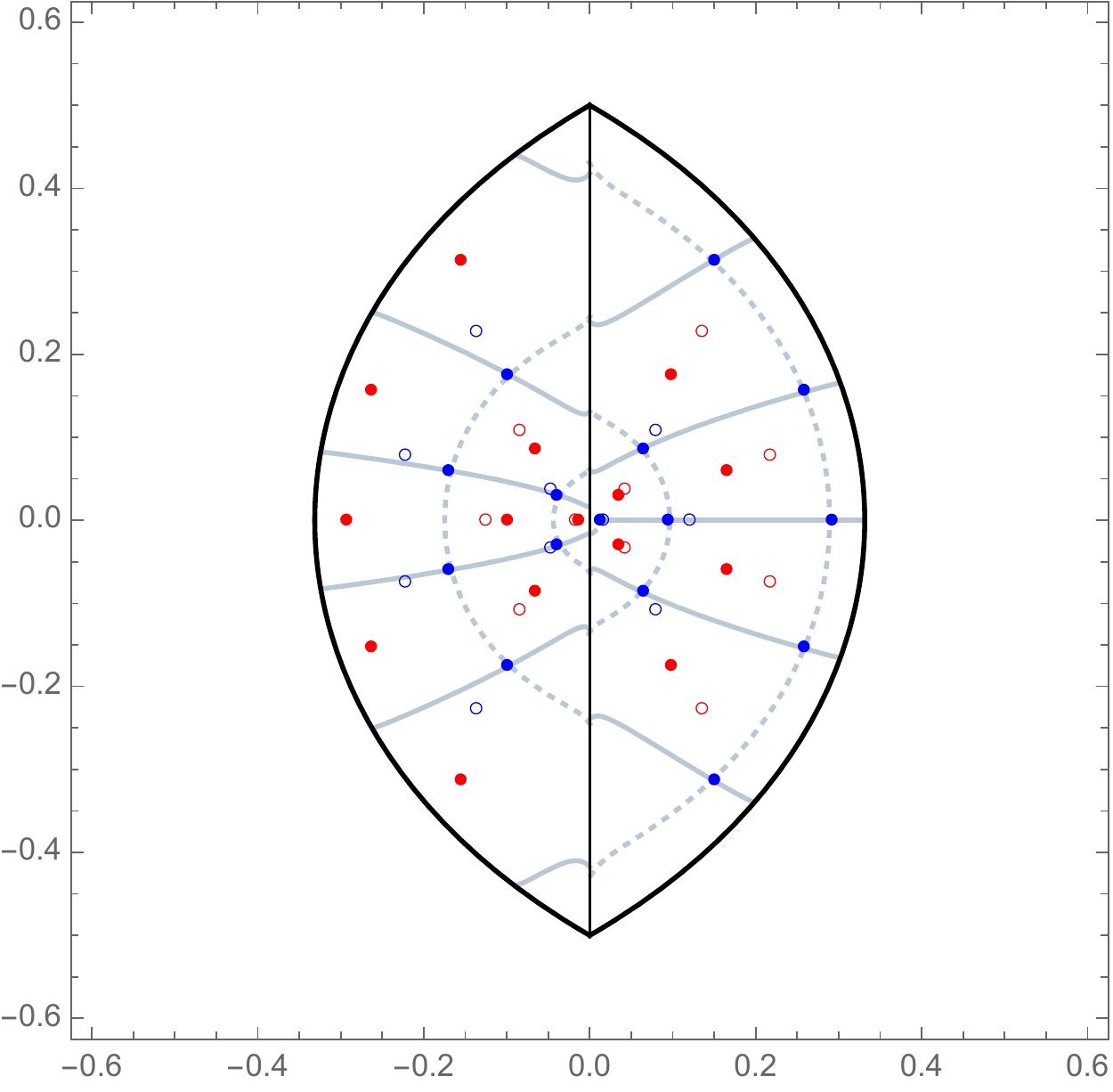}\hfill%
\includegraphics[width=0.3\textwidth]{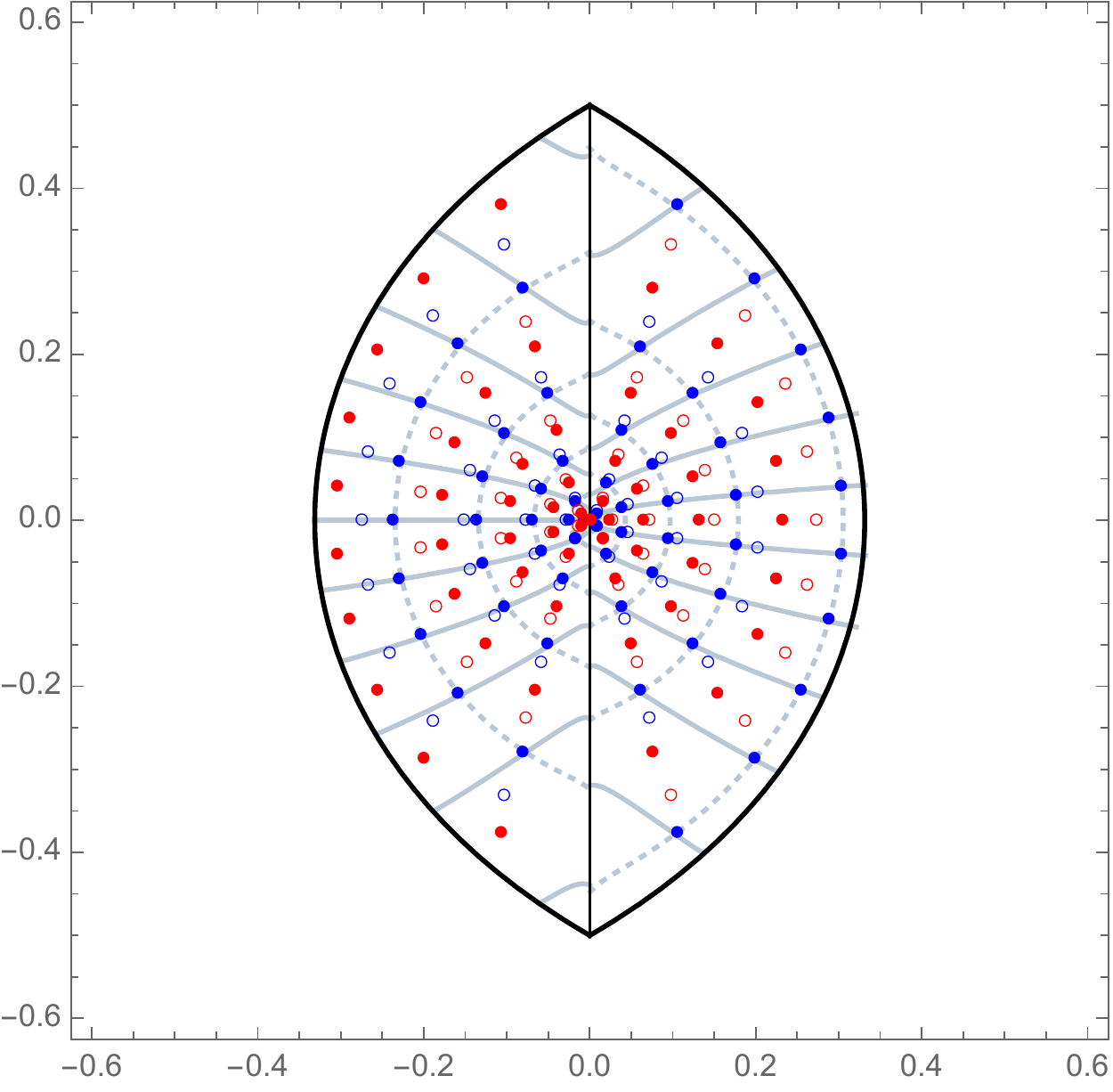}\hfill%
\includegraphics[width=0.3\textwidth]{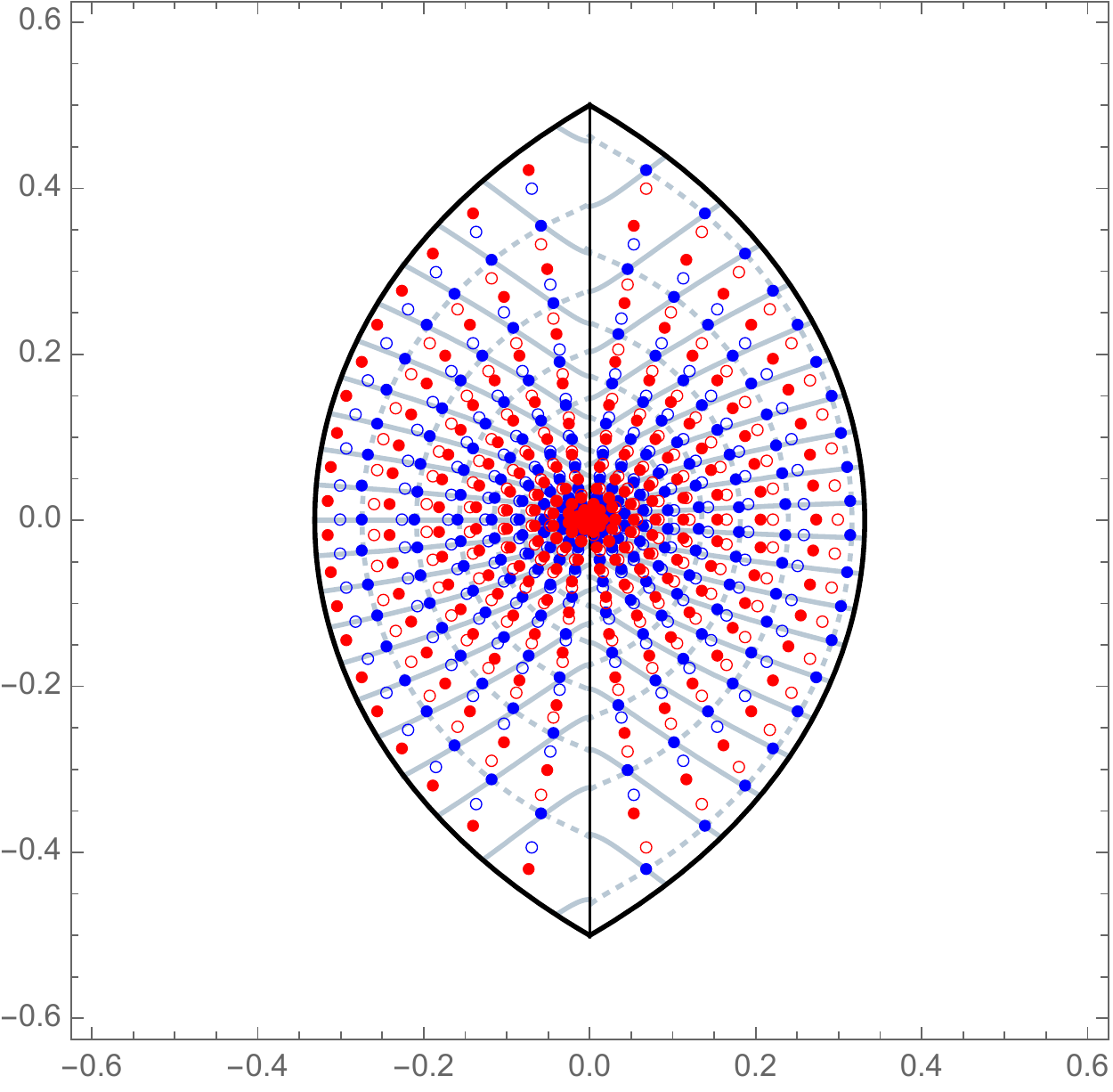}%
\end{center}
\caption{As in Figure~\ref{fig:m0-ZerosPlus} but here the light blue curves are $\alpha_n^{0,-}(y,0,m)\in\mathbb{Z}$ (solid) and $\beta_n^{0,-}(y,0,m)\in\mathbb{Z}$ (dotted); see \eqref{eq:Zero-Quantum} for definitions of these functions.  Note how the filled blue dots are attracted toward the intersections of the curves, which are now the zeros of $\dot{u}_n(y,0;m)$ arising from roots of $\mathcal{Z}_n^\bullet(y,0;m)$.}
\label{fig:m0-ZerosMinus}
\end{figure}
\begin{figure}[h]
\begin{center}
\includegraphics[width=0.3\textwidth]{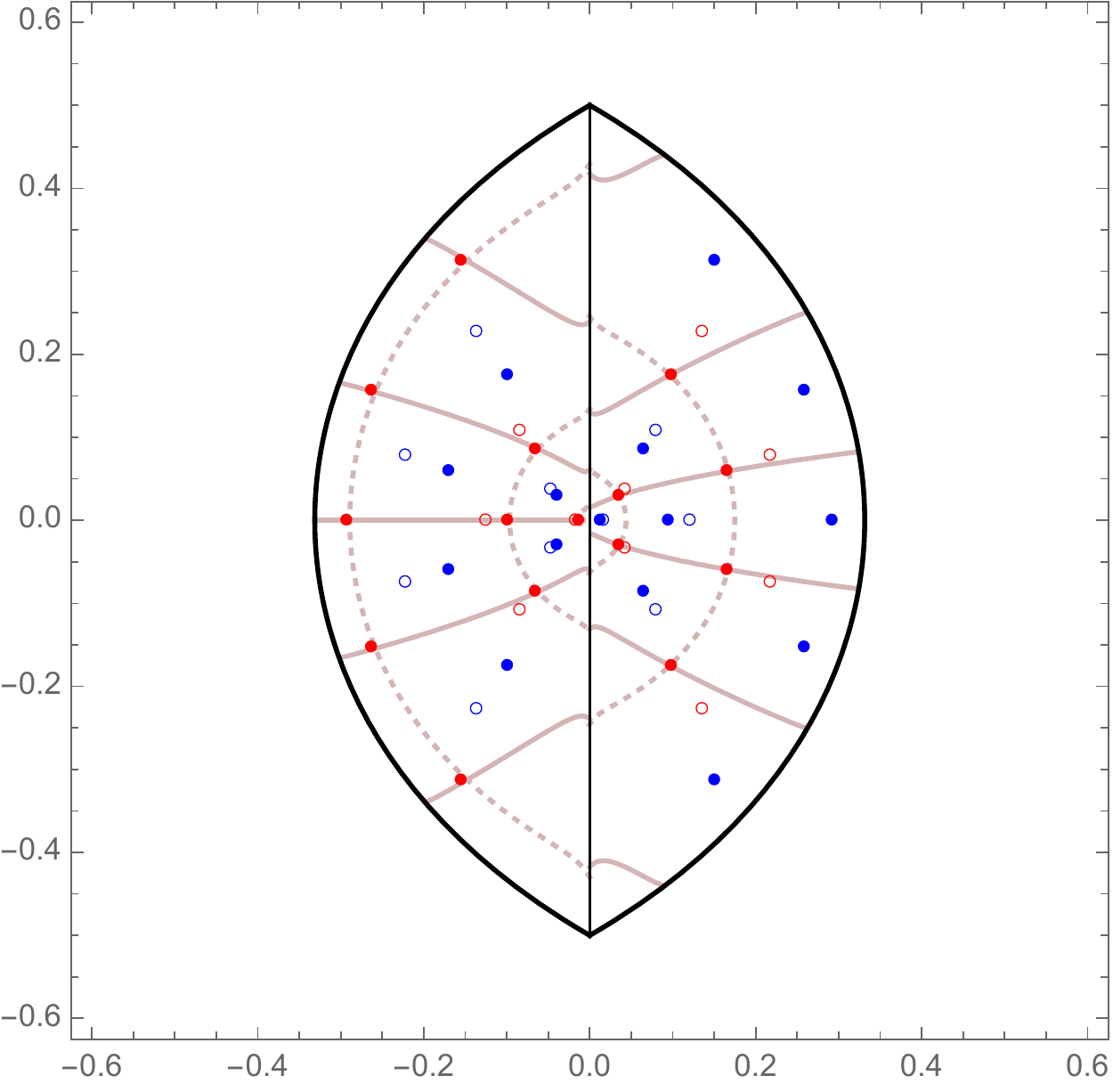}\hfill%
\includegraphics[width=0.3\textwidth]{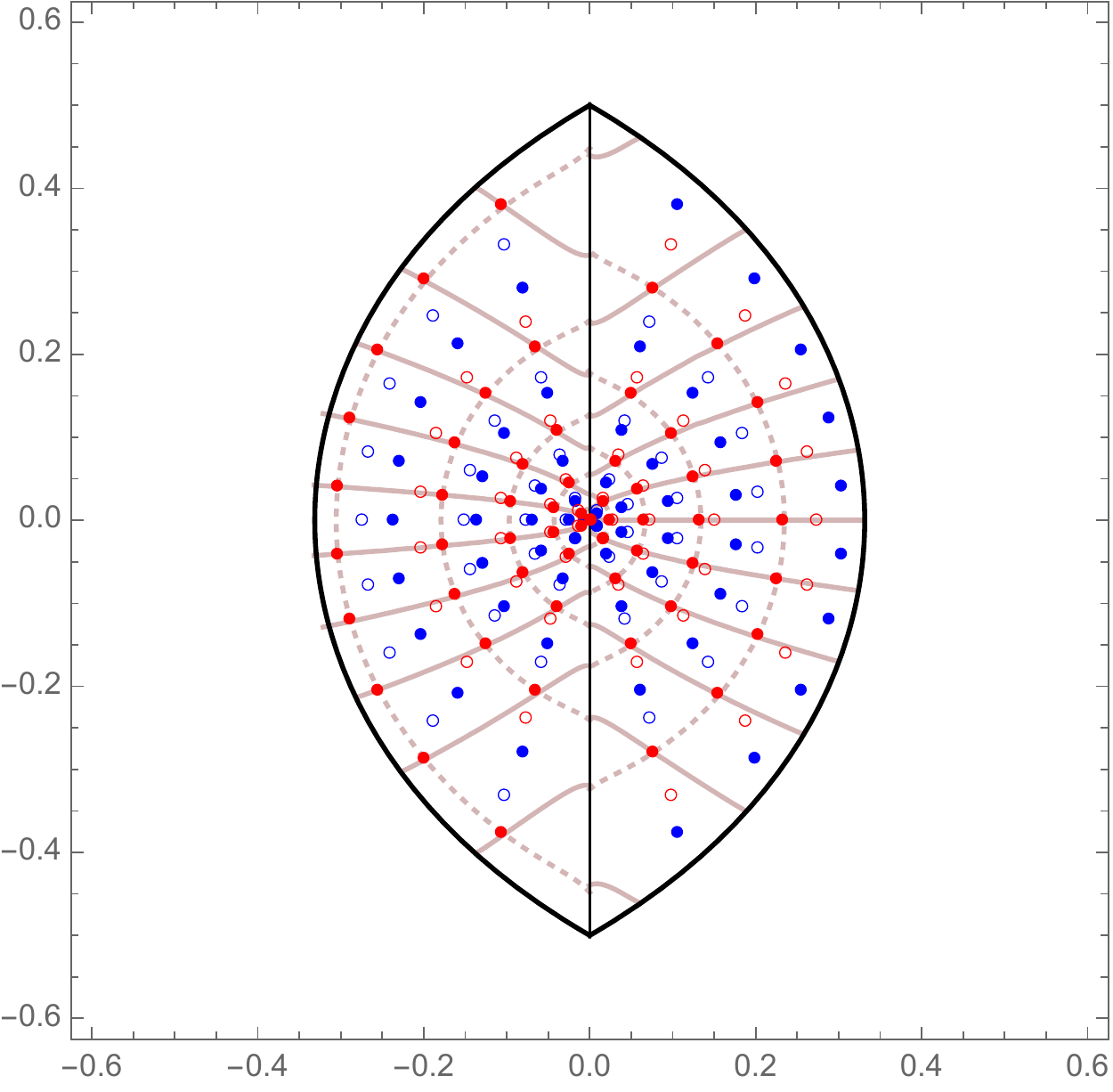}\hfill%
\includegraphics[width=0.3\textwidth]{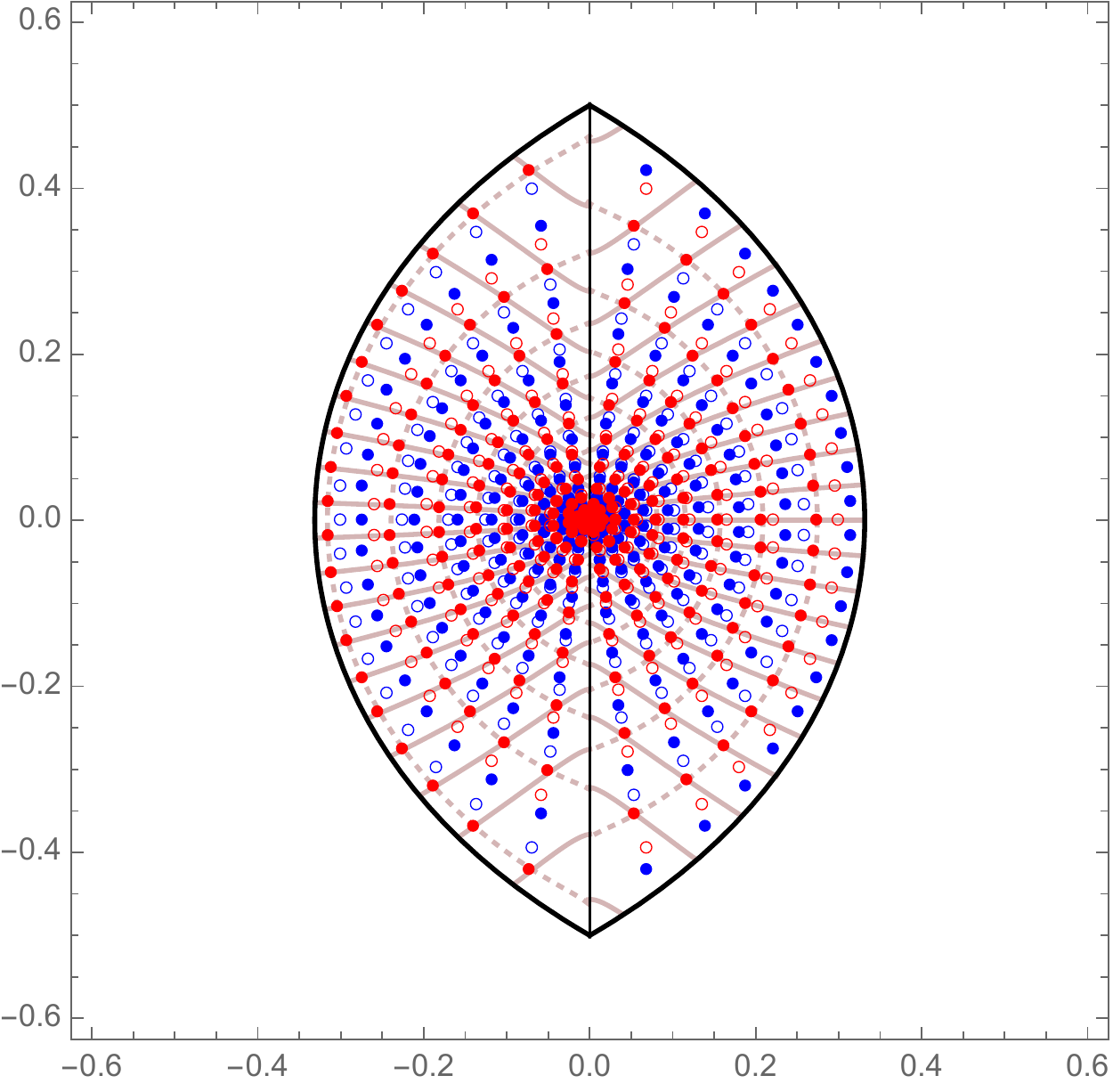}%
\end{center}
\caption{As in Figure~\ref{fig:m0-ZerosPlus} but here the light red curves are $\alpha_n^{\infty,+}(y,0,m)\in\mathbb{Z}$ (solid) and $\beta_n^{\infty,+}(y,0,m)\in\mathbb{Z}$ (dotted); see \eqref{eq:Pole-Quantum} for definitions of these functions.  Note how the filled red dots are attracted toward the intersections of the curves, which are the singularities of $\dot{u}_n(y,0;m)$ arising from roots of $\mathcal{P}_n^\bullet(y,0;m)$.}
\label{fig:m0-PolesPlus}
\end{figure}
\begin{figure}[h]
\begin{center}
\includegraphics[width=0.3\textwidth]{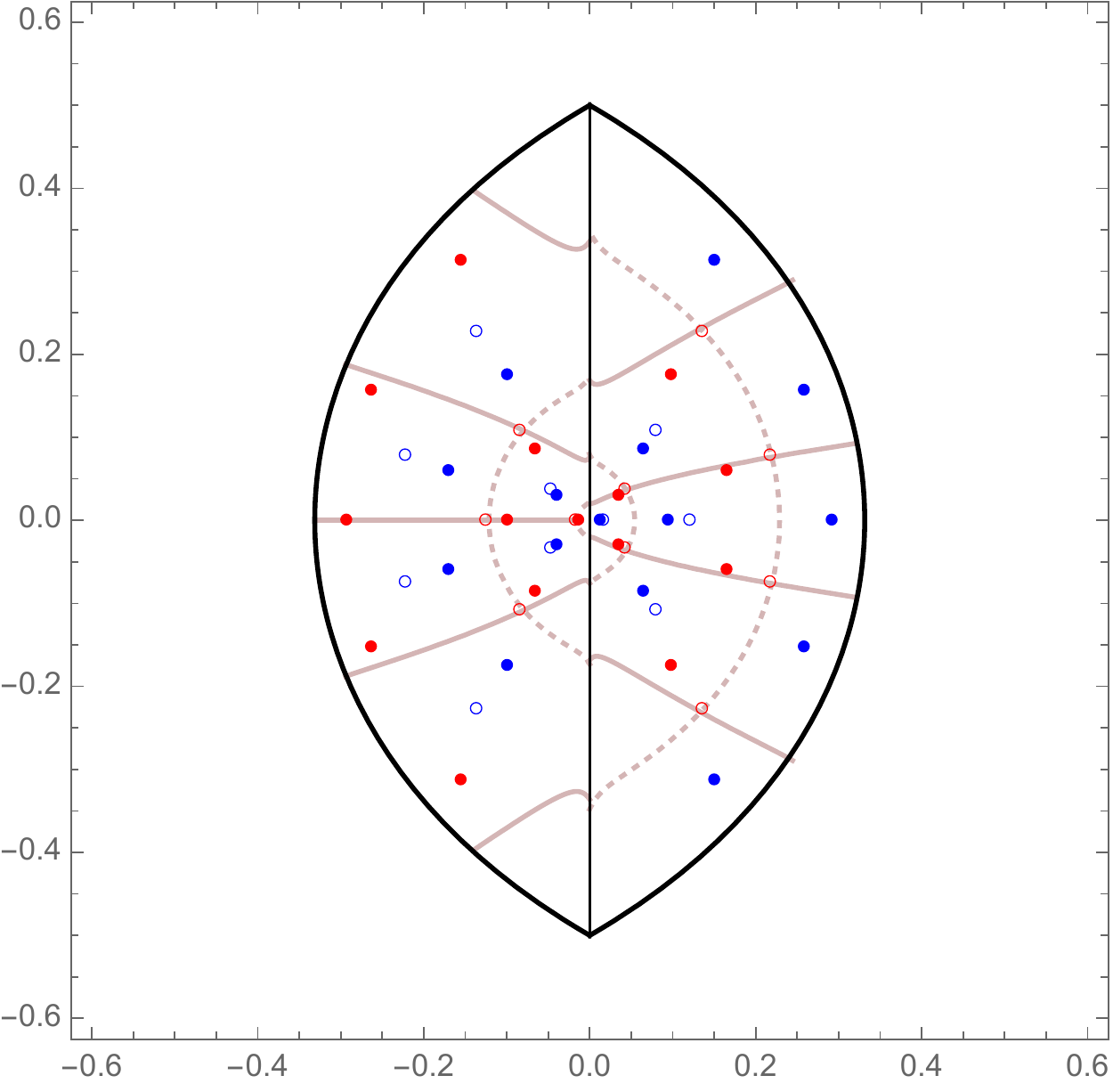}\hfill%
\includegraphics[width=0.3\textwidth]{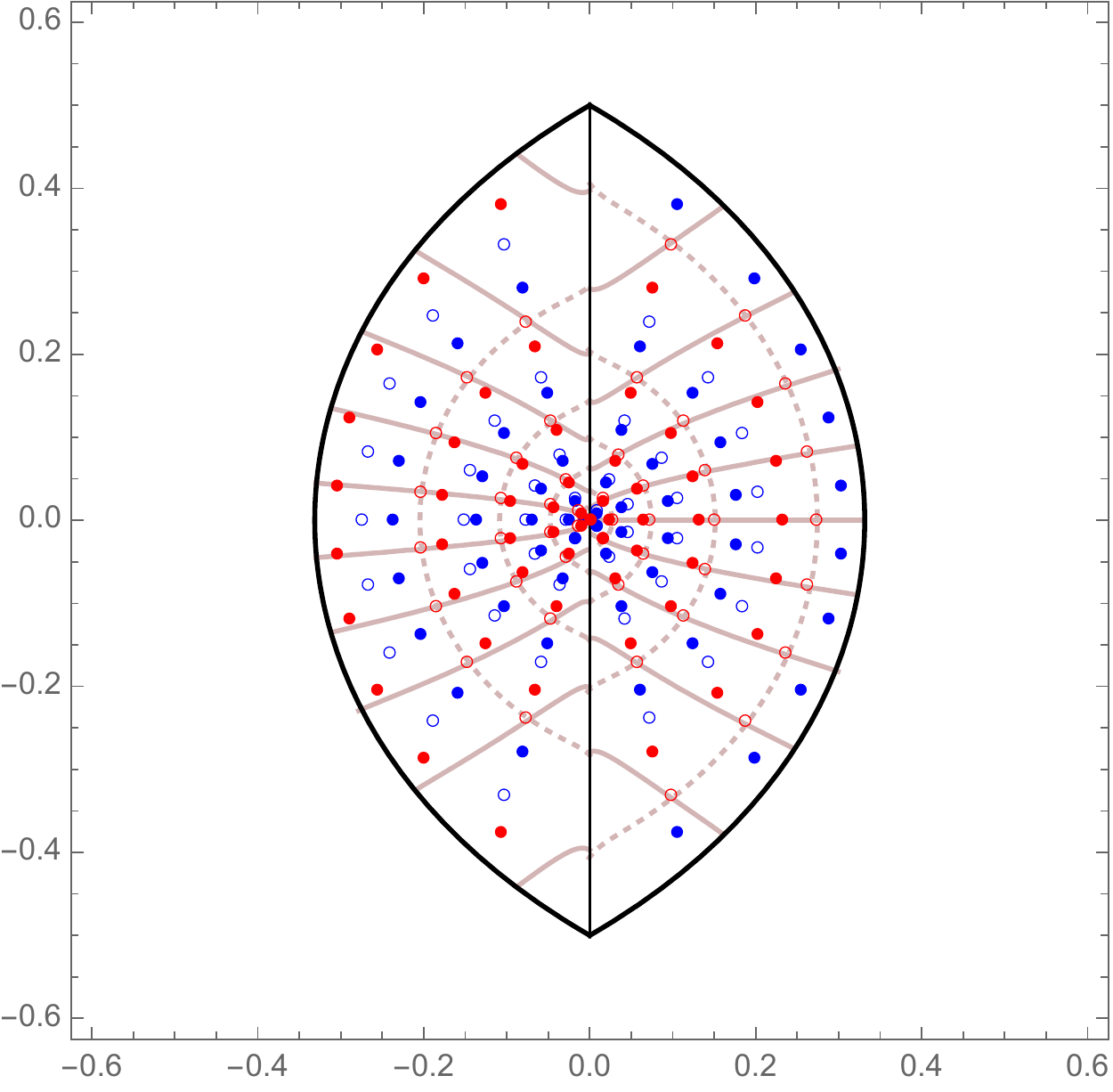}\hfill%
\includegraphics[width=0.3\textwidth]{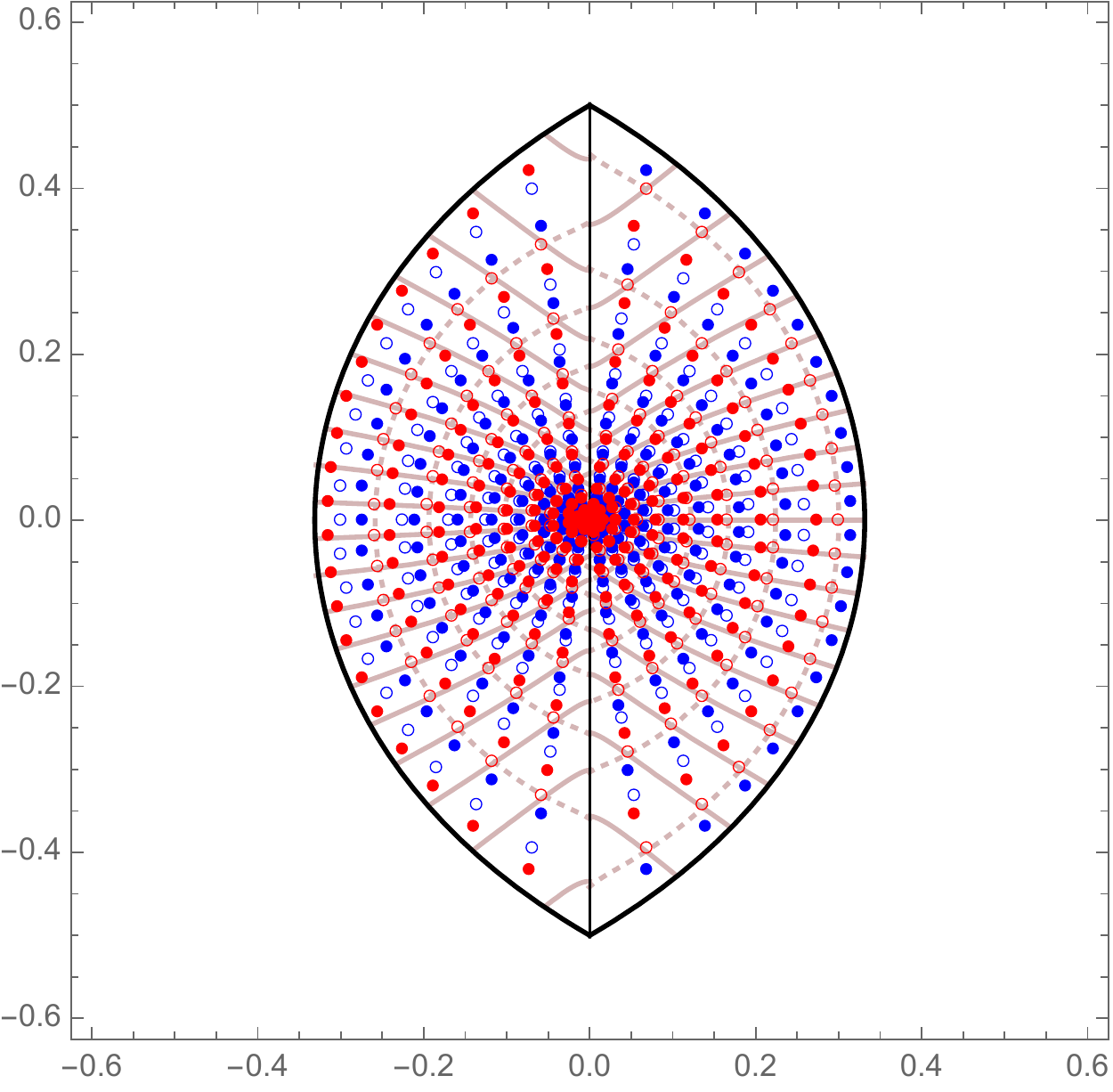}%
\end{center}
\caption{As in Figure~\ref{fig:m0-ZerosPlus} but here the light red curves are $\alpha_n^{\infty,-}(y,0,m)\in\mathbb{Z}$ (solid) and $\beta_n^{\infty,-}(y,0,m)\in\mathbb{Z}$ (dotted); see \eqref{eq:Pole-Quantum} for definitions of these functions.  Note how the unfilled red dots are attracted toward the intersections of the curves, which are now the singularities of $\dot{u}_n(y,0;m)$ arising from roots of $\mathcal{P}_n^\circ(y,0;m)$.}
\label{fig:m0-PolesMinus}
\end{figure}
\begin{figure}[h]
\begin{center}
\includegraphics[width=0.3\textwidth]{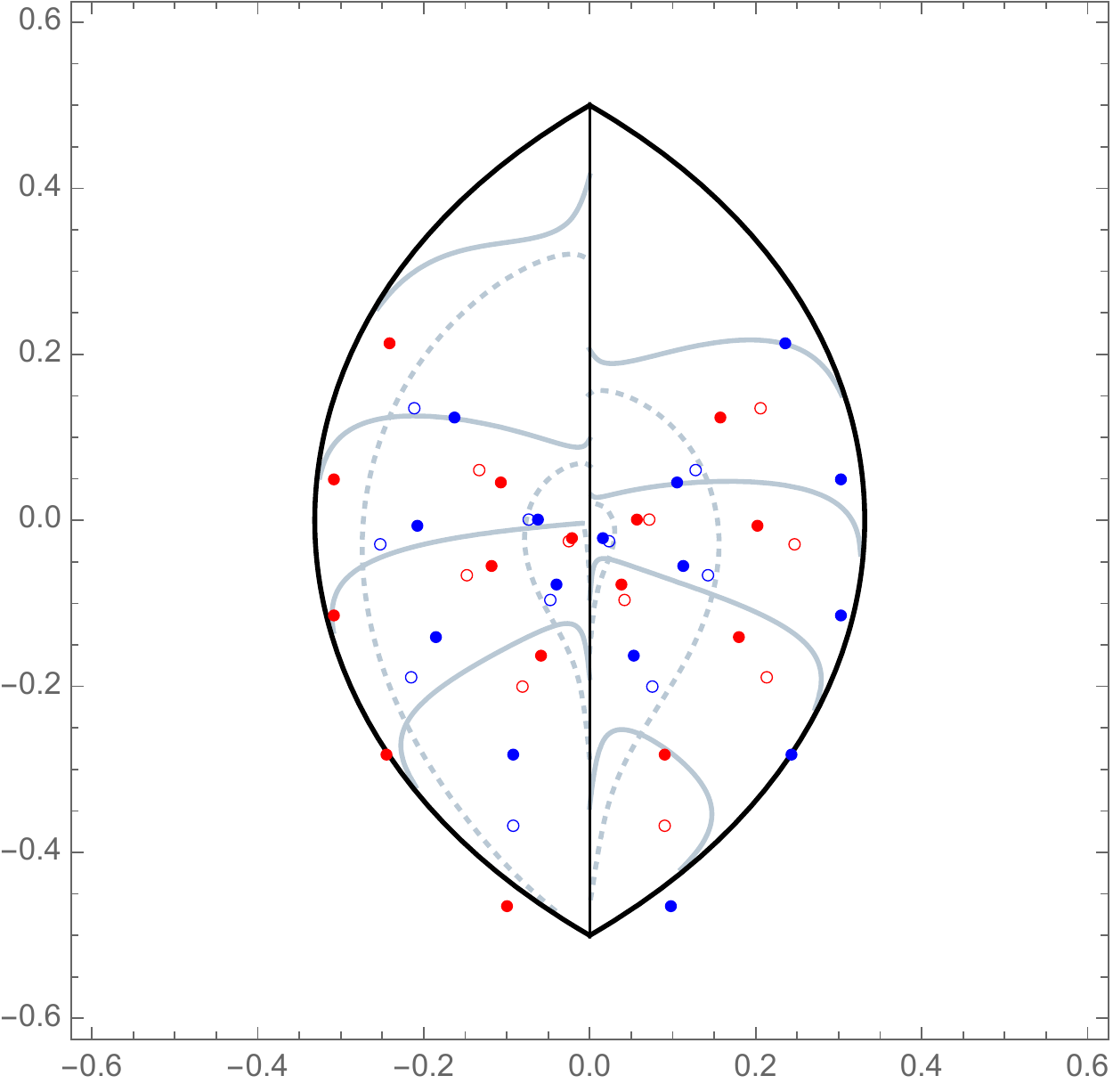}\hfill%
\includegraphics[width=0.3\textwidth]{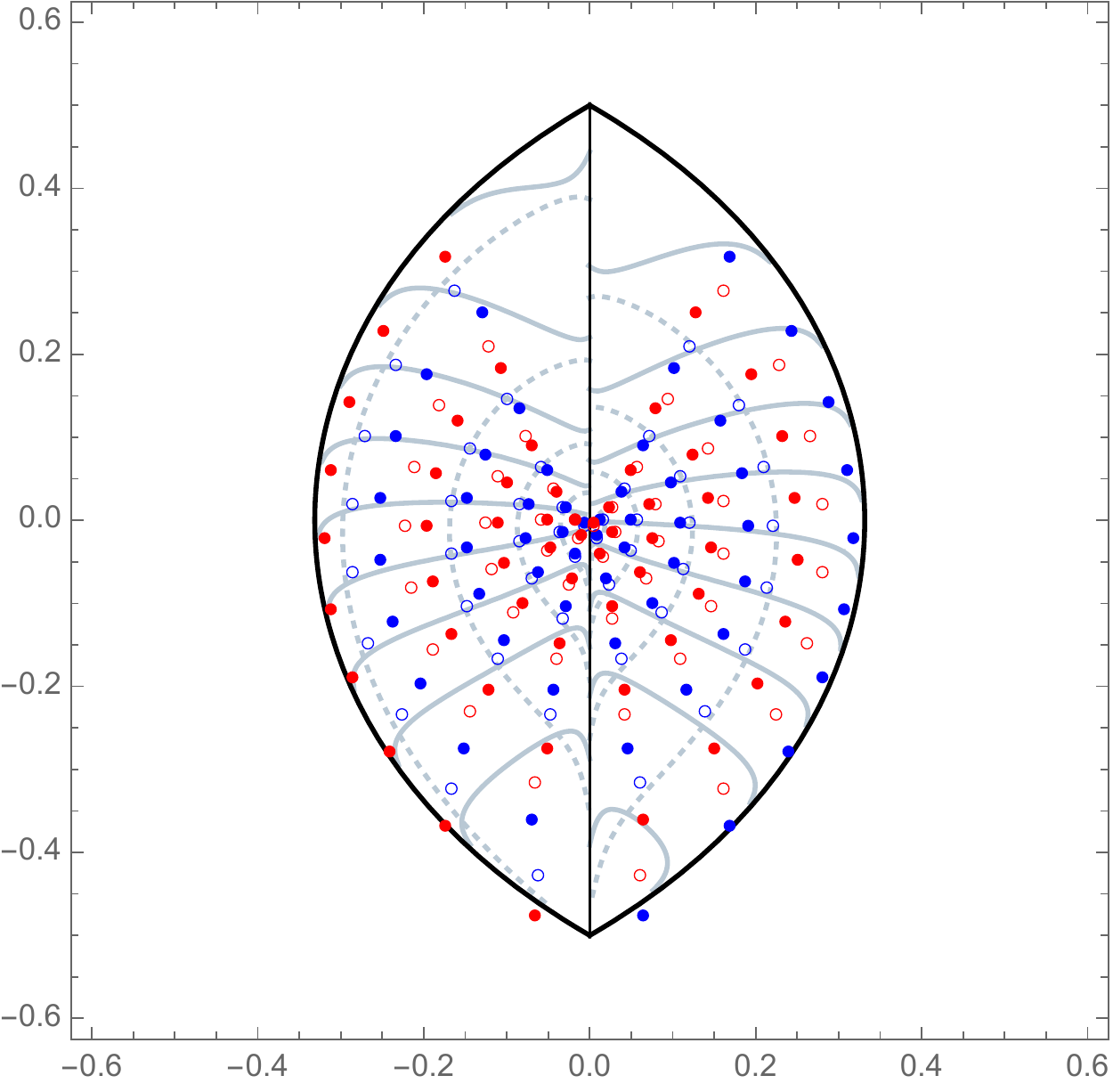}\hfill%
\includegraphics[width=0.3\textwidth]{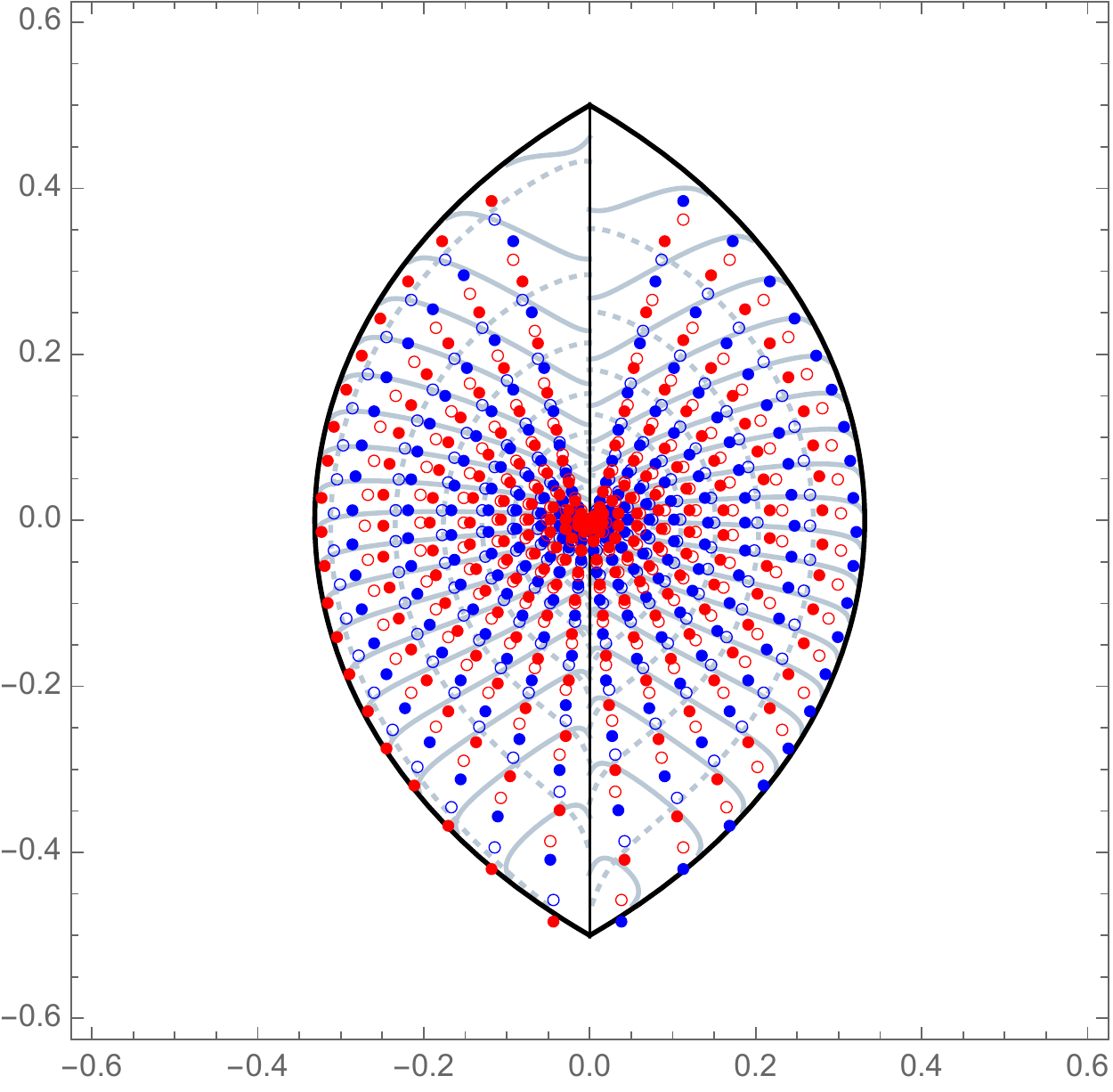}%
\end{center}
\caption{As in Figure~\ref{fig:m0-ZerosPlus} but for $m=\tfrac{4}{5}\ii$.}
\label{fig:m4iOver5-ZerosPlus}
\end{figure}
\begin{figure}[h]
\begin{center}
\includegraphics[width=0.3\textwidth]{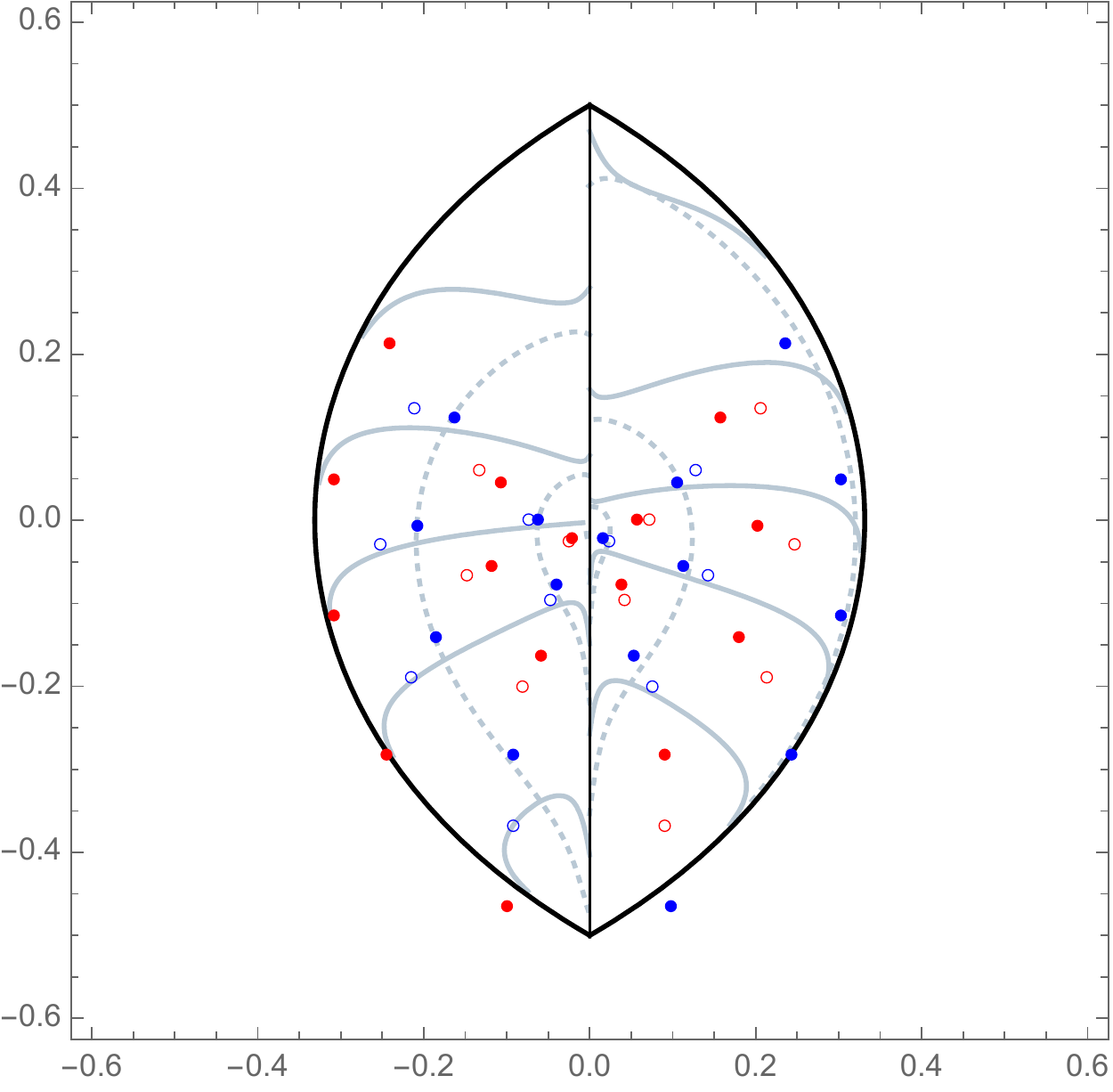}\hfill%
\includegraphics[width=0.3\textwidth]{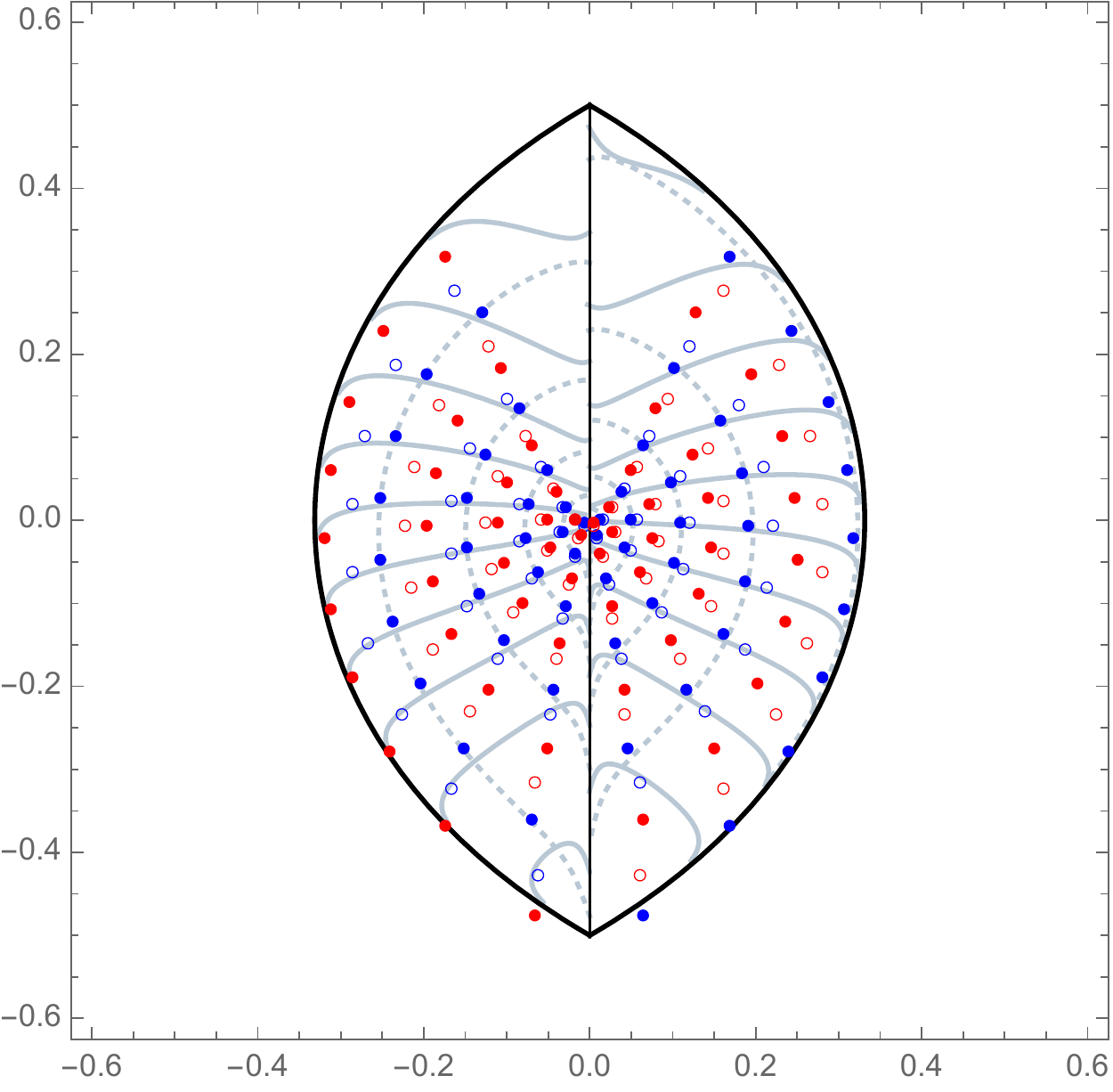}\hfill%
\includegraphics[width=0.3\textwidth]{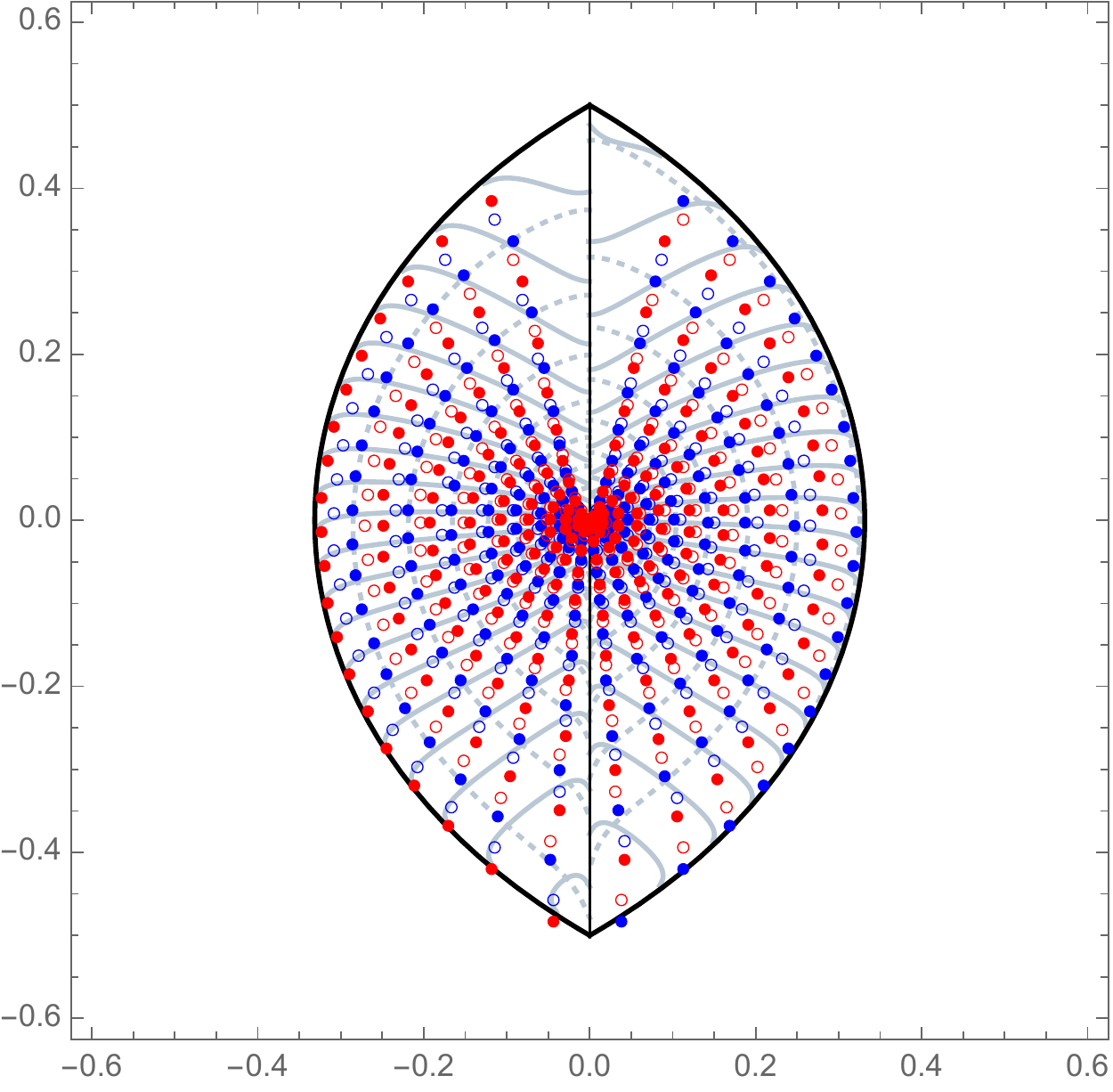}%
\end{center}
\caption{As in Figure~\ref{fig:m0-ZerosMinus} but for $m=\tfrac{4}{5}\ii$.}
\label{fig:m4iOver5-ZerosMinus}
\end{figure}
\begin{figure}[h]
\begin{center}
\includegraphics[width=0.3\textwidth]{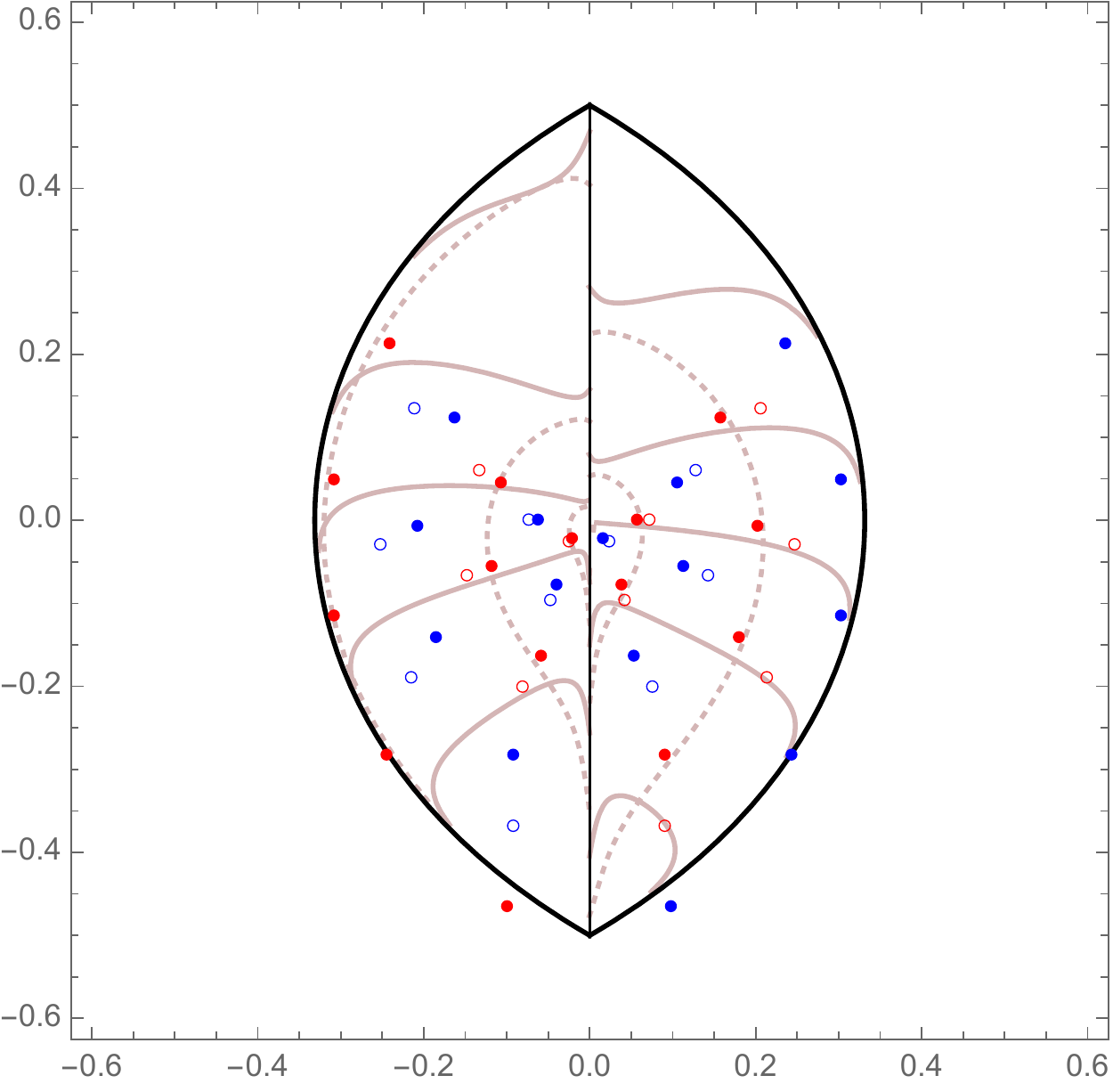}\hfill%
\includegraphics[width=0.3\textwidth]{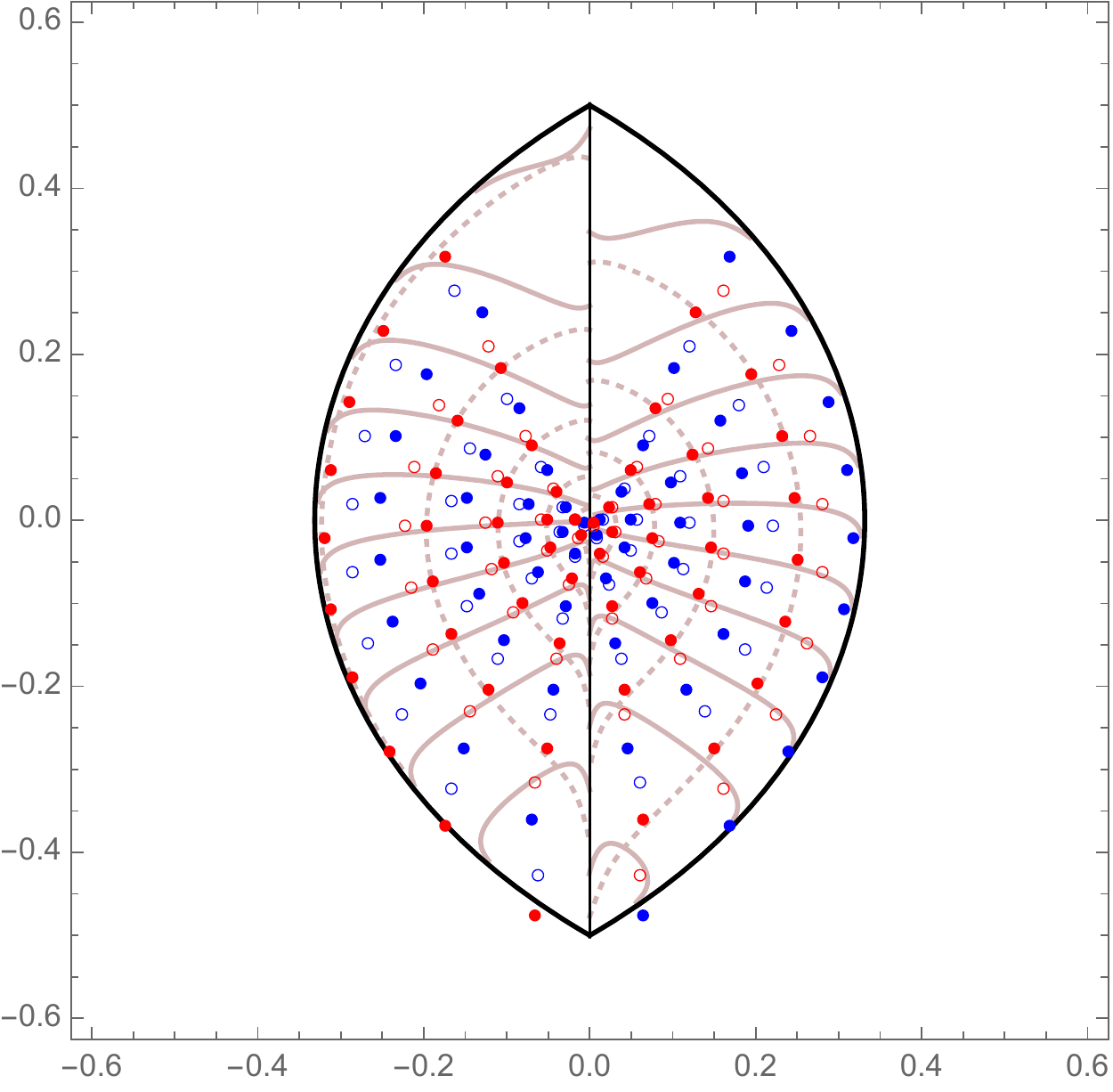}\hfill%
\includegraphics[width=0.3\textwidth]{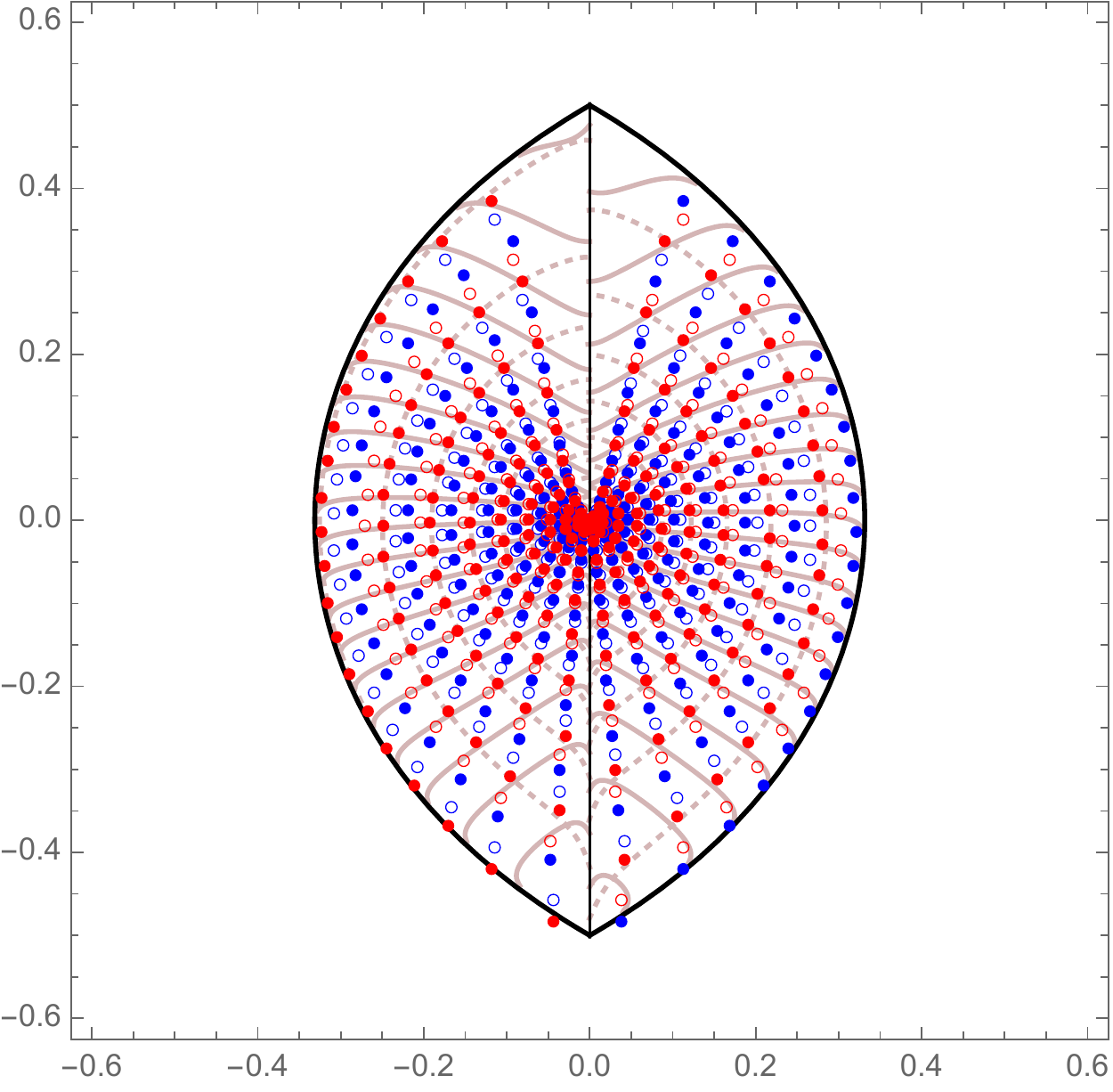}%
\end{center}
\caption{As in Figure~\ref{fig:m0-PolesPlus} but for $m=\tfrac{4}{5}\ii$.}
\label{fig:m4iOver5-PolesPlus}
\end{figure}
\begin{figure}[h]
\begin{center}
\includegraphics[width=0.3\textwidth]{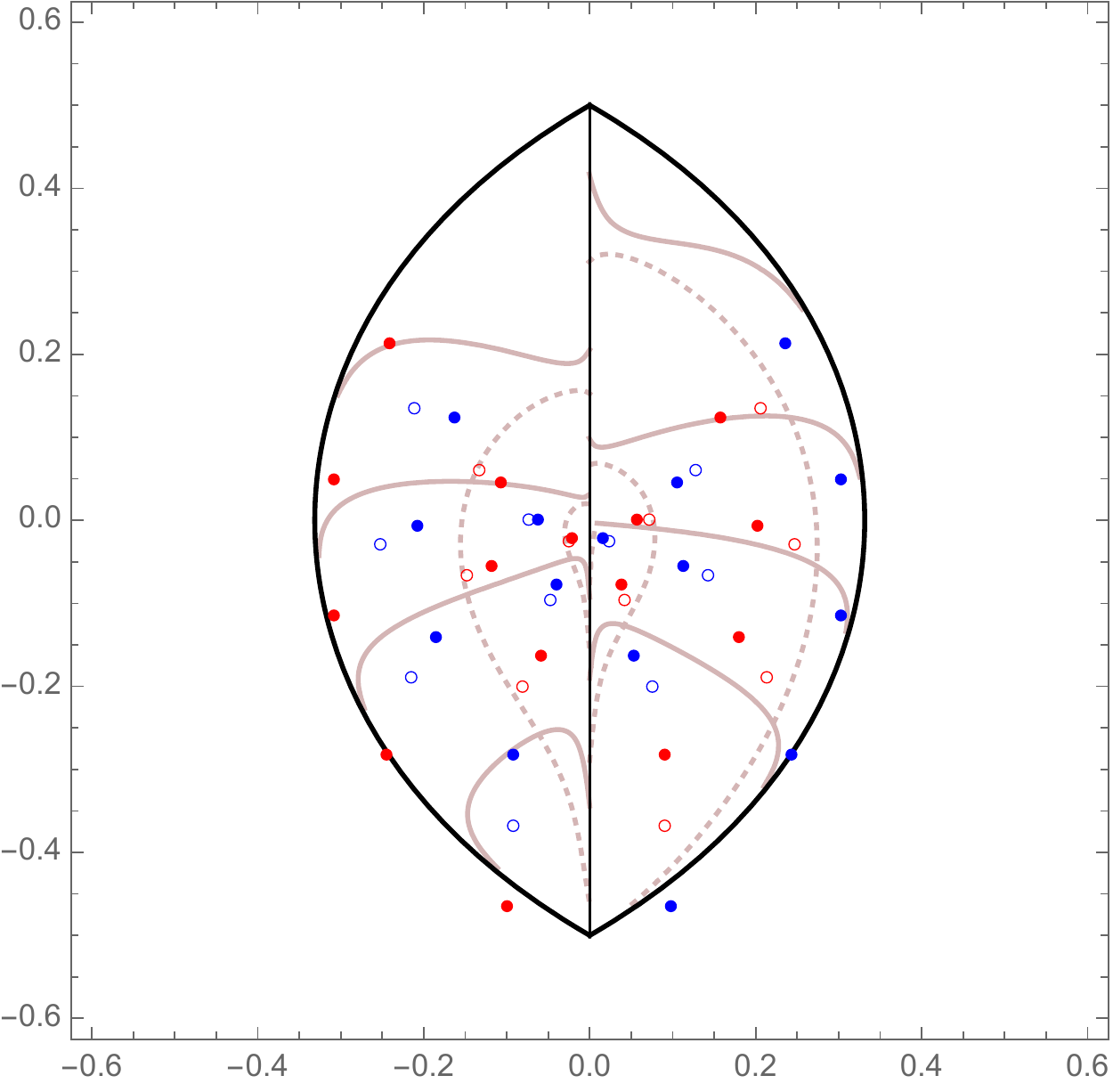}\hfill%
\includegraphics[width=0.3\textwidth]{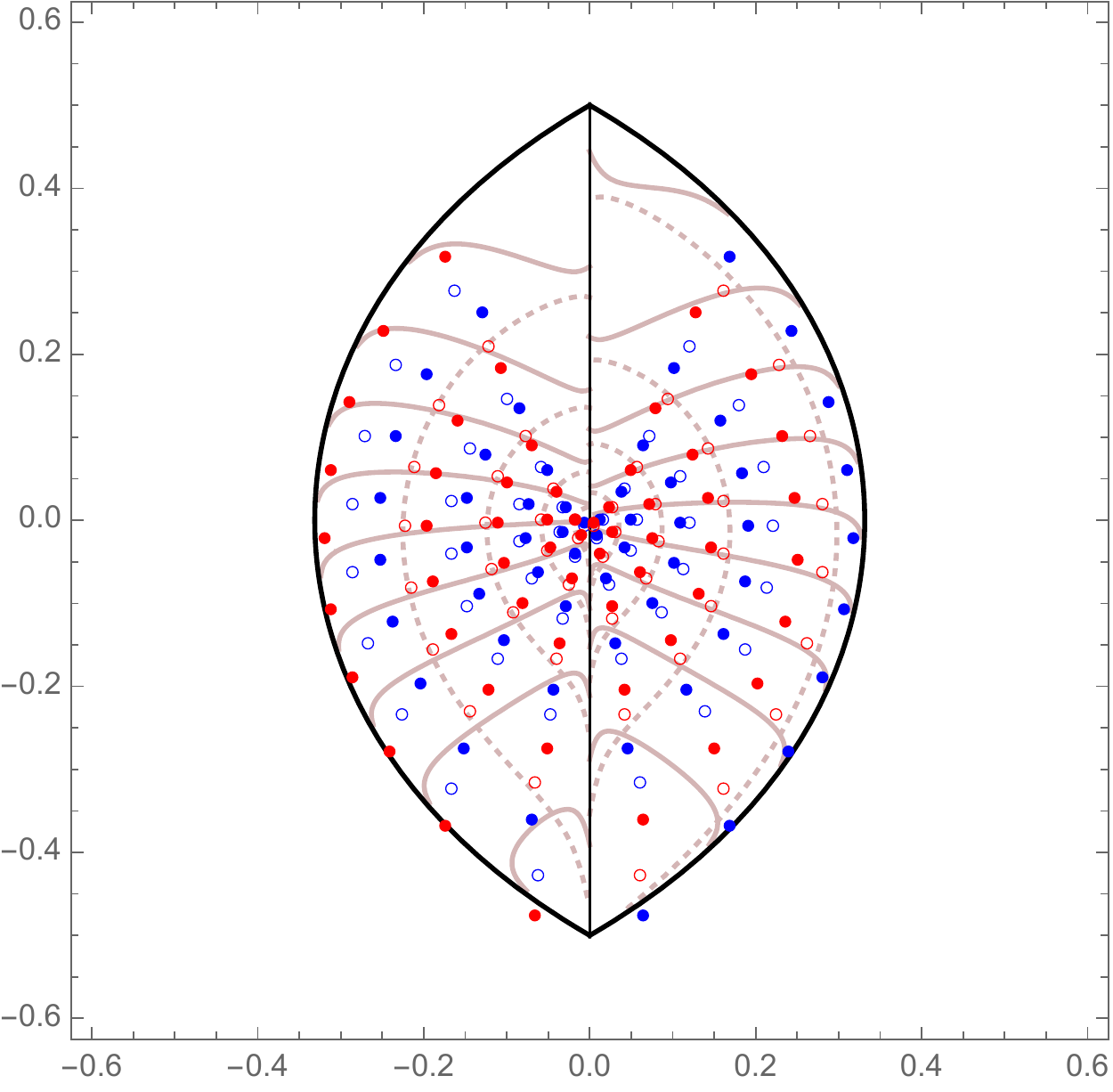}\hfill%
\includegraphics[width=0.3\textwidth]{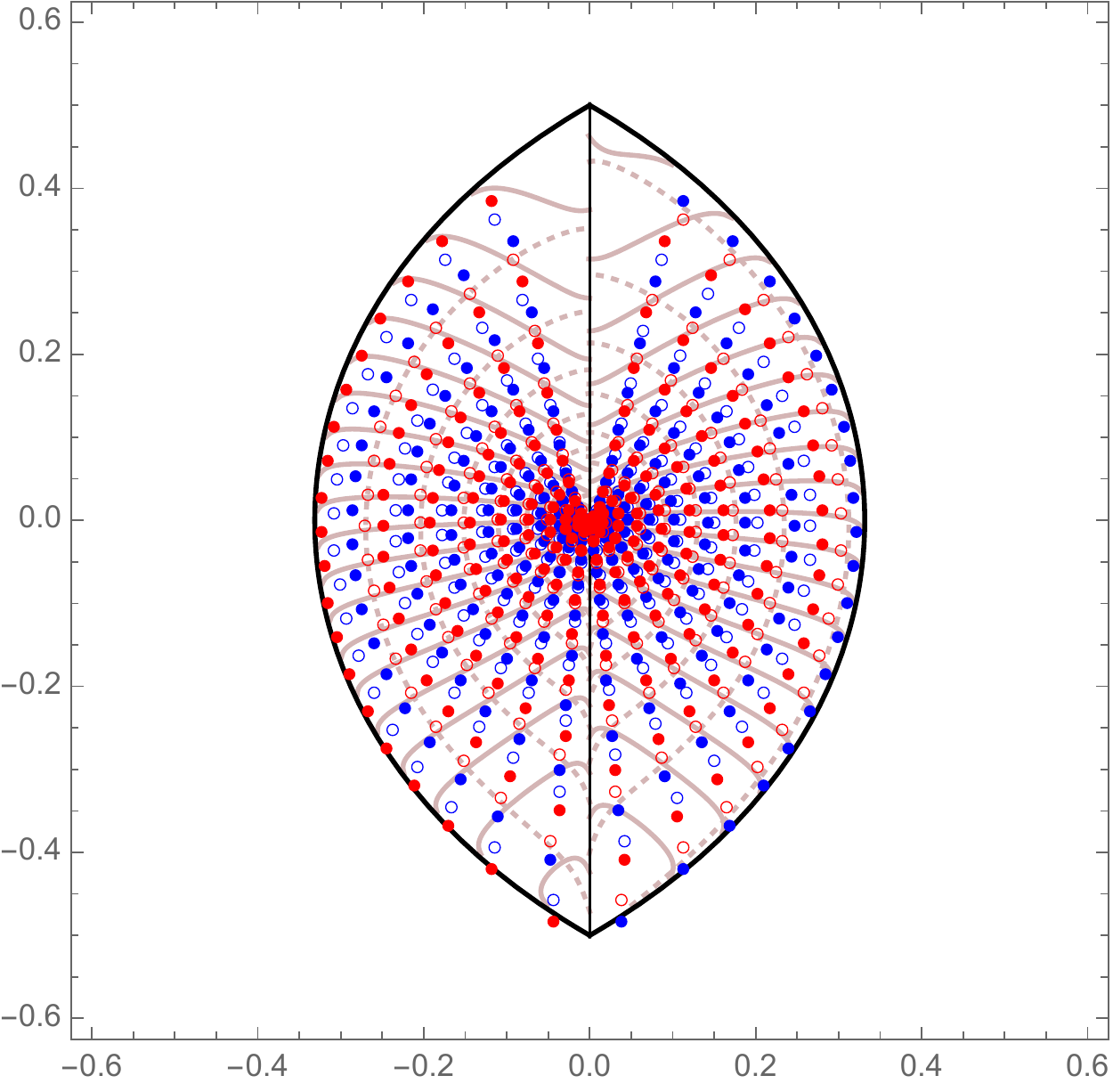}%
\end{center}
\caption{As in Figure~\ref{fig:m0-PolesMinus} but for $m=\tfrac{4}{5}\ii$.}
\label{fig:m4iOver5-PolesMinus}
\end{figure}
\begin{figure}[h]
\begin{center}
\includegraphics[width=0.3\textwidth]{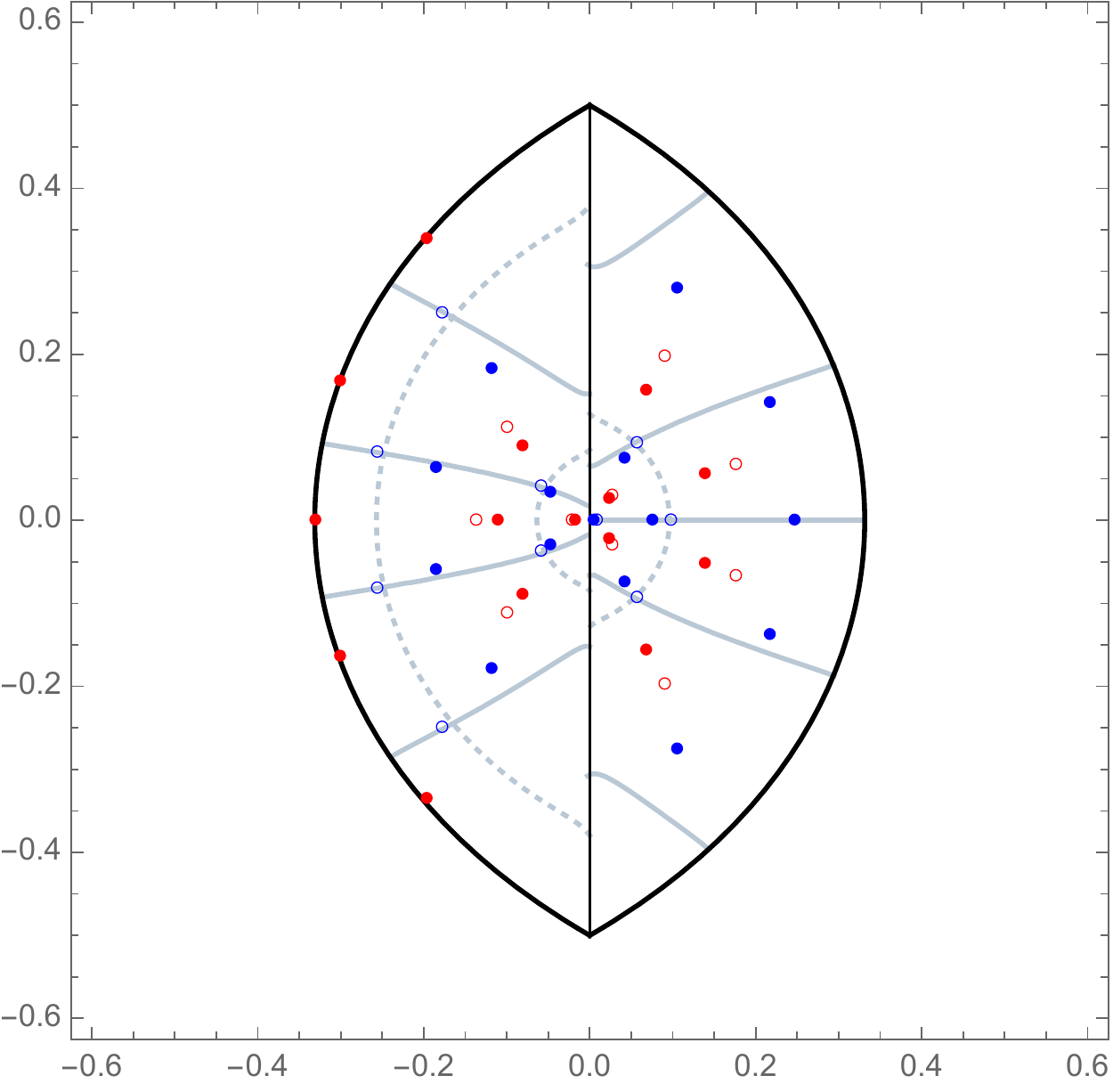}\hfill%
\includegraphics[width=0.3\textwidth]{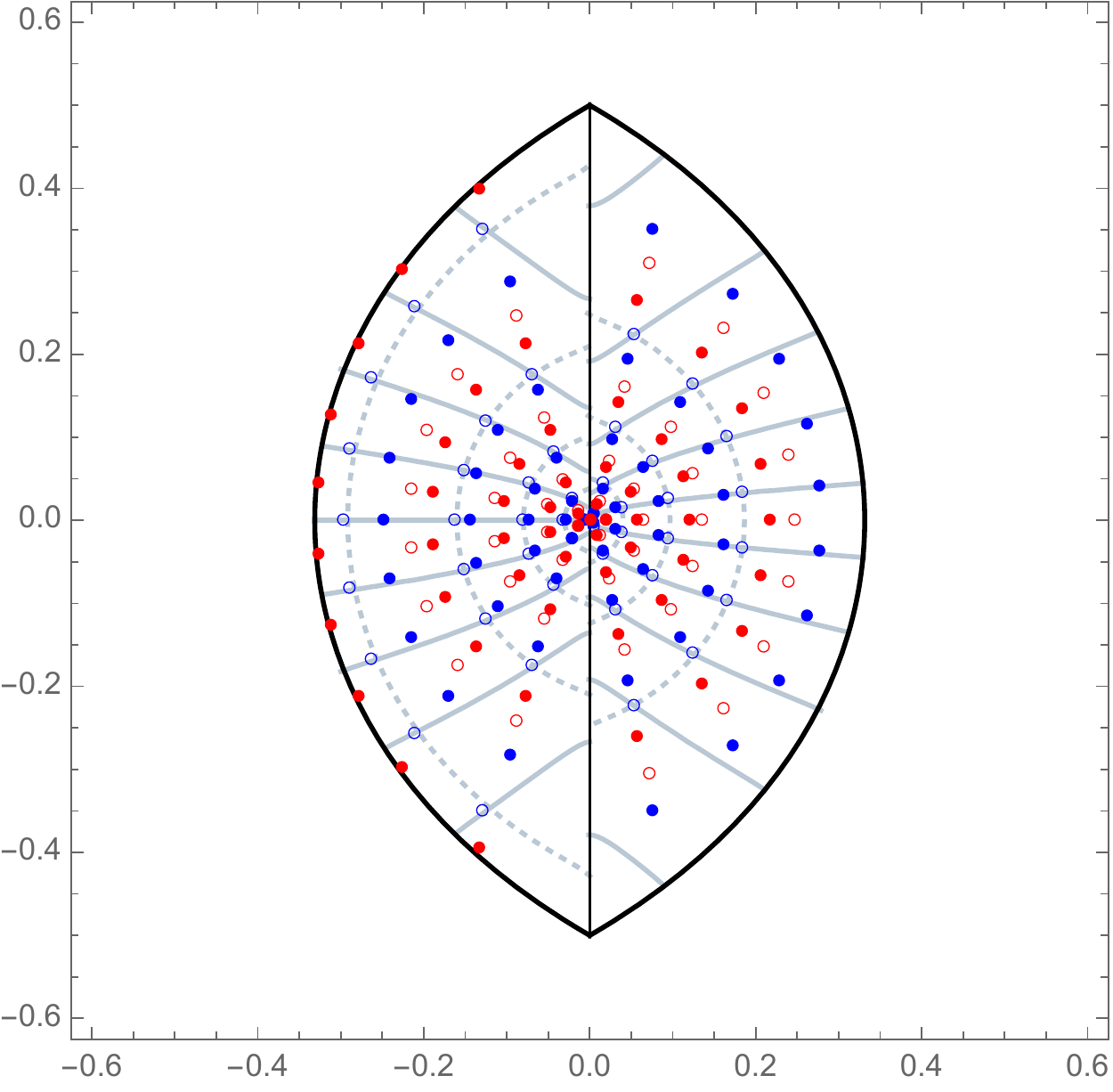}\hfill%
\includegraphics[width=0.3\textwidth]{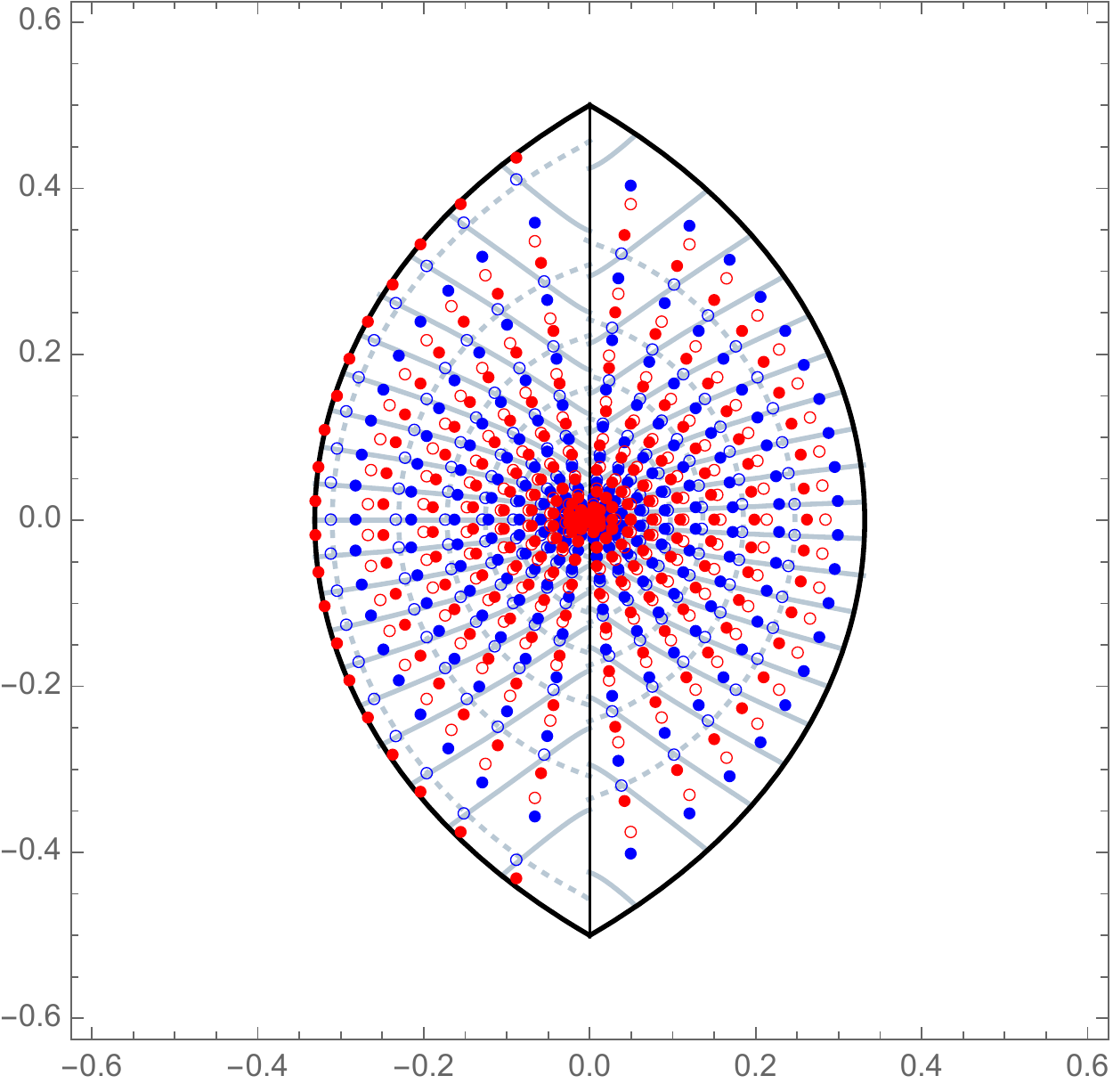}%
\end{center}
\caption{As in Figure~\ref{fig:m0-ZerosPlus} but for $m=\tfrac{1}{4}$.}
\label{fig:m1Over4-ZerosPlus}
\end{figure}
\begin{figure}[h]
\begin{center}
\includegraphics[width=0.3\textwidth]{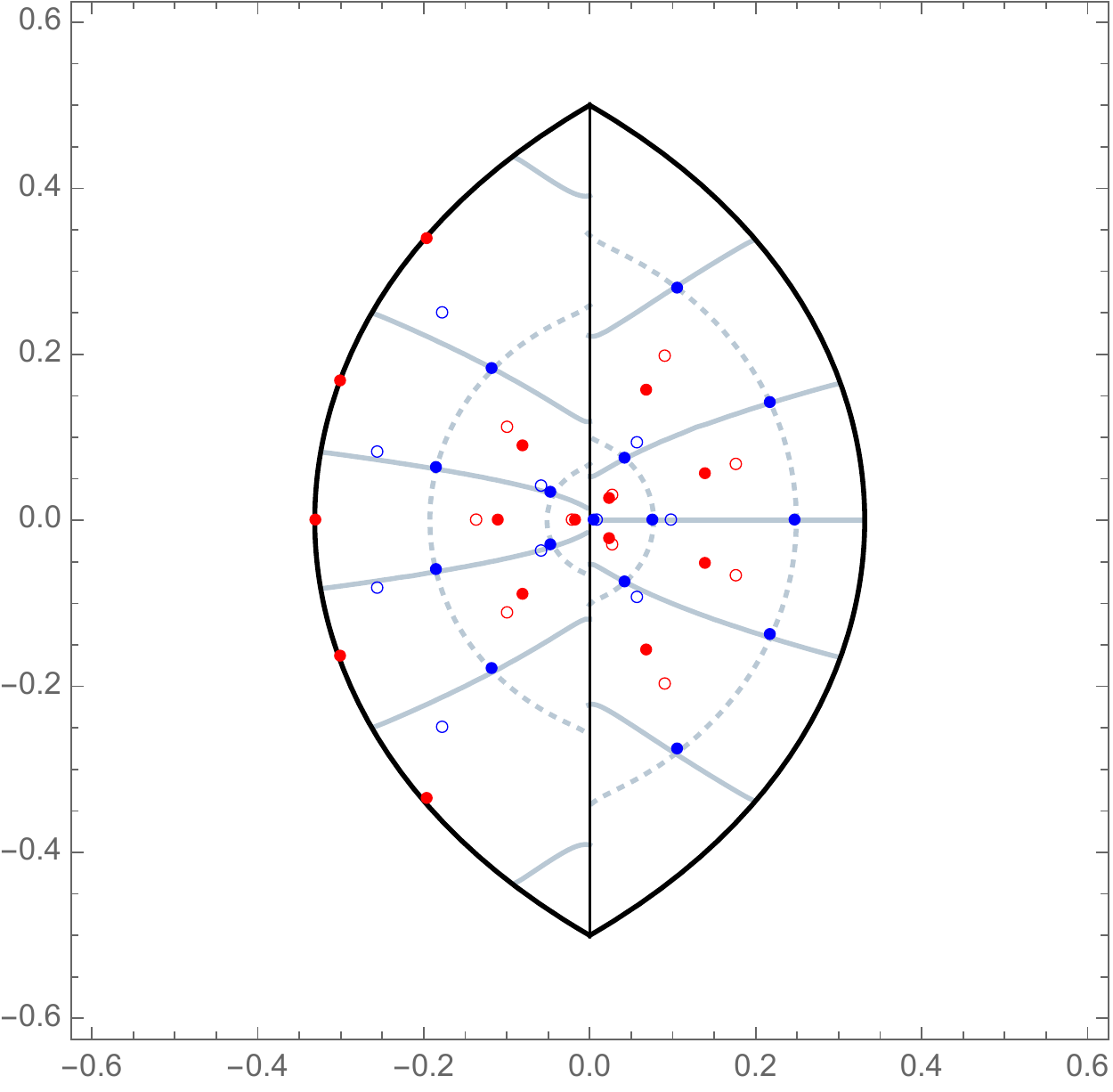}\hfill%
\includegraphics[width=0.3\textwidth]{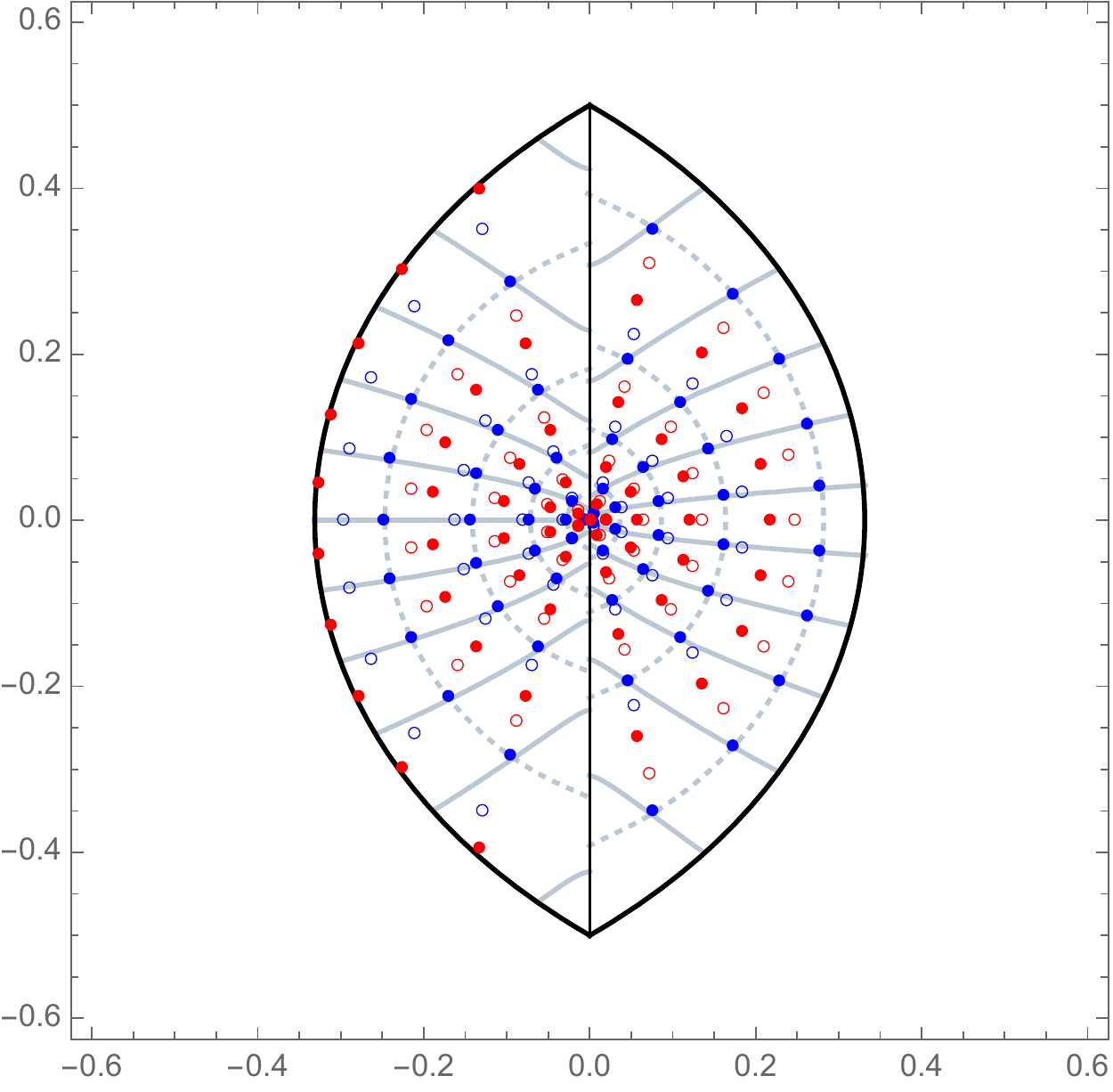}\hfill%
\includegraphics[width=0.3\textwidth]{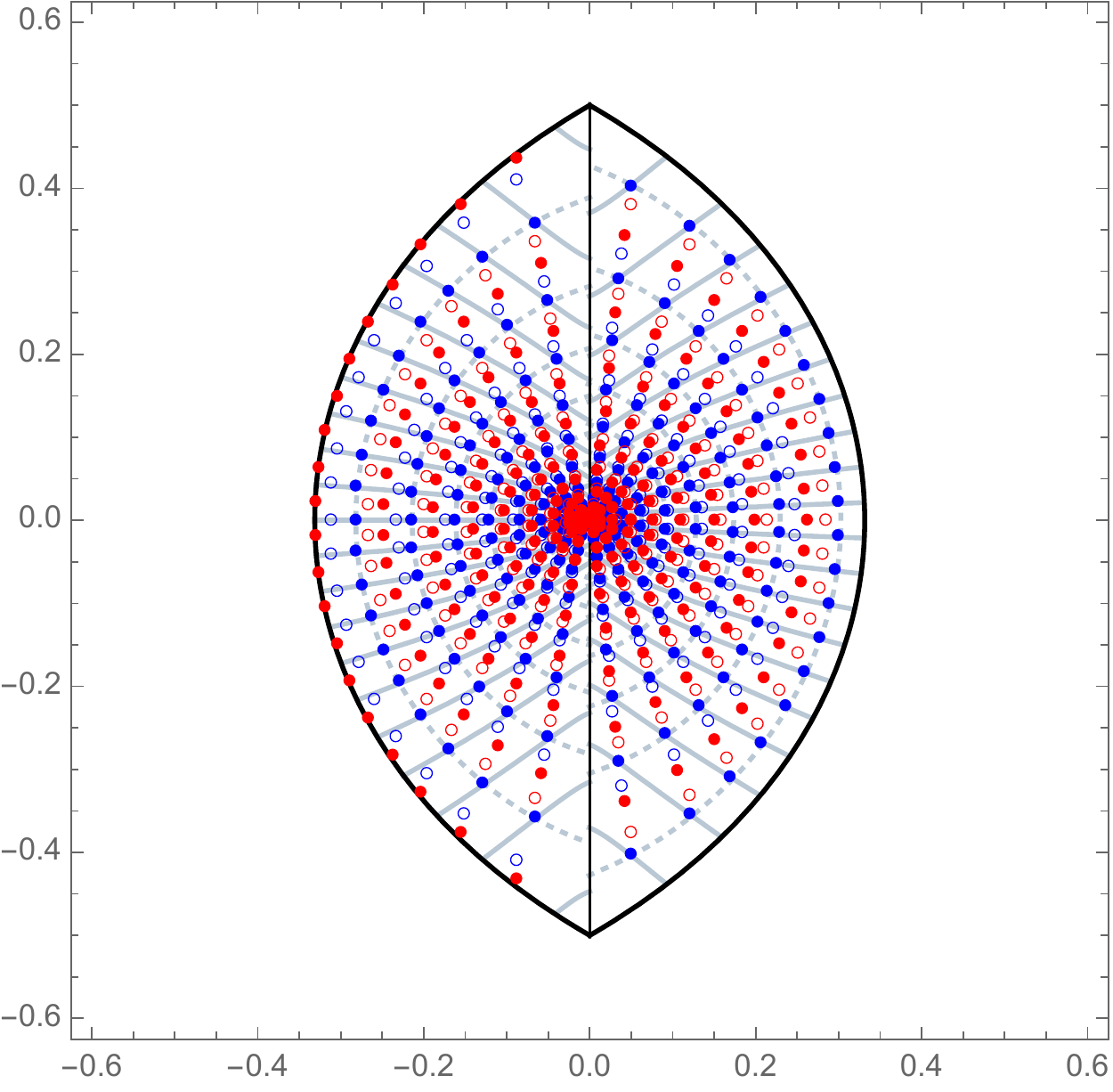}%
\end{center}
\caption{As in Figure~\ref{fig:m0-ZerosMinus} but for $m=\tfrac{1}{4}$.}
\label{fig:m1Over4-ZerosMinus}
\end{figure}
\begin{figure}[h]
\begin{center}
\includegraphics[width=0.3\textwidth]{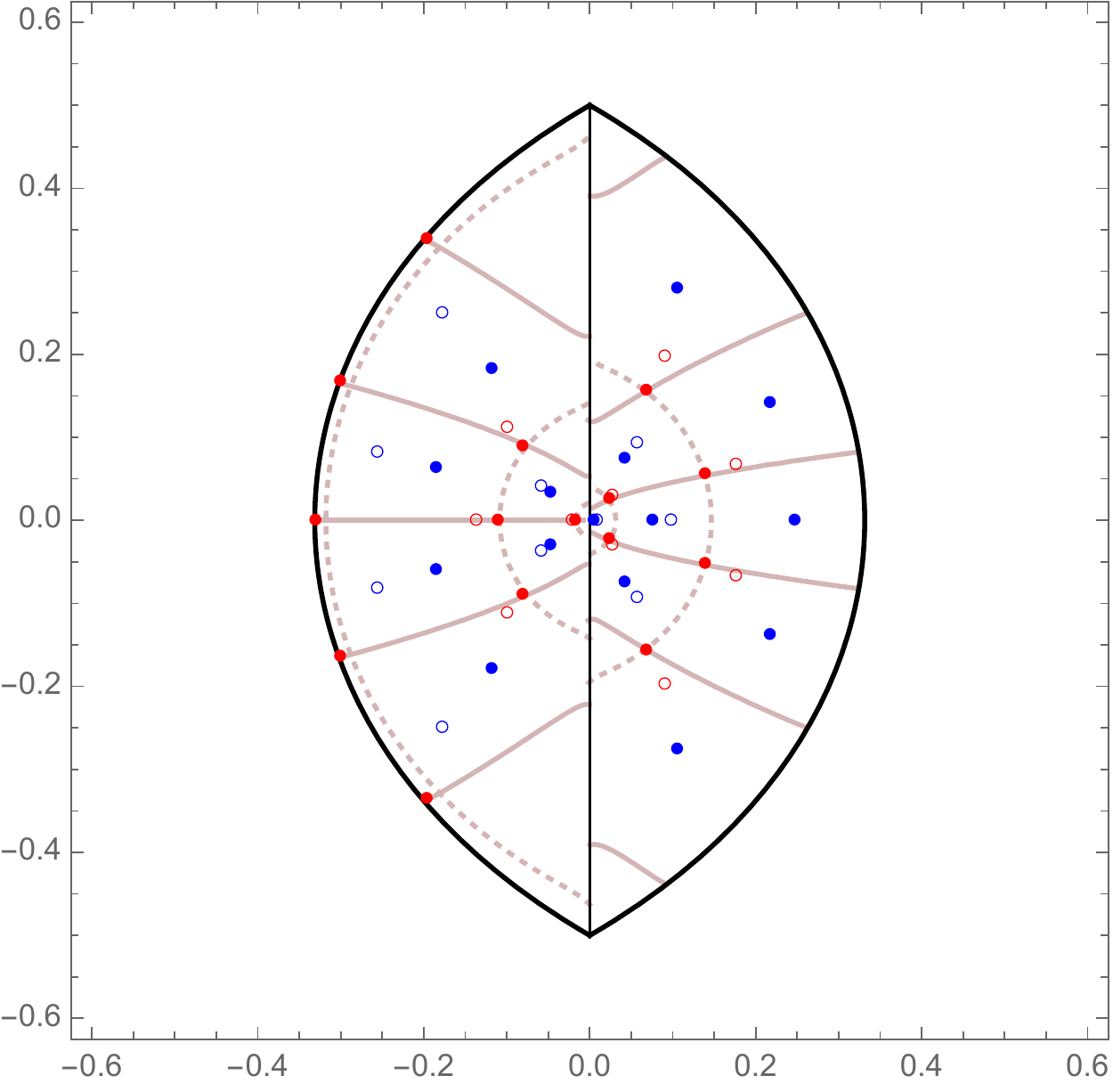}\hfill%
\includegraphics[width=0.3\textwidth]{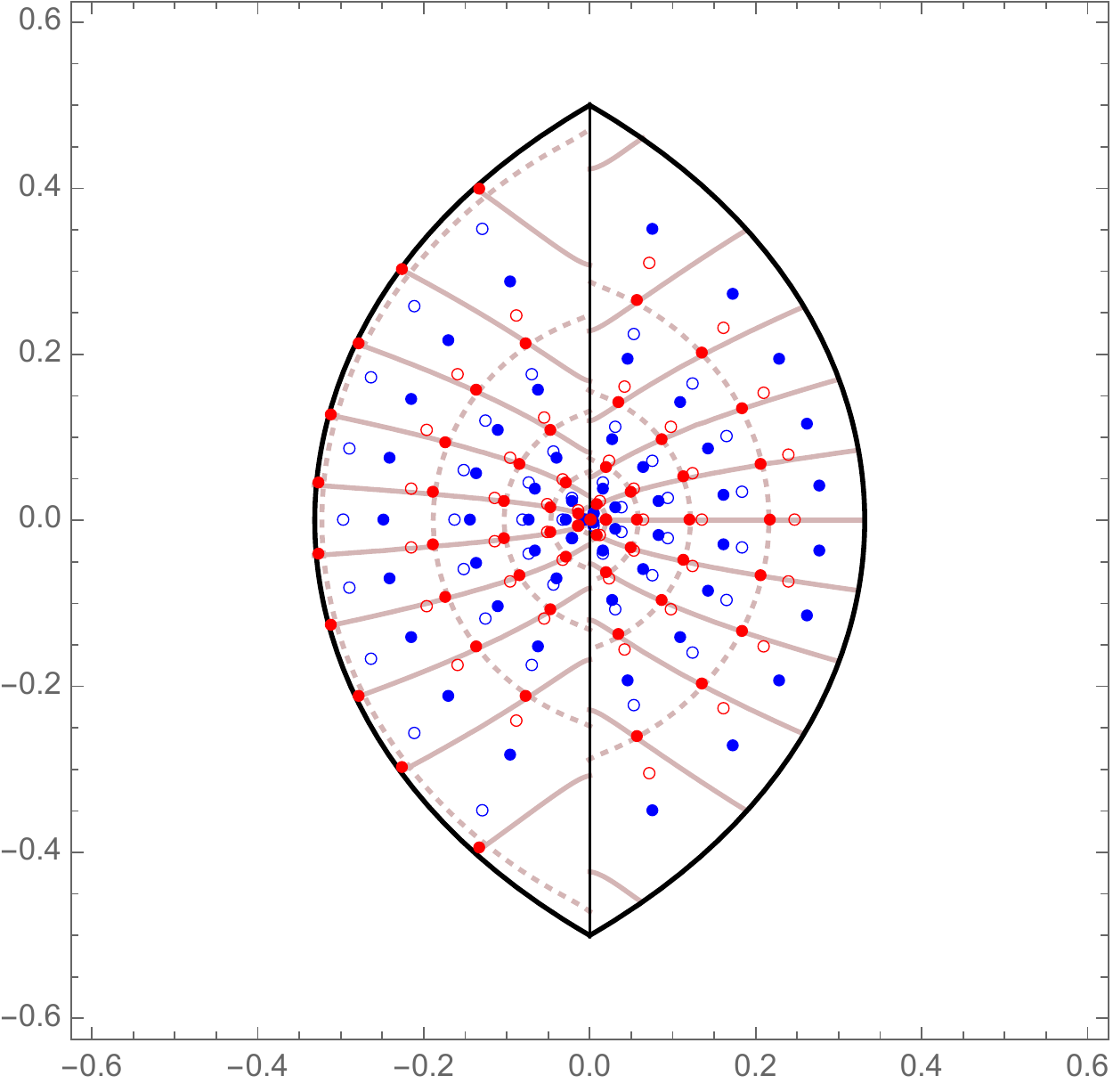}\hfill%
\includegraphics[width=0.3\textwidth]{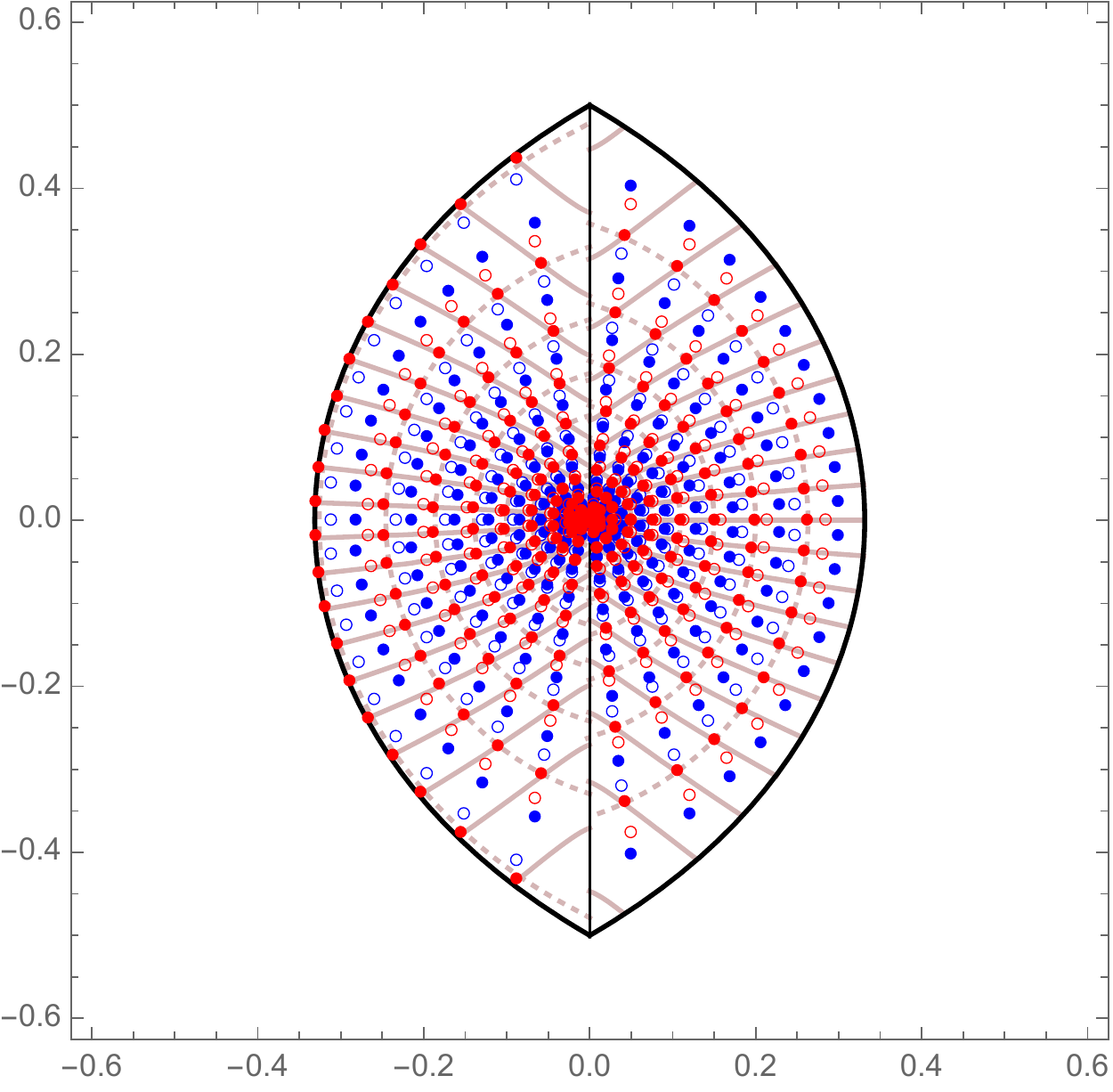}%
\end{center}
\caption{As in Figure~\ref{fig:m0-PolesPlus} but for $m=\tfrac{1}{4}$.}
\label{fig:m1Over4-PolesPlus}
\end{figure}
\begin{figure}[h]
\begin{center}
\includegraphics[width=0.3\textwidth]{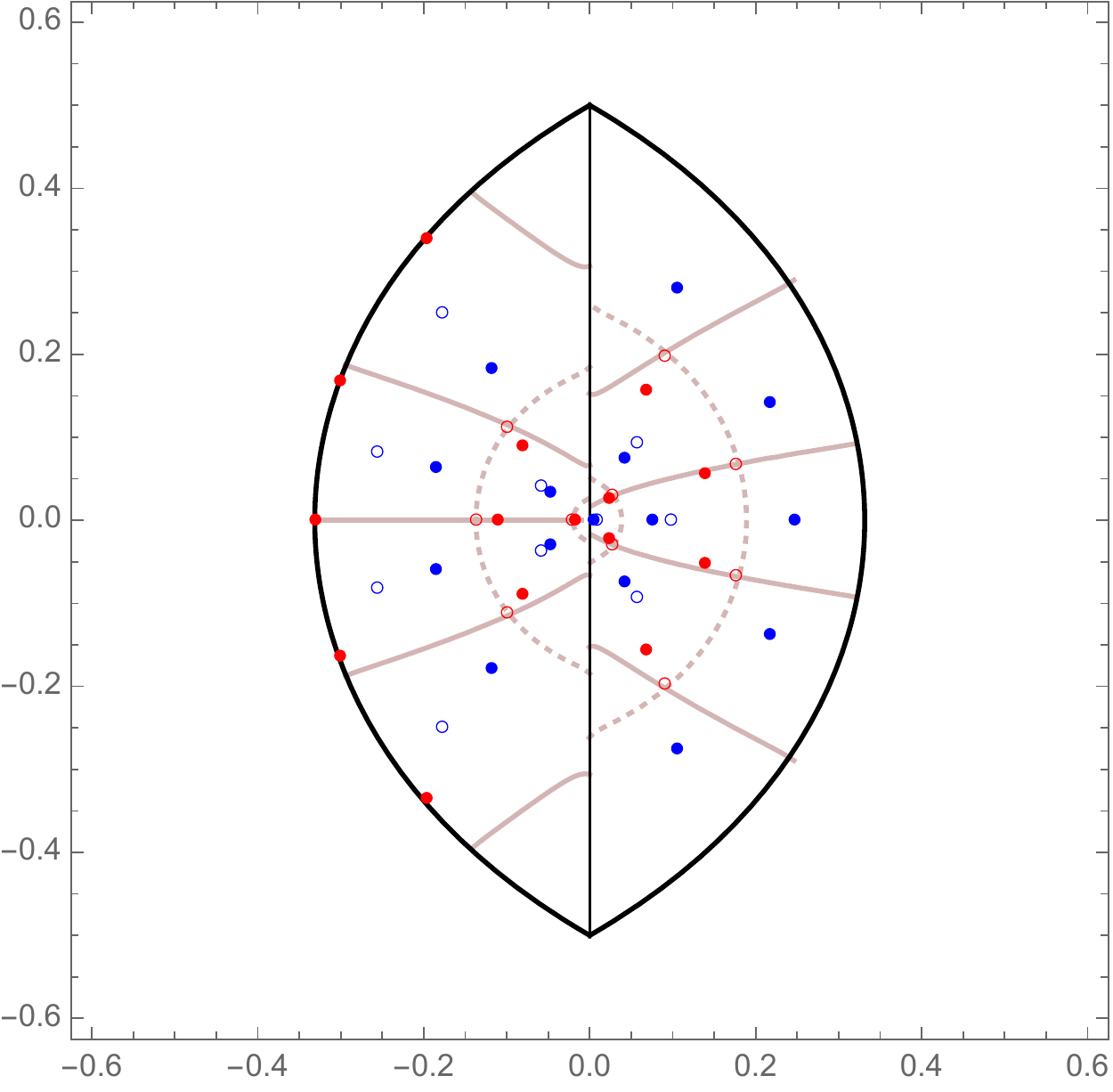}\hfill%
\includegraphics[width=0.3\textwidth]{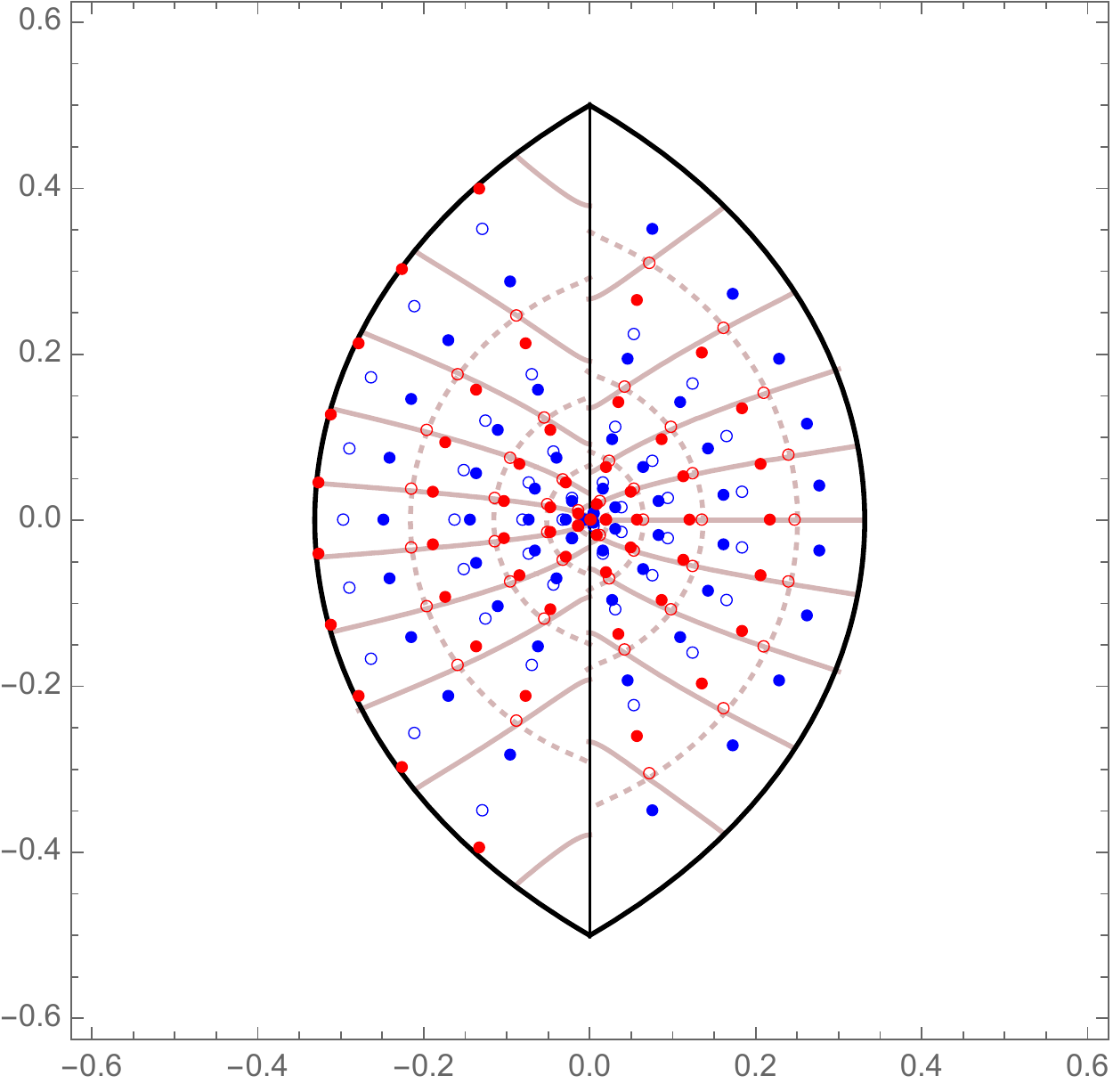}\hfill%
\includegraphics[width=0.3\textwidth]{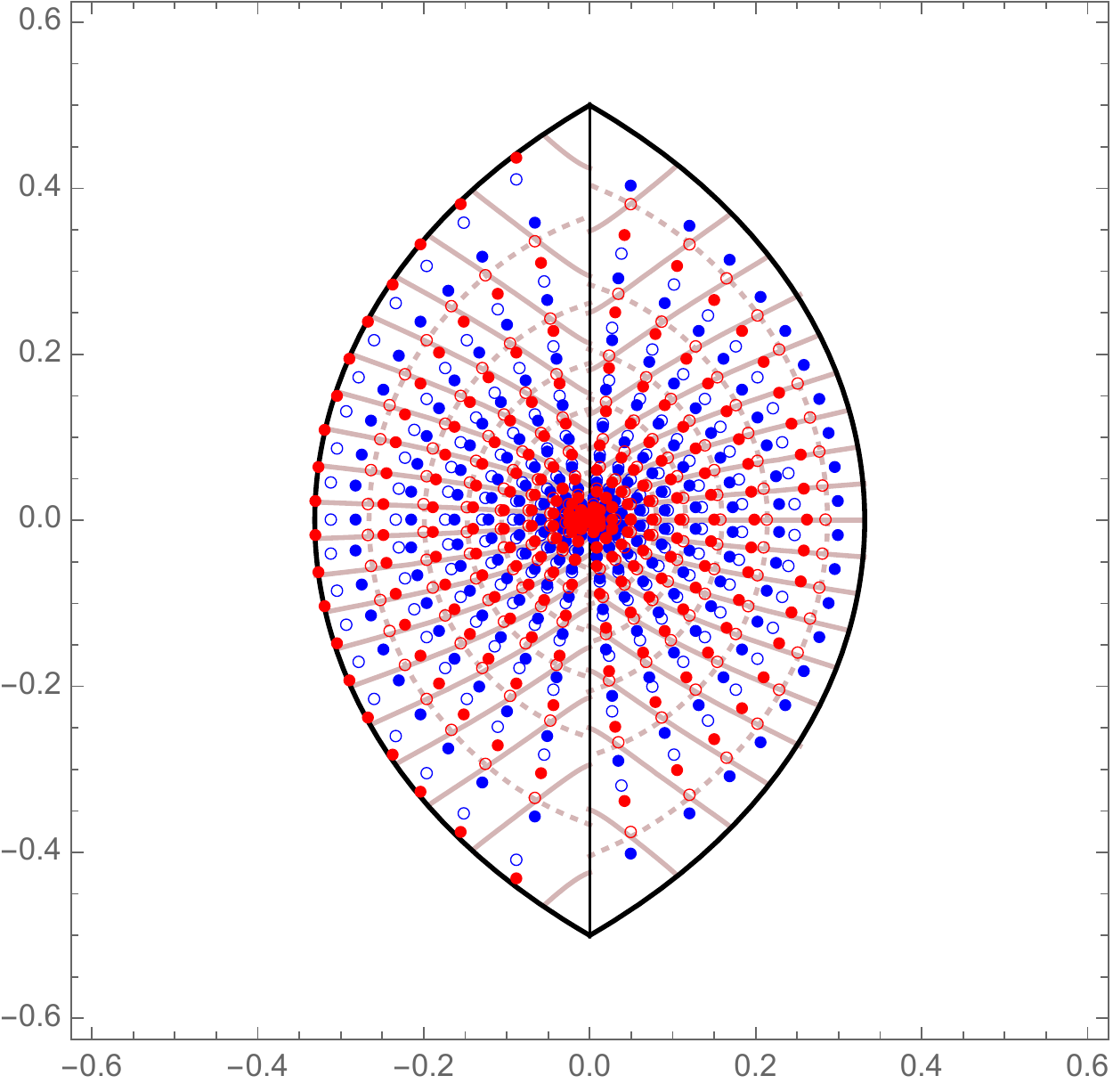}%
\end{center}
\caption{As in Figure~\ref{fig:m0-PolesMinus} but for $m=\tfrac{1}{4}$.}
\label{fig:m1Over4-PolesMinus}
\end{figure}%

Another feature of the plots in Figures~\ref{fig:m0-ZerosPlus}--\ref{fig:m1Over4-PolesMinus} is that only one pole or zero is evidently attracted to each crossing point of the curves, which suggests that the excess pairing phenomenon that cannot be ruled out by our index-based proof of Corollary~\ref{corollary:eye-zeros-and-poles:better} does in fact not occur.  Finally, these plots illustrate the most important properties of the pole/zero density function $\rho(y)$ described in Theorem~\ref{theorem:density}, namely the infinite density at the origin and the dilution of poles/zeros near the boundaries of $\partial E_\mathrm{L}$ and $\partial E_\mathrm{R}$ (which include the imaginary axis vertically bisecting $E$).

Clearly, when $m\in\mathbb{C}\setminus(\mathbb{Z}+\tfrac{1}{2})$ there are many poles and zeros in the domains $E_\mathrm{L}$ and $E_\mathrm{R}$ when $n$ is large, and in this situation we say that the \emph{eye is open}.  
On the other hand, the large-$n$ asymptotic behavior of $u_n(x;m)$ when $n^{-1}x$ is in a neighborhood of the eye $E$ is \emph{completely different} than described above when $m\in\mathbb{Z}+\tfrac{1}{2}$.  We refer to the closures (i.e., including endpoints) of the arcs of $\partial E_\mathrm{L}$ and $\partial E_\mathrm{R}$ in the open left and right half-planes respectively as the ``eyebrows'' of the eye $E$, denoting them by $\pEZeroBlue$ and $\pEInftyRed$, respectively.
Our first result is that, in a sense, the eye is closed when $m\in\mathbb{Z}+\tfrac{1}{2}$.
\begin{theorem}[Equilibrium asymptotics of $u_n(x;m)$ for $m\in\mathbb{Z}+\tfrac{1}{2}$]
Suppose that $m=-(k+\tfrac{1}{2})$ for $k\in\mathbb{Z}_{\ge 0}$.  Let $K\subset\mathbb{C}\setminus\pEInftyRed$ be bounded away from $\pEInftyRed$, i.e., $\mathrm{dist}(y,\pEInftyRed)>0$.  Then
\begin{equation}
\frac{1}{u_n(ny;m)}=\frac{1}{\ii\lNaughtInftyRed(y)} + \mathcal{O}(n^{-1}),\quad n\to+\infty,\quad y\in K,
\end{equation}
where $\lNaughtInftyRed(y)$ denotes the meromorphic continuation of $\lNaught(y)$ from a neighborhood of $y=\infty$ to the maximal domain $\mathbb{C}\setminus\pEInftyRed$  as a non-vanishing function whose only singularity is a simple pole at the origin $y=0$, and the error estimate is uniform for $y\in K$.  Likewise, if $m=k+\tfrac{1}{2}$ for $k\in\mathbb{Z}_{\ge 0}$ and $K\subset\mathbb{C}\setminus\pEZeroBlue$ is bounded away from $\pEZeroBlue$, then
\begin{equation}
u_n(ny;m)=\ii\lNaughtZeroBlue(y)+\mathcal{O}(n^{-1}),\quad n\to+\infty,\quad y\in K,
\end{equation}
where $\lNaughtZeroBlue(y)$ denotes the analytic continuation of $\lNaught(y)$ from a neighborhood of $y=\infty$ to the maximal domain $\mathbb{C}\setminus\pEZeroBlue$ as a function whose only zero is simple and lies at the origin, and the error estimate is uniform for $y\in K$.
\label{theorem:closed-eye-equilibrium}
\end{theorem}
The functions $\lNaughtInftyRed(y)$ and $\lNaughtZeroBlue(y)$ both agree with $\lNaught(y)$ for $y\in\mathbb{C}\setminus E$, and they are reciprocals of one another when $y\in E$.
Theorem~\ref{theorem:closed-eye-equilibrium} is proved in Section~\ref{sec:Half-Integer-Away-From-Edge}.  Note that this result is consistent with Theorem~\ref{theorem:outside}, which does not require any condition on $m\in\mathbb{C}$.  Moreover, it gives a far-reaching generalization of Theorem~\ref{theorem:outside} for the special case of $m\in\mathbb{Z}+\tfrac{1}{2}$.  The uniform nature of the convergence implies that $u_n(ny;m)$ can have no poles or zeros in $K$ for sufficiently large $n$, unless the set $K$ contains the origin, in which case an index argument predicts a unique simple pole near the origin for $m=-(k+\tfrac{1}{2})$ and a unique simple zero near the origin for $m=k+\tfrac{1}{2}$.  However, it is proven in \cite{ClarksonLL16} that there is a simple pole or zero \emph{exactly} at the origin if $n$ is sufficiently large (given $k\in\mathbb{Z}_{\ge 0}$).  Therefore, we have the following.
\begin{corollary}
Suppose that $m=-(k+\tfrac{1}{2})$, $k\in\mathbb{Z}_{\ge 0}$.  If $K\subset\mathbb{C}$ is bounded away from $\pEInftyRed$, then
$u_n(\cdot;m)$ has no zeros or poles in the set $nK$ for $n$ sufficiently large, except for a simple pole at the origin.  On the other hand, if $m=k+\tfrac{1}{2}$, $k\in\mathbb{Z}_{\ge 0}$ and $K\subset\mathbb{C}$ is bounded away from $\pEZeroBlue$, then $u_n(\cdot;m)$ has no zeros or poles in the set $nK$ for $n$ sufficiently large, except for a simple zero at the origin.
\label{corollary:half-integer-m-outside-zeros-poles}
\end{corollary}
This result can be combined with Theorem~\ref{theorem:closed-eye-equilibrium} to show immediately as in \eqref{eq:u-prime-outside} that the convergence of $u_n(ny;m)$ for $y\in K$ extends to all derivatives.
Corollary~\ref{corollary:half-integer-m-outside-zeros-poles} also shows that if $m\in\mathbb{Z}+\tfrac{1}{2}$, all of the poles/zeros but one are attracted toward one or the other of the eyebrows as $n\to+\infty$, depending on the sign of $m$; this is what we mean when we say that the \emph{eye is closed}.  Counting arguments suggest it is reasonable that the poles and zeros should be organized near curves rather than in a two-dimensional area such as $E_\mathrm{L}\cup E_\mathrm{R}$ in this case.  Indeed, in \cite{ClarksonLL16} it is also shown that the total number of zeros and poles of $u_n(x;m)$ scales as $n$ as $n\to+\infty$ when $m\in\mathbb{Z}+\tfrac{1}{2}$, while for $m\in\mathbb{C}\setminus(\mathbb{Z}+\tfrac{1}{2})$ the number scales as $n^2$.  Our methods allow for the following precise statement concerning the nature of convergence of the poles/zeros to one or the other of the eyebrows for $m\in\mathbb{Z}+\tfrac{1}{2}$. 
The following results refer to a ``tubular neighborhood'' $T$ of the eyebrow $\pEInftyRed$ defined as follows: for sufficiently small positive constants $\delta_1$ and $\delta_2$, 
\begin{equation}
T=T_{\delta_1,\delta_2}:=\left\{y\in\mathbb{C}: |\arg(y)|\le\frac{\pi}{2}-\delta_1,\;|\mathrm{Re}(V(\lNaught(y);y))|\le\delta_2\right\}.
\label{eq:T-define}
\end{equation}
Since points on the eyebrow $\pEInftyRed$ satisfy $\mathrm{Re}(V(\lNaught(y));y)=0$, the set $T$ contains points on both sides of $\pEInftyRed$, and the angular condition bounds the set $T$ away from the endpoints $y=\pm\tfrac{1}{2}\ii$ of $\pEInftyRed$.  Note that $V(\lNaught(y)^{-1};y)=-V(\lNaught(y);y)\pmod{2\pi\ii}$.
\begin{theorem}[Layered trigonometric asymptotics of $u_n(x;m)$ for $m\in\mathbb{Z}+\tfrac{1}{2}$]
Let $m=-(\tfrac{1}{2}+k)$, $k\in\mathbb{Z}_{\ge 0}$, and let 
$T$ be as defined in \eqref{eq:T-define}.  Then the following asymptotic formul\ae\ hold in which the error terms are uniform on the indicated sub-domains of $T$ from which small discs of radius proportional to an arbitrarily small multiple of $n^{-1}$
centered at each zero or pole of the indicated approximation are excised:
\begin{itemize}
\item If $y\in T$ with $\mathrm{Re}(V(\lNaught(y)^{-1};y))\le -\tfrac{1}{2}kn^{-1}\ln(n)$, then $u_n(ny;m)=\dot{u}_n+\mathcal{O}(n^{-1})$ where $\dot{u}_n$ is given explicitly by \eqref{eq:dot-u-UU-left}.
\item For $\ell=1,\dots,k$,
\begin{itemize}
\item
If $y\in T$ with $-\tfrac{1}{2}(k-2\ell+2)n^{-1}\ln(n)\le\mathrm{Re}(V(\lNaught(y)^{-1};y))\le -\tfrac{1}{2}(k-2\ell+\tfrac{3}{2})n^{-1}\ln(n)$, then $u_n(ny;m)=\dot{u}_n+\mathcal{O}(n^{-1/2})$ where $\dot{u}_n$ is given explicitly by \eqref{eq:dot-u-UU-first}.
\item
If $y\in T$ with $-\tfrac{1}{2}(k-2\ell+\tfrac{3}{2})n^{-1}\ln(n)\le\mathrm{Re}(V(\lNaught(y)^{-1};y))\le-\tfrac{1}{2}(k-2\ell+\tfrac{1}{2})n^{-1}\ln(n)$, then $u_n(ny;m)=\dot{u}_n+\mathcal{O}(n^{-1/2})$ where $\dot{u}_n$ is given explicitly by \eqref{eq:dot-u-UL-1} or \eqref{eq:dot-u-UL-2}.
\item
If $y\in T$ with $-\tfrac{1}{2}(k-2\ell+\tfrac{1}{2})n^{-1}\ln(n)\le\mathrm{Re}(V(\lNaught(y)^{-1};y))\le-\tfrac{1}{2}(k-2\ell)n^{-1}\ln(n)$, then $u_n(ny;m)=\dot{u}_n+\mathcal{O}(n^{-1/2})$ where $\dot{u}_n$ is given explicitly by \eqref{eq:dot-u-UU-last}.
\end{itemize}
\item
If $y\in T$ with $\mathrm{Re}(V(\lNaught(y)^{-1};y))\ge \tfrac{1}{2}kn^{-1}\ln(n)$, then $u_n(ny;m)=\dot{u}_n+\mathcal{O}(n^{-1})$ where $\dot{u}_n$ is given explicitly by \eqref{eq:dot-u-UU-right}.
\end{itemize}
These results imply corresponding asymptotic formul\ae\ for $u_n(ny;m)$ if $m=\tfrac{1}{2}+k$, $k\in\mathbb{Z}_{\ge 0}$ by the exact symmetry \eqref{eq:u-n-exact-symmetry}; in particular the eyebrow near which the asymptotics are nontrivial is then the left one, $\pEZeroBlue$.
\label{thm:edge-formulae}
\end{theorem}
The inequalities on $y$ in the statement of the theorem describe a dissection of $T$ into finitely-many (depending on $k$) ``layers'' roughly parallel to the right eyebrow $\pEInftyRed$ and overlapping at their common boundaries.  The order of the layers as written in the theorem corresponds to $y$ crossing $\pEInftyRed$ from inside $E$ to outside, and the ``interior'' layers described by the index $\ell$ are each of width proportional to $n^{-1}\ln(n)$.  The approximation $\dot{u}_n$ assigned to each layer is a fractional linear (M\"obius) function of $n^\beta\ee^{2nV(\lNaught(y)^{-1};y)}$ where the power $\beta$ and the coefficients of the linear expressions in the numerator/denominator depend on the layer.  The latter coefficients are relatively slowly-varying functions of $y$ alone that are explicitly built from $\lNaught(y)$, and hence the dominant local behavior in any given layer is essentially trigonometric with respect to $y$.  We wish to stress that, unlike the approximation formula \eqref{eq:udot-elliptic} whose ingredients involve implicitly-defined functions of $y\in E_\mathrm{R}$ and elements of algebraic geometry, the approximation $\dot{u}_n$ in each layer is an elementary function of $V(\lambda;y)$ and $\lNaught(y)$.  In particular, it is easy to check that when $y$ is in the innermost or outermost layers but bounded away from $\pEInftyRed$ (the ``overlap domain''), Theorem~\ref{thm:edge-formulae} is consistent with Theorem~\ref{theorem:closed-eye-equilibrium}.\bigskip 

The analogue of Corollary~\ref{corollary:eye-zeros-and-poles:better} in the present context is the following.
\begin{corollary}
Let $m=-(k+\tfrac{1}{2})$, $k\in\mathbb{Z}_{\ge 0}$, and let $T$ be defined as in \eqref{eq:T-define}.  If $\{y_n\}_{n=N}^\infty\subset T$ is a sequence for which $y_n$ is a zero of $\dot{u}_n$ for all $n\ge N$,
then for each $\epsilon>0$ sufficiently small there is exactly one simple zero, and possibly a group of an equal number of additional zeros and poles, of $u_n(ny;m)$ within $|y-y_n|<\epsilon n^{-1}$ for $n$ sufficiently large.  Likewise, if $\{y_n\}_{n=N}^\infty\subset T$ is a sequence for which $y_n$ is a pole of $\dot{u}_n$ for all $n\ge N$,
then for each $\epsilon>0$ sufficiently small there is exactly one simple pole, and possibly a group of an equal number of additional zeros and poles, of $u_n(ny;m)$ within $|y-y_n|<\epsilon n^{-1}$ for $n$ sufficiently large.
\label{corollary:eyebrow-zeros-and-poles}
\end{corollary}
As before, we suspect that with additional work one should be able to preclude the excess pairing phenomenon, so that the poles and zeros of $u_n(x;m)$ and its approximation $\dot{u}_n$ are in one-to-one correspondence.  Now in each layer of $T$, the poles and zeros of $\dot{u}_n$ are easily seen to lie exactly along certain explicit curves roughly parallel to the eyebrow.
\begin{theorem}
Suppose that $m=-(\tfrac{1}{2}+k)$, $k\in\mathbb{Z}_{\ge 0}$ and let $T$ be as in \eqref{eq:T-define}.  The zeros and poles of the piecewise-meromorphic approximating function $\dot{u}_n$ on $T$ lie on a system of $4k+2$ non-intersecting curves roughly parallel to the eyebrow $\pEInftyRed$.  From left-to-right, these are:
\begin{itemize}
\item a curve of poles given by \eqref{eq:left-pole-curve}
\item a curve of zeros given by \eqref{eq:left-zero-curve}
\item For $\ell=1,\dots,k$,
\begin{itemize}
\item a curve of zeros given by \eqref{eq:mid-UL-zero-curve}
\item a curve of poles given by \eqref{eq:mid-UL-pole-curve}
\item a curve of poles given by \eqref{eq:mid-UU-pole-curve}
\item a curve of zeros given by \eqref{eq:mid-UU-zero-curve}.
\end{itemize}
\end{itemize}
Analogous results hold for the approximation to $u_n(ny;m)$ for $m=\tfrac{1}{2}+k$, $k\in\mathbb{Z}_{\ge 0}$, obtained from $\dot{u}_n$ via the symmetry \eqref{eq:u-n-exact-symmetry} ($y\mapsto -y$, $m\mapsto -m$, $\dot{u}_n\mapsto \dot{u}_n^{-1}$).
\label{theorem:eyebrow-curves}
\end{theorem}
Corollary~\ref{corollary:eyebrow-zeros-and-poles} and Theorem~\ref{theorem:eyebrow-curves} are proved in Section~\ref{sec:eyebrow-zeros-and-poles}.  To illustrate the accuracy of these results, we compare the exact locations of zeros and poles of $u_n(ny;m)$ for $m=-(k+\tfrac{1}{2})$, $k\in\mathbb{Z}_{\ge 0}$, with the curves described in Theorem~\ref{theorem:eyebrow-curves} in Figures~\ref{fig:EdgeCurves-k0}--\ref{fig:EdgeCurves-k2}.
\begin{figure}[h]
\begin{center}
\includegraphics[width=0.3\linewidth]{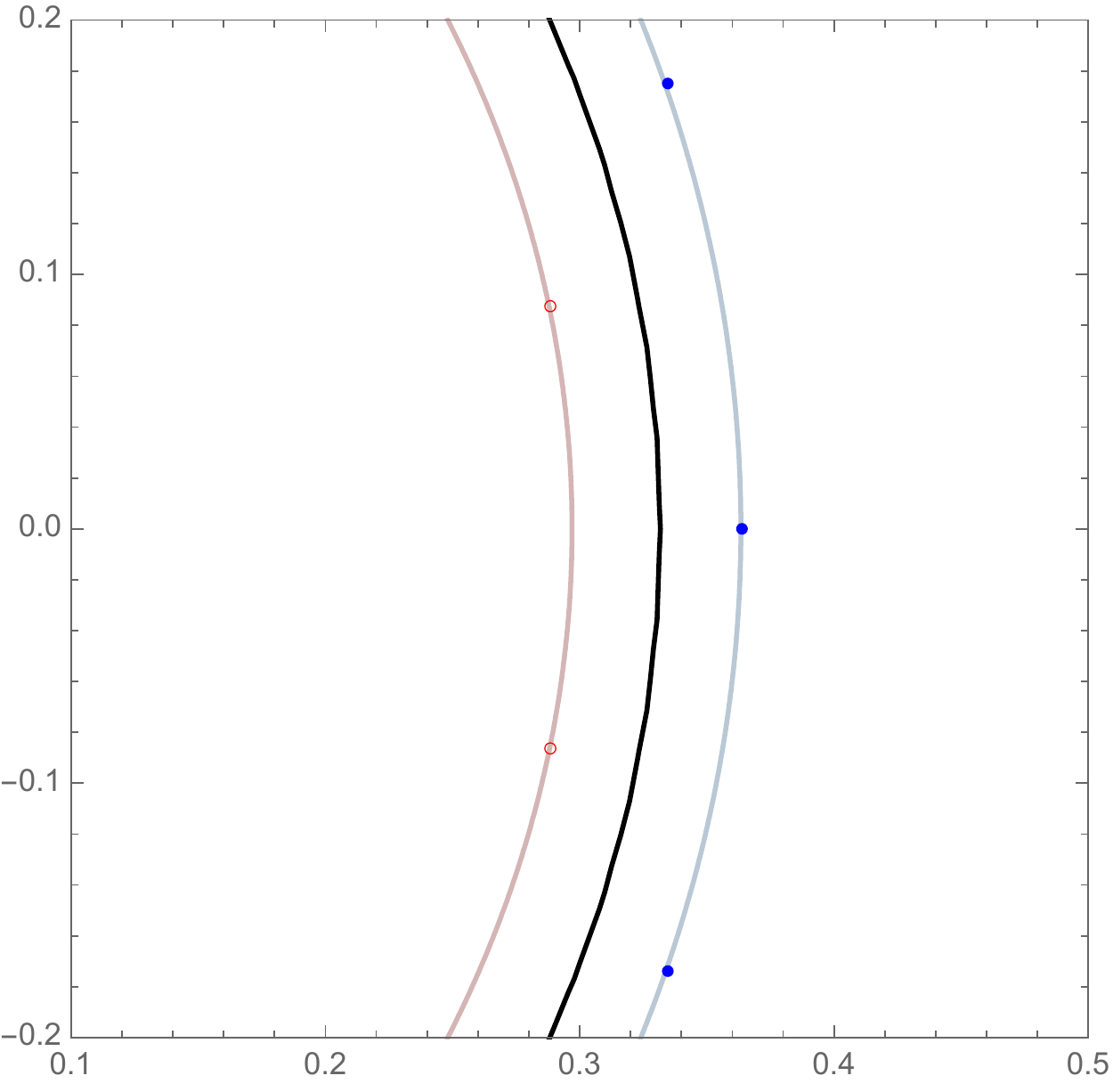}%
\hspace{0.04\linewidth}%
\includegraphics[width=0.3\linewidth]{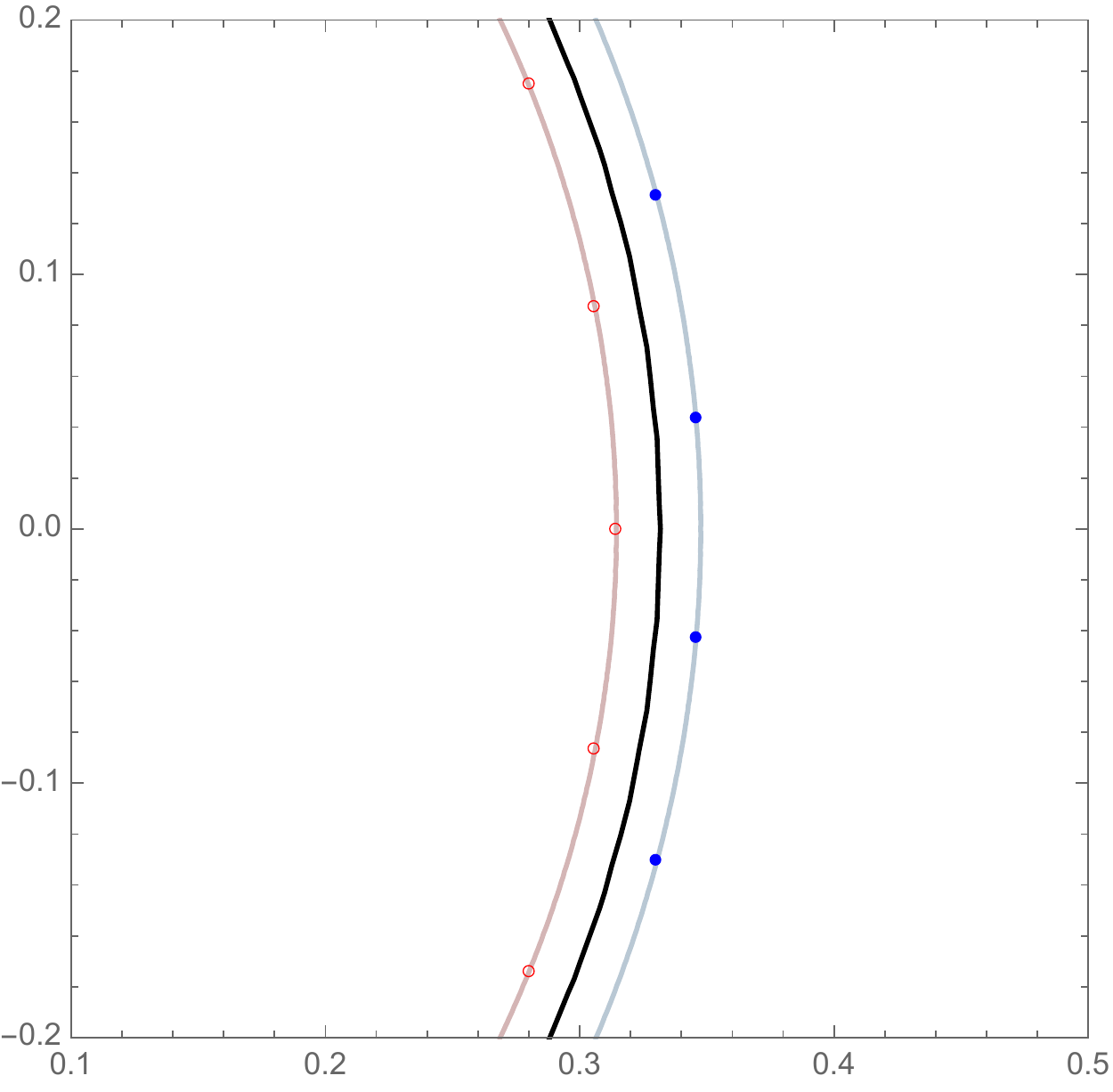}%
\hspace{0.04\linewidth}%
\includegraphics[width=0.3\linewidth]{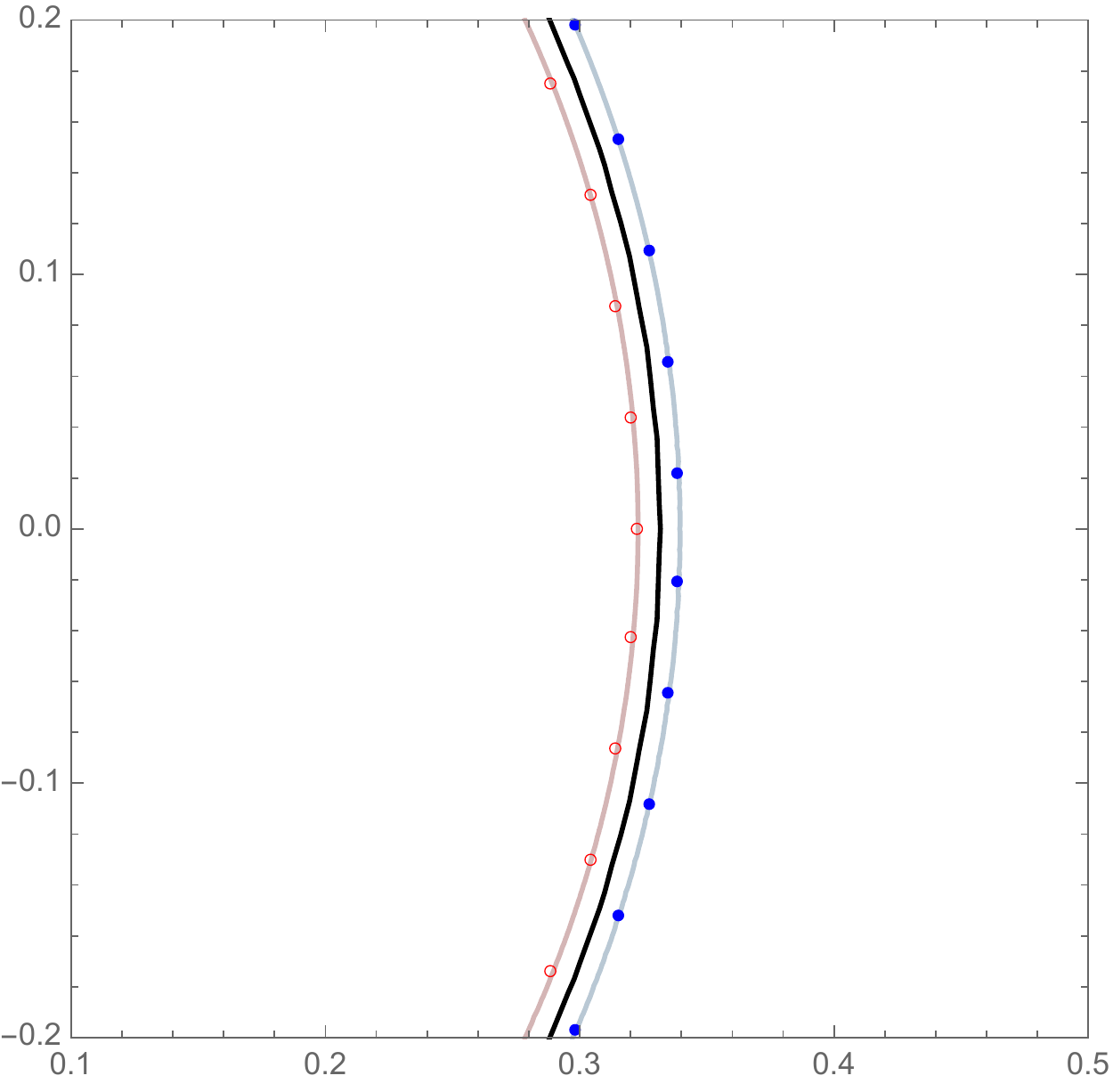}%
\end{center}
\caption{The pole (red) and zero (blue) curves of $\dot{u}$ for $k=0$ and $n=5,10,20$ from left-to-right, shown together with the actual poles (red dots) and zeros (blue dots) of $u_n(ny;-(\tfrac{1}{2}+k))$ and the eyebrow $\pEInftyRed$ (black curve).}
\label{fig:EdgeCurves-k0}
\end{figure}
\begin{figure}[h]
\begin{center}
\includegraphics[width=0.3\linewidth]{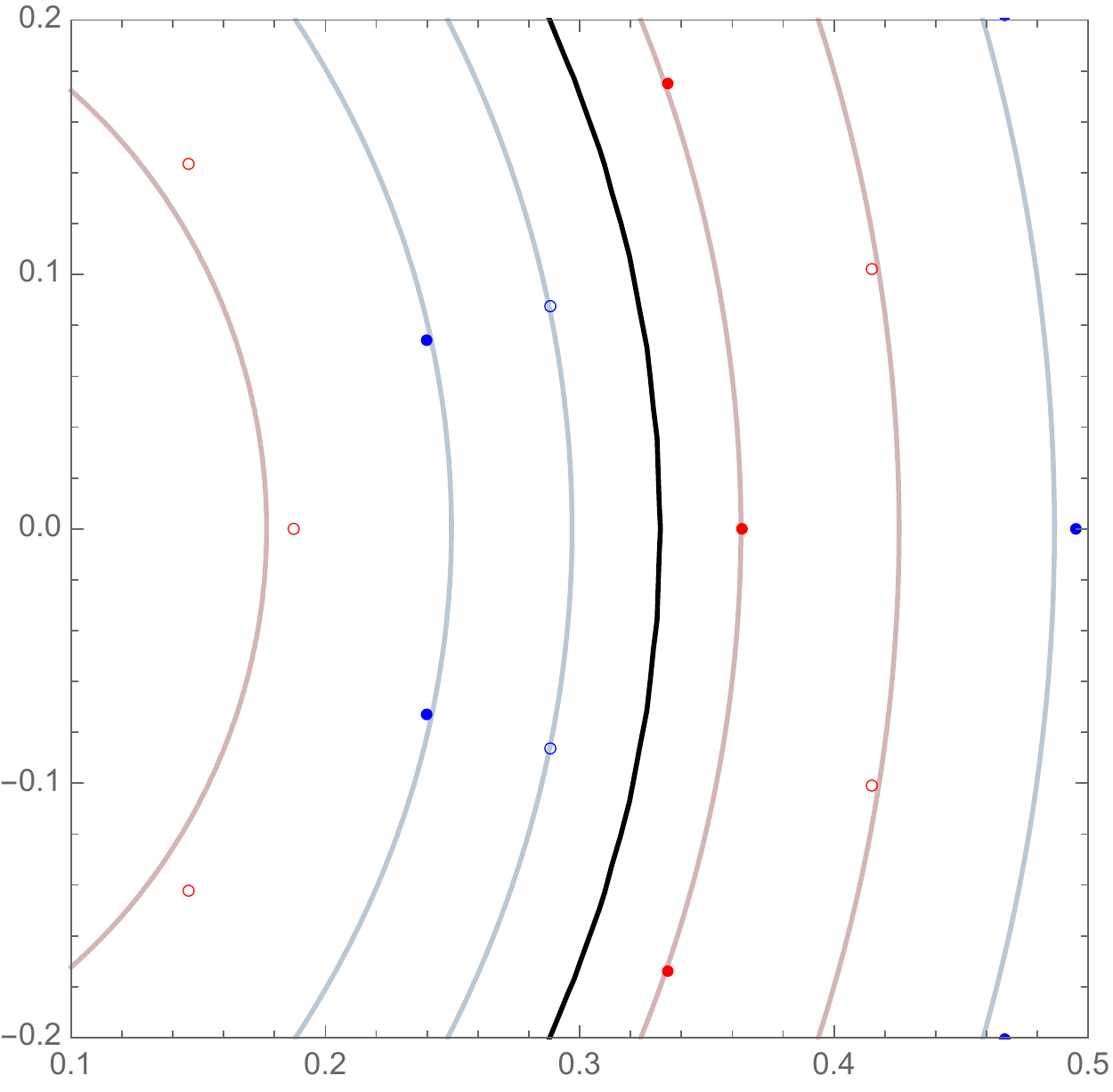}%
\hspace{0.04\linewidth}%
\includegraphics[width=0.3\linewidth]{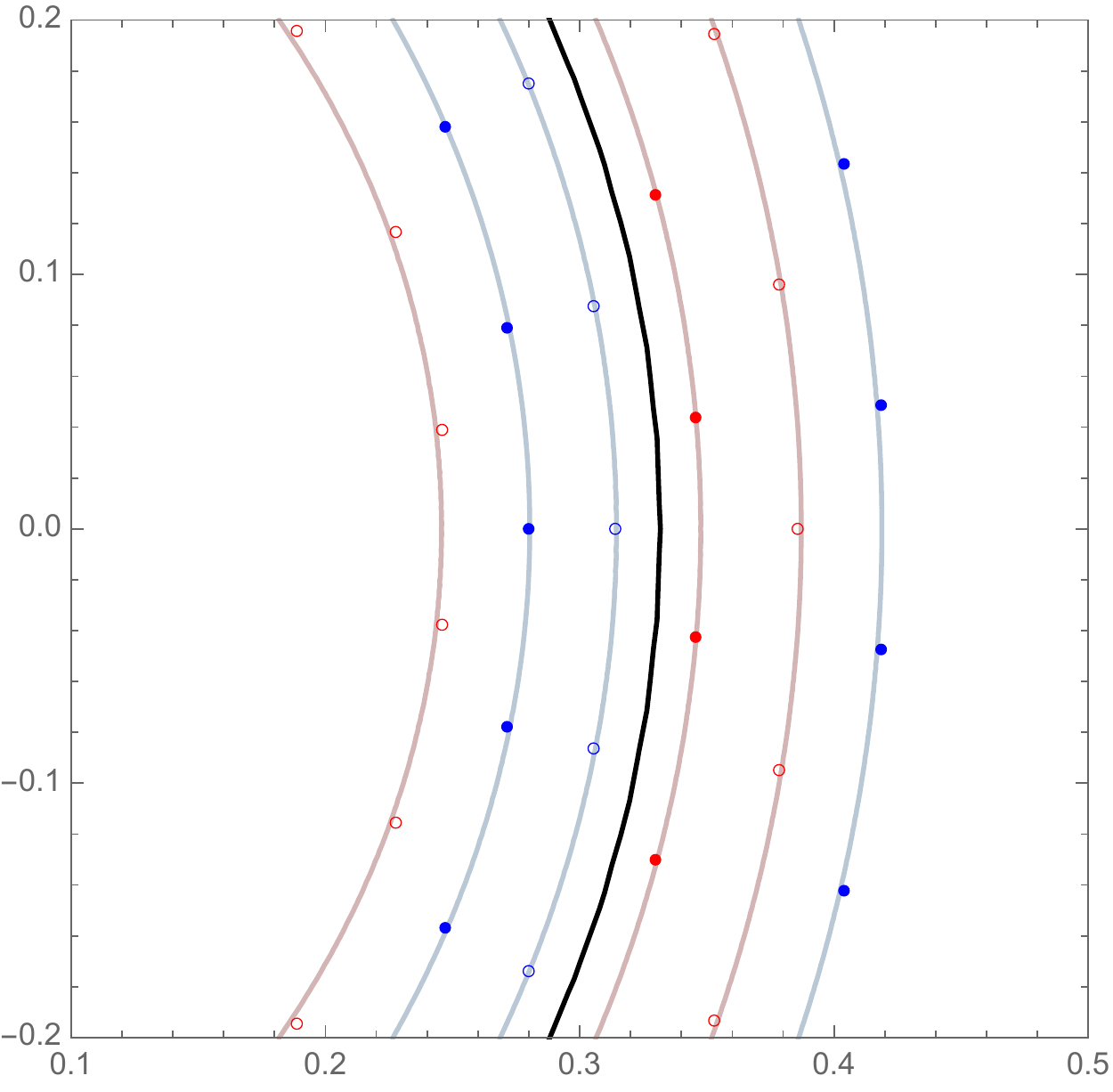}%
\hspace{0.04\linewidth}%
\includegraphics[width=0.3\linewidth]{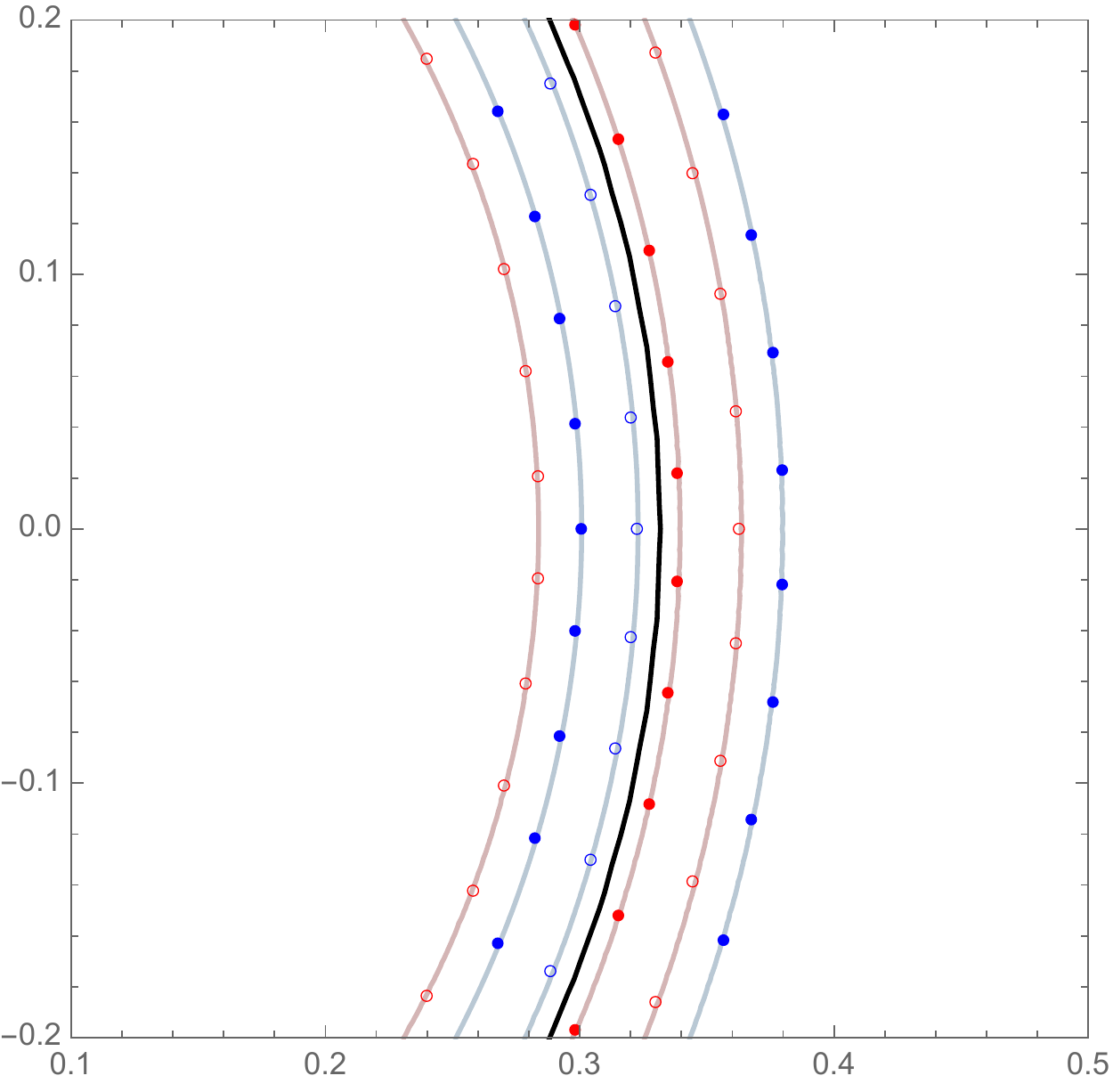}%
\end{center}
\caption{As in Figure~\ref{fig:EdgeCurves-k0} but for $k=1$.}
\label{fig:EdgeCurves-k1}
\end{figure}
\begin{figure}[h]
\begin{center}
\includegraphics[width=0.3\linewidth]{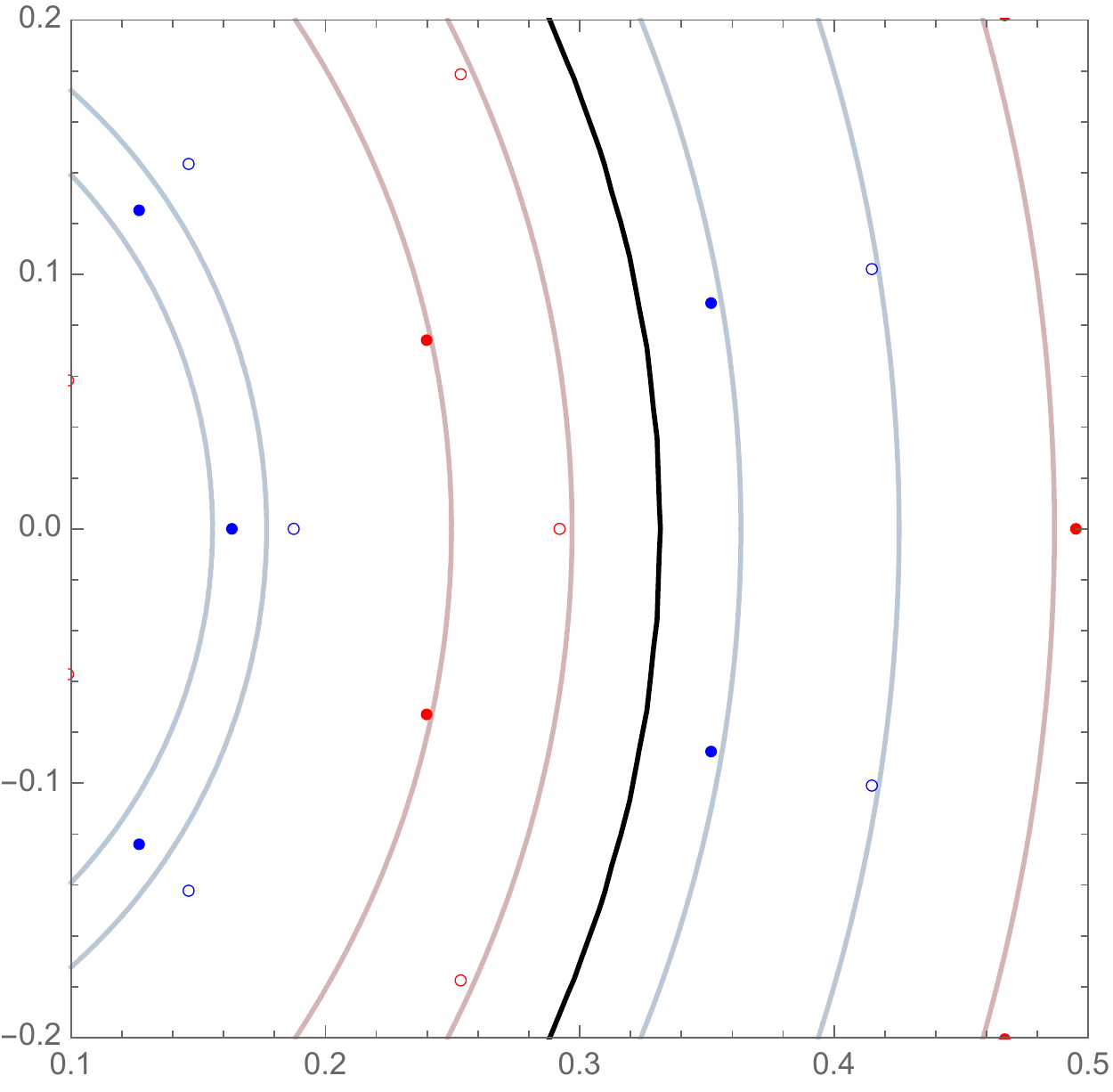}%
\hspace{0.04\linewidth}%
\includegraphics[width=0.3\linewidth]{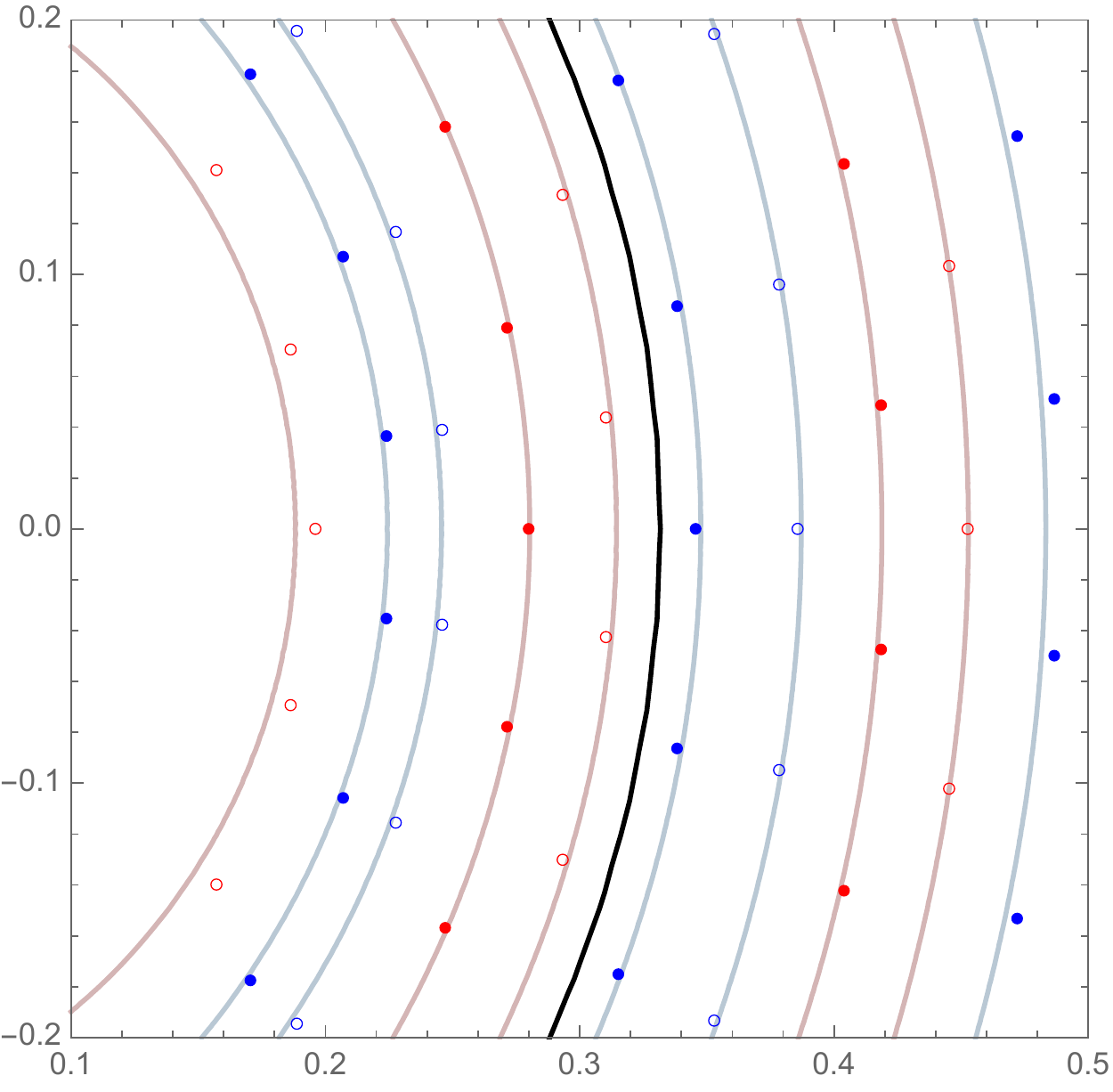}%
\hspace{0.04\linewidth}%
\includegraphics[width=0.3\linewidth]{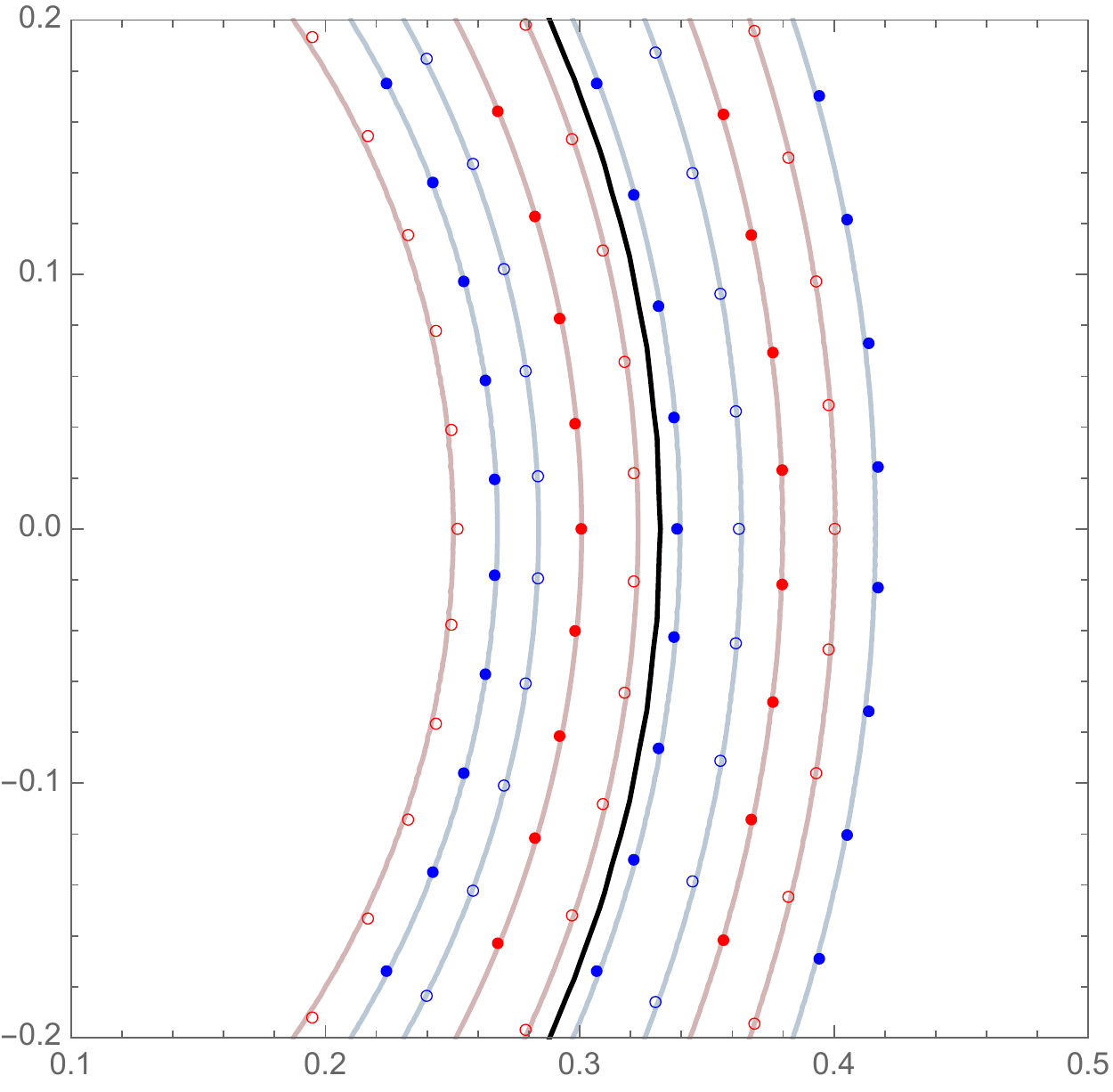}%
\end{center}
\caption{As in Figure~\ref{fig:EdgeCurves-k0} but for $k=2$.}
\label{fig:EdgeCurves-k2}
\end{figure}
In addition to illustrating the accuracy of the approximation by $\dot{u}_n$, these figures demonstrate another phenomenon for which we do not yet have an explanation:  for any given curve, the poles/zeros attracted are those contributed by exactly one of the four polynomial factors in \eqref{eq:un-exact-fraction}.  Furthermore, there appears again to be no excess pairing of poles and zeros.\bigskip

Evidently, the large-$n$ asymptotic behavior of $u_n(x;m)$ is completely different for $m=\pm (k+\tfrac{1}{2})$, $k\in\mathbb{Z}_{\ge 0}$, and for $m=\pm (k+\tfrac{1}{2})+\epsilon$, however small $\epsilon\neq 0$ is.  In other words, even crude aspects of the large-$n$ asymptotic behavior of $u_n(x;m)$ for $n^{-1}x$ in a neighborhood of the eye $E$ fail to be uniformly valid with respect to the second parameter $m$ near half-integer values of the latter.   Thus, given $m\in\mathbb{C}$, the eye is either open or closed in the large-$n$ limit.  On the other hand, the polynomials $s_n(x;m)$ in the formula \eqref{eq:un-exact-fraction} are actually polynomials in both arguments $x$ and $m$ \cite{ClarksonLL16}, and in this sense the limits of $n\to+\infty$ and $m\to \mathbb{Z}+\tfrac{1}{2}$ do not commute. Capturing the process of the closing of the eye requires connecting $m$ with $n$ in a suitable \emph{double-scaling limit} so that $m$ tends to a given half-integer as $n\to+\infty$.  In a subsequent paper, we will show that in the right double-scaling limit, \emph{all three types of solutions of the autonomous model equation \eqref{eq:autonomous} play a role in describing $u_n(ny;m)$ as $n\to+\infty$.}    



\section{Spectral Curve and $g$-function}
%
When $n$ is large, the exponential factors $\ee^{\pm nV(\lambda;y)}$ appearing in the jump conditions \eqref{eq:Yjump-1}--\eqref{eq:Yjump-4} need to be balanced in general by some compensating factors that can be used to control exponential growth.  We therefore introduce a ``$g$-function'' $g(\lambda;y)$ that is taken to be bounded and analytic in $\mathbb{C}\setminus L$ with $g(\lambda;y)\to g_\infty(y)$ as $\lambda\to\infty$ for some $g_\infty(y)$ to be determined, and we set
\begin{equation}
\mathbf{M}_n(\lambda;y,m):=\ee^{ng_\infty(y)\sigma_3}\mathbf{Y}_n(\lambda;ny,m)\ee^{-ng(\lambda;y)\sigma_3}.
\label{eq:Y-M-g-function}
\end{equation}
Thus, representing \eqref{eq:Yjump-1}--\eqref{eq:Yjump-4} in the general form $\mathbf{Y}_{n+}(\lambda;x,m)=\mathbf{Y}_{n-}(\lambda;x,m)\mathbf{V}(\lambda;x,m)$, we obtain the corresponding jump conditions for $\mathbf{M}_n(\lambda;y,m)$ in the form $\mathbf{M}_{n+}(\lambda;y,m)=\mathbf{M}_{n-}(\lambda;y,m)\ee^{ng_-(\lambda;y)\sigma_3}\mathbf{V}(\lambda;ny,m)\ee^{-ng_+(\lambda;y)\sigma_3}$.  Noting that
\begin{equation}
\ee^{ng_-(\lambda;y)\sigma_3}\mathbf{V}(\lambda;ny,m)\ee^{-ng_+(\lambda;y)\sigma_3}=
\begin{bmatrix}
\ee^{-n(g_+(\lambda;y)-g_-(\lambda;y))}V_{11}(\lambda;ny,m) & \ee^{n(g_+(\lambda;y)+g_-(\lambda;y))}V_{12}(\lambda;ny,m)\\
\ee^{-n(g_+(\lambda;y)+g_-(\lambda;y))}V_{21}(\lambda;ny,m) & \ee^{n(g_+(\lambda;y)-g_-(\lambda;y))}V_{22}(\lambda;ny,m)\end{bmatrix}
\end{equation}
we place the following conditions on $g$.  We want $g$ to be chosen so that $L$ can be deformed and then split into several arcs along each of which one of the following alternatives holds (recall that $V$ is defined by \eqref{eq:V-define}):
\begin{itemize}
\item $g_+(\lambda;y)-g_-(\lambda;y)=\ii K$ where $K\in\mathbb{R}$ is constant (implying that $g'(\lambda;y)$ has no jump discontinuity across the arc), and $\mathrm{Re}(2g_\pm(\lambda;y)-V(\lambda;y))<0$, or
\item $g_+(\lambda;y)+g_-(\lambda;y)-V(\lambda;y)=\ii K$ where $K\in\mathbb{R}$ is constant (implying that $g_+'(\lambda;y)+g_-'(\lambda;y)-V'(\lambda;y)=0$ holds along the arc), 
while $\mathrm{Re}(2g(\lambda;y)-V(\lambda;y))>0$ on both sides of the arc, or
\item $g_+(\lambda;y)+g_-(\lambda;y)-V(\lambda;y)=\ii K$ where $K\in\mathbb{R}$ is constant (implying that $g_+'(\lambda;y)+g_-'(\lambda;y)-V'(\lambda;y)=0$ holds along the arc),
while $\mathrm{Re}(2g(\lambda;y)-V(\lambda;y))<0$ on both sides of the arc, or
\item $g_+(\lambda;y)-g_-(\lambda;y)=\ii K$ where $K\in\mathbb{R}$ is constant (implying that $g'(\lambda;y)$ has no jump discontinuity across the arc), and $\mathrm{Re}(2g_\pm(\lambda;y)-V(\lambda;y))>0$.
\end{itemize}
The real constant $K$ will generally be different in each maximal arc.

\subsection{The spectral curve and its degenerations}
If we assume that $g'(\lambda;y)$ has a finite number of arcs of discontinuity along $L\setminus\{0\}$, then obviously $(g'(\lambda;y)-\tfrac{1}{2}V'(\lambda;y))_+=(g'(\lambda;y)-\tfrac{1}{2}V'(\lambda;y))_-$ except along these arcs.  Along the arcs of discontinuity where instead the condition $g_+(\lambda;y)+g_-(\lambda;y)-V(\lambda;y)=\ii K$ holds, by differentiation along the arc we have $(g'(\lambda;y)-\tfrac{1}{2}V'(\lambda;y))_+=-(g'(\lambda;y)-\tfrac{1}{2}V'(\lambda;y))_-$.  It follows that $(g'(\lambda;y)-\tfrac{1}{2}V'(\lambda;y))^2$ is an analytic function of $\lambda$ except at $\lambda=0$, which is the only singularity of $V'(\lambda;y)$.  Now since $g'(\lambda;y)=\mathcal{O}(\lambda^{-2})$ as $\lambda\to\infty$ and $g'(\lambda;y)=\mathcal{O}(1)$ as $\lambda\to 0$, it follows that
\begin{equation}
g'(\lambda;y)-\frac{1}{2}V'(\lambda;y)=\begin{cases} \displaystyle \frac{\ii y}{2}\lambda^{-2} + \frac{1}{2}\lambda^{-1} + \mathcal{O}(1),&\quad\lambda\to 0\smallskip\\
\displaystyle \frac{\ii y}{2} + \frac{1}{2}\lambda^{-1} + \mathcal{O}(\lambda^{-2}),&\quad\lambda\to\infty
\end{cases},
\label{eq:gprimeminushalfVprime-asymp}
\end{equation}
and hence if $y\neq 0$,
\begin{equation}
\left(g'(\lambda;y)-\frac{1}{2}V'(\lambda;y)\right)^2 = \begin{cases}
\displaystyle -\frac{y^2}{4}\lambda^{-4} + \frac{\ii y}{2}\lambda^{-3}+\mathcal{O}(\lambda^{-2}),&\quad
\lambda\to 0\smallskip\\
\displaystyle -\frac{y^2}{4}+\frac{\ii y}{2}\lambda^{-1}+\mathcal{O}(\lambda^{-2}),&\quad\lambda\to\infty
\end{cases}.
\label{eq:gprimeminushalfVprime-squared-asymp}
\end{equation}

Therefore, if $y=0$, Liouville's theorem shows that
\begin{equation}
g'(\lambda;0)-\frac{1}{2}V'(\lambda;0)=\frac{1}{2}\lambda^{-1},
\end{equation}
while if $y\neq 0$ we necessarily have that
\begin{equation}
\left(g'(\lambda;y)-\frac{1}{2}V'(\lambda;y)\right)^2=\frac{1}{\lambda^4}P(\lambda;y,C),
\label{eq:tildeP4}
\end{equation}
where $P(\cdot;y,C)$ is the quartic polynomial defined by \eqref{eq:dotV-ODE} and it only remains to determine $C$.  
Since the zero locus of $P(\lambda;y,C)$ is obviously symmetric with respect to the involution $\lambda\mapsto \lambda^{-1}$, the following configurations for $P(\lambda;y,C)$ include all possibilities, given that $y\neq 0$:
\begin{itemize}
\item[(i)] All four roots coincide, in which case the four-fold root must lie at either $\lambda=1$ or $\lambda=-1$, i.e., $P(\lambda;y,C)=-\tfrac{1}{4}y^2(\lambda\mp 1)^4 = -\tfrac{1}{4}y^2\lambda^4\pm y^2\lambda^3 -\tfrac{3}{2}y^2\lambda^2 \pm y^2\lambda -\tfrac{1}{4}y^2$.  Comparing with \eqref{eq:dotV-ODE}, we see that this situation occurs only if $y=\pm\tfrac{1}{2}\ii$, and then only if also $C=-\tfrac{3}{2}y^2=\tfrac{3}{8}$.  In this case, since $P(\lambda;y,C)$ is a perfect square, we have either $g'(\lambda;y)-\tfrac{1}{2}V'(\lambda;y)=\tfrac{1}{2}\ii y(1\mp\lambda^{-1})^2=\tfrac{1}{2}\ii y(1\mp 2\lambda^{-1}+\lambda^{-2})$ or $g'(\lambda;y)-\tfrac{1}{2}V'(\lambda;y)=-\tfrac{1}{2}\ii y(1\mp\lambda^{-1})^2=-\tfrac{1}{2}\ii y(1\mp 2\lambda^{-1}+\lambda^{-2})$.
Since $g'(\lambda;y)=\mathcal{O}(\lambda^{-2})$ as $\lambda\to\infty$, only the former is consistent with \eqref{eq:V-define}, and then we see that in fact $g'(\lambda;y)-\tfrac{1}{2}V'(\lambda;y)=-\tfrac{1}{2}V'(\lambda;y)$, i.e., $g'(\lambda;y)=0$ in this case, which implies that $g(\lambda;y)=g_\infty(y)$.  This case turns out to be relevant exactly for $y=\pm\tfrac{1}{2}\ii$.
\item[(ii)] There are two double roots that are exchanged\footnote{That it is impossible to have two double roots that are fixed individually by the involution can be seen as follows.  It would be necessary to have one double root at $\lambda=1$ and another double root at $\lambda=-1$, and therefore $P(\lambda;y;C)=-\tfrac{1}{4}y^2(\lambda^2-1)^2=-\tfrac{1}{4}y^2\lambda^4+\tfrac{1}{2}y^2\lambda^2 -\tfrac{1}{4}y^2$.  Comparing with \eqref{eq:dotV-ODE} shows that this situation cannot occur for $y\neq 0$.} by the involution, in which case there is a number $\lNaught\neq \pm 1$ such that $P(\lambda;y;C)=-\tfrac{1}{4}y^2(\lambda-\lNaught)^2(\lambda-\lNaught^{-1})^2 = 
-\tfrac{1}{4}y^2\lambda^4+\tfrac{1}{2}y^2(\lNaught+\lNaught^{-1})\lambda^3 -\tfrac{1}{4}y^2(\lNaught^2+4+\lNaught^{-2})\lambda^2 +\tfrac{1}{2}y^2(\lNaught+\lNaught^{-1})\lambda -\tfrac{1}{4}y^2$.  Comparing with \eqref{eq:dotV-ODE} shows that this is possible for all $y\neq 0$, provided that $\lNaught$ is determined as a function of $y$ up to reciprocation by $\lNaught+\lNaught^{-1}=\ii y^{-1}$ and then $C$ is given the value $C=-\tfrac{1}{4}y^2(\lNaught^2+4+\lNaught^{-2})=-\tfrac{1}{4}(2y^2-1)$.  In this case, $P(\lambda;y,C)$ is again a perfect square and hence either $g'(\lambda;y)-\tfrac{1}{2}V'(\lambda;y)=\tfrac{1}{2}\ii y(1-\lNaught\lambda^{-1})(1-\lNaught^{-1}\lambda^{-1})$ or $g'(\lambda;y)-\tfrac{1}{2}V'(\lambda;y)=-\tfrac{1}{2}\ii y(1-\lNaught\lambda^{-1})(1-\lNaught^{-1}\lambda^{-1})$.  Only the former is consistent with \eqref{eq:V-define} given that $g'(\lambda;y)=\mathcal{O}(\lambda^{-2})$ as $\lambda\to\infty$ and again we deduce that $g'(\lambda;y)=0$ and hence also $g(\lambda;y)=g_\infty(y)$.  This turns out to be the case corresponding to $y\in\mathbb{C}\setminus E$.
\item[(iii)]
There is one double root and two simple roots, with the double root being fixed by the involution and hence occurring at $\lambda=\pm 1$ and the simple roots being permuted by the involution and hence being given by $\lambda=\lambda_0$ and $\lambda=\lambda_0^{-1}$ for some $\lambda_0\neq\pm 1$.  Thus $P(\lambda;y,C)=-\tfrac{1}{4}y^2(\lambda\mp 1)^2(\lambda-\lambda_0)(\lambda-\lambda_0^{-1})=-\tfrac{1}{4}y^2\lambda^4 +\tfrac{1}{4}y^2(\lambda_0\pm 2+\lambda_0^{-1})\lambda^3-\tfrac{1}{2}y^2(1\pm(\lambda_0+\lambda_0^{-1}))\lambda^2 +\tfrac{1}{4}y^2(\lambda_0\pm 2+\lambda_0^{-1})\lambda-\tfrac{1}{4}y^2$.  Comparing with \eqref{eq:dotV-ODE} shows that this configuration is possible for all $y\neq 0$, provided that $\lambda_0$ is determined up to reciprocation by $\lambda_0+\lambda_0^{-1}=2\ii y^{-1}\mp 2$ and that $C$ is assigned the value $C=-\tfrac{1}{2}y^2(1\pm (\lambda_0+\lambda_0^{-1}))=\tfrac{1}{2}y^2\mp \ii y$.  This case turns out to be relevant only when $y\in (E\cap\ii\mathbb{R})\setminus\{\pm\tfrac{1}{2}\ii\}$.
\item[(iv)] There are four simple roots, none of which equal\footnote{If there are four roots and one of them is $\lambda=\pm 1$, then the others are $\lambda=\mp 1$, $\lambda=\lambda_0$ and $\lambda=\lambda_0^{-1}$ with $\lambda_0^2\neq 1$.  Thus $P(\lambda;y,C)=-\tfrac{1}{4}y^2(\lambda^2-1)(\lambda-\lambda_0)(\lambda-\lambda_0^{-1})=-\tfrac{1}{4}y^2\lambda^4+\tfrac{1}{4}y^2(\lambda_0+\lambda_0^{-1})\lambda^3-\tfrac{1}{4}y^2(\lambda_0+\lambda_0^{-1})\lambda +\tfrac{1}{4}y^2$.  Comparing with \eqref{eq:dotV-ODE} shows that this case is not possible for $y\neq 0$.} $1$ or $-1$,
in which case for some $\lambda_0$ and $\lambda_1$ with $\lambda_0^2\neq 1$, $\lambda_1^2\neq 1$, $\lambda_1\neq\lambda_0$ and $\lambda_1\neq\lambda_0^{-1}$, we have
$P(\lambda;y,C)=-\tfrac{1}{4}y^2(\lambda-\lambda_0)(\lambda-\lambda_0^{-1})(\lambda-\lambda_1)(\lambda-\lambda_1^{-1})=-\tfrac{1}{4}y^2\lambda^4+\tfrac{1}{4}y^2(\lambda_0+\lambda_0^{-1}+\lambda_1+\lambda_1^{-1})\lambda^3-\tfrac{1}{4}y^2((\lambda_0+\lambda_0^{-1})(\lambda_1+\lambda_1^{-1})+2)\lambda^2 + \tfrac{1}{4}y^2(\lambda_0+\lambda_0^{-1}+\lambda_1+\lambda_1^{-1})\lambda-\tfrac{1}{4}y^2$.
Comparing with \eqref{eq:dotV-ODE} shows that this case is possible for all $y\neq 0$ with arbitrary $C$, and that then $\lambda_0$ and $\lambda_1$ are determined up to reciprocation and exchange by the identities $\lambda_0+\lambda_0^{-1}+\lambda_1+\lambda_1^{-1}=2\ii y^{-1}$ and $(\lambda_0+\lambda_0^{-1})(\lambda_1+\lambda_1^{-1})=-2-4Cy^{-2}$.  This turns out to be the case for $y\in E_\mathrm{L}\cup E_\mathrm{R}$.
\end{itemize}\smallskip
Note that in either of the cases that $P(\lambda;y,C)$ is not a perfect square it is necessary to take care in placing the branch cuts of the square root to obtain $g'(\lambda;y)-\tfrac{1}{2}V'(\lambda;y)$ from $(g'(\lambda;y)-\tfrac{1}{2}V'(\lambda;y))^2=\lambda^{-4}P(\lambda;y,C)$ in order that the asymptotic relations \eqref{eq:gprimeminushalfVprime-asymp} hold rather than just \eqref{eq:gprimeminushalfVprime-squared-asymp}.

\subsection{Boutroux integral conditions}
\label{sec:Boutroux}
In order to ensure that the constant $K$ associated with each distinguished arc of $L$ is real, it is necessary in the above case (iv) to impose further conditions.  Given $y$ and $C$ such that this is the case, let $\Gamma=\{(\lambda,\mu):\,\mu^2=\lambda^{-4}P(\lambda;y,C)\}$ be the genus-$1$ Riemann surface or algebraic variety associated with the equation $\mu^2=\lambda^{-4}P(\lambda;y,C)$ in $\mathbb{C}^2$ with coordinates $(\lambda,\mu)$.  Let $(\mathfrak{a},\mathfrak{b})$ be a canonical homology basis on $\Gamma$ and take concrete representatives that do not pass through the preimages on $\Gamma$ of each of the two points $\lambda=0$ or $\lambda=\infty$.  Then we impose the \emph{Boutroux conditions}
\begin{equation}
\mathfrak{B}_\mathfrak{a}(u,v;y):=\mathrm{Re}\left(\oint_\mathfrak{a}\mu\,\dd\lambda\right)=0\quad\text{and}\quad
\mathfrak{B}_\mathfrak{b}(u,v;y):=\mathrm{Re}\left(\oint_\mathfrak{b}\mu\,\dd\lambda\right)=0,
\label{eq:Boutroux}
\end{equation}
where $C=u+\ii v$, i.e., $u:=\mathrm{Re}(C)$ and $v:=\mathrm{Im}(C)$.
It follows from \eqref{eq:gprimeminushalfVprime-asymp} that the differential $\mu\,\dd\lambda$ has real residues at the two points of $\Gamma$ over $\lambda=0$ and the two points over $\lambda=\infty$; therefore taken together the conditions \eqref{eq:Boutroux} do not depend on the choice of homology basis.  We expect that the two real conditions \eqref{eq:Boutroux} should determine $u$ and $v$ as functions of $y\in\mathbb{C}$.  Differentiation of the algebraic identity relating $\mu$ and $\lambda$ gives
\begin{equation}
2\mu\frac{\partial\mu}{\partial u}=\frac{1}{\lambda^2}\quad\text{and}\quad 2\mu\frac{\partial\mu}{\partial v} = \frac{\ii}{\lambda^2}
\end{equation}
from which it follows (since the paths $\mathfrak{a}$ and $\mathfrak{b}$ may be locally taken to be independent of $y$ and $C$) that
\begin{equation}
\frac{\partial\mathfrak{B}_\mathfrak{a,b}}{\partial u}(u,v;y) = \frac{1}{2}\mathrm{Re}\left(\oint_{\mathfrak{a},\mathfrak{b}}\frac{\dd\lambda}{\mu\lambda^2}\right)\quad\text{and}\quad
\frac{\partial\mathfrak{B}_\mathfrak{a,b}}{\partial v}(u,v;y)=-\frac{1}{2}\mathrm{Im}\left(\oint_{\mathfrak{a},\mathfrak{b}}\frac{\dd\lambda}{\mu\lambda^2}\right).
\end{equation}
Therefore, the Jacobian determinant of the equations \eqref{eq:Boutroux} equals 
\begin{equation}
\begin{split}
\det\left(\frac{\partial(\mathfrak{B}_\mathfrak{a},\mathfrak{B}_\mathfrak{b})}{\partial (u,v)}\right) &= -\frac{1}{4}\mathrm{Re}\left(\oint_\mathfrak{a}\frac{\dd\lambda}{\mu\lambda^2}\right)\mathrm{Im}\left(\oint_\mathfrak{b}\frac{\dd\lambda}{\mu\lambda^2}\right) +\frac{1}{4}
\mathrm{Re}\left(\oint_\mathfrak{b}\frac{\dd\lambda}{\mu\lambda^2}\right)\mathrm{Im}\left(\oint_\mathfrak{a}\frac{\dd\lambda}{\mu\lambda^2}\right)\\
&=\frac{1}{4}\mathrm{Im}\left(\left[\oint_\mathfrak{a}\frac{\dd\lambda}{\mu\lambda^2}\right]\left[\oint_\mathfrak{b}\frac{\dd\lambda}{\mu\lambda^2}\right]^*\right).
\end{split}
\label{eq:Jacobian}
\end{equation}
Noting that $\mu^{-1}\lambda^{-2}\dd\lambda$ is a nonzero differential spanning the (one-dimensional) vector space of holomorphic differentials on $\Gamma$, it follows from \cite[Chapter II, Corollary 1]{Dubrovin81} that the above Jacobian is strictly negative under the assumption that the four roots of $P(\lambda;y,C)$ are distinct.  Thus, an application of the implicit function theorem allows us to extend any solution of the integral conditions \eqref{eq:Boutroux} for which $P(\lambda;y_0,u_0+\ii v_0)$ has distinct roots to a neighborhood of $y_0$ on which $u$ and $v$ are smooth real-valued functions of $y$ satisfying $u(y_0)=u_0$ and $v(y_0)=v_0$. In fact, one can show that the Jacobian determinant \eqref{eq:Jacobian} \emph{blows up} as the spectral curve degenerates, and it is in this way that the implicit function theorem ultimately fails.


\section{Asymptotics of $u_n(ny;m)$ for $y\in \mathbb{C}\setminus E$} 
\label{sec:outside}

In this section, we study Riemann-Hilbert Problem~\ref{rhp:renormalized} with $x=ny$ (i.e., we set $w=0$) and assume that $y$ lies in a neighborhood of $y=\infty$ to be determined.

\subsection{Placement of arcs of $L$ and determination of $\partial E$}
\label{sec:placement-of-arcs}
We first show that for $y$ sufficiently large in magnitude, we may take $C=-\tfrac{1}{4}(2y^2-1)$ and hence $P(\lambda;y,C)$ has two double roots; therefore the spectral curve is reducible leading to $g(\lambda;y)=g_\infty(y)$ for a suitable value of the latter constant.  For $y$ large we take the double root $\lNaught=\lNaught(y)$ satisfying $\lNaught+\lNaught^{-1}=\ii y^{-1}$ to be the branch for which $\lNaught(y)=-\ii(1-\tfrac{1}{2}y^{-1}+\mathcal{O}(y^{-2}))$ as $y\to\infty$.  Then, we choose $g_\infty(y):=\tfrac{1}{2}V(\lNaught(y);y)$.  Thus,
\begin{equation}
\begin{split}
2g(\lambda;y)-V(\lambda;y)&=2g_\infty(y)-V(\lambda;y)=V(\lNaught(y);y)-V(\lambda;y)\\
&=\log(\lambda)-\log(\lNaught(y))+\ii y(\lambda-\lambda^{-1})-\ii y(\lNaught(y)-\lNaught(y)^{-1})
=\ii y(\lambda+2\ii-\lambda^{-1}) + \mathcal{O}(1),
\end{split}
\end{equation}
where the $\mathcal{O}(1)$ error term applies in the limit $y\to\infty$ uniformly for $\lambda$ in compact subsets of $\mathbb{C}\setminus\{0\}$.  
%
Taking into account that $2g(\lambda;y)-V(\lambda;y)$ has a double zero at $\lambda=\lNaught(y)$,
one can show that
if $|y|$ is sufficiently large, taking the common intersection point of all four contour arcs to be the point $\lNaught(y)$, it is possible to arrange the arcs so that $\mathrm{Re}(2g(\lambda;y)-V(\lambda;y))<0$ (resp., $\mathrm{Re}(2g(\lambda;y)-V(\lambda;y))>0$) holds on $\LInftyRed\cup\LZeroRed$ (resp., on $\LInftyBlue\cup\LZeroBlue$), with the inequality being strict except at the intersection point $\lambda=\lNaught(y)$, compare Figure~\ref{fig:y-5-Exp-i-pi-over-4}.
\begin{figure}[h]
\begin{center}
\includegraphics{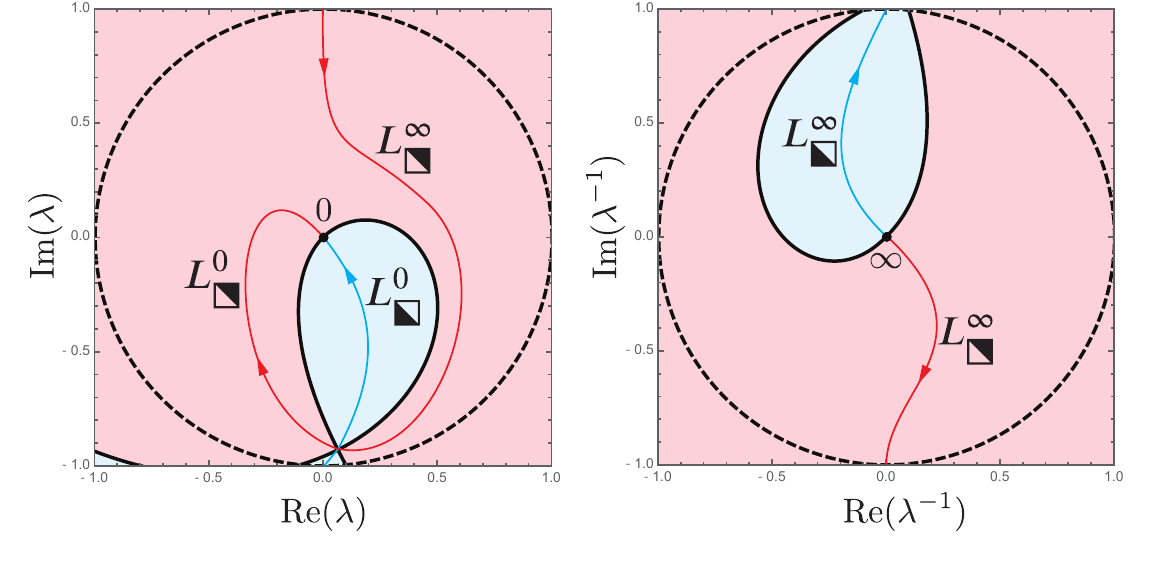}
\end{center}
\caption{For $y=5\ee^{\ii\pi/4}$, the domain where $\mathrm{Re}(2g(\lambda;y)-V(\lambda;y))<0$ is shaded in red, and the domain where $\mathrm{Re}(2g(\lambda;y)-V(\lambda;y))>0$ is shaded in blue. Left panel:  the $\lambda$-plane.  Right panel:  the $\lambda^{-1}$-plane.  The unit circle is shown with a dashed curve in each plot.  Suitable contour arcs matching the scheme described in Section~\ref{sec:analytic-representation} including the argument increment conditions \eqref{eq:increment-argument-red}--\eqref{eq:increment-argument-blue} (for one choice of the arbitrary sign in \eqref{eq:increment-argument-red}) are also shown.}
\label{fig:y-5-Exp-i-pi-over-4}
\end{figure}
The function $\lNaught(y)$ has an analytic continuation from the neighborhood of $y=\infty$ to the maximal domain $y\in\mathbb{C}\setminus I$, where $I$ denotes the imaginary segment connecting the two branch points $\pm\tfrac{1}{2}\ii$.  As $y$ is brought in from the point at infinity, it remains possible to place the arcs of the contour $L$ as described above at least until  either $y$ meets the branch cut $I$ of $\lNaught(y)$ or the topology of the zero level set of $\mathrm{Re}(2g(\lambda;y)-V(\lambda;y))$ changes.  The latter occurs precisely when the only other critical point $\lambda=\lNaught(y)^{-1}$ moves onto the zero level set;  since $\mathrm{Re}(V(\lambda^{-1};y))=-\mathrm{Re}(V(\lambda;y))$, whenever both $\lNaught(y)$ and $\lNaught(y)^{-1}$ lie on the same level of $\mathrm{Re}(V(\lambda;y))$ we necessarily have $\mathrm{Re}(V(\lNaught(y);y))=0$.
The set of $y\in\mathbb{C}\setminus I$ where the latter condition holds true is plotted in Figure~\ref{fig:partialE}.
\begin{figure}[h]
\begin{center}
\includegraphics{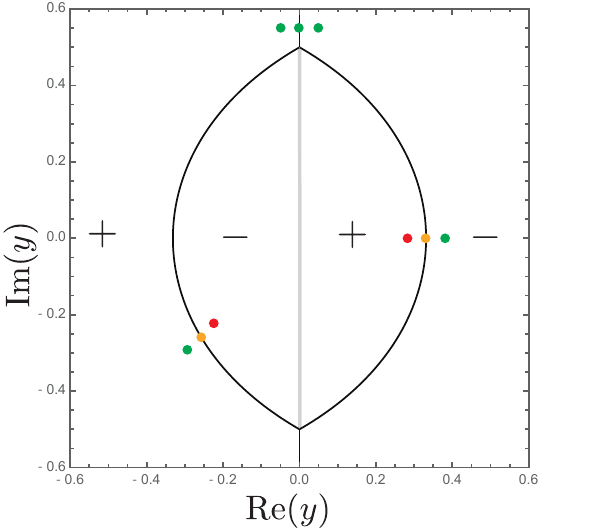}
\end{center}
\caption{The curves in the $y$-plane where $\mathrm{Re}(V(\lNaught(y);y))=0$.  The branch cut $I$ of $\lNaught(y)$ (gray line) joins the two junction points.  The dots correspond to plots in subsequent figures.  The domain $E$ is defined as the bounded region between the two black curves, bisected by the branch cut of $\lNaught(y)$.  The sign of $\mathrm{Re}(V(\lNaught(y);y))$ is indicated in each region.  In particular, $E_\mathrm{R}$ (resp., $E_\mathrm{L}$) is the bounded region where $\mathrm{Re}(V(\lNaught(y);y))>0$ (resp., $\mathrm{Re}(V(\lNaught(y);y))<0$) holds.}
\label{fig:partialE}
\end{figure}
Because $y\in\ii\mathbb{R}\setminus I$ implies that $|\lNaught(y)|=1$, it is easy to confirm that indeed $\mathrm{Re}(V(\lNaught(y);y))=0$ for such $y$, see also Figure~\ref{fig:partialE}.  The rest of the points comprise a closed curve $\partial E$ with two smooth arcs meeting at the branch points $\pm \tfrac{1}{2}\ii$ and bounding the eye-shaped domain $E$ defined in Section~\ref{sec:results}.  

The following figures illustrate how the domains such as shown in Figure~\ref{fig:y-5-Exp-i-pi-over-4} change as the value of $y$ varies near the arcs of the curve shown in Figure~\ref{fig:partialE}.  Figure~\ref{fig:y-real} 
\begin{figure}[h]
\begin{center}
\includegraphics{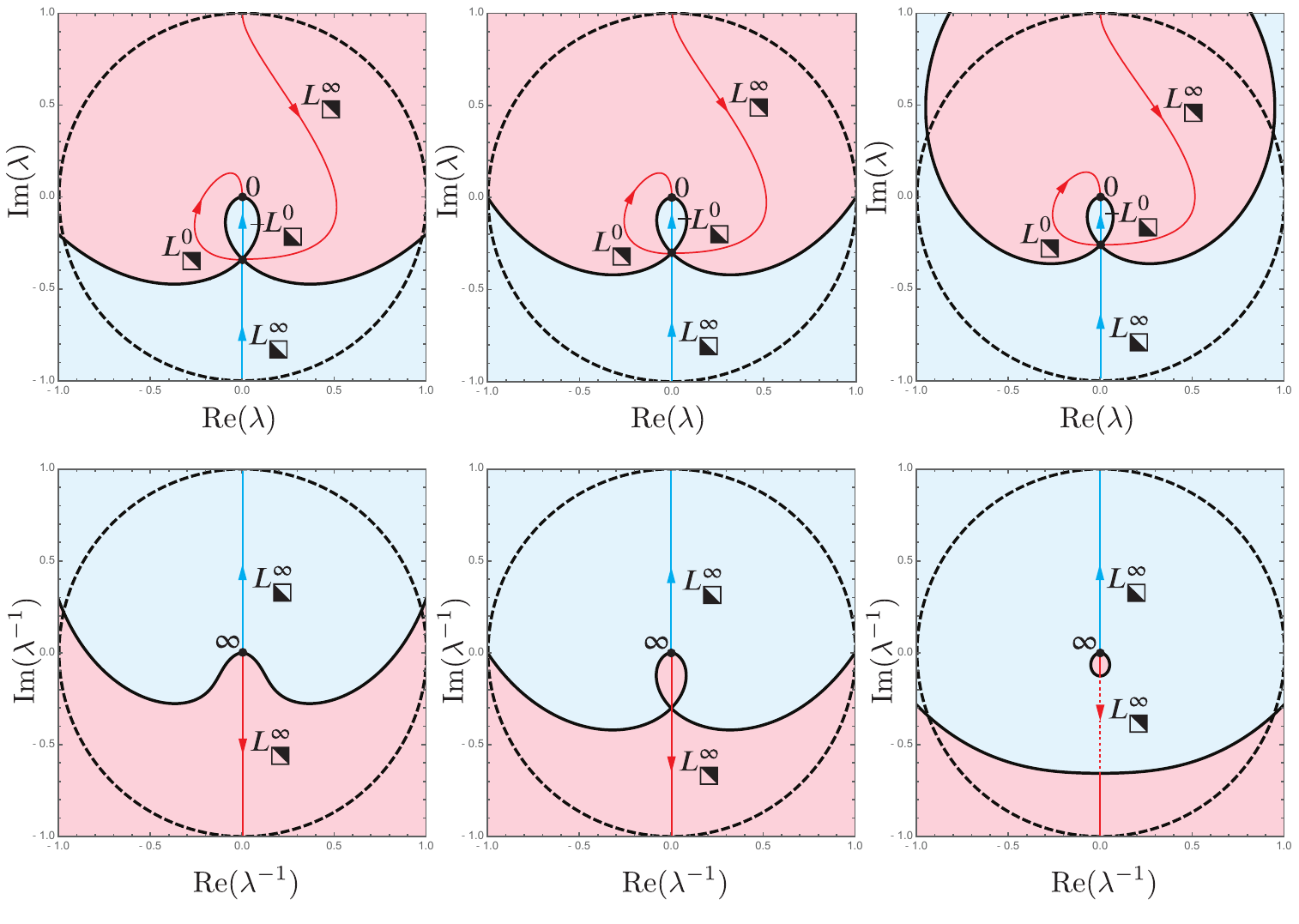}
\end{center}
\caption{The domain where $\mathrm{Re}(2g(\lambda;y)-V(\lambda;y))<0$ in red and where $\mathrm{Re}(2g(\lambda;y)-V(\lambda;y))>0$ in blue for $y=0.381372$ (left column), $y=0.331372$ (middle column), and $y=0.281372$ (right column), corresponding to the green, amber, and red dots, respectively, on the real axis in Figure~\ref{fig:partialE}.  The top row shows a neighborhood of the unit disk in the $\lambda$-plane, while the bottom row shows the exterior of the unit disk in the $\lambda^{-1}$-plane.  In the plots in the right-hand column, the level curve has broken and it is no longer possible to place the contour arc $\LInftyRed$ connecting $\lNaught(y)$ and $\infty$ completely within the red region.  This phase transition, which apparently occurs only on the right edge of the domain $E$, is only relevant if the jump matrix on $\LInftyRed$ is not the identity, i.e., if $m-\tfrac{1}{2}\notin\mathbb{Z}_{\ge 0}$.  These plots show contours with $\Delta\arg(\squareurblack)=2\mathrm{Arg}(x)-2\pi=-2\pi$.  The other choice $\Delta\arg(\squareurblack)=2\mathrm{Arg}(x)+2\pi=2\pi$ would also be compatible with the sign chart.}
\label{fig:y-real}
\end{figure}
concerns the three points on the real axis and Figure~\ref{fig:y-diag} 
\begin{figure}[h]
\begin{center}
\includegraphics{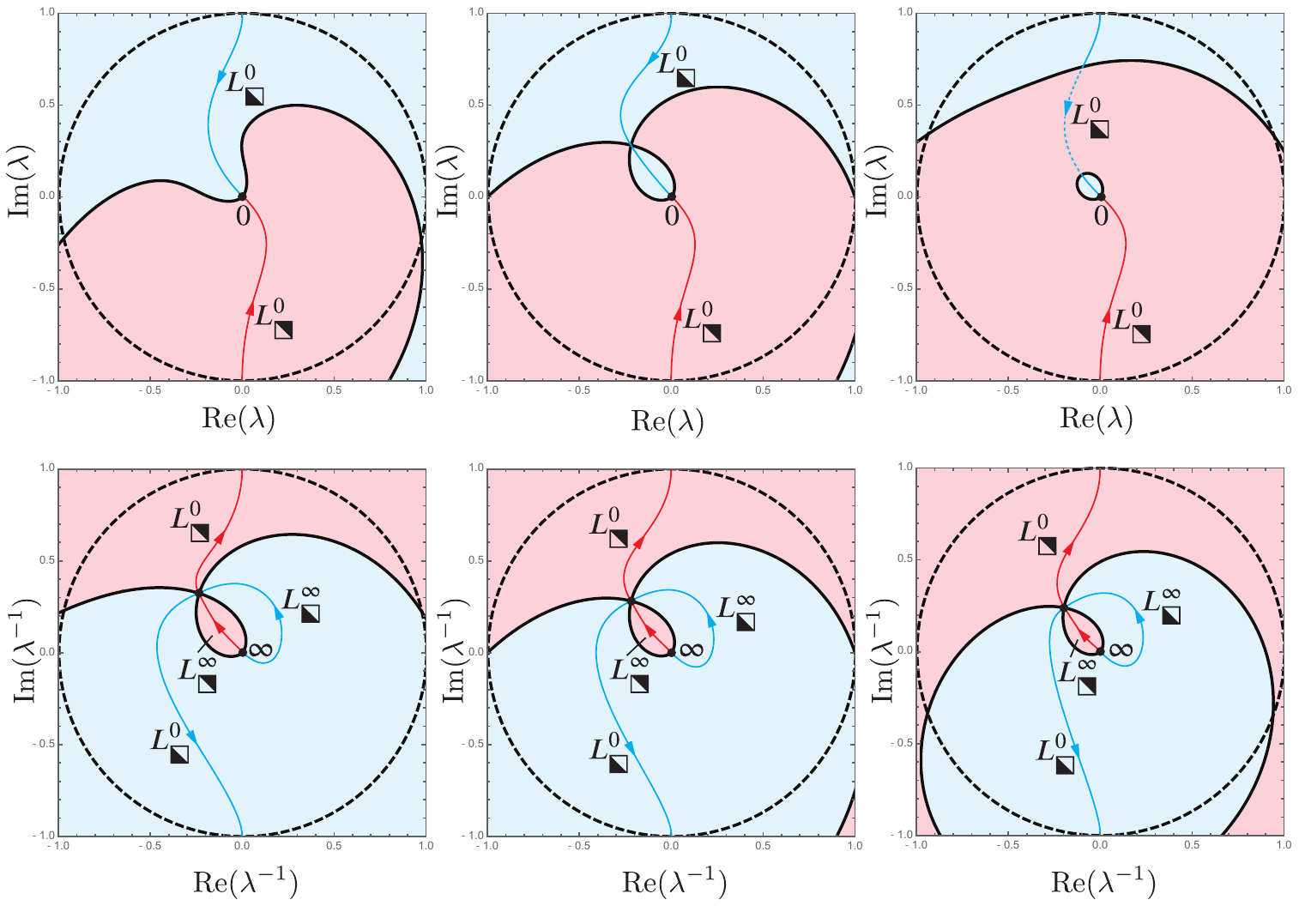}
\end{center}
\caption{As in Figure~\ref{fig:y-real} except for $y=0.414768\ee^{-3\pi\ii/4}$ (left column), $y=0.364768\ee^{-3\pi\ii/4}$ (middle column), and $y=0.314768\ee^{-3\pi\ii/4}$ (right column), corresponding to the green, amber, and red dots, respectively, on the diagonal in Figure~\ref{fig:partialE}.  In the plots in the right-hand column, the level curve has broken and it is no longer possible to place the contour arc $\LZeroBlue$ connecting $\lNaught(y)$ and $0$ completely within the blue region.
This phase transition, which apparently occurs only on the left edge of the domain $E$,
is only relevant if the jump matrix on $\LZeroBlue$ is not the identity, i.e., if $-m-\tfrac{1}{2}\notin\mathbb{Z}_{\ge 0}$.  These plots show contours with $\Delta\arg(\squareurblack)=2\mathrm{Arg}(x)+2\pi=\pi/2$.  In this case, the other choice of $\Delta\arg(\squareurblack)=2\mathrm{Arg}(x)-2\pi$ could only be arranged by a surgery of $\LInftyRed\cup\LZeroRed$ that would result in contours incompatible with the sign chart.}
\label{fig:y-diag}
\end{figure}
concerns the three points on the diagonal.  Figure~\ref{fig:y-imag} 
\begin{figure}[h]
\begin{center}
\includegraphics{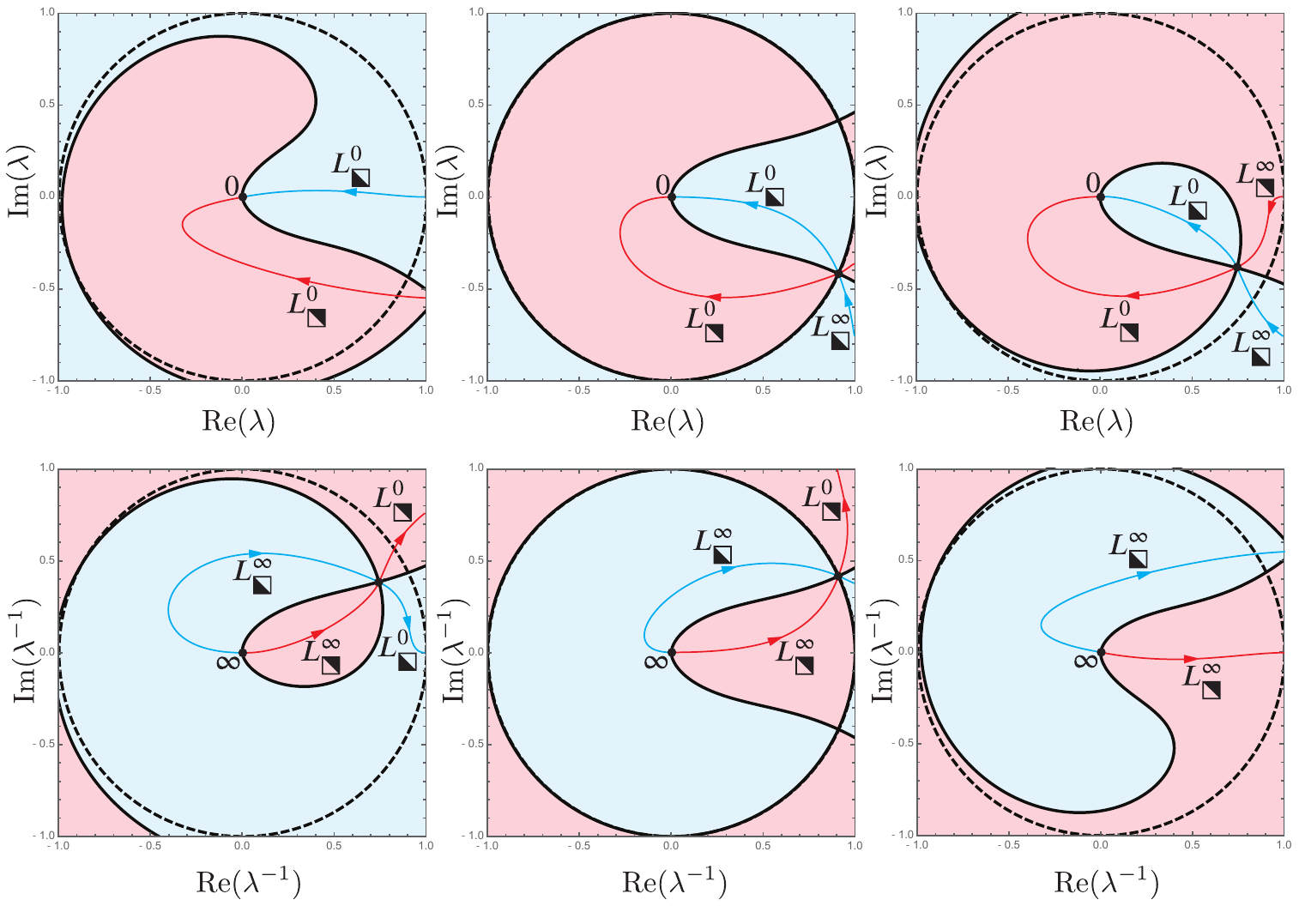}
\end{center}
\caption{As in Figure~\ref{fig:y-real} except for $y=-0.05+0.55\ii$ (left column), $y=0.55\ii$ (middle column), and $y=0.05+0.55\ii$ (right column), corresponding to the three green dots near the positive imaginary axis in Figure~\ref{fig:partialE}.  The topological change in sign chart has no effect on the placement of the contours.  The configurations with $\mathrm{Re}(y)\le 0$ require the choice $\Delta\arg(\squareurblack)=2\mathrm{Arg}(x)-2\pi\approx -\pi$.  On the other hand, the configuration with $\mathrm{Re}(y)>0$, although pictured here with $\Delta\arg(\squareurblack)=2\mathrm{Arg}(x)-2\pi$, admits surgery of the contours $\LZeroRed$ and $\LInftyRed$ within the domain $\mathrm{Re}(2g(\lambda;y)-V(\lambda;y))<0$ and hence the choice $\Delta\arg(\squareurblack)=2\mathrm{Arg}(x)+2\pi$ is also possible.}
\label{fig:y-imag}
\end{figure}
shows that although there is a topological change in the level curve as $y$ crosses the imaginary axis in the exterior of $E$, this does not obstruct the placement of the contour arcs of $L$.  On the other hand, the topological change that occurs when $y$ lies along the arc of $\partial E$ in the right half-plane (resp., left half-plane) only obstructs placement of the arc $\LInftyRed$ (resp., $\LZeroBlue$) and therefore we write $\partial E$ as the union of two closed arcs:  $\partial E=\pEInftyRed\cup\pEZeroBlue$. 
Note that the surgery allowing for a sign change $\Delta\arg(\squareurblack)=2\mathrm{Arg}(x)-2\pi\leftrightarrow\Delta\arg(\squareurblack)=2\mathrm{Arg}(x)+2\pi$ (see Remark~\ref{rmk:surgery}) is compatible with the sign-chart/contour placement scheme provided that the domain $\mathrm{Re}(2g(\lambda;y)-V(\lambda;y))<0$ consists of a single component.  If it consists of two components, then the contours $\LZeroRed$ and $\LInftyRed$ necessarily lie in distinct components and the surgery becomes impossible.  The former holds in the exterior of $E$ for $\mathrm{Re}(y)>0$ and the latter for $\mathrm{Re}(y)\le 0$.

\subsection{Parametrix construction}  
\label{sec:parametrix-for-outside}
Let $y$ be fixed outside of $E$, let $\delta>0$ be a fixed sufficiently small (given $y$) constant, and let $D$ denote the simply-connected neighborhood of $\lambda=\lNaught(y)$ defined by the inequality $|2g(\lambda;y)-V(\lambda;y)|<\delta^2$.  We will define a parametrix $\dot{\mathbf{M}}_n(\lambda;y,m)$ for $\mathbf{M}_n(\lambda;y,m)$ in \eqref{eq:Y-M-g-function} by a piecewise formula:
\begin{equation}
\dot{\mathbf{M}}_n(\lambda;y,m)=\begin{cases} \dot{\mathbf{M}}^{\mathrm{out}}(\lambda;y,m),\quad & \lambda\in\mathbb{C}\setminus(L\cup\overline{D})\\
\dot{\mathbf{M}}_n^{\mathrm{in}}(\lambda;y,m),\quad &\lambda\in D\setminus L.
\end{cases}
\label{eq:M-dot-piecewise}
\end{equation}

Noting that the jump matrix for $\mathbf{M}_n(\lambda;y,m)$ converges uniformly on $L\setminus D$ (with exponential accuracy) to $\mathbb{I}$ except on $\LZeroBlue$, where the limit is instead $-\ee^{2\pi\ii m\sigma_3}$, and that $\mathbf{M}_n(\lambda;y,m)\OurPower{\lambda}{-(m+1/2)\sigma_3}$ should have a limit as $\lambda\to 0$, we define $\dot{\mathbf{M}}^{\mathrm{out}}(\lambda;y,m)$ as the following diagonal matrix:
\begin{equation}
\dot{\mathbf{M}}^{\mathrm{out}}(\lambda;y,m):= \left[\frac{\lambda}{\lambda-\lNaught(y)}\right]^{(m+\tfrac{1}{2})\sigma_3},\ \ \lambda\in\mathbb{C}\setminus(L\cup\overline{D}),
\label{eq:M-dot-out}
\end{equation}
where the branch cut is taken to be $\LZeroBlue$ and the branch is chosen such that the right-hand side tends to $\mathbb{I}$ as $\lambda\to\infty$.  
In order to define $\dot{\mathbf{M}}_n^{\mathrm{in}}(\lambda;y,m)$, we will find a certain canonical matrix function that satisfies exactly the jump conditions of $\mathbf{M}_n(\lambda;y,m)$ within the neighborhood $D$ and then we will multiply the result on the left by a matrix holomorphic in $D$ to arrange a good match with $\dot{\mathbf{M}}^{\mathrm{out}}(\lambda;y,m)$ on $\partial D$.  For the first part, we introduce a conformal mapping $W:D\to\mathbb{C}$ by the following relation:
\begin{equation}
W^2 = 2g(\lambda;y)-V(\lambda;y),\ \ \ \ \lambda\in D.
\label{eq:w-definition}
\end{equation}
Because $2g(\lambda;y)-V(\lambda;y)$ is a locally analytic function vanishing precisely to second order\footnote{Indeed, $2g''(\lNaught(y);y)-V''(\lNaught(y);y)=-\lNaught(y)^{-2}-2\ii y\lNaught(y)^{-3}=(1-\lNaught(y)^2)\lNaught(y)^{-2}(1+\lNaught(y)^2))^{-1}$ can only vanish if $\lNaught(y)=\pm 1$ which corresponds to $y=\pm \tfrac{1}{2}\ii$.} at $\lambda=\lNaught(y)$, the relation \eqref{eq:w-definition} defines two different analytic functions of $\lambda$ both vanishing to first order at $\lambda=\lNaught(y)$.  We choose the analytic solution $W=W(\lambda;y)$
that is negative real in the direction tangent to $\LZeroBlue$.  Then we deform the arcs of $L$ within $D$ so that in this neighborhood $\LZeroBlue$ and $\LInftyBlue$ correspond exactly to negative and positive real values of $W$, while $\LZeroRed$ and $\LInftyRed$ correspond exactly to negative and positive imaginary values of $W$.  By definition of $D$, its image $W(D;y)$ under $W$ is the disk of radius $\delta$ centered at the origin, see Figure~\ref{fig:Outside-Local}.
\begin{figure}[h]
\begin{center}
\includegraphics{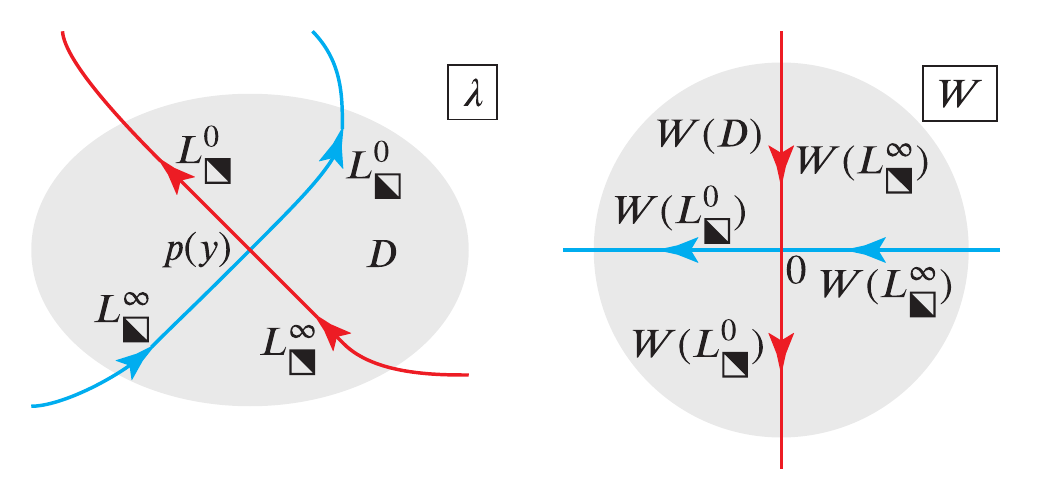}
\end{center}
\caption{The neighborhood $D$ and its image under the conformal mapping $W=W(\lambda)=W(\lambda;y)$.}
\label{fig:Outside-Local}
\end{figure}
Consider the matrix $\mathbf{N}_n(\lambda;y,m)$ defined in terms of $\mathbf{M}_n(\lambda;y,m)$ for $\lambda\in D$ by $\mathbf{N}_n(\lambda;y,m):=
\mathbf{M}_n(\lambda;y,m)d(\lambda;y,m)^{\sigma_3}$, where
\begin{equation}
d(\lambda;y,m):=\begin{cases} \OurPower{\lambda}{-\tfrac{1}{2}(m+1)}(-\ee^{-\ii\pi m}),&\mathrm{Im}(W(\lambda;y))>0,\; \mathrm{Re}(W(\lambda;y))\neq 0,\smallskip\\
\OurPower{\lambda}{-\tfrac{1}{2}(m+1)},&\mathrm{Im}(W(\lambda;y))<0,\;\mathrm{Re}(W(\lambda;y))\neq 0.
\end{cases}
\label{eq:scalar-d-define}
\end{equation}
Recalling the precise definition of $\OurPower{\lambda}{\pm\tfrac{1}{2}(m+1)}$, with its cut along $\LInftyBlue\cup\LZeroBlue$, it follows that $d(\lambda;y,m)$ can be continued to the whole domain $D$ as an analytic nonvanishing function.  The jump conditions satisfied by $\mathbf{N}_n(\lambda;y,m)$ within $D$ are then the following:
\begin{equation}
\mathbf{N}_{n+}(\lambda;y,m)=\mathbf{N}_{n-}(\lambda;y,m)\begin{bmatrix}1 & \displaystyle -\frac{\sqrt{2\pi}}{\Gamma(\tfrac{1}{2}-m)}\ee^{nW(\lambda;y)^2}\\
0 & 1\end{bmatrix},\quad \lambda\in \LZeroRed\cap D,
\label{eq:N-jump-first}
\end{equation}
\begin{equation}
\mathbf{N}_{n+}(\lambda;y,m)=\mathbf{N}_{n-}(\lambda;y,m)\begin{bmatrix}1 & \displaystyle \frac{\sqrt{2\pi}\ee^{2\pi\ii m}}{\Gamma(\tfrac{1}{2}-m)}\ee^{nW(\lambda;y)^2}\\
0 & 1\end{bmatrix},\quad \lambda\in \LInftyRed\cap D,
\end{equation}
\begin{equation}
\mathbf{N}_{n+}(\lambda;y,m)=\mathbf{N}_{n-}(\lambda;y,m)\begin{bmatrix}1 & 0\\
\displaystyle -\frac{\sqrt{2\pi}\ee^{-\ii\pi m}}{\Gamma(\tfrac{1}{2}+m)}\ee^{-nW(\lambda;y)^2} & 1\end{bmatrix},\quad\lambda\in\LInftyBlue\cap D,
\end{equation}
and
\begin{equation}
\mathbf{N}_{n+}(\lambda;y,m)=\mathbf{N}_{n-}(\lambda;y,m)\begin{bmatrix}-\ee^{2\pi\ii m} & 0\\
\displaystyle -\frac{\sqrt{2\pi}\ee^{-\ii\pi m}}{\Gamma(\tfrac{1}{2}+m)}\ee^{-nW(\lambda;y)^2} & -\ee^{-2\pi\ii m}\end{bmatrix},\quad\lambda\in\LZeroBlue\cap D.
\label{eq:N-jump-last}
\end{equation}
Although we will only use its values for $\lambda\in\mathbb{C}\setminus D$, the outer parametrix $\dot{\mathbf{M}}^{\mathrm{out}}(\lambda;y,m)$ has a convenient representation also for $\lambda\in D$ in terms of the conformal coordinate $W=W(\lambda)$:
\begin{equation}
\dot{\mathbf{M}}^{\mathrm{out}}(\lambda;y,m)=f(\lambda;y,m)^{\sigma_3}W(\lambda;y)^{-(m+\tfrac{1}{2})\sigma_3},\quad \lambda\in D,
\label{eq:scalar-f-define}
\end{equation}
where the power function of $W$ refers to the principal branch cut for $W<0$, and where $f(\lambda;y,m)$ is holomorphic and nonvanishing in $D$.
Now letting $\zeta:=n^{1/2}W(\lambda;y)$, we define precisely a matrix $\mathbf{P}(\zeta;m)$ as the solution of the following model Riemann-Hilbert problem.
\begin{rhp}
Given any $m\in\mathbb{C}$, seek a $2\times 2$ matrix function $\zeta\mapsto\mathbf{P}(\zeta;m)$ with the following properties:
\begin{itemize}
\item[1.]\textbf{Analyticity:}  $\zeta\mapsto\mathbf{P}(\zeta;m)$ is analytic for $\mathrm{Re}(\zeta^2)\neq 0$, taking continuous boundary values on the four rays of $\mathrm{Re}(\zeta^2)=0$ oriented as in the right-hand panel of Figure~\ref{fig:Outside-Local}.
\item[2.]\textbf{Jump conditions:}  The boundary values $\mathbf{P}_\pm(\zeta;m)$ taken on the four rays of $\mathrm{Re}(\zeta^2)=0$ satisfy the following jump conditions (cf., \eqref{eq:N-jump-first}--\eqref{eq:N-jump-last}):
\begin{equation}
\mathbf{P}_+(\zeta;m)=\mathbf{P}_-(\zeta;m)\begin{bmatrix}1 & \displaystyle -\frac{\sqrt{2\pi}}{\Gamma(\tfrac{1}{2}-m)}\ee^{\zeta^2}\\0 & 1\end{bmatrix},\quad \arg(\zeta)=-\frac{\pi}{2},
\label{eq:N-zeta-jump-first}
\end{equation}
\begin{equation}
\mathbf{P}_+(\zeta;m)=\mathbf{P}_-(\zeta;m)\begin{bmatrix}1 & \displaystyle \frac{\sqrt{2\pi}\ee^{2\pi\ii m}}{\Gamma(\tfrac{1}{2}-m)}\ee^{\zeta^2}\\0 & 1\end{bmatrix},\quad \arg(\zeta)=\frac{\pi}{2},
\end{equation}
\begin{equation}
\mathbf{P}_+(\zeta;m)=\mathbf{P}_-(\zeta;m)\begin{bmatrix}1 & 0\\\displaystyle -\frac{\sqrt{2\pi}\ee^{-\ii\pi m}}{\Gamma(\tfrac{1}{2}+m)}\ee^{-\zeta^2} & 1\end{bmatrix},\quad \arg(\zeta)=0,
\end{equation}
and
\begin{equation}
\mathbf{P}_+(\zeta;m)=\mathbf{P}_-(\zeta;m)\begin{bmatrix}-\ee^{2\pi\ii m} & 0\\
\displaystyle -\frac{\sqrt{2\pi}\ee^{-\ii\pi m}}{\Gamma(\tfrac{1}{2}+m)}\ee^{-\zeta^2} & -\ee^{-2\pi\ii m}\end{bmatrix},\quad \arg(-\zeta)=0.
\label{eq:N-zeta-jump-last}
\end{equation}
\item[3.]\textbf{Asymptotics:}
$\mathbf{P}(\zeta;m)$ is required to satisfy the normalization condition 
\begin{equation}
\lim_{\zeta\to\infty}\mathbf{P}(\zeta;m)\zeta^{(m+\tfrac{1}{2})\sigma_3}=\mathbb{I}.
\label{eq:N-zeta-normalize}
\end{equation}
Here, $\zeta^p$ refers to the principal branch.
\end{itemize}
\label{rhp:ParabolicCylinder}
\end{rhp}
This problem will be solved in all details in Appendix~\ref{app:PC}.  From it, we define the inner parametrix $\dot{\mathbf{M}}_n^{\mathrm{in}}(\lambda;y,m)$ as follows:
\begin{equation}
\dot{\mathbf{M}}_n^{\mathrm{in}}(\lambda;y,m):=d(\lambda;y,m)^{\sigma_3}n^{\tfrac{1}{2}(m+\tfrac{1}{2})\sigma_3}f(\lambda;y,m)^{\sigma_3}\mathbf{P}(n^{1/2}W(\lambda;m);m)d(\lambda;y,m)^{-\sigma_3},\quad \lambda\in D\setminus L.
\end{equation}
As shown in Appendix~\ref{app:PC}, $\mathbf{P}(\zeta;m)\zeta^{(m+\tfrac{1}{2})\sigma_3}$ has a complete asymptotic expansion in descending powers of $\zeta$ as $\zeta\to\infty$  (see \eqref{eq:N-expansion}).
Taking into account the explicit leading terms from the expansion \eqref{eq:N-expansion} and using  the fact that $W(\lambda;y)$ is bounded away from zero for $\lambda\in\partial D$, we get
\begin{equation}
\dot{\mathbf{M}}_n^{\mathrm{in}}(\lambda;y,m)\dot{\mathbf{M}}^\mathrm{out}(\lambda;y,m)^{-1}=
n^{m\sigma_3/2}\begin{bmatrix}1+\mathcal{O}(n^{-1}) & a(\lambda;y,m)+\mathcal{O}(n^{-1})\\
\mathcal{O}(n^{-1}) & 1+\mathcal{O}(n^{-1})\end{bmatrix}n^{-m\sigma_3/2},\quad n\to+\infty,\quad\lambda\in\partial D,
\label{eq:Mismatch-0}
\end{equation}
holding uniformly for the indicated values of $\lambda\in \partial D$, where
\begin{equation}
a(\lambda;y,m):=-\ii\ee^{\ii\pi m}2^md(\lambda;y,m)^2f(\lambda;y,m)^2W(\lambda;y)^{-1},\quad\lambda\in\partial D.
\end{equation}

\subsection{Error analysis and proof of Theorem~\ref{theorem:outside}}
\label{sec:error-outside}
To compare the unknown $\mathbf{M}_n(\lambda;y,m)$ with its parametrix $\dot{\mathbf{M}}_n(\lambda;y,m)$, note the constant conjugating factors in \eqref{eq:Mismatch-0} and consider the matrix function $\mathbf{F}_n(\lambda;y,m)$ defined by $\mathbf{F}_n(\lambda;y,m):=n^{-m\sigma_3/2}\mathbf{M}_n(\lambda;y,m)\dot{\mathbf{M}}_n(\lambda;y,m)^{-1}n^{m\sigma_3/2}$, which is well-defined for $\lambda\in \mathbb{C}\setminus (L\cup\partial D)$.  This matrix satisfies a jump condition of the form $\mathbf{F}_{n+}(\lambda;y,m)=\mathbf{F}_{n-}(\lambda;y,m)(\mathbb{I}+\text{exponentially small})$ as $n\to+\infty$ uniformly for $\lambda\in L$ (in fact, $\mathbf{F}_{n+}(\lambda;y,m)=\mathbf{F}_{n-}(\lambda;y,m)$ exactly if $\lambda\in L\cap D$).  To see this for $\lambda\in L\setminus\overline{D}$, note that
\begin{equation}
\mathbf{F}_{n+}(\lambda;y,m)=\mathbf{F}_{n-}(\lambda;y,m)n^{-m\sigma_3/2}\dot{\mathbf{M}}^\mathrm{out}_-(\lambda;y,m)
\mathbf{V}_n(\lambda;y,m)\dot{\mathbf{V}}(\lambda;y,m)^{-1}\dot{\mathbf{M}}^\mathrm{out}_-(\lambda;y,m)^{-1}n^{m\sigma_3/2},\quad\lambda\in L\setminus\overline{D},
\end{equation}
where $\mathbf{M}_{n+}(\lambda;y,m)=\mathbf{M}_{n-}(\lambda;y,m)\mathbf{V}_n(\lambda;y,m)$ and $\dot{\mathbf{V}}(\lambda;y,m)$ is the corresponding jump matrix for $\dot{\mathbf{M}}^\mathrm{out}(\lambda;y,m)$, which is just the diagonal part of $\mathbf{V}_n(\lambda;y,m)$ and which hence reduces to the identity matrix except on $\LZeroBlue$.  The desired result then follows because $\mathbf{V}_n(\lambda;y,m)\dot{\mathbf{V}}(\lambda;y,m)^{-1}-\mathbb{I}$ is exponentially small in the limit $n\to +\infty$, while $\dot{\mathbf{M}}^\mathrm{out}_-(\lambda;y,m)$ and its inverse are independent of $n$ and bounded because $\lambda$ is excluded from $D$.  Finally, 
for $\lambda\in\partial D$, taken with clockwise orientation, we have
\begin{equation}
\mathbf{F}_{n+}(\lambda;y,m)=\mathbf{F}_{n-}(\lambda;y,m)n^{-m\sigma_3/2}
\dot{\mathbf{M}}_n^{\mathrm{in}}(\lambda;y,m)
\dot{\mathbf{M}}^{\mathrm{out}}(\lambda;y,m)^{-1}n^{m\sigma_3/2},\quad\lambda\in\partial D.
\label{eq:Fjump-circle}
\end{equation}
The jump contour for $\mathbf{F}_n(\lambda;y,m)$ (and also for the related matrix $\mathbf{E}_n(\lambda;y,m)$ defined below) in a typical case of $y\in\mathbb{C}\setminus E$ is shown in Figure~\ref{fig:OutsideError}.\smallskip
\begin{figure}[h]
\begin{center}
\includegraphics{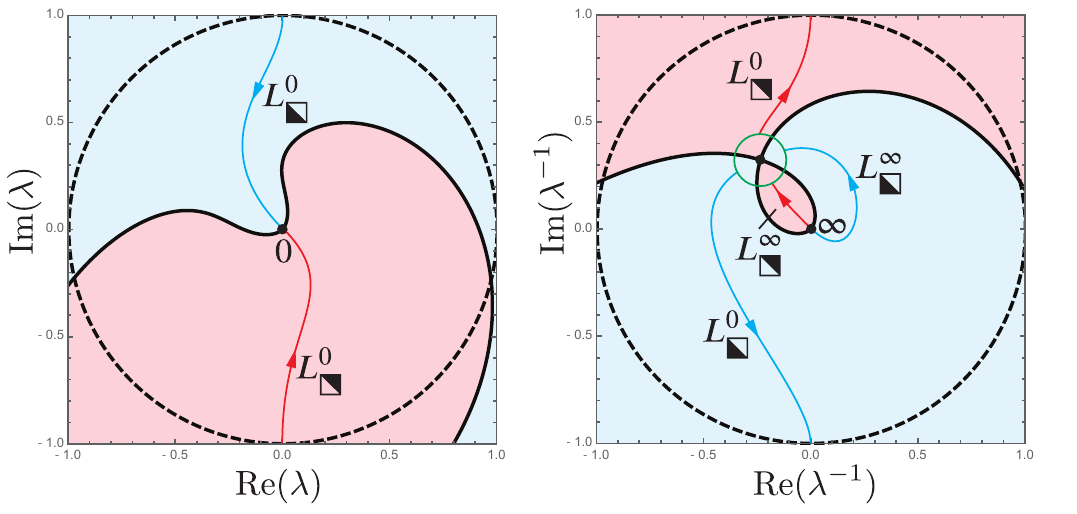}
\end{center}
\caption{The jump contour for $\mathbf{F}_n(\lambda;y,m)$ and $\mathbf{E}_n(\lambda;y,m)$ for  $y=0.364768\ee^{-3\pi\ii/4}\in\mathbb{C}\setminus E$ corresponding to the left-most column of Figure~\ref{fig:y-diag}.  The jump contour is the union of the cyan and red arcs (which comprise $L\setminus \overline{D}$), and the green circle $\partial D$, which is taken to have clockwise orientation for the purposes of defining the boundary values.}
\label{fig:OutsideError}
\end{figure}

Taking into account the leading term on the upper off-diagonal in \eqref{eq:Mismatch-0},
we define a parametrix for $\mathbf{F}_n(\lambda;y,m)$ as a triangular matrix independent of $n$:
\begin{equation}
\dot{\mathbf{F}}(\lambda;y,m):=\mathbb{I}+\frac{1}{2\pi\ii}\oint_{\partial D}
\begin{bmatrix}
0 & a(\xi;y,m)\\
0 & 0\end{bmatrix}\frac{\dd\xi}{\xi-\lambda} \quad\text{(clockwise orientation for $\partial D$)}.
\label{eq:dot-F-def}
\end{equation}
Since $d(\cdot;y,m)$ and $f(\cdot;y,m)$ are analytic and nonvanishing within $D$, and since $W(\cdot;y)$ is univalent on $D$ with $W(\lNaught(y);y)=0$, 
the above Cauchy integral
can be evaluated by residues.  In particular, if $\lambda\in\mathbb{C}\setminus\overline{D}$, then
\begin{equation}
\dot{\mathbf{F}}(\lambda;y,m)=\begin{bmatrix}
1 & \ii\ee^{\ii\pi m}2^md(\lNaught(y);y,m)^{2}f(\lNaught(y);y,m)^{2}W'(\lNaught(y);y)^{-1}(\lNaught(y)-\lambda)^{-1} \\0 & 1\end{bmatrix},
\quad\lambda\in\mathbb{C}\setminus\overline{D}.
\label{eq:F-dot-outside}
\end{equation}
Note that $\dot{\mathbf{F}}(\lambda;y,m)$ is analytic for $\lambda\in\mathbb{C}\setminus\partial D$,
$\dot{\mathbf{F}}(\lambda;y,m)\to\mathbb{I}$ as $\lambda\to\infty$, and across $\partial D$ satisfies (by the Plemelj formula from \eqref{eq:dot-F-def})  the jump condition
\begin{equation}
\dot{\mathbf{F}}_+(\lambda;y,m)=\dot{\mathbf{F}}_-(\lambda;y,m)\begin{bmatrix}1 & a(\lambda;y,m)\\0&1\end{bmatrix},\quad\text{$\lambda\in\partial D$ with clockwise orientation.}
\label{eq:Fdot-circle-jump}
\end{equation}

At last, we consider the matrix $\mathbf{E}_n(\lambda;y,m):=\mathbf{F}_n(\lambda;y,m)\dot{\mathbf{F}}(\lambda;y,m)^{-1}$.  The matrix $\mathbf{E}_n(\lambda;y,m)$ is analytic for $\lambda\in\mathbb{C}\setminus (L\cup\partial D)$, tends to the identity as $\lambda\to\infty$, and takes continuous boundary values from each component of its domain of analyticity, including at the origin.  
The jump conditions satisfied by $\mathbf{E}_n(\lambda;y,m)$ are as follows.  Firstly, since $\mathbf{F}_n(\lambda;y,m)$ extends continuously to the arcs of $L\cap D$ while $\dot{\mathbf{F}}(\lambda;y,m)^{-1}$ is analytic for $\lambda\in\mathbb{C}\setminus\partial D$, it follows by Morera's theorem that $\mathbf{E}_n(\lambda;y,m)$ is in fact analytic on the arcs of $L\cap D$.  For $\lambda\in L\setminus\overline{D}$, $\dot{\mathbf{F}}(\lambda;y,m)$ is analytic with analytic inverse, both of which are bounded; since $\mathbf{F}_{n+}(\lambda;y,m)=\mathbf{F}_{n-}(\lambda;y,m)(\mathbb{I}+\text{exponentially small})$, it follows that as $n\to+\infty$, $\mathbf{E}_{n-}(\lambda;y,m)^{-1}\mathbf{E}_{n+}(\lambda;y,m)-\mathbb{I}$ is small beyond all orders uniformly for bounded $m\in\mathbb{C}$.  Finally, for $\lambda\in\partial D$, 
\begin{equation}
\begin{split}
\mathbf{E}_{n+}(\lambda;y,m)&=\mathbf{F}_{n+}(\lambda;y,m)\dot{\mathbf{F}}_+(\lambda;y,m)^{-1}\\
&=\mathbf{F}_{n-}(\lambda;y,m)\begin{bmatrix}1+\mathcal{O}(n^{-1}) & a(\lambda;y,m)+\mathcal{O}(n^{-1})\\
\mathcal{O}(n^{-1}) & 1+\mathcal{O}(n^{-1})\end{bmatrix}\begin{bmatrix}1 & -a(\lambda;y,m)\\0 & 1\end{bmatrix}\dot{\mathbf{F}}_-(\lambda;y,m)^{-1}\\
&=\mathbf{E}_{n-}(\lambda;y,m)\dot{\mathbf{F}}_-(\lambda;y,m)\begin{bmatrix}1+\mathcal{O}(n^{-1}) & a(\lambda;y,m)+\mathcal{O}(n^{-1})\\
\mathcal{O}(n^{-1}) & 1+\mathcal{O}(n^{-1})\end{bmatrix}\begin{bmatrix}1 & -a(\lambda;y,m)\\0 & 1\end{bmatrix}\dot{\mathbf{F}}_-(\lambda;y,m)^{-1}\\
&=\mathbf{E}_{n-}(\lambda;y,m)(\mathbb{I}+\mathcal{O}(n^{-1})),\quad n\to+\infty,\quad\lambda\in\partial D,
\end{split}
\label{eq:E-jump}
\end{equation}
where the $\mathcal{O}(n^{-1})$ terms are uniform on $\partial D$.
Here, we used \eqref{eq:Mismatch-0}, \eqref{eq:Fjump-circle}, and \eqref{eq:Fdot-circle-jump} on the second line.  The jump contour for $\mathbf{E}_n(\lambda;y,m)$ is therefore exactly the same as that for $\mathbf{F}_n(\lambda;y,m)$; see Figure~\ref{fig:OutsideError}.  
From these considerations, we see that uniformly for $(y,m)$ in compact subsets of $(\mathbb{C}\setminus \overline{E})\times\mathbb{C}$, 
$\mathbf{E}_n(\lambda;y,m)$ satisfies the conditions of a small-norm Riemann-Hilbert problem for $|n|$ sufficiently large, and the unique solution satisfies $\mathbf{E}_n(\lambda;y,m)=\mathbb{I}+\mathcal{O}(n^{-1})$
uniformly for $\lambda\in\mathbb{C}\setminus (L\cup\partial D)$.  Moreover, $\mathbf{E}_n(\lambda;y,m)$ is 
well-defined at $\lambda=0$ with $\mathbf{E}_n(0;y,m)=\mathbb{I}+\mathcal{O}(n^{-1})$ as $n\to +\infty$, and
$\mathbf{E}_n(\lambda;y,m)=\mathbb{I}+\mathbf{E}_{n,1}(y,m)\lambda^{-1}+\mathcal{O}(\lambda^{-2})$ as $\lambda\to\infty$ with $\mathbf{E}_{n,1}(y,m)=\mathcal{O}(n^{-1})$
as $n\to +\infty$.  
Now, we have the exact identity 
\begin{equation}
\begin{split}
\mathbf{M}_n(\lambda;y,m)&=n^{m\sigma_3/2}\mathbf{F}_n(\lambda;y,m)n^{-m\sigma_3/2}\dot{\mathbf{M}}_n(\lambda;y,m)\\ &=n^{m\sigma_3/2}\mathbf{E}_n(\lambda;y,m)\dot{\mathbf{F}}(\lambda;y,m)n^{-m\sigma_3/2}\dot{\mathbf{M}}_n(\lambda;y,m),
\end{split}
\end{equation}
and therefore from \eqref{eq:Y-M-g-function} and $g(\lambda;y)=g_\infty(y)=\tfrac{1}{2}V(\lNaught(y);y)$, we get 
\begin{equation}
\mathbf{Y}_n(\lambda;ny,m)=\ee^{-nV(\lNaught(y);y)\sigma_3/2}n^{m\sigma_3/2}\mathbf{E}_n(\lambda;y,m)\dot{\mathbf{F}}(\lambda;y,m)n^{-m\sigma_3/2}\dot{\mathbf{M}}_n(\lambda;y,m)\ee^{nV(\lNaught(y);y)\sigma_3/2}.
\label{eq:Y-representation-I}
\end{equation}
Now recall the formula \eqref{eq:u-n-from-Y-formula} for the rational solution $u=u_n(x;m)$ of the Painlev\'e-III equation \eqref{eq:PIII}.
Since to calculate the quantities in this formula we only need $\mathbf{Y}_n(\lambda;ny,m)$ for $\lambda$ in neighborhoods of the origin and infinity, we can safely replace $\dot{\mathbf{M}}^{(n)}(\lambda;y,m)$ in \eqref{eq:Y-representation-I} by the diagonal outer parametrix $\dot{\mathbf{M}}^\mathrm{out}(\lambda;y,m)$ which commutes with $n^{-m\sigma_3/2}$.  Therefore, if $\mathbf{Q}_{n}(\lambda;y,m):=\mathbf{E}_n(\lambda;y,m)\dot{\mathbf{F}}(\lambda;y,m)\dot{\mathbf{M}}^\mathrm{out}(\lambda;y,m)$, then \eqref{eq:u-n-from-Y-formula} can be rewritten as
\begin{equation}
u_n(ny;m)=\frac{-\ii\displaystyle\lim_{\lambda\to\infty}\lambda Q_{n,12}(\lambda;y,m)}{\displaystyle\left[\lim_{\lambda\to 0}Q_{n,11}(\lambda;y,m)\OurPower{\lambda}{-(m+1/2)}\right]
\left[\lim_{\lambda\to 0}Q_{n,12}(\lambda;y,m)\OurPower{\lambda}{m+1/2}\right]}.
\label{eq:u-from-Y-rewrite}
\end{equation}
Now all three factors of $\mathbf{Q}_n(\lambda;y,m)$ tend to $\mathbb{I}$ as $\lambda\to\infty$ but $\dot{\mathbf{M}}^\mathrm{out}(\lambda;y,m)$ is also diagonal, 
\begin{equation}
\begin{split}
\lim_{\lambda\to\infty}\lambda Q_{n,12}(\lambda;y,m)&=\lim_{\lambda\to\infty}\lambda \dot{F}_{12}(\lambda;y,m)+E_{n,1,12}(y,m)\\
&=-\ii\ee^{\ii\pi m}2^md(\lNaught(y);y,m)^2f(\lNaught(y);y,m)^2W'(\lNaught(y);y)^{-1}+\mathcal{O}(n^{-1})
\end{split}
\end{equation}
where in the second line we used \eqref{eq:F-dot-outside} and $\mathbf{E}_{n,1}(y,m)=\mathcal{O}(n^{-1})$.
Also, since $\dot{\mathbf{M}}^\mathrm{out}(\lambda;y,m)\OurPower{\lambda}{-(m+1/2)\sigma_3}$ tends to a limit of the form $h(y,m)^{\sigma_3}$ as $\lambda\to 0$, where $h(y,m)\neq 0$, 
\begin{multline}
\left[\lim_{\lambda\to 0}Q_{n,11}(\lambda;y,m)\OurPower{\lambda}{-(m+1/2)}\right]
\left[\lim_{\lambda\to 0}Q_{n,12}(\lambda;y,m)\OurPower{\lambda}{m+1/2}\right]\\
\begin{aligned}
&=\left[
E_{n,11}(0;y,m)\dot{F}_{11}(0;y,m)+E_{n,12}(0;y,m)\dot{F}_{21}(0;y,m)\right]
\left[E_{n,11}(0;y,m)\dot{F}_{12}(0;y,m)+E_{n,12}(0;y,m)\dot{F}_{22}(0;y,m)\right]\\
&=\left[E_{n,11}(0;y,m)\right]\left[E_{n,11}(0;y,m)\ii\ee^{\ii\pi m}2^md(\lNaught(y);y,m)^2f(\lNaught(y);y,m)^2W'(\lNaught(y);y)^{-1}\lNaught(y)^{-1}+ E_{n,12}(0;y,m)\right]\\
&=\ii\ee^{\ii\pi m}2^md(\lNaught(y);y,m)^2f(\lNaught(y);y,m)^2w'(\lNaught(y);y)^{-1}\lNaught(y)^{-1}+\mathcal{O}(n^{-1}),
\end{aligned}
\end{multline}
where in the third line we used \eqref{eq:F-dot-outside} and in the fourth line we used $\mathbf{E}_n(0;y,m)=\mathbb{I}+\mathcal{O}(n^{-1})$.
Using these results in \eqref{eq:u-from-Y-rewrite} then gives the asymptotic formula \eqref{eq:u-outside} and completes the proof of Theorem~\ref{theorem:outside}.

\section{Asymptotics of $u_n(ny+w;m)$ for $y\in E$ and $m\in\mathbb{C}\setminus(\mathbb{Z}+\tfrac{1}{2})$}
\label{sec:interior}
To study $u_n(x;m)$ for values of $x$ corresponding to the interior of $E$, we wish to capture two different effects: (i) the rapid oscillation visible in plots showing a locally regular pattern of poles and zeros on a microscopic length scale $\Delta x\sim 1$ and (ii) the gradual modulation of this pattern over macroscopic length scales $\Delta x\sim n$.  To separate these scales, we write $x=ny+w$ as described in Section~\ref{sec:results}.  As mentioned in Remark~\ref{remark:two-vars}, considering $u_n(ny+w;m)$ as a function of $w$ for fixed $y\in E$ captures the microscopic behavior of $u_n$, while setting $w=0$ and considering $u_n(ny;m)$ as a function of $y$ captures instead the macroscopic behavior of $u_n$.  A similar approach to the rational solutions of the Painlev\'e-II equation was taken in \cite{BuckinghamM14}.  In this section we will develop an approximation of $u_n(ny+w;m)$ that depends not on the combination $ny+m$ but rather separately on $y$ and $m$ in such a way as to explicitly separate these scales.  In particular, it will turn out that the approximation is meromorphic in $w$ for each fixed $y$ but generally is not analytic at all in $y$.

\subsection{Spectral curves satisfying the Boutroux integral conditions for $y\in E$}
\label{sec:Boutroux-in-E}
We tie the spectral curve to the value $y$ of the macroscopic coordinate and compensate for nonzero values of the microscopic coordinate $w$ later in the construction of a parametrix.
\subsubsection{Solving the Boutroux integral conditions for $y$ small}
\label{sec:Boutroux-y-small}
To construct a $g$-function for $y$ small, we assume that the spectral curve corresponds to a polynomial $P(\lambda;y,C)$ with four distinct roots.  We write $y$ in polar form as $y=r\ee^{\ii\theta}$ and we write $C$ in the form $C=y\tilde{C}$.  For $r>0$ we may divide the equations \eqref{eq:Boutroux} through by $\sqrt{r}$ and consider instead the \emph{renormalized Boutroux integral conditions}
\begin{equation}
\tilde{\mathfrak{B}}_\mathfrak{a}(\tilde{C};r,\theta):=\mathrm{Re}\left(\oint_\mathfrak{a}\tilde{\mu}\,\dd\lambda\right)=0\quad\text{and}\quad
\tilde{\mathfrak{B}}_\mathfrak{b}(\tilde{C};r,\theta):=\mathrm{Re}\left(\oint_\mathfrak{b}\tilde{\mu}\,\dd\lambda\right)=0
\label{eq:renormalized-Boutroux}
\end{equation}
where
\begin{equation}
\tilde{\mu}^2 = \ee^{\ii\theta}\left[\frac{1}{2}\ii\lambda^{-1}+\tilde{C}\lambda^{-2}+\frac{1}{2}\ii\lambda^{-3}\right]-\frac{1}{4}r\ee^{2\ii\theta}(1+\lambda^{-4}).
\label{eq:renormalized-rho}
\end{equation}
Note that if $\tilde{u}:=\mathrm{Re}(\tilde{C})$ and $\tilde{v}:=\mathrm{Im}(\tilde{C})$, then just as in \eqref{eq:Jacobian} one has that
\begin{equation}
\det\left(\frac{\partial(\tilde{\mathfrak{B}}_\mathfrak{a},\tilde{\mathfrak{B}}_\mathfrak{b})}{\partial (\tilde{u},\tilde{v})}\right) =\frac{1}{4}\mathrm{Im}\left(\left[\oint_\mathfrak{a}\frac{\dd\lambda}{\tilde{\mu}\lambda^2}\right]\left[\oint_\mathfrak{b}\frac{\dd\lambda}{\tilde{\mu}\lambda^2}\right]^*\right)
\label{eq:renormalized-Jacobian}
\end{equation}
which is nonzero as long as $\tilde{\mu}$ (see the algebraic relation \eqref{eq:renormalized-rho}) has distinct branch points on the Riemann sphere of the $\lambda$-plane, see \cite[Chapter II, Corollary 1]{Dubrovin81}. We now first set $r=0$ and attempt to determine $\tilde{C}$ as a function of $\theta$.  It is convenient to then reduce the cycle integrals in \eqref{eq:renormalized-Boutroux} to contour integrals connecting pairs of branch points in the finite $\lambda$-plane, and since when $r=0$ the differential $\tilde{\mu}\,\dd\lambda$ has a double pole with zero residue (in an appropriate local coordinate) at the branch point $\lambda=0$ we can integrate by parts to transfer ``half'' of the double pole to each of the finite nonzero roots of $\tilde{\mu}^2$ which (again in appropriate local coordinates) are simple zeros of $\tilde{\mu}\,\dd\lambda$.  In this way we obtain conditions equivalent to \eqref{eq:renormalized-Boutroux} for $r=0$ involving a differential that is holomorphic at all three branch points in the finite $\lambda$-plane.  These conditions are the following:
\begin{equation}
\tilde{\mathfrak{B}}^0_\mathfrak{a}(\tilde{C};\theta):=\mathrm{Re}\left(\oint_\mathfrak{a}\tilde{\mu}_0\,\dd\lambda\right)=0\quad\text{and}\quad
\tilde{\mathfrak{B}}^0_\mathfrak{b}(\tilde{C};\theta):=\mathrm{Re}\left(\oint_\mathfrak{b}\tilde{\mu}_0\,
\dd\lambda\right)=0,
\label{eq:Boutroux-r=0}
\end{equation}
where
\begin{equation}
\tilde{\mu}_0^2:=\ii\ee^{\ii\theta}\frac{(\lambda-\ii \tilde{C})^2}{\lambda(\lambda^2-2\ii \tilde{C}\lambda+1)}.
\end{equation}
The desired simplification is then that the cycle integrals in \eqref{eq:Boutroux-r=0} over $\mathfrak{a}$ and $\mathfrak{b}$ may be replaced (up to a harmless factor of $2$) by path integrals from $\lambda=0$ to the two roots of the quadratic $\lambda^2-2\ii \tilde{C}\lambda+1$ respectively.\bigskip

If $\ee^{\ii\theta}=1$, we may solve \eqref{eq:Boutroux-r=0} in this simplified form by assuming $\tilde{C}$ to be real and positive.  Indeed, then the roots of $\lambda^2-2\ii \tilde{C}\lambda+1$ are the values $\lambda=\ii(\tilde{C}\pm\sqrt{\tilde{C}^2+1})$ which lie on the positive and negative imaginary axes.  It is easy to see that when $\theta=0$, $\tilde{\mu}_0^2>0$ holds for purely imaginary $\lambda$ between $\lambda=\ii(\tilde{C}-\sqrt{\tilde{C}^2+1})$ and $\lambda=0$.  Therefore it is immediate that 
\begin{equation}
\mathrm{Re}\left(\int_0^{\ii (\tilde{C}-\sqrt{\tilde{C}^2+1})}\tilde{\mu}_0\,\dd\lambda\right)=0,\quad \theta=0,\quad \tilde{C}>0.
\end{equation}
The remaining Boutroux integral condition then reduces under the hypotheses $\theta=0$ and $\tilde{C}>0$ to a purely real-valued integral condition on $\tilde{C}$:
\begin{equation}
J(\tilde{C}):=\int_0^{\tilde{C}+\sqrt{\tilde{C}^2+1}}\frac{t-\tilde{C}}{\sqrt{t}\sqrt{1+2\tilde{C}t-t^2}}\,\dd t=0
\end{equation}
Obviously $\lim_{\tilde{C}\downarrow 0}J(\tilde{C})$ exists and the limit is positive.  Also, by rescaling $t=\tilde{C}s$, 
\begin{equation}
J(\tilde{C})=-\sqrt{\tilde{C}}\int_0^1\frac{1-s}{\sqrt{s}\sqrt{\tilde{C}^{-2}+2s-s^2}}\,\dd s + \sqrt{\tilde{C}}\int_1^2\frac{s-1}{s\sqrt{2-s}}\,\dd s + o(\sqrt{\tilde{C}}),\quad \tilde{C}\to\infty,
\end{equation}
and clearly the first term is the dominant one so $J(\tilde{C})<0$ for large positive $\tilde{C}$.  Also, by direct calculation, 
\begin{equation}
J'(\tilde{C})=-\frac{1}{2}\int_0^{\tilde{C}+\sqrt{\tilde{C}^2+1}}\frac{\dd t}{\sqrt{t}\sqrt{1+2\tilde{C}t-t^2}}<0,\quad \tilde{C}>0,
\end{equation}
so there exists a unique simple root $\tilde{C}_0>0$ of $J(\tilde{C})$.  Numerical computation shows that $\tilde{C}_0\approx 0.860437$.\bigskip  

If $\ee^{\ii\theta}=-1$, we can invoke the symmetry $\lambda\mapsto -\lambda$ and $\tilde{C}\mapsto -\tilde{C}$ of $\tilde{\mu}_0^2$ to deduce that the equations \eqref{eq:Boutroux-r=0} hold for $\tilde{C}=-\tilde{C}_0\approx -0.860437$.\bigskip  

When $r=0$, the elliptic curve given by \eqref{eq:renormalized-rho} has distinct branch points on the Riemann sphere unless $\tilde{C}=\pm\ii$,
and hence the Jacobian \eqref{eq:renormalized-Jacobian} of the equations \eqref{eq:Boutroux-r=0} is nonzero for $\ee^{\ii\theta}=\pm 1$.
The solution of the $r=0$ system can therefore be continued to other values of $\ee^{\ii\theta}$ until the condition $\tilde{C}\neq\pm\ii$ is violated.
It is easy to check that $\tilde{C}=\pm\ii$ is consistent with \eqref{eq:Boutroux-r=0} only for $\ee^{\ii\theta}=\mp\ii$.  Therefore the solutions of the $r=0$ system \eqref{eq:Boutroux-r=0} obtained for $\ee^{\ii\theta}=\pm 1$ can be uniquely continued by the implicit function theorem to fill out an infinitesimal circle surrounding the origin $y=0$ with the possible exception of its intersection with the imaginary axis. Fixing any phase factor $\ee^{\ii\theta}\neq \pm\ii$, we can then continue the solution of the full (rescaled) system \eqref{eq:renormalized-Boutroux} to small $r>0$ (in fact, also for $r<0$, although the solution is not relevant), and the radial continuation can only be obstructed if branch points collide.  

\subsubsection{Degenerate spectral curves satisfying the Boutroux integral conditions}
\label{sec:Boutroux-degenerate}
The only possible values of $y\in\mathbb{C}$ for which all four roots of $P(\lambda;y,C)$ coincide are $y=\pm\frac{1}{2}\ii$, which lie on the boundary of $E$.  For all $y\in\mathbb{C}$ it is possible to have either a pair of distinct double roots or a double root and two simple roots, provided $C$ is appropriately chosen as a function of $y$.  We will now show that these degenerate configurations are inconsistent with the Boutroux integral conditions \eqref{eq:Boutroux}, which have to be interpreted in a limiting sense, provided that $y$ lies in the interior of $E$ but does not also lie on the imaginary axis.\smallskip

Consider first a nearly degenerate configuration of roots in which two simple roots of $P$ are very close to a point $\lambda=\lNaught$ and two reciprocal simple roots are very close to $\lambda=\lNaught^{-1}$.  Then we may choose the cycle $\mathfrak{a}$ to encircle the pair of roots near, say, $\lambda=\lNaught$.  As the spectral curve degenerates with the cycle $\mathfrak{a}$ fixed, we may observe that $\mu$ becomes in the limit an analytic function of $\lambda$ in the interior of $\mathfrak{a}$ and therefore $\oint_\mathfrak{a}\mu\,\dd\lambda\to 0$ and hence $\mathrm{Re}(\oint_\mathfrak{a}\mu\,\dd\lambda)\to 0$, so one of the Boutroux integral conditions is automatically satisfied in the limit.  The cycle $\mathfrak{b}$ should then be chosen to connect the small branch cut near $\lambda=\lNaught$ with the small reciprocal branch cut near $\lambda=\lNaught^{-1}$.  In the limit that the spectral curve degenerates and $\mu^2$ becomes a perfect square, the second Boutroux integral condition becomes
\begin{equation}
\mathrm{Re}\left(\oint_\mathfrak{b}\mu\,\dd\lambda\right)\to 2\mathrm{Re}\left(\int_{\lNaught}^{\lNaught^{-1}}\frac{\ii y}{2}\frac{(\lambda-\lNaught)(\lambda-\lNaught^{-1})}{\lambda^2}\,\dd\lambda\right)=\mathrm{Re}(V(\lNaught;y)-V(\lNaught^{-1};y))
\end{equation}
where $\lNaught+\lNaught^{-1}=\ii y^{-1}$.  \emph{The condition on $y\in\mathbb{C}$ that this quantity vanishes is precisely that either $y\in\partial E$ or $y$ lies on the imaginary axis outside of $E$.  Therefore the Boutroux conditions cannot be satisfied by such a degenerate spectral curve if $y$ is in the interior of $E$.}\smallskip

Next consider a nearly degenerate configuration in which a pair of simple roots of $P$ lie very close to $\lambda=\pm 1$ and another pair of reciprocal simple roots tend to distinct reciprocal limits satisfying whose sum is $2\ii y^{-1}\mp 2$.  Again taking $\mathfrak{a}$ to surround the coalescing pair of roots shows that $\mathrm{Re}(\oint_\mathfrak{a}\mu\,\dd\lambda)\to 0$ in the limit.  Then, in the same limit, up to signs,
\begin{equation}
\mathrm{Re}\left(\oint_\mathfrak{b}\mu\,\dd\lambda\right)\to 2\mathrm{Re}\left(\int_{\lambda^\pm}^{\pm 1}\frac{\ii y}{2}\frac{\lambda\mp 1}{\lambda^2}r^\pm(\lambda;y)\,\dd\lambda\right)=:F^\pm(y),
\end{equation}
where $\lambda^\pm+(\lambda^\pm)^{-1}=2\ii y^{-1}\mp 2$ and $r^\pm(\lambda;y)^2=(\lambda-\lambda^\pm)(\lambda-(\lambda^\pm)^{-1})$ with $r^\pm$ having a branch cut connecting the two roots of $r^\pm(\lambda;y)^2$ and, say, $r^\pm=\lambda+\mathcal{O}(1)$ as $\lambda\to\infty$.  It is easy to show that $F^+(y)=0$ for $y$ on the segment between $y=0$ and $y=\tfrac{1}{2}\ii$, and that $F^-(y)=0$ for $y$ on the segment between $y=0$ and $y=-\tfrac{1}{2}\ii$.  However, neither function $F^\pm(y)$ vanishes identically, so the equations $F^\pm(y)=0$ define a system of curves in the complex $y$-plane.  The only branches of these curves in the interior of $E$ lie on the imaginary axis as illustrated in Figure~\ref{fig:FplusFminus}.
\begin{figure}[h]
\begin{center}
\includegraphics{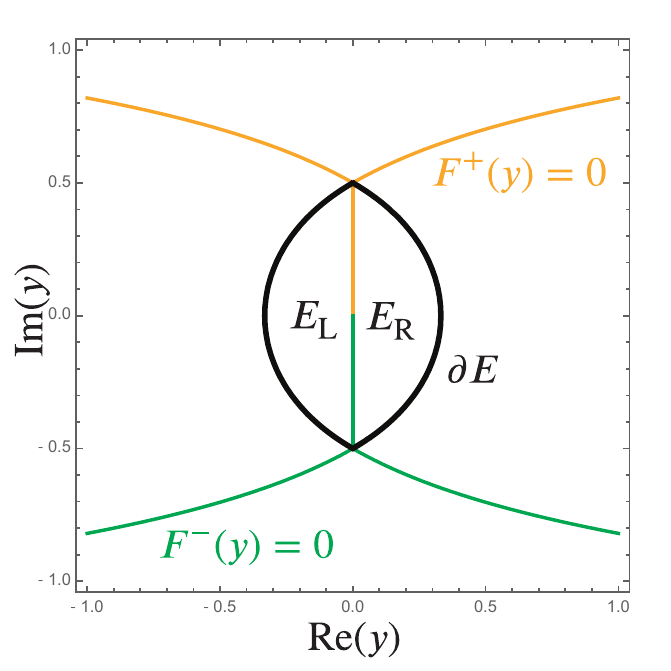}
\end{center}
\caption{The locus $F^+(y)=0$ (orange) and $F^-(y)=0$ (green).  The curve $\partial E$ is shown in black and the sub-domains $E_\mathrm{L}$ and $E_\mathrm{R}$ of $E\setminus ((\partial E)\cup\ii\mathbb{R})$ are indicated.}
\label{fig:FplusFminus}
\end{figure}
Therefore, continuation along radial paths of the Boutroux conditions from the infinitesimal semicircles about the origin in the right and left half-planes defines a unique spectral curve for each $y\in E_\mathrm{L}\cup E_\mathrm{R}$, recalling that $E_\mathrm{L}$ ($E_\mathrm{R}$) is the part of the interior of $E$ in the open left (right) half-plane. 
%

\subsection{Stokes graph and construction of the $g$-function}
For the rest of Section~\ref{sec:interior} we will be concerned with the approximation of $u_n(ny+w;m)$ for large $n$ when $m\in\mathbb{C}\setminus (\mathbb{Z}+\tfrac{1}{2})$ and 
$w$ is bounded, while $y\in E_\mathrm{L}\cup E_\mathrm{R}$.  Actually, due to the exact symmetry \eqref{eq:u-n-exact-symmetry},
it is sufficient to assume that $y\in E_\mathrm{R}$, as $E_\mathrm{L}$ is the reflection through the origin of $E_\mathrm{R}$.  Thus we assume for the rest of Section~\ref{sec:interior} that $y\in E_\mathrm{R}$ and at the end invoke \eqref{eq:u-n-exact-symmetry} to extend the results to $y\in E_\mathrm{L}$.\smallskip

Given $y\in E_\mathrm{R}$, let $C=C(y)$ be determined by the procedure described in Section~\ref{sec:Boutroux-in-E}  so that the Boutroux conditions \eqref{eq:Boutroux} are satisfied.  The \emph{Stokes graph} of $y$ is the system of arcs (\emph{edges}) in the complex $\lambda$-plane emanating from the four distinct roots of $P(\lambda;y,C(y))$ (\emph{vertices}, when taken along with $\lambda=0,\infty$) along which the condition $(g'(\lambda;y)-\tfrac{1}{2}V'(\lambda;y))^2\,\dd\lambda^2=\lambda^{-4}P(\lambda;y,C(y))\,\dd\lambda^2<0$ holds.  The Boutroux conditions \eqref{eq:Boutroux} imply that the Stokes graph is connected.  
In particular, each pair of roots of $P(\lambda;y,C(y))$ that coalesce at $\partial E$ is directly connected by an edge of the Stokes graph.  Denoting the union of these two edges by $\Sigma^\mathrm{out}(y)$, let $R(\lambda;y)$ be the function analytic 
for $\lambda\in\mathbb{C}\setminus\Sigma^\mathrm{out}(y)$ that satisfies $R(\lambda;y)^2=P(\lambda;y,C(y))$ and $R(\lambda;y)=\tfrac{1}{2}\ii y\lambda^2+\mathcal{O}(\lambda)$ as $\lambda\to\infty$.  According to \eqref{eq:gprimeminushalfVprime-asymp} and \eqref{eq:tildeP4}, $g'(\lambda;y)$ may then be defined by
\begin{equation}
g'(\lambda;y)=\frac{1}{2}V'(\lambda;y)+\frac{R(\lambda;y)}{\lambda^2},\quad\lambda\in\mathbb{C}\setminus\Sigma^\mathrm{out}(y),\quad y\in  E_\mathrm{R}.
\label{eq:gprime-E}
\end{equation}
Note that the apparent singularity at $\lambda=0$ is removable, and $g'(\lambda;y)$ is integrable at $\lambda=\infty$. Figures~\ref{fig:StokesGraphReal}, \ref{fig:StokesGraphUpperArc}, and \ref{fig:StokesGraphLowerArc} below illustrate how the Stokes graph varies with $y\in E_\mathrm{R}$.
\begin{figure}[h]
\begin{center}
\includegraphics{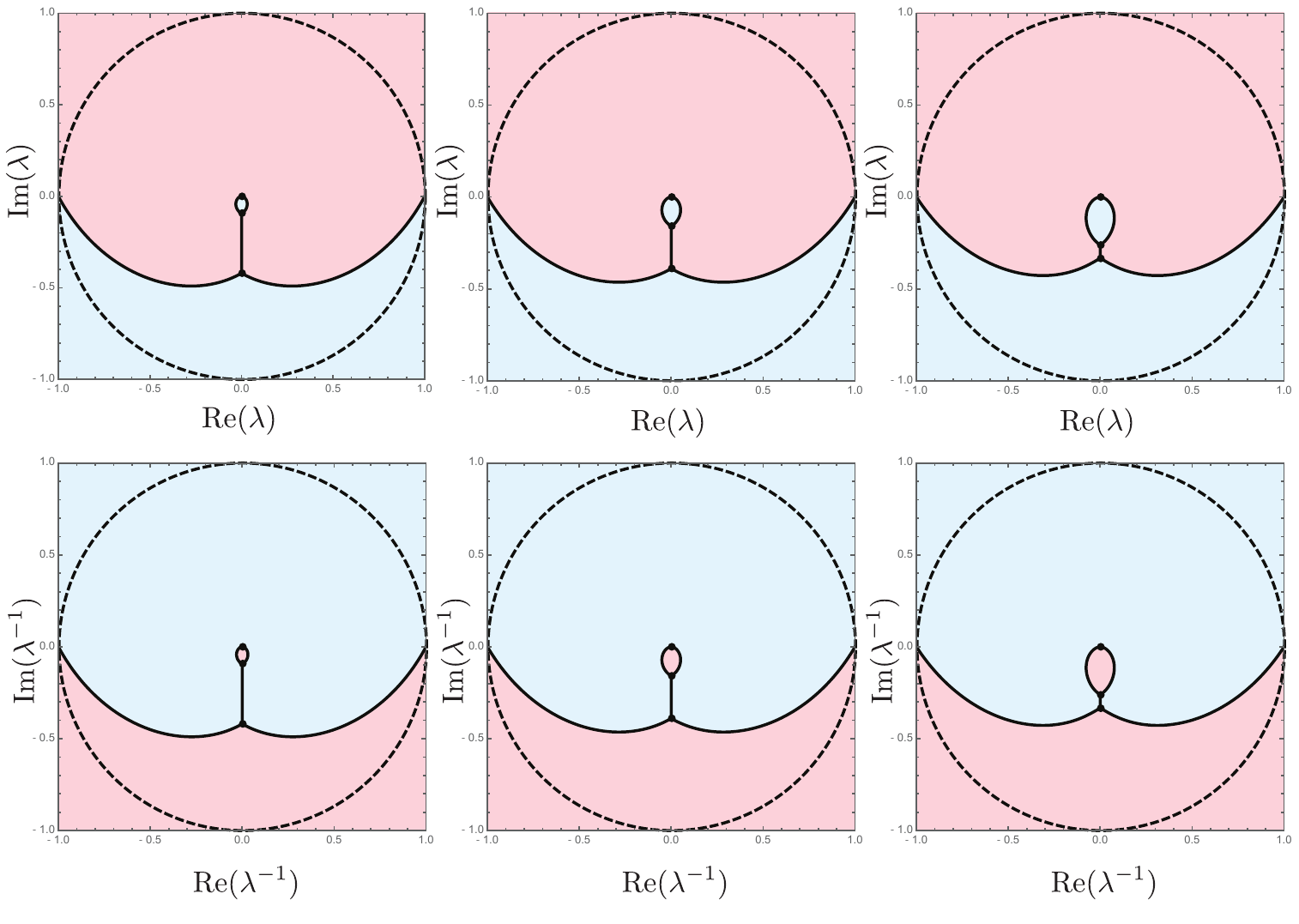}
\end{center}
\caption{The Stokes graph and sign charts for $\mathrm{Re}(2g(\lambda;y)-V(\lambda;y))$ (red/blue for negative/positive) for $y=0.16$ (left column), $y=0.24$ (center column), and $y=0.32$ (right column).  Observe as expected that for small $y$ there are roots close to $\lambda=0$ and $\lambda=\infty$, while as $y\to\partial E$  there are two coalescing pairs of roots, each pair connected by an edge of the graph.}
\label{fig:StokesGraphReal}
\end{figure}
\begin{figure}[h]
\begin{center}
\includegraphics{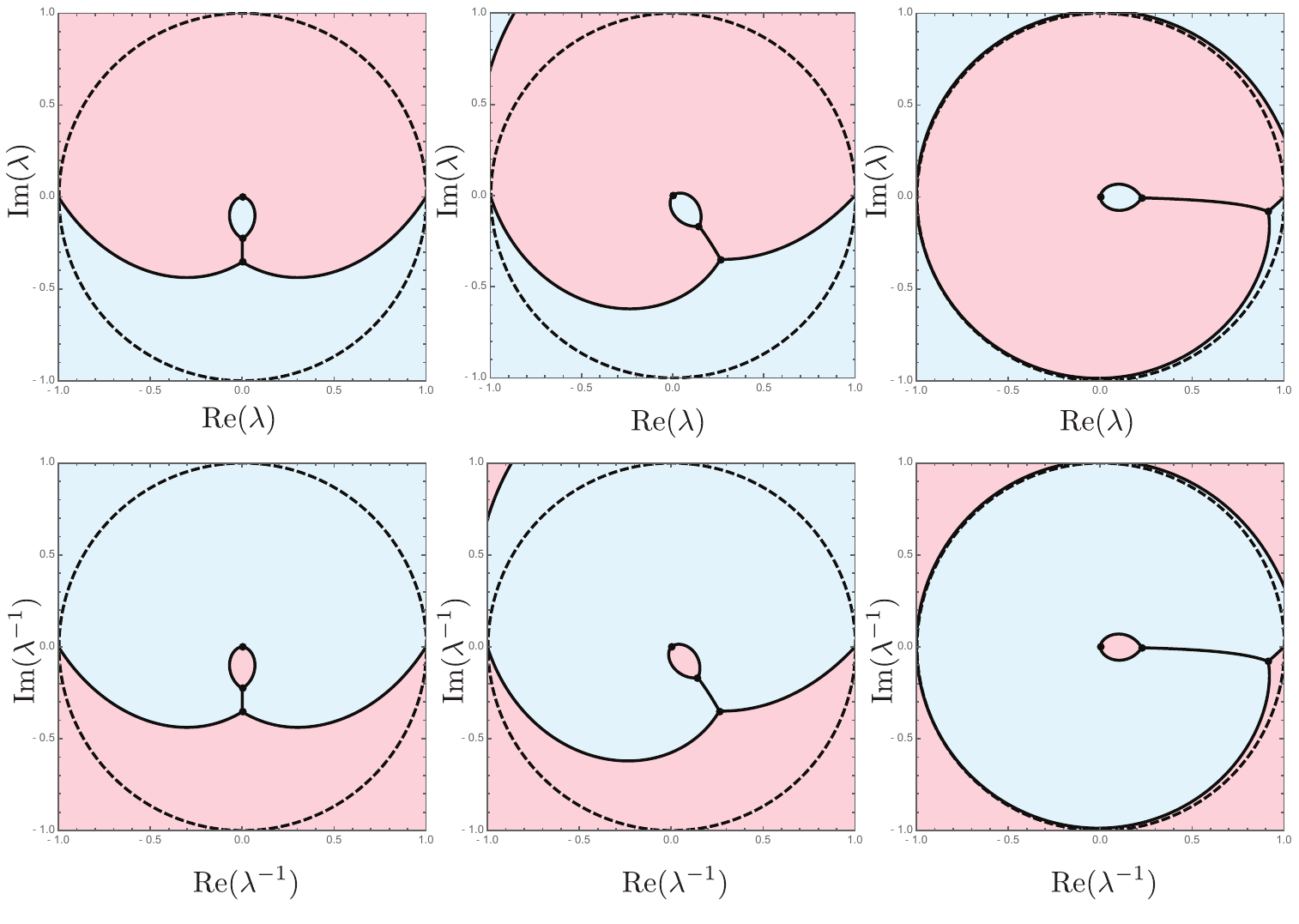}
\end{center}
\caption{As in Figure~\ref{fig:StokesGraphReal}, but for $y=0.3$ (left column), $y=0.3\ee^{\ii\pi/4}$ (center column), and $y=0.3\ee^{99\pi\ii/200}$ (right column).  Observe that as $y$ approaches the positive imaginary axis (where $F^+(y)=0$) from within $E_\mathrm{R}$, a pair of roots coalesce at $\lambda=1$.}
\label{fig:StokesGraphUpperArc}
\end{figure}
\begin{figure}[h]
\begin{center}
\includegraphics{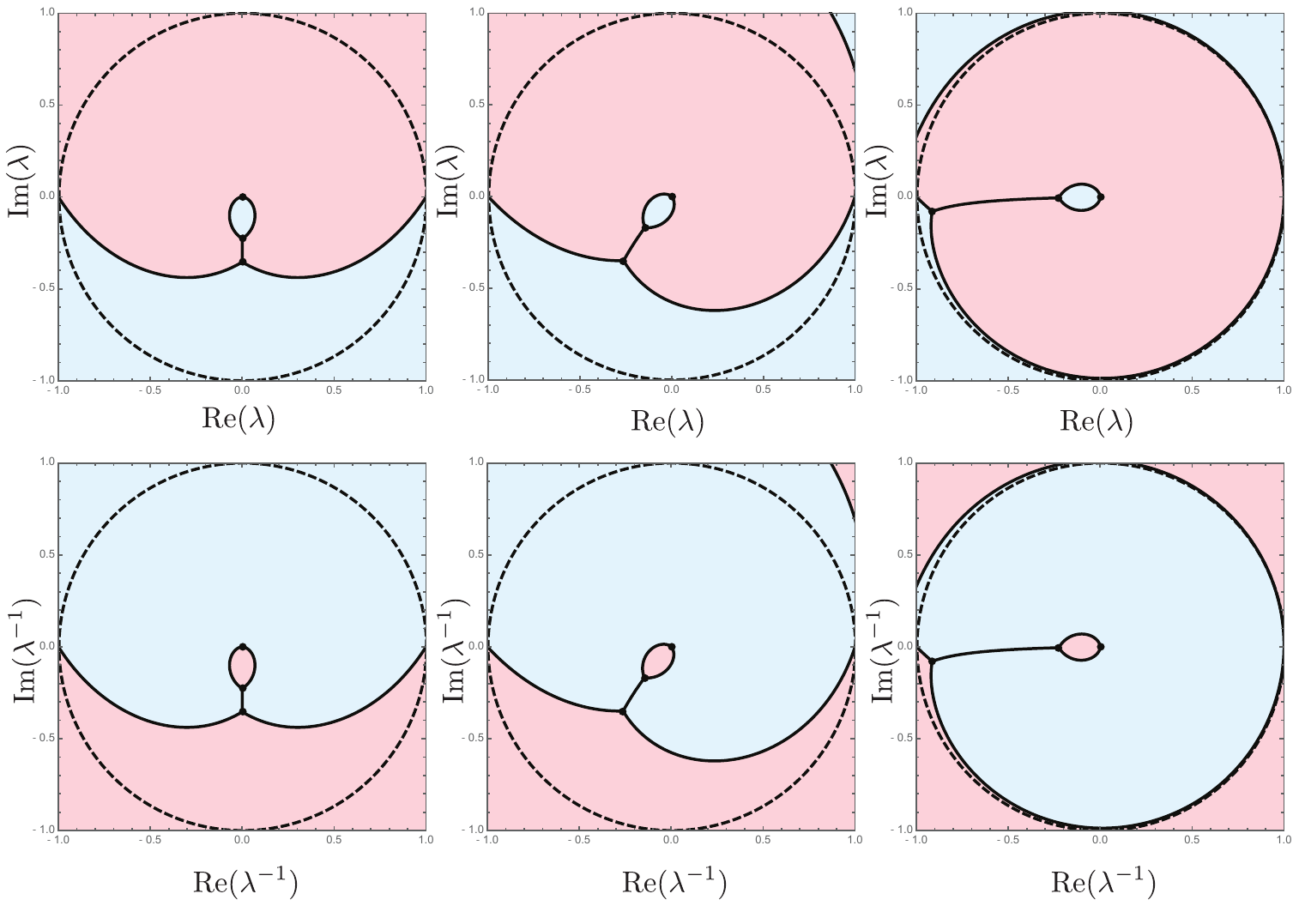}
\end{center}
\caption{As in Figure~\ref{fig:StokesGraphReal}, but for $y=0.3$ (left column), $y=0.3\ee^{-\ii\pi/4}$ (center column), and $y=0.3\ee^{-99\pi\ii/200}$ (right column).  Observe that as $y$ approaches the negative imaginary axis (where $F^-(y)=0$) from within $E_\mathrm{R}$, a pair of roots coalesce at $\lambda=-1$.}
\label{fig:StokesGraphLowerArc}
\end{figure}
A comparison of the top and bottom rows of these figures illustrates the fact that the Stokes graph of $y\in E_\mathrm{R}$ is invariant while $\mathrm{Re}(2g(\lambda;y)-V(\lambda;y))$ changes sign under the involution $\lambda\mapsto \lambda^{-1}$.\bigskip  


Given the Stokes graph, we may lay over the arcs $\Sigma^\mathrm{out}(y)$ and in the complement of the Stokes graph a contour $L$ consisting of arcs $\LInftyRed$, $\LZeroRed$, $\LInftyBlue$, and $\LZeroBlue$ that satisfy the increment-of-argument conditions \eqref{eq:increment-argument-red}--\eqref{eq:increment-argument-blue}.  There are two topologically distinct cases differentiated by the sign $\mathrm{Im}(y)$, as illustrated in Figure~\ref{fig:StokesGraph} for $y\in E_\mathrm{R}$ with $\mathrm{Im}(y)>0$   
\begin{figure}[h]
\begin{center}
\includegraphics{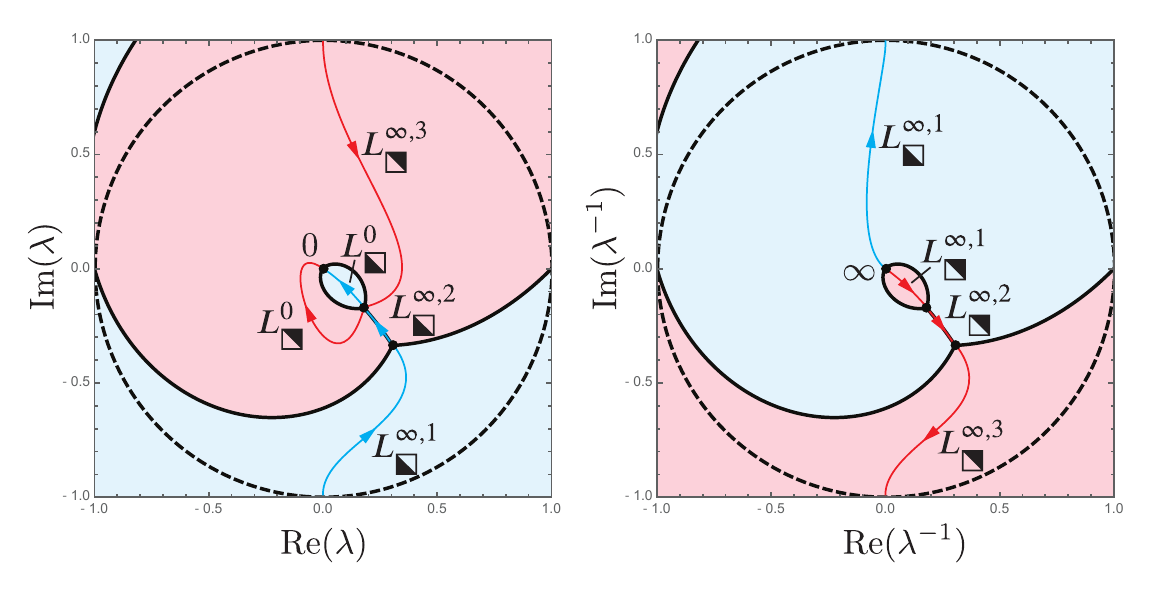}
\end{center}
\caption{The Stokes graph (black curves) for $y=0.2+0.25\ii\in E_\mathrm{R}$.  Suitable contour arcs matching the argument increment conditions \eqref{eq:increment-argument-red}--\eqref{eq:increment-argument-blue} are also shown.  As in preceding figures, the sign of $\mathrm{Re}(2g(\lambda;y)-V(\lambda;y))$ is indicated with red (negative) and blue (positive) shading.}
\label{fig:StokesGraph}
\end{figure}
and in Figure~\ref{fig:StokesGraphLowerRight} for $y\in E_\mathrm{R}$ with $\mathrm{Im}(y)<0$.
\begin{figure}[h]
\begin{center}
\includegraphics{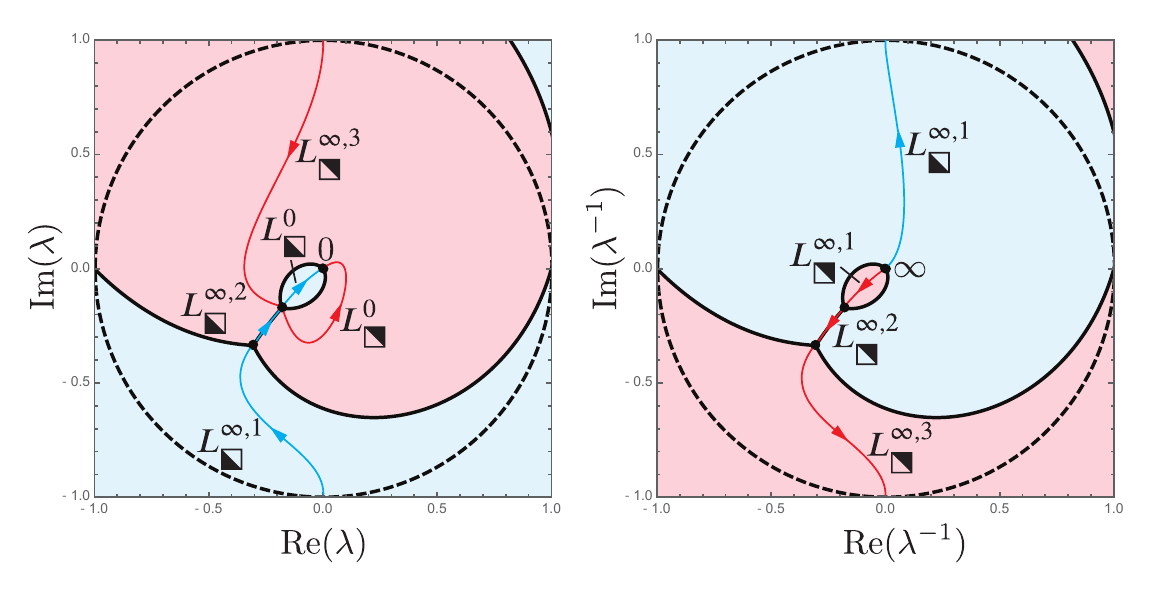}
\end{center}
\caption{As in Figure~\ref{fig:StokesGraph} but for $y=0.2-0.25\ii\in E_\mathrm{R}$.}
\label{fig:StokesGraphLowerRight}
\end{figure}
If $y\in E_\mathrm{R}$ with $\mathrm{Im}(y)=0$, we may use either configuration and obtain consistent results because as a rational function $u_n(x;m)$ is single-valued.  In the rest of this section, we will for simplicity suppose frequently that $y\in E_\mathrm{R}\setminus\mathbb{R}$ simply for the convenience of being able to speak of contour $L$ as a well-defined notion.
The vertices of the Stokes graph on the Riemann sphere are the four roots of $P(\lambda;y,C(y))$, each of which has degree $3$, and the points $\{0,\infty\}$, each of which has degree $2$.\smallskip  

The solution $\mathbf{Y}$ of Riemann-Hilbert Problem~\ref{rhp:renormalized} depends parametrically on $x=ny+w$, and when we consider $y\in E_\mathrm{R}$ we are introducing a $g$-function $g$ that depends on $y$ but not on $w$.  Therefore, in this setting the analogue of the definition \eqref{eq:Y-M-g-function} is instead
\begin{equation}
\mathbf{M}_n(\lambda;y,w,m):=\ee^{ng_\infty(y)\sigma_3}\mathbf{Y}_n(\lambda;ny+w,m)\ee^{-ng(\lambda;y)\sigma_3},
\label{eq:Y-M-Eye}
\end{equation}
i.e., the matrix $\mathbf{M}$ related to $\mathbf{Y}$ via \eqref{eq:Y-M-Eye} will depend on both $y$ and $w$ as independent parameters.  

\subsubsection{The $g$-function and its properties}
When $y\in E_\mathrm{R}\setminus\mathbb{R}$, 
the self-intersection point of $L$ is identified with the root of $P(\lambda;y,C(y))$ adjacent to $0$  in the Stokes graph.
Therefore, for $y\in E_\mathrm{R}\setminus\mathbb{R}$, the arcs $\LZeroBlue$ and $\LZeroRed$ each connect two distinct vertices of the Stokes graph, while $\LInftyBlue$ joins three consecutive vertices and $\LInftyRed$ joins four consecutive vertices.  We break these latter arcs at the intermediate vertices; thus $\LInftyBlue=\LInftyBlueOne\cup\LInftyBlueTwo$ and $\LInftyRed=\LInftyRedOne\cup\LInftyRedTwo\cup\LInftyRedThree$ with the components ordered by orientation away from $\infty$ and where $\LInftyBlueTwo$ and $\LInftyRedTwo$ are the two disjoint components of $\Sigma^\mathrm{out}(y)$.  The different sub-arcs are illustrated in Figures~\ref{fig:StokesGraph} and \ref{fig:StokesGraphLowerRight}.

With these definitions, $g(\lambda;y)$ is determined up to an integration constant by \eqref{eq:gprime-E} and the condition that $g(\lambda;y)$ is analytic for $\lambda\in\mathbb{C}\setminus(\Sigma^\mathrm{out}(y)\cup\LInftyRedThree)$.  Then, assuming that the branch cut of $\log(\lambda)$ in \eqref{eq:V-define} is disjoint from the contour $L$,  $g_+(\lambda;y)+g_-(\lambda;y)-V(\lambda;y)$ is constant along the two arcs of $\Sigma^\mathrm{out}(y)$, and we choose the integration constant (given the arbitrary choice of overall branch for $\log(\lambda)$ in \eqref{eq:V-define}) so that $g_+(\lambda;y)+g_-(\lambda;y)-V(\lambda;y)=0$ holds as an identity for $\lambda\in\LInftyRedTwo\subset\Sigma^\mathrm{out}(y)$.  In particular, $g(\infty;y)$ is well-defined mod $2\pi\ii\mathbb{Z}$.  The Stokes graph of $y$ then coincides with the zero level set of the function $\mathrm{Re}(2g(\lambda;y)-V(\lambda;y))$.  In Figures~\ref{fig:StokesGraph} and \ref{fig:StokesGraphLowerRight}, the region where $\mathrm{Re}(2g(\lambda;y)-V(\lambda;y))<0$ is shaded red while the region where $\mathrm{Re}(2g(\lambda;y)-V(\lambda;y))>0$ is shaded blue. The advantage of placing the arcs of $L$ in relation to the Stokes graph of $y$ as shown in Figures~\ref{fig:StokesGraph} and \ref{fig:StokesGraphLowerRight} is that the following conditions hold:
\begin{itemize}
\item For $\lambda\in\LZeroRed$, $g_+(\lambda;y)=g_-(\lambda;y)$ and $\mathrm{Re}(2g(\lambda;y)-V(\lambda;y))<0$.
\item For $\lambda\in\LZeroBlue$, $g_+(\lambda;y)=g_-(\lambda;y)$ and $\mathrm{Re}(2g(\lambda;y)-V(\lambda;y))>0$.
\item For $\lambda\in\LInftyRedOne$, $g_+(\lambda;y)=g_-(\lambda;y)$ and $\mathrm{Re}(2g(\lambda;y)-V(\lambda;y))<0$.
\item For $\lambda\in\LInftyRedTwo$, $g_+(\lambda;y)+g_-(\lambda;y)-V(\lambda;y)=0$ (by choice of integration constant) and $\mathrm{Re}(2g(\lambda;y)-V(\lambda;y))>0$ on both left and right sides of $\LInftyRedTwo$.
\item For $\lambda\in\LInftyRedThree$, we can use \eqref{eq:gprime-E} to deduce that
\begin{equation}
g_+(\lambda;y)-g_-(\lambda;y)=-\oint\frac{R(\ell;y)}{\ell^2}\,\dd\ell,\quad\lambda\in \LInftyRedThree,
\end{equation}
where the integration is over a counterclockwise-oriented loop surrounding $\LInftyRedTwo$.  As this loop can be interpreted as one of the homology cycles $(\mathfrak{a},\mathfrak{b})$ on the Riemann surface of the equation $\mu^2=\lambda^{-4}P(\lambda;y,C(y))$, by the Boutroux conditions \eqref{eq:Boutroux} we therefore have $g_+(\lambda;y)-g_-(\lambda;y)=\ii K_1$ where $K_1\in\mathbb{R}$ is a real constant (independent of $\lambda\in\LInftyRedThree$, but depending on $y\in E_\mathrm{R}$).  Also for $\lambda\in\LInftyRedThree$ we have $\mathrm{Re}(g_+(\lambda;y)+g_-(\lambda;y)-V(\lambda;y))<0$.
\item For $\lambda\in\LInftyBlueOne$, $g_+(\lambda;y)=g_-(\lambda;y)$ and $\mathrm{Re}(2g(\lambda;y)-V(\lambda;y))>0$.
\item For $\lambda\in\LInftyBlueTwo$, since $g_+(\lambda;y)+g_-(\lambda;y)-V(\lambda;y)=0$ holds on $\LInftyRedTwo$, integration of \eqref{eq:gprime-E} along $\LInftyRedThree$ gives
\begin{equation}
g_+(\lambda;y)+g_-(\lambda;y)-V(\lambda;y)=2\int_\LInftyRedThree\frac{R(\ell;y)}{\ell^2}\,\dd\ell.
\label{eq:LooksDifferentFor-m-NearNegativeHalfInteger}
\end{equation}
The right-hand side can also be identified with a cycle integral on the Riemann surface, so by the Boutroux conditions \eqref{eq:Boutroux} we deduce that $g_+(\lambda;y)+g_-(\lambda;y)-V(\lambda;y)=\ii K_2$ holds on $\LInftyBlueTwo$, where $K_2=K_2(y)\in\mathbb{R}$ is a real constant.
Also $\mathrm{Re}(2g(\lambda;y)-V(\lambda;y))<0$ holds on either side of the arc $\LInftyBlueTwo$.
\end{itemize}

\subsection{Szeg\H{o} function}
The Szeg\H{o} function is a kind of lower-order correction to the $g$-function.  Its dual purpose is  to remove the weak $\lambda$-dependence from the jump matrices on $\Sigma^\mathrm{out}(y)$ for $\mathbf{M}_n(\lambda;y,w,m)$ defined in \eqref{eq:Y-M-g-function} while simultaneously repairing the singularity at the origin captured by the condition that $\mathbf{M}_n(\lambda;y,w,m)\OurPower{\lambda}{-(m+1/2)\sigma_3}$ must be well-defined at $\lambda=0$.  We write the scalar Szeg\H{o} function $S(\lambda;y,m)$ in the form of an exponential:  $S(\lambda;y,m)=\ee^{L(\lambda;y,m)}$ where $L(\lambda;y,m)$ is bounded except near the origin and is analytic for $\lambda\in\mathbb{C}\setminus (\Sigma^\mathrm{out}(y)\cup\LZeroBlue)$.  The Szeg\H{o} function is then used to define a new unknown $\mathbf{N}_n(\lambda;y,w,m)$, by the formula
\begin{equation}
\mathbf{N}_n(\lambda;y,w,m):=S(\infty;y,m)^{-\sigma_3}\mathbf{M}_n(\lambda;y,w,m)S(\lambda;y,m)^{\sigma_3}=\ee^{-L(\infty;y,m)\sigma_3}\mathbf{M}_n(\lambda;y,w,m)\ee^{L(\lambda;y,m)\sigma_3}.
\end{equation}
To define the Szeg\H{o} function, we insist that the boundary values taken by $L(\lambda;y,m)$ on the arcs of its jump contour $\Sigma^\mathrm{out}(y)\cup \LZeroBlue$ are related as follows:
\begin{itemize}
\item 
For $\lambda\in \LZeroBlue$,
$L_+(\lambda;y,m)-L_-(\lambda;y,m)=-2\pi\ii (m+\tfrac{1}{2})$.
\item 
For $\lambda\in\LInftyRedTwo$, $L_+(\lambda;y,m)+L_-(\lambda;y,m)=\tfrac{1}{2}\ln(2\pi)-\log(\Gamma(\tfrac{1}{2}-m))-(m+1)\OurLog{\lambda}-\tfrac{1}{2}\ii\pi$.  
\item 
For $\lambda\in \LInftyBlueTwo$, $L_+(\lambda;y,m)+L_-(\lambda;y,m)=-(m+1)\OurLogAvg{\lambda}+\tfrac{1}{2}\ii\pi +\gamma(y,m)$.  
\end{itemize}
Here $\log(\Gamma(\tfrac{1}{2}-m))$ is an arbitrary value of the (generally complex) logarithm, we recall that $\OurLog{\lambda}:=\ln|\lambda|+\ii\OurArg{\lambda}$, and $\OurLogAvg{\lambda}$ refers to the average of the two boundary values of $\OurLog{\lambda}$ 
taken on $\LInftyBlueTwo$.  Also, $\gamma(y,m)$ is a constant to be determined so that $L(\lambda;y,m)$ tends to a well-defined limit $L(\infty;y,m)$ as $\lambda\to\infty$. Writing $L(\lambda;y,m)=R(\lambda;y)k(\lambda;y,m)$ and solving for $k$ using the Plemelj formula we obtain
\begin{multline}
k(\lambda;y,m)=-(m+\tfrac{1}{2})\int_\LZeroBlue\frac{\dd\ell}{R(\ell;y)(\ell-\lambda)}\\
{}+\frac{1}{2\pi\ii}\int_\LInftyRedTwo\frac{\tfrac{1}{2}\ln(2\pi)-\log(\Gamma(\tfrac{1}{2}-m))-(m+1)\OurLog{\ell}-\tfrac{1}{2}\ii\pi}{R_+(\ell;y)(\ell-\lambda)}\,\dd\ell\\{}+\frac{1}{2\pi\ii}
\int_\LInftyBlueTwo\frac{-(m+1)\OurLogAvg{\ell}+\tfrac{1}{2}\ii\pi+\gamma(y,m)}{R_+(\ell;y)(\ell-\lambda)}\,\dd\ell.
\end{multline}
Since $R(\lambda;y)=\mathcal{O}(\lambda^2)$ as $\lambda\to\infty$,  we need $k(\lambda;y,m)=\mathcal{O}(\lambda^{-2})$ in the same limit, which gives the condition determining $\gamma(y,m)$:
\begin{multline}
0=-(m+\tfrac{1}{2})\int_\LZeroBlue\frac{\dd\ell}{R(\ell;y)}\\
{}+\frac{1}{2\pi\ii}\int_\LInftyRedTwo\frac{\tfrac{1}{2}\ln(2\pi)-\log(\Gamma(\tfrac{1}{2}-m))-(m+1)\OurLog{\ell}-\tfrac{1}{2}\ii\pi}{R_+(\ell;y)}\,\dd\ell\\
{}+\frac{1}{2\pi\ii}\int_\LInftyBlueTwo\frac{-(m+1)\OurLogAvg{\ell}+\tfrac{1}{2}\ii\pi+\gamma(y,m)}{R_+(\ell;y)}\,\dd\ell.
\end{multline}
Note that the coefficient of $\gamma(y,m)$ is necessarily nonzero as a complete elliptic integral of the first kind.  We note also the identity
\begin{equation}
\int_\LInftyRedTwo\frac{\dd\ell}{R_+(\ell;y)} = -\int_\LInftyBlueTwo\frac{\dd\ell}{R_+(\ell;y)},
\label{eq:loop-swap}
\end{equation}
from which it follows that 
\begin{equation}
\gamma(y,m)=\tfrac{1}{2}\ln(2\pi)-\log(\Gamma(\tfrac{1}{2}-m))-\ii\pi +\tilde{\gamma}(y,m)
\label{eq:gamma-gammatilde}
\end{equation}
where
\begin{multline}
\tilde{\gamma}(y,m):=\left(\int_{\LInftyBlueTwo}\frac{\dd\lambda}{R_+(\lambda;y)}\right)^{-1}\left[2\pi\ii(m+\tfrac{1}{2})\int_{\LZeroBlue}\frac{\dd\lambda}{R(\lambda;y)} \right.\\
\left. {}+ (m+1)\left(\int_{\LInftyRedTwo}\frac{\OurLog{\lambda}\,\dd\lambda}{R_+(\lambda;y)} +
\int_{\LInftyBlueTwo}\frac{\OurLogAvg{\lambda}\,\dd\lambda}{R_+(\lambda;y)}\right)\right].
\end{multline}
Since $k(\lambda;y,m)$ exhibits negative one-half power singularities at each of the four roots of $P(\lambda;y,C(y))$, $L(\lambda;y,m)$ is bounded near these points.  Near the origin, we have $L(\lambda;y,m)=-(m+\tfrac{1}{2})\OurLog{\lambda} + \mathcal{O}(1)$, and therefore $\mathbf{N}_n(\lambda;y,w,m)$ is bounded near $\lambda=0$.\bigskip  

The jump conditions satisfied by $\mathbf{N}_n(\lambda;y,w,m)$ on the arcs of $L$ when $y\in E_\mathrm{R}\setminus\mathbb{R}$ are then as follows: 
\begin{equation}
\mathbf{N}_{n+}(\lambda;y,w,m)=\mathbf{N}_{n-}(\lambda;y,w,m)\begin{bmatrix}
1 & \displaystyle\frac{\sqrt{2\pi}\OurPower{\lambda}{-(m+1)}\ee^{-2L(\lambda;y,m)}\ee^{\ii w\varphi(\lambda)}}{\Gamma(\tfrac{1}{2}-m)}
\ee^{n(2g(\lambda;y)-V(\lambda;y))}\\
0 & 1\end{bmatrix},\quad \lambda\in\LInftyRedOne.
\label{eq:NLInftyRedOne}
\end{equation}
\begin{multline}
\mathbf{N}_{n+}(\lambda;y,w,m)=\\
\mathbf{N}_{n-}(\lambda;y,w,m)\begin{bmatrix}
\ee^{-\ii n K_1(y)} & \displaystyle\frac{\sqrt{2\pi}\OurPower{\lambda}{-(m+1)}\ee^{-2L(\lambda;y,m)}\ee^{\ii w\varphi(\lambda)}}{\Gamma(\tfrac{1}{2}-m)}
\ee^{n(g_+(\lambda;y)+g_-(\lambda;y)-V(\lambda;y))}\\
0 & \ee^{\ii n K_1(y)}\end{bmatrix},\quad \lambda\in\LInftyRedThree.
\end{multline}
\begin{multline}
\mathbf{N}_{n+}(\lambda;y,w,m)=\\
\mathbf{N}_{n-}(\lambda;y,w,m)\begin{bmatrix}
1 & 0\\
\displaystyle\frac{\sqrt{2\pi}(\OurPower{\lambda}{(m+1)/2})_+(\OurPower{\lambda}{(m+1)/2})_-\ee^{2L(\lambda;y,m)}\ee^{-\ii w\varphi(\lambda)}}{\Gamma(\tfrac{1}{2}+m)}\ee^{-n(2g(\lambda;y)-V(\lambda;y))} & 1\end{bmatrix},\quad \lambda\in\LInftyBlueOne.
\end{multline}
\begin{equation}
\mathbf{N}_{n+}(\lambda;y,w,m)=\mathbf{N}_{n-}(\lambda;y,w,m)\begin{bmatrix}
1 & \displaystyle -\frac{\sqrt{2\pi}\OurPower{\lambda}{-(m+1)}\ee^{-2L(\lambda;y,m)}\ee^{\ii w\varphi(\lambda)}}{\Gamma(\tfrac{1}{2}-m)}
\ee^{n(2g(\lambda;y)-V(\lambda;y))}\\
0 & 1\end{bmatrix},\quad\lambda\in \LZeroRed.
\end{equation}
\begin{multline}
\mathbf{N}_{n+}(\lambda;y,w,m)=\\
{}\mathbf{N}_{n-}(\lambda;y,w,m)\begin{bmatrix}
1 & 0\\
\displaystyle\frac{\sqrt{2\pi}(\OurPower{\lambda}{(m+1)/2})_+(\OurPower{\lambda}{(m+1)/2})_-\ee^{L_+(\lambda;y,m)+L_-(\lambda;y,m)}\ee^{-\ii w\varphi(\lambda)}}{\Gamma(\tfrac{1}{2}+m)}\ee^{-n(2g(\lambda;y)-V(\lambda;y))} & 1\end{bmatrix},\\
\lambda\in\LZeroBlue.
\label{eq:NLZeroBlue}
\end{multline}
\begin{multline}
\mathbf{N}_{n+}(\lambda;y,w,m)\begin{bmatrix}1 & 0\\\displaystyle
-\frac{\Gamma(\tfrac{1}{2}-m)\ee^{2L_+(\lambda;y,m)}\ee^{-\ii w\varphi(\lambda)}}{\sqrt{2\pi}\OurPower{\lambda}{-(m+1)}}\ee^{-n(2g_+(\lambda;y)-V(\lambda;y))} & 1\end{bmatrix}=\\
\mathbf{N}_{n-}(\lambda;y,w,m)\begin{bmatrix}1 & 0\\\displaystyle
\frac{\Gamma(\tfrac{1}{2}-m)\ee^{2L_-(\lambda;y,m)}\ee^{-\ii w\varphi(\lambda)}}{\sqrt{2\pi}\OurPower{\lambda}{-(m+1)}}\ee^{-n(2g_-(\lambda;y)-V(\lambda;y))} & 1\end{bmatrix}\ii\sigma_1\ee^{-\ii w\varphi(\lambda)\sigma_3}
,\quad\lambda\in\LInftyRedTwo.
\label{eq:NLInftyRedTwo}
\end{multline}
\begin{multline}
\mathbf{N}_{n+}(\lambda;y,w,m)\begin{bmatrix}1 & \displaystyle-\ee^{\ii\pi(m+1)}\frac{\Gamma(\tfrac{1}{2}+m)\ee^{-2L_+(\lambda;y,m)}\ee^{\ii w\varphi(\lambda)}}{\sqrt{2\pi}(\OurPower{\lambda}{m+1})_+}\ee^{n(2g_+(\lambda;y)-V(\lambda;y))}\\0 & 1\end{bmatrix}=\\
\mathbf{N}_{n-}(\lambda;y,w,m)
\begin{bmatrix}
1 & \displaystyle\ee^{-\ii\pi (m+1)}\frac{\Gamma(\tfrac{1}{2}+m)\ee^{-2L_-(\lambda;y,m)}\ee^{\ii w\varphi(\lambda)}}{\sqrt{2\pi}(\OurPower{\lambda}{m+1})_-}\ee^{n(2g_-(\lambda;y)-V(\lambda;y))}\\
0 & 1
\end{bmatrix}\ii\sigma_1\ee^{\delta(y,m)\sigma_3}\ee^{-\ii n K_2(y)\sigma_3}\ee^{-\ii w\varphi(\lambda)\sigma_3},\\\lambda\in\LInftyBlueTwo,
\label{eq:NLInftyBlueTwo}
\end{multline}
where 
\begin{equation}
\ee^{\delta(y,m)}:=-2\cos(\pi m)\ee^{\tilde{\gamma}(y,m)},
\label{eq:e-to-the-delta}
\end{equation}
in which the product $\Gamma(\tfrac{1}{2}-m)\Gamma(\tfrac{1}{2}+m)$ has been eliminated using \cite[Eq.\@ 5.5.3]{DLMF}.

\begin{rem}
Referring to \eqref{eq:e-to-the-delta}, it is the fact that $\ee^{\delta(y,m)}=0$ and hence $\delta(y,m)$ is undefined when $m\in\mathbb{Z}+\tfrac{1}{2}$ that excludes the latter values from consideration in this section and hence in the statements of Theorem~\ref{theorem:eye}, Corollary~\ref{corollary:eye-zeros-and-poles:better}, and Theorem~\ref{theorem:density}.
\end{rem}

\subsection{Steepest descent.  Outer model problem and its solution}
\subsubsection{Steepest descent and the derivation of the outer model Riemann-Hilbert problem}
For the steepest descent step, we take advantage of the factorization of the jump matrix evidenced in the formul\ae\ \eqref{eq:NLInftyRedTwo} and \eqref{eq:NLInftyBlueTwo}.  Let $\LensRedPM$ denote lens-shaped domains immediately to the left ($+$) and right ($-$) of $\LInftyRedTwo$.  Define
\begin{equation}
\mathbf{O}_n(\lambda;y,w,m):=\mathbf{N}_n(\lambda;y,w,m)\begin{bmatrix}1 & 0\\
\displaystyle\mp\frac{\Gamma(\tfrac{1}{2}-m)\ee^{2L(\lambda;y,m)}\ee^{-\ii w\varphi(\lambda)}}{\sqrt{2\pi}\OurPower{\lambda}{-(m+1)}}\ee^{-n(2g(\lambda;y)-V(\lambda;y))} & 1\end{bmatrix},\quad \lambda\in\LensRedPM.
\end{equation}
Similarly, let $\LensBluePM$ denote lens-shaped domains immediately to the left ($+$) and right ($-$) of $\LInftyBlueTwo$, and define
\begin{equation}
\mathbf{O}_n(\lambda;y,w,m):=\mathbf{N}_n(\lambda;y,w,m)\begin{bmatrix}1 &
\displaystyle
\mp\ee^{\pm\ii\pi(m+1)}\frac{\Gamma(\tfrac{1}{2}+m)\ee^{-2L(\lambda;y,m)}\ee^{\ii w\varphi(\lambda)}}{\sqrt{2\pi}\OurPower{\lambda}{m+1}}\ee^{n(2g(\lambda;y)-V(\lambda;y))}\\0 & 1\end{bmatrix},\quad\lambda\in\LensBluePM.
\end{equation}
For all other values of $\lambda$ for which $\mathbf{N}_n(\lambda;y,w,m)$ is well-defined, we simply set $\mathbf{O}_n(\lambda;y,w,m):=\mathbf{N}_n(\lambda;y,w,m)$.  If we denote by $\partial\LensRedPM$ (resp., $\partial\LensBluePM$) the arc of the boundary of $\LensRedPM$ (resp., $\LensBluePM$) distinct from $\LInftyRedTwo$ (resp., $\LInftyBlueTwo$), but with the same initial and terminal endpoints, then the boundary values taken by $\mathbf{O}_n(\lambda;y,w,m)$ on these arcs satisfy the jump conditions
\begin{equation}
\mathbf{O}_{n+}(\lambda;y,w,m)=\mathbf{O}_{n-}(\lambda;y,w,m)\begin{bmatrix}
1 & 0\\
\displaystyle\frac{\Gamma(\tfrac{1}{2}-m)\ee^{2L(\lambda;y,m)}\ee^{-\ii w\varphi(\lambda)}}{\sqrt{2\pi}\OurPower{\lambda}{-(m+1)}}\ee^{-n(2g(\lambda;y)-V(\lambda;y))} & 1\end{bmatrix},\quad\lambda\in\partial\LensRedPM,
\end{equation}
and
\begin{equation}
\mathbf{O}_{n+}(\lambda;y,w,m)=\mathbf{O}_{n-}(\lambda;y,w,m)\begin{bmatrix}
1 & \displaystyle\ee^{\pm\ii\pi (m+1)}\frac{\Gamma(\tfrac{1}{2}+m)\ee^{-2L(\lambda;y,m)}\ee^{\ii w\varphi(\lambda)}}{\sqrt{2\pi}\OurPower{\lambda}{m+1}}\ee^{n(2g(\lambda;y)-V(\lambda;y))}\\
0 & 1\end{bmatrix},\quad\lambda\in\partial\LensBluePM.
\end{equation}
The effect of the transformation from $\mathbf{N}_n(\lambda;y,w,m)$ to $\mathbf{O}_n(\lambda;y,w,m)$ is that the jump matrices for $\mathbf{O}_n(\lambda;y,w,m)$ on $\LInftyRedTwo$ and $\LInftyBlueTwo$ are now simply off-diagonal matrices:
\begin{equation}
\mathbf{O}_{n+}(\lambda;y,w,m)=\mathbf{O}_{n-}(\lambda;y,w,m)\ii\sigma_1\ee^{-\ii w\varphi(\lambda)\sigma_3},\quad\lambda\in\LInftyRedTwo,
\end{equation}
and
\begin{equation}
\mathbf{O}_{n+}(\lambda;y,w,m)=\mathbf{O}_{n-}(\lambda;y,w,m)\ii\sigma_1\ee^{\delta(y,m)\sigma_3}\ee^{-\ii nK_2(y)\sigma_3}\ee^{-\ii w\varphi(\lambda)\sigma_3},\quad\lambda\in\LInftyBlueTwo.
\end{equation}
On all remaining arcs of $L$, the boundary values of $\mathbf{O}_n(\lambda;y,w,m)$ agree with those of $\mathbf{N}_n(\lambda;y,w,m)$, which are related by the jump conditions \eqref{eq:NLInftyRedOne}--\eqref{eq:NLZeroBlue}.  Finally, we note that $\lambda\mapsto\mathbf{O}_n(\lambda;y,w,m)$ is analytic for $\lambda\in \mathbb{C}\setminus\Sigma_\mathbf{O}$, where $\Sigma_\mathbf{O}:=L\cup\partial\LensRedPM\cup\partial\LensBluePM$, taking continuous boundary values from each component of its domain of analyticity, and satisfies $\mathbf{O}_n(\lambda;y,w,m)\to\mathbb{I}$ as $\lambda\to\infty$.  

The placement of the arcs of $L$ relative to the Stokes graph of $y$ now ensures that all jump matrices converge exponentially fast to the identity as $n\to+\infty$ with the exception of those on the arcs $\LInftyRedTwo\cup\LInftyRedThree\cup\LInftyBlueTwo$.  The convergence holds uniformly on compact subsets of each open contour arc, as well as uniformly in neighborhoods of $\lambda=0$ and $\lambda=\infty$.  Building in suitable assumptions about the behavior near the four roots of $P(\lambda;y,C(y))$, we postulate the following model Riemann-Hilbert problem as an asymptotic description of $\mathbf{O}_n(\lambda;y,w,m)$ away from these four points.
\begin{rhp}[Outer model problem]
Given $n\in\mathbb{Z}$, $y\in E_\mathrm{R}\setminus\mathbb{R}$, $w\in\mathbb{C}$, and $m\in\mathbb{C}\setminus(\mathbb{Z}+\tfrac{1}{2})$, seek a $2\times 2$ matrix function $\lambda\mapsto\dot{\mathbf{O}}^\mathrm{out}(\lambda)=\dot{\mathbf{O}}_n^{\mathrm{out}}(\lambda;y,w,m)$ with the following properties:
\begin{itemize}
\item[1.]\textbf{Analyticity:}  $\lambda\mapsto\dot{\mathbf{O}}^{\mathrm{out}}(\lambda)$ is analytic in the domain $\lambda\in\mathbb{C}\setminus(\LInftyRedTwo\cup\LInftyRedThree\cup\LInftyBlueTwo)$.  It takes continuous boundary values on the three indicated arcs of $L$ except at the four endpoints $\lambda_j(y)$ at which we require that all four matrix elements are $\mathcal{O}((\lambda-\lambda_j(y))^{-1/4})$.
\item[2.]\textbf{Jump conditions:}  The boundary values $\dot{\mathbf{O}}^{\mathrm{out}}_\pm(\lambda)$ are related on each arc of the jump contour by the following formul\ae:
\begin{equation}
\dot{\mathbf{O}}^{\mathrm{out}}_+(\lambda)=\dot{\mathbf{O}}^{\mathrm{out}}_-(\lambda)\ii\sigma_1\ee^{-\ii w\varphi(\lambda)\sigma_3},\quad\lambda\in\LInftyRedTwo,
\end{equation}
\begin{equation}
\dot{\mathbf{O}}^{\mathrm{out}}_+(\lambda)=\dot{\mathbf{O}}^{\mathrm{out}}_-(\lambda)\ee^{-\ii n K_1(y)\sigma_3},\quad\lambda\in\LInftyRedThree,
\end{equation}
and
\begin{equation}
\dot{\mathbf{O}}^{\mathrm{out}}_+(\lambda)=\dot{\mathbf{O}}^{\mathrm{out}}_-(\lambda)\ii\sigma_1\ee^{\delta(y,m)\sigma_3}\ee^{-\ii n K_2(y)\sigma_3}\ee^{-\ii w\varphi(\lambda)\sigma_3},\quad\lambda\in\LInftyBlueTwo.
\end{equation}
\item[3.]\textbf{Asymptotics:}  $\dot{\mathbf{O}}^{\mathrm{out}}(\lambda)\to\mathbb{I}$ as $\lambda\to\infty$.
\end{itemize}
\label{rhp:outer-elliptic}
\end{rhp}
The jump diagram for Riemann-Hilbert Problem~\ref{rhp:outer-elliptic} is illustrated in Figure~\ref{fig:OuterEllipticR}.  The solution of this problem (see Section~\ref{sec:elliptic-outer-solution} below) is called the \emph{outer parametrix}.
\begin{figure}[h]
\begin{center}
\includegraphics{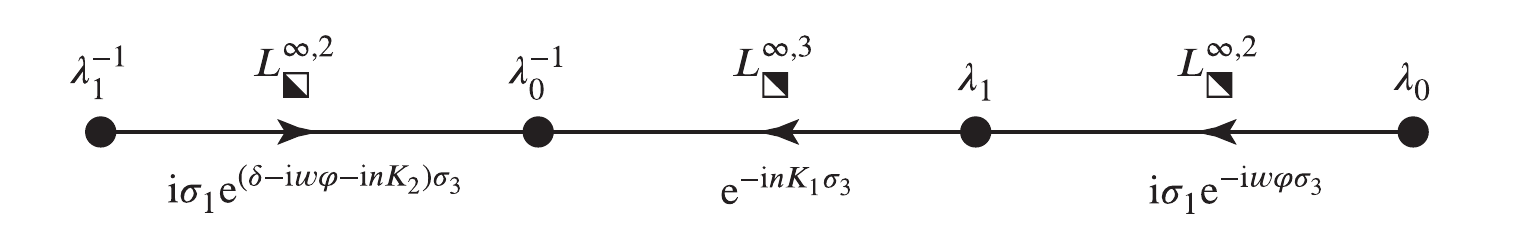}
\end{center}
\caption{The jump contour and jump conditions for Riemann-Hilbert Problem~\ref{rhp:outer-elliptic}.  Note that we denote by $\lambda_0$ the vertex of the Stokes graph adjacent to $\infty$.}
\label{fig:OuterEllipticR}
\end{figure}

\subsubsection{Solution of the outer model Riemann-Hilbert problem}
\label{sec:elliptic-outer-solution}
To solve Riemann-Hilbert Problem~\ref{rhp:outer-elliptic}, first let $H(\lambda;y)$ be defined by the formula
\begin{multline}
H(\lambda;y):=\frac{R(\lambda;y)K_1(y)}{2\pi}\int_{\LInftyRedThree}\frac{\dd\ell}{R(\ell;y)(\ell-\lambda)}+\frac{R(\lambda;y)(\ii K_2(y)+\eta(y))}{2\pi\ii}\int_{\LInftyBlueTwo}\frac{\dd\ell}{R_+(\ell;y)(\ell-\lambda)},\\\lambda\in\mathbb{C}\setminus(\LInftyRedTwo\cup\LInftyRedThree\cup\LInftyBlueTwo).
\label{eq:H-exponent}
\end{multline}
Here, $\eta(y)$ is uniquely determined so that $H(\infty;y)$ is well-defined:
\begin{equation}
\ii K_1(y)\int_\LInftyRedThree\frac{\dd\ell}{R(\ell;y)} + (\ii K_2(y)+\eta(y))\int_\LInftyBlueTwo\frac{\dd\ell}{R_+(\ell;y)}=0.
\label{eq:eta-defining-equation}
\end{equation}
Unlike the real-valued quantities $K_1(y)$ and $K_2(y)$, $\eta(y)$ is complex-valued, and it is well-defined because its coefficient is a complete elliptic integral, necessarily nonzero.  The boundary values taken by $H(\lambda;y)$ on its jump contour are related by the conditions
\begin{equation}
H_+(\lambda;y)+H_-(\lambda;y)=0,\quad\lambda\in\LInftyRedTwo,
\label{eq:Hjump-1}
\end{equation}
\begin{equation}
H_+(\lambda;y)-H_-(\lambda;y)=\ii K_1(y),\quad\lambda\in \LInftyRedThree,
\end{equation}
and
\begin{equation}
H_+(\lambda;y)+H_-(\lambda;y)=\ii K_2(y)+\eta(y),\quad\lambda\in\LInftyBlueTwo.
\label{eq:Hjump-3}
\end{equation}
We also define a related function $h(\lambda;y)$ by
\begin{equation}
h(\lambda;y):=\frac{1}{2}\varphi(\lambda)+\frac{R(\lambda;y)}{\ii y\lambda}+\frac{\nu(y)R(\lambda;y)}{2\pi\ii}\int_{\LInftyBlueTwo}\frac{\dd\ell}{R_+(\ell;y)(\ell-\lambda)},\quad
\lambda\in\mathbb{C}\setminus(\LInftyRedTwo\cup\LInftyBlueTwo)
\label{eq:h-exponent}
\end{equation}
(note that $h(\lambda;y)$ is analytic at $\lambda=0$ because $R(0;y)=\tfrac{1}{2}\ii y$), in which $\nu(y)$ is a constant determined uniquely by setting to zero the coefficient of the dominant term proportional to $\lambda$ in the Laurent series of $h$ at $\lambda=\infty$, making $h(\infty;y)$ well-defined:
\begin{equation}
1-\frac{\nu(y)y}{4\pi}\int_\LInftyBlueTwo\frac{\dd\ell}{R_+(\ell;y)}=0.
\label{eq:nu-defining-equation}
\end{equation}
The analogues of the conditions \eqref{eq:Hjump-1}--\eqref{eq:Hjump-3} for $h$ are
\begin{equation}
h_+(\lambda;y)+h_-(\lambda;y)=\varphi(\lambda),\quad\lambda\in\LInftyRedTwo,
\end{equation}
\begin{equation}
h_+(\lambda;y)-h_-(\lambda;y)=0,\quad\lambda\in\LInftyRedThree,
\end{equation}
and
\begin{equation}
h_+(\lambda;y)+h_-(\lambda;y)=\varphi(\lambda)+\nu(y),\quad\lambda\in\LInftyBlueTwo.
\end{equation}

It follows that the matrix $\mathbf{P}_n(\lambda;y,w,m)$ related to the solution $\dot{\mathbf{O}}_n^{\mathrm{out}}(\lambda;y,w,m)$ of Riemann-Hilbert Problem~\ref{rhp:outer-elliptic} by 
\begin{equation}
\mathbf{P}_n(\lambda;y,w,m):=\ee^{-(nH(\infty;y)+\ii w h(\infty,y))\sigma_3}\dot{\mathbf{O}}_n^{\mathrm{out}}(\lambda;y,w,m)\ee^{(nH(\lambda;y)+\ii w h(\lambda;y))\sigma_3}
\label{eq:POout}
\end{equation}
has the same properties of analyticity, boundedness, and identity normalization at $\lambda=\infty$ as does
$\dot{\mathbf{O}}_n^{\mathrm{out}}(\lambda;y,w,m)$, but the jump conditions for $\mathbf{P}_n(\lambda;y,w,m)$ take the form
\begin{equation}
\mathbf{P}_{n+}(\lambda;y,w,m)=\mathbf{P}_{n-}(\lambda;y,w,m)\ii\sigma_1,\quad\lambda\in\LInftyRedTwo,
\label{eq:PLInftyRedTwo}
\end{equation}
\begin{equation}
\mathbf{P}_{n+}(\lambda;y,w,m)=\mathbf{P}_{n-}(\lambda;y,w,m),\quad\lambda\in\LInftyRedThree,
\label{eq:PLInftyRedThree}
\end{equation}
and
\begin{equation}
\mathbf{P}_{n+}(\lambda;y,w,m)=\mathbf{P}_{n-}(\lambda;y,w,m)\ii\sigma_1\ee^{(\delta(y,m)+\ii w\nu(y))\sigma_3}\ee^{n\eta(y)\sigma_3},\quad\lambda\in\LInftyBlueTwo.
\label{eq:PLInftyBlueTwo}
\end{equation}
The jump condition \eqref{eq:PLInftyRedThree} together with the continuity of the boundary values taken by $\mathbf{P}_n(\lambda;y,m)$ on $\LInftyRedThree$ from both sides indicates that the domain of analyticity of $\mathbf{P}_n(\lambda;y,w,m)$ is precisely the ``two-cut'' contour $\Sigma^\mathrm{out}(y)=\LInftyRedTwo\cup\LInftyBlueTwo$.

Let $\mathfrak{a}$ denote a counterclockwise-oriented loop surrounding the cut $\LInftyRedTwo$, and define the \emph{Abel mapping} $A(\lambda;y)$ by
\begin{equation}
A(\lambda;y):=2\pi\ii\left[\oint_\mathfrak{a}\frac{\dd\ell}{R(\ell;y)}\right]^{-1}\int_{\lambda_0(y)}^\lambda\frac{\dd\ell}{R(\ell;y)},\quad\lambda\in\mathbb{C}\setminus(\LInftyRedTwo\cup\LInftyRedThree\cup\LInftyBlueTwo),
\label{eq:Abel-def}
\end{equation}
where $\lambda_0(y)$ is the vertex adjacent to $\infty$ on the Stokes graph of $y$ (hence the initial endpoint of $\LInftyRedTwo$).  Note that $A(\lambda;y)$ is well-defined because $1/R(\lambda;y)$ is integrable at $\lambda=\infty$.  The integral over the corresponding $\mathfrak{b}$-cycle (in the canonical homology basis determined from $\mathfrak{a}$) of the $\mathfrak{a}$-normalized holomorphic differential that is the integrand of $A(\lambda;y)$ is then given by 
\begin{equation}
B(y):=-4\pi\ii\left[\oint_\mathfrak{a}\frac{\dd\ell}{R(\ell;y)}\right]^{-1}\int_\LInftyRedThree\frac{\dd\ell}{R(\ell;y)}.
\label{eq:B-of-y}
\end{equation}
Since 
\begin{equation}
\oint_\mathfrak{a}\frac{\dd\ell}{R(\ell;y)}=-2\int_\LInftyRedTwo\frac{\dd\ell}{R_+(\ell;y)} = 2\int_\LInftyBlueTwo\frac{\dd\ell}{R_+(\ell;y)},
\end{equation}
with the second equality following from \eqref{eq:loop-swap}, we can use \eqref{eq:eta-defining-equation} to write $\eta(y)$ in the form 
\begin{equation}
\eta(y)=-\ii K_2(y)+\ii K_1(y)\frac{B(y)}{2\pi\ii}.
\label{eq:eta-again}
\end{equation}
It is a general fact \cite{Dubrovin81} that $\mathrm{Re}(B(y))<0$, which implies that therefore $\mathrm{Re}(\eta(y))\neq 0$ unless $K_1(y)=0$.  More concretely, by comparing with the Stokes graphs illustrated in Figure~\ref{fig:StokesGraphReal}, it is easy to see that for $y>0$ in the domain $E$, $B(y)$ is real and strictly negative.  The Abel mapping satisfies the following jump conditions:
\begin{equation}
A_+(\lambda;y)+A_-(\lambda;y)=0,\quad\lambda\in\LInftyRedTwo,
\label{eq:ALInftyRedTwo}
\end{equation}
\begin{equation}
A_+(\lambda;y)-A_-(\lambda;y)=-2\pi\ii,\quad\lambda\in\LInftyRedThree,
\end{equation}
and
\begin{equation}
A_+(\lambda;y)+A_-(\lambda;y)=-B(y),\quad\lambda\in\LInftyBlueTwo.
\label{eq:ALInftyBlueTwo}
\end{equation}
We now recall the \emph{Riemann theta function} $\Theta(z,B)$ defined by the series
\begin{equation}
\Theta(z,B):=\sum_{k\in\mathbb{Z}}\ee^{kz+\tfrac{1}{2}Bk^2},\quad z\in\mathbb{C},\quad \mathrm{Re}(B)<0.
\label{eq:Riemann-Theta}
\end{equation}
In the notation of \cite[\S20]{DLMF}, $\Theta(z,B)=\theta_3(w|\tau)$ where $z=2\ii w$ and $B=2\pi\ii\tau$ (i.e., in the currently relevant genus-$1$ setting the Riemann theta function basically coincides with one of the Jacobi theta functions). For each $B$ in the left half-plane, $\Theta(z,B)$ is an entire function of $z$ with the automorphic properties
\begin{equation}
\Theta(-z,B)=\Theta(z,B),\quad
\Theta(z+2\pi\ii,B)=\Theta(z,B),\quad\text{and}\quad\Theta(z\pm B,B)=\ee^{-\tfrac{1}{2}B}\ee^{\mp z}\Theta(z,B).
\label{eq:automorphic}
\end{equation}
The function $\Theta(z,B)$ has simple zeros only, at each of the lattice points $z=(j+\tfrac{1}{2})2\pi\ii + (k+\tfrac{1}{2})B$, $(j,k)\in\mathbb{Z}^2$.  Given a point $\kappa\in\mathbb{C}\setminus(\LInftyRedTwo\cup\LInftyRedThree\cup\LInftyBlueTwo)$ and a complex number $s$, consider the meromorphic functions defined by
\begin{equation}
q^\pm(\lambda;\kappa,s,y):=\frac{\Theta(A(\lambda;y)\pm A(\kappa;y)\pm\ii\pi\pm\tfrac{1}{2}B(y)-s,B(y))}{\Theta(A(\lambda;y)\pm A(\kappa;y)\pm\ii\pi\pm\tfrac{1}{2}B(y),B(y))},\quad\lambda\in\mathbb{C}\setminus(\LInftyRedTwo\cup\LInftyRedThree\cup\LInftyBlueTwo).
\label{eq:gpm-def}
\end{equation}
In fact, $q^+(\lambda;\kappa,s,y)$ is analytic for $\lambda$ in its domain of definition, but $q^-(\lambda;\kappa,s,y)$ has a simple pole at $\lambda=\kappa$ as its only singularity (unless $s$ is an integer linear combination of $2\pi\ii$ and $B(y)$ in which case the singularity is cancelled and $q^-(\lambda;\kappa,s,y)$ is analytic as well).  Consider the matrix function 
\begin{equation}
\mathbf{Q}(\lambda;\kappa,s,y):=\begin{bmatrix}q^+(\lambda;\kappa,s,y) & -\ii q^-(\lambda;\kappa,-s,y)\\
\ii q^-(\lambda;\kappa,s,y) & q^+(\lambda;\kappa,-s,y)\end{bmatrix}.
\label{eq:Q-def}
\end{equation}
Then from the jump conditions \eqref{eq:ALInftyRedTwo}--\eqref{eq:ALInftyBlueTwo} and the automorphic properties \eqref{eq:automorphic}, it is easy to check that $\mathbf{Q}(\lambda;\kappa,s,y)$ satisfies the jump conditions:
\begin{equation}
\mathbf{Q}_+(\lambda;\kappa,s,y)=\mathbf{Q}_-(\lambda;\kappa,s,y)\begin{bmatrix}0 & -\ii\\\ii & 0\end{bmatrix},\quad\lambda\in\LInftyRedTwo,
\label{eq:QLInftyRedTwo}
\end{equation}
\begin{equation}
\mathbf{Q}_+(\lambda;\kappa,s,y)=\mathbf{Q}_-(\lambda;\kappa,s,y),\quad\lambda\in\LInftyRedThree,
\end{equation}
and
\begin{equation}
\mathbf{Q}_+(\lambda;\kappa,s,y)=\mathbf{Q}_-(\lambda;\kappa,s,y)\begin{bmatrix}0 & -\ii\ee^{s}\\\ii\ee^{-s} & 0\end{bmatrix},\quad \lambda\in\LInftyBlueTwo.
\label{eq:QLInftyBlueTwo}
\end{equation}

To construct $\mathbf{P}_n(\lambda;y,w,m)$ from $\mathbf{Q}$, we need to remove the pole from the off-diagonal elements of $\mathbf{Q}$ while slightly modifying the jump conditions on $\Sigma^\mathrm{out}(y)$.  To this end, we observe that we have the freedom to introduce mild singularities into $\mathbf{P}_n(\lambda;y,w,m)$ at the four roots of $P(\cdot;y,C(y))$, here denoted $\lambda_0(y)$ (adjacent to $\infty$ in the Stokes graph of $y$), $\lambda_1(y)$ (adjacent to $\lambda_0(y)$ in the Stokes graph), $\lambda_0(y)^{-1}$, and $\lambda_1(y)^{-1}$.  Let $\phi(\lambda;y)$ denote the unique function analytic for $\lambda\in\Sigma^\mathrm{out}(y)$ with $\phi(\lambda;y)\to 1$ as $\lambda\to\infty$ that satisfies
\begin{equation}
\phi(\lambda;y)^4=\frac{(\lambda-\lambda_0(y))(\lambda-\lambda_1(y)^{-1})}{(\lambda-\lambda_1(y))(\lambda-\lambda_0(y)^{-1})}.
\label{eq:phi-to-the-fourth}
\end{equation}
Then set 
\begin{equation}
f^\mathrm{D}(\lambda;y):=\frac{1}{2}(\phi(\lambda;y)+\phi(\lambda;y)^{-1}),\quad
f^\mathrm{OD}(\lambda;y):=\frac{1}{2\ii}(\phi(\lambda;y)-\phi(\lambda;y)^{-1}),\quad\lambda\in\mathbb{C}\setminus\Sigma^\mathrm{out}(y).
\label{eq:qDqOD}
\end{equation}
It is easy to see that on both arcs of $\Sigma^\mathrm{out}(y)$, the jump condition $\phi_+(\lambda;y)=-\ii\phi_-(\lambda;y)$ holds.  This implies the corresponding jump conditions
\begin{equation}
f^\mathrm{D}_+(\lambda;y)=f^\mathrm{OD}_-(\lambda;y),\quad
f^\mathrm{OD}_+(\lambda;y)=-f^\mathrm{D}_-(\lambda;y),\quad\lambda\in\Sigma^\mathrm{out}(y).
\label{eq:qDOD}
\end{equation}
The functions $\lambda\mapsto f^\mathrm{D}(\lambda;y)$ and $\lambda\mapsto f^\mathrm{OD}(\lambda;y)$ are analytic in their domain of definition, and they are bounded except near the four roots of $P(\lambda;y,C(y))$, where they exhibit negative one-fourth root singularities.  Also,
\begin{equation}
f^\mathrm{D}(\lambda;y)=1+\mathcal{O}(\lambda^{-2})\quad\text{and}\quad
f^\mathrm{OD}(\lambda;y)=\frac{1}{4}\ii(\lambda_0(y)+\lambda_1(y)^{-1}-\lambda_1(y)-\lambda_0(y)^{-1})\lambda^{-1}+\mathcal{O}(\lambda^{-2}),\quad\lambda\to\infty.
\end{equation}
Observe that
\begin{equation}
\begin{split}
f^\mathrm{D}(\lambda;y)f^\mathrm{OD}(\lambda;y)&=\frac{1}{4\ii\phi(\lambda;y)^2}(\phi(\lambda;y)^4-1)\\
&=\frac{(\lambda-\lambda_0(y))(\lambda-\lambda_1(y)^{-1})-(\lambda-\lambda_1(y))(\lambda-\lambda_0(y)^{-1})}{4\ii\phi(\lambda;y)^2(\lambda-\lambda_1(y))(\lambda-\lambda_0(y)^{-1})}\\
&=\frac{(\lambda_1(y)+\lambda_0(y)^{-1}-\lambda_0(y)-\lambda_1(y)^{-1})(\lambda-\kappa(y))}{4\ii\phi(\lambda;y)^2(\lambda-\lambda_1(y))(\lambda-\lambda_0(y)^{-1}),}
\end{split}
\end{equation}
where
\begin{equation}
\kappa(y):=\frac{\lambda_1(y)\lambda_0(y)^{-1}-\lambda_0(y)\lambda_1(y)^{-1}}{\lambda_1(y)+\lambda_0(y)^{-1}-\lambda_0(y)-\lambda_1(y)^{-1}}=\frac{\lambda_0(y)+\lambda_1(y)}{1+\lambda_0(y)\lambda_1(y)}.
\label{eq:mu-def}
\end{equation}
Therefore the product $f^\mathrm{D}(\lambda;y)f^\mathrm{OD}(\lambda;y)$ has precisely one simple zero in its domain of definition, namely $\lambda=\kappa(y)$, and this value is either a zero of $f^\mathrm{D}(\lambda;y)$ or $f^\mathrm{OD}(\lambda;y)$ but not both.  In the case that $y>0$, the roots of $P(\lambda;y,C(y))$ lie on the imaginary axis with $1<|\lambda_1(y)|<|\lambda_0(y)|$.  It is easy to check that $\phi(\lambda;y)$ is positive on the imaginary axis excluding the jump contour $\Sigma^\mathrm{out}(y)$, which also implies that $f^\mathrm{D}(\lambda;y)>0$ for such $\lambda$.    The inequality $1<|\lambda_1(y)|<|\lambda_0(y)|$ implies that $\kappa(y)$ is negative imaginary, and that $|\kappa(y)|>|\lambda_1(y)^{-1}|$.  Thus $\kappa(y)$ lies below both intervals of the jump contour $\Sigma^\mathrm{out}(y)$ on the imaginary axis, and hence $f^\mathrm{D}(\kappa(y);y)>0$.  It therefore follows that for $y>0$, $\kappa(y)$ is a zero of $f^\mathrm{OD}(\lambda;y)$.  This will remain so as $y$ varies in $E_\mathrm{R}$ so long as $\kappa(y)$ does not pass through either arc of $\Sigma^\mathrm{out}(y)$.  We proceed under the assumption that $\lambda=\kappa(y)$ is a simple zero of $f^\mathrm{OD}(\lambda;y)$, and indicate below how the procedure should be modified if $\kappa(y)$ should ever intersect $\Sigma^\mathrm{out}(y)$, a possibility which is difficult to rule out analytically, although we have never observed it numerically.\smallskip  

We may obtain $\mathbf{P}_n(\lambda;y,w,m)$ from $\mathbf{Q}(\lambda;\mu(y),s,y)$ by multiplying the diagonal elements by $f^\mathrm{D}(\lambda;y)$ and the off-diagonal elements by $\pm f^\mathrm{OD}(\lambda;y)$, and by normalizing the result via left-multiplication by a constant matrix:
\begin{multline}
\mathbf{P}_n(\lambda;y,w,m):=\begin{bmatrix} Q_{11}(\infty;\kappa(y),s,y)^{-1} & 0\\
0 & Q_{22}(\infty;\kappa(y),s,y)^{-1}\end{bmatrix}\\
{}\cdot
\begin{bmatrix} f^\mathrm{D}(\lambda;y)Q_{11}(\lambda;\kappa(y),s,y) & f^\mathrm{OD}(\lambda;y)Q_{12}(\lambda;\kappa(y),s,y)\\
-f^\mathrm{OD}(\lambda;y)Q_{21}(\lambda;\kappa(y),s,y) & f^\mathrm{D}(\lambda;y)Q_{22}(\lambda;\kappa(y),s,y)
\end{bmatrix}.
\label{eq:P-formula}
\end{multline}
Combining \eqref{eq:QLInftyRedTwo}--\eqref{eq:QLInftyBlueTwo} with \eqref{eq:qDOD} shows that $\mathbf{P}_n(\lambda;y,w,m)$ satisfies the prescribed jump conditions \eqref{eq:PLInftyRedTwo}--\eqref{eq:PLInftyBlueTwo} provided that the free parameter $s$ is given the value
\begin{equation}
s=s_n(y,w,m):=-\delta(y,m)-\ii w\nu(y)-n\eta(y).
\label{eq:sn}
\end{equation}
Since the zero $\lambda=\kappa(y)$ of $f^\mathrm{OD}(\lambda;y)$ cancels the simple pole of $Q_{12}(\lambda;\kappa(y),s,y)$ and $Q_{21}(\lambda;\kappa(y),s,y)$, the singularity is removable and hence $\mathbf{P}_n(\lambda;y,w,m)$ is indeed analytic for $\lambda\in\mathbb{C}\setminus\Sigma^\mathrm{out}(y)$ with negative one-fourth root singularities at the roots of $P(\lambda;y,C(y))$.  
Finally, the constant matrix pre-factor ensures the asymptotic normalization condition that $\mathbf{P}_n(\infty;y,w,m)=\mathbb{I}$.  Now that $\mathbf{P}_n(\lambda;y,w,m)$ has been determined, we recover $\dot{\mathbf{O}}_n^{\mathrm{out}}(\lambda;y,w,m)$ using \eqref{eq:H-exponent}, \eqref{eq:POout}, \eqref{eq:B-of-y}, and \eqref{eq:eta-again}.\smallskip

Finally we indicate what changes if $\kappa(y)$ passes through an arc of $\Sigma^\mathrm{out}(y)$ as $y$ varies in $E_\mathrm{R}$.  It is easy to see that each time $\kappa(y)$ crosses an arc of $\Sigma^\mathrm{out}(y)$ transversely, the simple zero at $\lambda=\kappa(y)$ is exchanged between the functions $f^\mathrm{D}(\lambda;y)$ and $f^\mathrm{OD}(\lambda;y)$.  To account for this correctly, one should define the value of $A(\kappa(y);y)$ appearing in \eqref{eq:gpm-def} by analytic continuation of the Abel mapping $A(\lambda;y)$ through the cuts, which has the effect of transferring the simple pole at $\lambda=\kappa(y)$ between the function $q^+(\lambda;\kappa(y),s,y)$ and $q^-(\lambda;\kappa(y),s,y)$ and hence between the off-diagonal and diagonal elements of $\mathbf{Q}(\lambda;\kappa(y),s,y)$.  With this interpretation of $A(\kappa(y);y)$ the formula \eqref{eq:P-formula} remains analytic in its domain of definition and yields the solution of Riemann-Hilbert Problem~\ref{rhp:outer-elliptic} through \eqref{eq:POout}.

\subsubsection{Properties of the solution of the outer model Riemann-Hilbert problem}
The constant pre-factor in \eqref{eq:P-formula} also introduces singularities in the parameter space.  In other words, $\mathbf{P}_n(\lambda;y,w,m)$ and hence also $\dot{\mathbf{O}}_n^{\mathrm{out}}(\lambda;y,w,m)$ will exist if and only if $Q_{11}(\infty;\kappa(y),s_n(y,w,m),y)Q_{22}(\infty;\kappa(y),s_n(y,w,m),y)\neq 0$.  This is equivalent to the condition $\Theta(A(\infty;y)+A(\kappa(y);y)+\ii\pi +\tfrac{1}{2}B(y)-s_n(y,w,m),B(y))\Theta(A(\infty;y)+A(\kappa(y);y)+\ii\pi+\tfrac{1}{2}B(y)+s_n(y,w,m),B(y))\neq 0$.  In other words, we see that Riemann-Hilbert Problem~\ref{rhp:outer-elliptic} has a unique solution provided that the parameters do not satisfy either (distinguished by a sign $\pm$) of the following conditions
\begin{multline}
\text{No solution of Riemann-Hilbert Problem~\ref{rhp:outer-elliptic}:}\\
A(\infty;y)+A(\kappa(y);y)\pm (\delta(y,m)+\ii w\nu(y)+n\eta(y))\in 2\pi\ii \mathbb{Z} + B(y)\mathbb{Z}.
\label{eq:outer-pole-divisor}
\end{multline}
\begin{lem}
For each $y\in E_\mathrm{R}$, the condition \eqref{eq:outer-pole-divisor} is independent of the choice of sign ($\pm$).
\label{lemma:one-divisor}
\end{lem}
\begin{proof}
Fix $y\in E_\mathrm{R}$.  It is sufficient to show that $2A(\infty;y)+2A(\kappa(y);y)\in 2\pi\ii\mathbb{Z}+B(y)\mathbb{Z}$.  Let $\Gamma$ denote the hyperelliptic Riemann surface associated with the equation $\tau^2=P(\lambda;y,C(y))$, which we model as two copies (``sheets'') of the complex $\lambda$-plane identified along the two cuts making up $\Sigma^\mathrm{out}(y)$.  Selecting one sheet on which $A(\lambda;y)$ is defined as in \eqref{eq:Abel-def}, we may extend the definition to the universal covering of $\Gamma$ by analytic continuation through the cuts or through $\LInftyRedThree$.  Then by taking the quotient of the continuation by the lattice of integer periods $\Lambda:=2\pi\ii\mathbb{Z}+ B(y)\mathbb{Z}$, we arrive at a well-defined function on $\Gamma$ taking values in the corresponding Jacobian variety $\mathrm{Jac}(\Gamma):=\mathbb{C}/\Lambda$, the Abel map of $\Gamma$ denoted $\mathcal{A}(Q)$, $Q\in\Gamma$.  Labeling the points $\lambda$ on the original sheet of definition of $A(\lambda;y)$ as $Q^+(\lambda)$, and their corresponding hyperelliptic involutes on the second sheet as $Q^-(\lambda)$, we observe that for any $\lambda$ not one of the four branch points of $\Gamma$, the equalities $A(\lambda;y)=\mathcal{A}(Q^+(\lambda))=-\mathcal{A}(Q^-(\lambda)))$ hold on $\mathrm{Jac}(\Gamma)$ because the base point of the integral in \eqref{eq:Abel-def} was chosen as a branch point.  We may also take $\lambda=\infty$ in the above relations, and hence we have $2A(\infty;y)+2A(\kappa(y);y)=\mathcal{A}(Q^+(\infty))+\mathcal{A}(Q^+(\kappa(y)))-\mathcal{A}(Q^-(\infty))-\mathcal{A}(Q^-(\kappa(y)))$, which is usually written as $\mathcal{A}(\mathscr{D})$ for $\mathscr{D}:=Q^+(\infty)+Q^+(\kappa(y))-Q^-(\infty)-Q^-(\kappa(y))$ when the action of $\mathcal{A}$ is extended to divisors $\mathscr{D}$ as formal sums of points with integer coefficients.  Therefore, $2A(\infty;y)+2A(\kappa(y);y)\in 2\pi\ii\mathbb{Z}+B(y)\mathbb{Z}$ is equivalent to the condition that $\mathcal{A}(\mathscr{D})=0$ in $\mathrm{Jac}(\Gamma)$ for the indicated divisor $\mathscr{D}$.  According to the Abel-Jacobi theorem, to establish this condition it suffices to construct a nonzero meromorphic function on $\Gamma$ with simple poles at $Q^-(\infty)$ and $Q^-(\kappa(y))$ and simple zeros at the hyperelliptic involutes $Q^+(\infty)$ and $Q^+(\kappa(y))$, with no other zeros or poles.  The existence of such a function must take advantage of the formula \eqref{eq:mu-def}, because the Riemann-Roch theorem asserts that the dimension of the linear space of meromorphic functions on the genus $g=1$ Riemann surface $\Gamma$ with divisor of the form $\mathscr{D}=Q^+(\infty)+Q^+(\zeta)-Q^-(\infty)-Q^-(\zeta)$ is $\mathrm{deg}(\mathscr{D})-g+1=0-1+1=0$ unless the divisor is \emph{special}, implying that $\zeta$ is non-generic.\smallskip  

In order to construct the required function, let $\tau(Q^\pm(\lambda)):=\pm R(\lambda;y)$ define $\tau$ properly as a function on $\Gamma$, and consider the function $f:\Gamma\to\mathbb{C}$ given by
\begin{equation}
f(Q):=\frac{\lambda(Q)^2+a\lambda(Q)+b + c\tau(Q)}{\lambda(Q)-\zeta}
\label{eq:f-of-P}
\end{equation}
for constants $a$, $b$, $c$, and $\zeta$.  The only possible singularities of this function are the two points on $\Gamma$ over $\lambda=\infty$ and the two points over $\lambda=\zeta$. Recall the roots of $P(\lambda;y,C(y))$:  $\lambda_j=\lambda_j(y)$, $j=0,1$, and their reciprocals. As $Q\to Q^\pm(\infty)$, we have $\tau(Q)=\pm\tfrac{1}{2}\ii y(\lambda(Q)^2  -\tfrac{1}{2}(\lambda_0+\lambda_1 +\lambda_0^{-1}+\lambda_1^{-1})\lambda(Q) + O(1))$, so to ensure that $f(Q^+(\infty))=0$ we must choose
\begin{equation}
c:=\frac{2\ii}{y} \quad\text{and}\quad a:=-\frac{1}{2}\left(\lambda_0+\lambda_1+\lambda_0^{-1}+\lambda_1^{-1}\right).
\label{eq:a-c}
\end{equation}
With the above choice of $c$ it is also clear that $f(Q)=2\lambda(Q)+O(1)$ as $Q\to Q^-(\infty)$, so $f$ has a simple pole at $Q=Q^-(\infty)$.  Given these choices and the divisor parameter $\zeta\in\mathbb{C}$, upon taking a generic value of $b$, $f(Q)$ will have simple poles at both $Q=Q^+(\zeta)$ and $Q=Q^-(\zeta)$.  We may obviously choose $b$ uniquely such that $f(Q)$ is holomorphic at $Q=Q^+(\zeta)$:
\begin{equation}
b:=-\zeta^2-a\zeta-c\tau(Q^+(\zeta))=-\zeta^2+\frac{1}{2}\left(\lambda_0+\lambda_1+\lambda_0^{-1}+\lambda_1^{-1}\right)\zeta-\frac{2\ii}{y} R(\zeta;y).
\label{eq:b}
\end{equation}
With $a,b,c$ determined for arbitrary fixed $\zeta$, there is no additional parameter available in the form \eqref{eq:f-of-P} to ensure that $f(Q^+(\zeta))=0$, a fact that is consistent with the Riemann-Roch argument given above.

So we take the point of view that $\zeta$ should be viewed as the additional parameter needed to guarantee that $f(Q^+(\zeta))=0$.  Indeed, for this to be the case, the derivative with respect to $\lambda$ of the numerator in \eqref{eq:f-of-P} should vanish at $Q=Q^+(\zeta)$; we therefore require:
\begin{equation}
2\zeta-\frac{1}{2}\left(\lambda_0+\lambda_1+\lambda_0^{-1}+\lambda_1^{-1}\right)=-\frac{2\ii}{y}\frac{\dd\tau}{\dd\lambda}(Q^+(\zeta)) = -\frac{2\ii}{y}\frac{\dd R}{\dd\lambda}(\zeta;y).
\label{eq:zeta-condition}
\end{equation}
By implicit differentiation,
\begin{multline}
\frac{\dd R}{\dd\lambda}(\lambda;y)=-\frac{y^2}{8R(\lambda;y)}\big((\lambda-\lambda_0)(\lambda-\lambda_1)(\lambda-\lambda_0^{-1}) + (\lambda-\lambda_0)(\lambda-\lambda_1)(\lambda-\lambda_1^{-1})\\
{}+(\lambda-\lambda_0)(\lambda-\lambda_0^{-1})(\lambda-\lambda_1^{-1}) + (\lambda-\lambda_1)(\lambda-\lambda_0^{-1})(\lambda-\lambda_1^{-1})\big).
\label{eq:rho-prime}
\end{multline}
Substituting \eqref{eq:rho-prime} into \eqref{eq:zeta-condition} and squaring both sides yields a cubic equation for $\zeta$ with solutions:
\begin{equation}
\zeta=\kappa(y),\quad\zeta=\kappa(y)^{-1},\quad\text{and}\quad\zeta=0,
\end{equation}
where $\kappa(y)$ is given by \eqref{eq:mu-def}.  These are precisely the three values of $\zeta\in\mathbb{C}$ for which the divisor $\mathscr{D}=Q^+(\infty)+Q^+(\zeta)-Q^-(\infty)-Q^-(\zeta)$ is special in the setting of the Riemann-Roch theorem.  Selecting the desired solution $\zeta=\kappa(y)$, it remains only to confirm that \eqref{eq:zeta-condition} holds \emph{without squaring both sides}.
But, since $-2a$ is the sum of roots of $P(\lambda;y,C(y))$, from \eqref{eq:dotV-ODE} we can also write $a=-\ii y^{-1}$, and then since when $y>0$ we know that the branch cuts of $R(\lambda;y)$ lie on opposite halves of the imaginary axis and $\kappa(y)$ lies on the imaginary axis below both cuts, it follows that both sides of \eqref{eq:zeta-condition} are negative imaginary for $\zeta=\kappa(y)$ and $y>0$.  The persistence of \eqref{eq:zeta-condition} for $\zeta=\kappa(y)$ as $y$ varies within $E_\mathrm{R}$ follows by analytic continuation, with the re-definition of $A(\kappa(y);y)$ as described in the last paragraph of Section~\ref{sec:elliptic-outer-solution}, should $\kappa(y)$ pass through $\Sigma^\mathrm{out}(y)$ as both move in the complex $\lambda$-plane.
\end{proof}
\begin{rem}
Numerical calculations allow us to find the exact lattice point corresponding to the sum of Abel maps appearing in the proof:  $2A(\infty;y)+2A(\kappa(y);y)=-B(y)$ holds as an identity on $y\in E_\mathrm{R}$.  In a similar way, one can also prove the identity $2A(0;y)-2A(\kappa(y);y)=2\pi\ii$.  
\end{rem}
The parameter values excluded by the (equivalent) conditions \eqref{eq:outer-pole-divisor} are said to form the \emph{Malgrange divisor} for Riemann-Hilbert 
Problem~\ref{rhp:outer-elliptic}.
We have the following result.
\begin{lem}
Riemann-Hilbert Problem~\ref{rhp:outer-elliptic} has a unique solution with unit determinant provided that $n=0,1,2,3,\dots$, $m\in\mathbb{C}\setminus(\mathbb{Z}+\tfrac{1}{2})$, $y\in E_\mathrm{R}$, and $w\in\mathbb{C}$ do not lie in the Malgrange divisor \eqref{eq:outer-pole-divisor}.  Moreover, for fixed $m\in\mathbb{C}\setminus(\mathbb{Z}+\tfrac{1}{2})$ and $\epsilon>0$ arbitrarily small, $\dot{\mathbf{O}}_n^{\mathrm{out}}(\lambda;y,w,m)$ is uniformly bounded on the set of $(\lambda,n,y,w)$ satisfying $|w|\le K$ for some $K>0$ and the conditions
\begin{equation}
\mathrm{dist}(\lambda,\{\lambda_0(y),\lambda_1(y),\lambda_0(y)^{-1},\lambda_1(y)^{-1}\})\ge\epsilon,
\label{eq:lambda-cheese}
\end{equation}
\begin{equation}
\quad\mathrm{dist}(A(\infty;y)+A(\kappa(y);y)\pm (\delta(y,m)+\ii w\nu(y)+n\eta(y)),2\pi\ii\mathbb{Z}+B(y)\mathbb{Z})\ge\epsilon,\quad\text{and}\quad \mathrm{dist}(y,\partial(E_\mathrm{R}))\ge\epsilon.
\label{eq:Swiss-Cheese}
\end{equation}
\label{lem:Outer-Bounded}
\end{lem}
Note that for fixed $n$ and $w=0$, the two conditions in \eqref{eq:Swiss-Cheese} 
bound $y$ within $E_\mathrm{R}$ 
by a distance $\epsilon$ from the boundary and also bound $y$ away from the points of the 
Malgrange divisor by a distance proportional to $\frac{\epsilon}{n}$, that is, an arbitrarily small fixed fraction of the spacing between the points of the divisor.  
\begin{proof}
The uniqueness and unimodularity of the solution given existence are standard results.  It remains to show the boundedness under the conditions $|w|\le K$, \eqref{eq:lambda-cheese}, and \eqref{eq:Swiss-Cheese},
which is not obvious because the solution formula for $\dot{\mathbf{O}}^{(n),\mathrm{out}}(\lambda;y,w,m)$ contains exponential factors and theta-function factors that grow exponentially with $n$, which is allowed to grow without bound.  However, the conditions of Riemann-Hilbert 
Problem~\ref{rhp:outer-elliptic}
only involve $n$ in the form of exponential factors $\ee^{\pm\ii n K_j(y)}$, $j=1,2$, which have unit modulus for all $n$ because $K_j(y)\in\mathbb{R}$ by the Boutroux conditions \eqref{eq:Boutroux}.  The parameter space for Riemann-Hilbert 
Problem~\ref{rhp:outer-elliptic} is therefore a subset of a compact set even though $n$ can become unbounded.  This fact leads to the claimed uniform boundedness.  See \cite[Proposition 8]{BuckinghamM14} for a similar argument with full details.
\end{proof}

\subsubsection{Defining the approximation $\dot{u}_n(y,w;m)$}
Reversing the substitutions $\mathbf{Y}_n\mapsto\mathbf{M}_n\mapsto\mathbf{N}_n\mapsto\mathbf{O}_n$ and using \eqref{eq:u-n-from-Y-formula} shows that the rational solution $u=u_n(x;m)$ of the Painlev\'e-III equation \eqref{eq:PIII} can be expressed for 
$y\in E_\mathrm{R}$
and $x=ny+w$ in terms of $\mathbf{O}_n(\lambda;y,w,m)$ in the form
\begin{equation}
u_n(ny+w;m)=\frac{-\ii Y^\infty_{1,12}(ny+w,m)}{Y^0_{0,11}(ny+w,m)Y^0_{0,12}(ny+w,m)}=\frac{-\ii O^\infty_{n,1,12}(y,w,m)}{O^0_{n,0,11}(y,w,m)O^0_{n,0,12}(y,w,m)}.
\label{eq:u_n-O}
\end{equation}
where $\mathbf{O}^0_{n,0}(y,w,m)=\mathbf{O}_n(0;y,w,m)$ and $\mathbf{O}^\infty_{n,1}(y,w,m)=\lim_{\lambda\to\infty}\lambda (\mathbf{O}_n(\lambda;y,w,m)-\mathbb{I})$.  Note that here we do not have to exclude real values of $y$, because $u_n(ny+w;m)$ is rational in $y$ meaning that \eqref{eq:u_n-O} must be consistent for positive $y$ in $E_\mathrm{R}$.
In Section~\ref{sec:Error-elliptic} we will show that under the conditions $|w|\le K$, \eqref{eq:lambda-cheese} and \eqref{eq:Swiss-Cheese}, 
the outer parametrix $\dot{\mathbf{O}}_n^{\mathrm{out}}(\lambda;y,w,m)$ is an accurate approximation of $\mathbf{O}_n(\lambda;y,w,m)$, from which $u_n(ny+w;m)$ can be extracted according to \eqref{eq:u_n-O}.  This motivates the introduction of an explicit approximation for $u_n(ny+w;m)$ obtained by replacing $\mathbf{O}_n(\lambda;y,w,m)$ by its outer parametrix in \eqref{eq:u_n-O}:
\begin{equation}
\dot{u}_n(y,w;m):=\frac{-\ii \dot{O}^\infty_{n,1,12}(y,w,m)}{\dot{O}^0_{n,0,11}(y,w,m)\dot{O}^0_{n,0,12}(y,w,m)},
\label{eq:elliptic-dot-u-def}
\end{equation}
where $\dot{\mathbf{O}}^0_{n,0}(y,w,m)=\dot{\mathbf{O}}_n^{\mathrm{out}}(0;y,w,m)$ and $\dot{\mathbf{O}}^\infty_{n,1}(y,w,m)=\lim_{\lambda\to\infty}\lambda (\dot{\mathbf{O}}_n^{\mathrm{out}}(\lambda;y,w,m)-\mathbb{I})$.
Using the formul\ae\ developed in Section~\ref{sec:elliptic-outer-solution} for the outer parametrix then yields the formula \eqref{eq:udot-elliptic} for $\dot{u}_n(y,w;m)$, in which
\begin{equation}
\begin{split}
\mathcal{Z}_n^\circ(y,w;m)&:=\Theta(A(\infty;y)+A(\kappa(y);y)+\ii\pi+\tfrac{1}{2}B(y)-s_n(y,w,m),B(y))\\
\mathcal{Z}_n^\bullet(y,w;m)&:=\Theta(A(\infty;y)-A(\kappa(y);y)-\ii\pi-\tfrac{1}{2}B(y)+s_n(y,w,m),B(y))\\
\mathcal{P}_n^\bullet(y,w;m)&:=\Theta(A(0;y)+A(\kappa(y);y)+\ii\pi+\tfrac{1}{2}B(y)-s_n(y,w,m),B(y))\\
\mathcal{P}_n^\circ(y,w;m)&:=\Theta(A(0;y)-A(\kappa(y);y)-\ii\pi-\tfrac{1}{2}B(y)+s_n(y,w,m),B(y))
\end{split}
\label{eq:four-factors}
\end{equation}
and, using the fact that $\phi(0;y)^2=\lambda_0(y)\lambda_1(y)^{-1}$, 
\begin{equation}
N(y):=\frac{\ii}{\kappa(y)}
\cdot
\frac{\Theta(A(0;y)+A(\kappa(y);y)+\ii\pi+\tfrac{1}{2}B(y),B(y))\Theta(A(0;y)-A(\kappa(y);y)-\ii\pi-\tfrac{1}{2}B(y),B(y))}{\Theta(A(\infty;y)+A(\kappa(y);y)+\ii\pi+\tfrac{1}{2}B(y),B(y))\Theta(A(\infty;y)-A(\kappa(y);y)-\ii\pi-\tfrac{1}{2}B(y),B(y))}.
\label{eq:N-of-y}
\end{equation}
We recall that for $y\in E_\mathrm{R}$,  $s_n(y,w,m)=-\delta(y,m)-\ii w\nu(y)-n\eta(y)$.
Observe that $N(y)$ is well-defined and nonzero for all 
$y\in E_\mathrm{R}$
and is independent of $n$ and $m$.  

\subsubsection{The differential equation satisfied by $\dot{\lNaught}(w)=-\ii\dot{u}_n(y,w;m)$}
\label{sec:DE-w}
Although $u_n(x;m)$ is a rational function of $x=ny+w$, the approximation $\dot{u}_n(y,w;m)$
is not a meromorphic function of $y$ because $C=C(y)$ is determined from the Boutroux equations \eqref{eq:Boutroux}, from which a direct computation shows that $\overline{\partial}C\neq 0$ in general, i.e., the real and imaginary parts of $C$ do not satisfy the Cauchy-Riemann equations with respect to the real and imaginary parts of $y$.
On the other hand, since $s_n(y,w,m)$ is linear in $w$, it is obvious from \eqref{eq:udot-elliptic} with \eqref{eq:four-factors}--\eqref{eq:N-of-y} that $\dot{u}_n(y,w;m)$ is a meromorphic function of $w$ for each fixed $y\in E_\mathrm{R}$.  In order to establish the first statement of Theorem~\ref{theorem:eye}, we will prove in this section that the related function $\dot{p}(w):=-\ii\dot{u}_n(y,w;m)$ is in fact an elliptic function of $w$ satisfying the differential equation \eqref{eq:dotV-ODE} in which the constant $C=C(y)$ is determined from the Boutroux equations \eqref{eq:Boutroux}.

Rather than try to deal directly with the explicit formula \eqref{eq:udot-elliptic}, we argue indirectly from the conditions of Riemann-Hilbert Problem~\ref{rhp:outer-elliptic}.  We first observe that the outer parametrix $\dot{\mathbf{O}}_n^{\mathrm{out}}(\lambda;y,w,m)$ satisfies a simple algebraic equation.  Indeed, it is straightforward to check that the matrix
\begin{equation}
\mathbf{G}(\lambda):=R(\lambda;y)\dot{\mathbf{O}}_n^{\mathrm{out}}(\lambda;y,w,m)\sigma_3
\dot{\mathbf{O}}_n^{\mathrm{out}}(\lambda;y,w,m)^{-1}
\label{eq:G-matrix-def}
\end{equation}
is an entire function; its continuous boundary values match along the three arcs of the jump contour of $R$ and $\dot{\mathbf{O}}_n^{\mathrm{out}}$, and it is clearly bounded near the four roots of $R^2$, hence analyticity in the whole complex $\lambda$-plane follows by Morera's theorem.  Moreover, since 
\begin{equation}
R(\lambda;y)=\frac{1}{2}\ii y\lambda^2+\frac{1}{2}\lambda +\ii\frac{1-4C(y)}{4y}+ \mathcal{O}(\lambda^{-1}),\quad\lambda\to\infty,
\end{equation}
Liouville's theorem shows that $\mathbf{G}(\lambda)$ is a quadratic matrix polynomial in $\lambda$.  Using the expansion $\dot{\mathbf{O}}_n^{\mathrm{out}}(\lambda;y,w,m)=\mathbb{I}+\lambda^{-1}\dot{\mathbf{O}}^\infty_{n,1}(y,w,m) + \lambda^{-2}\dot{\mathbf{O}}^\infty_{n,2}(y,w,m)+\mathcal{O}(\lambda^{-3})$ as $\lambda\to\infty$ shows that
\begin{equation}
\mathbf{G}(\lambda)=\frac{1}{2}\ii y\sigma_3\lambda^2 + \frac{1}{2}\left(\sigma_3+\ii y\left[\dot{\mathbf{O}}^\infty_{n,1}(y,w,m),\sigma_3\right]\right)\lambda + 
\mathbf{G}^\infty + 
\mathcal{O}(\lambda^{-1}),
\quad\lambda\to\infty,
\label{eq:G-infty}
\end{equation}
where
\begin{equation}
\mathbf{G}^\infty:=\ii\frac{1-4C(y)}{4y}\sigma_3 +\frac{1}{2}\left[\dot{\mathbf{O}}^\infty_{n,1}+\ii y\dot{\mathbf{O}}^\infty_{n,2},\sigma_3\right]-\frac{1}{2}\ii y\left[\dot{\mathbf{O}}^\infty_{n,1},\sigma_3\right]\dot{\mathbf{O}}^\infty_{n,1}.
\label{eq:G0-def}
\end{equation}
Also, using 
\begin{equation}
R(\lambda;y)=\frac{1}{2}\ii y +\frac{1}{2}\lambda+\mathcal{O}(\lambda^2),\quad\lambda\to 0
\end{equation}
and the expansion $\dot{\mathbf{O}}_n^{\mathrm{out}}(\lambda;y,w,m)=\dot{\mathbf{O}}^0_{n,0}(y,w,m)+\dot{\mathbf{O}}^0_{n,1}(y,w,m)\lambda+\mathcal{O}(\lambda^2)$ as $\lambda\to 0$ gives
\begin{equation}
\mathbf{G}(\lambda)
=\frac{1}{2}\ii y\dot{\mathbf{O}}^0_{n,0}(y,w,m)\sigma_3\dot{\mathbf{O}}^0_{n,0}(y,w,m)^{-1} + \mathbf{G}_1^0\lambda+\mathcal{O}(\lambda^2),\quad\lambda\to 0,
\label{eq:G-zero}
\end{equation}
where
\begin{equation}
\mathbf{G}_1^0:=\frac{1}{2}\dot{\mathbf{O}}_{n,0}^0\sigma_3(\dot{\mathbf{O}}_{n,0}^0)^{-1} +\frac{1}{2}\ii y\dot{\mathbf{O}}_{n,0}^0
\left[(\dot{\mathbf{O}}_{n,0}^0)^{-1}\dot{\mathbf{O}}_{n,1}^0,\sigma_3\right](\dot{\mathbf{O}}_{n,0}^0)^{-1}.
\label{eq:G1-def}
\end{equation}
Therefore $\mathbf{G}(\lambda)$ is the quadratic matrix polynomial
\begin{equation}
\mathbf{G}(\lambda)=\frac{1}{2}\ii y\sigma_3\lambda^2 + \frac{1}{2}\left(\sigma_3+\ii y\mathbf{A}(w)\right)\lambda+\frac{1}{2}\ii y\mathbf{B}(w), 
\label{eq:G-of-lambda}
\end{equation}
where, suppressing explicit dependence on the parameters $y\in E$ and $m\in\mathbb{C}$,
\begin{equation}
\mathbf{A}(w):=\left[\dot{\mathbf{O}}^\infty_{n,1}(y,w,m),\sigma_3\right]\quad \text{and}\quad
\mathbf{B}(w):=\dot{\mathbf{O}}^0_{n,0}(y,w,m)\sigma_3\dot{\mathbf{O}}^0_{n,0}(y,w,m)^{-1}.
\end{equation}
These matrices have the forms
\begin{equation}
\mathbf{A}(w)=\begin{bmatrix}0 & A_{12}(w)\\A_{21}(w) & 0\end{bmatrix}\quad\text{and}\quad
\mathbf{B}(w)=\begin{bmatrix}\beta(w) & B_{12}(w)\\B_{21}(w) & -\beta(w)\end{bmatrix}
\end{equation}
where
\begin{equation}
\det(\mathbf{B}(w))=-1\quad\implies\quad \beta(w)^2=1-B_{12}(w)B_{21}(w).
\label{eq:beta-identity}
\end{equation}
Comparing the constant terms between the expansions \eqref{eq:G-infty} and \eqref{eq:G-zero} yields the identity
\begin{equation}
\mathbf{G}^\infty=\frac{1}{2}\ii y\mathbf{B}(w),
\label{eq:G0-eqn}
\end{equation}
where $\mathbf{G}^\infty$ is given by \eqref{eq:G0-def},
and comparing the terms proportional to $\lambda$ in the same expansions yields
\begin{equation}
\mathbf{G}^0_{1}=\frac{1}{2}\left(\sigma_3+\ii y \mathbf{A}(w)\right),
\label{eq:G1-eqn}
\end{equation}
where $\mathbf{G}^0_{1}$ is given by \eqref{eq:G1-def}.
Since $\sigma_3^2=\mathbb{I}$, it is also clear from \eqref{eq:G-matrix-def} that the square of the matrix polynomial $\mathbf{G}(\lambda)$ is a multiple of the identity, i.e., a specific scalar polynomial:
\begin{equation}
\mathbf{G}(\lambda)^2 = R(\lambda;y)^2\mathbb{I}=P(\lambda;y,C(y))\mathbb{I},
\label{eq:G-squared-P}
\end{equation}
where $P$ is the quartic in \eqref{eq:dotV-ODE}.  On the other hand, calculating the square directly from \eqref{eq:G-of-lambda} gives
\begin{multline}
\mathbf{G}(\lambda)^2 = \frac{1}{4}\left(-y^2\lambda^4 + 2\ii y\lambda^3 + \left(1-y^2A_{12}(w)A_{21}(w)-2y^2\beta(w)\right)\lambda^2\right. \\{}\left.+ \ii y\left(2\beta(w)+\ii y(A_{12}(w)B_{21}(w)+A_{21}(w)B_{12}(w))\right)\lambda-y^2\right)\mathbb{I}.
\label{eq:G-squared-direct}
\end{multline}
Comparing the coefficient of $\lambda$ between \eqref{eq:G-squared-direct} and $P(\lambda;y,C(y))\mathbb{I}$ using \eqref{eq:dotV-ODE} yields the identity
\begin{equation}
\beta(w)=1-\frac{1}{2}\ii y(A_{12}(w)B_{21}(w)+A_{21}(w)B_{12}(w)).
\label{eq:lambda-1-identity}
\end{equation}
Using \eqref{eq:lambda-1-identity} to eliminate $\beta(w)$ from \eqref{eq:G-squared-direct} and comparing again with \eqref{eq:G-squared-P} gives the identity
\begin{multline}
P(\lambda;y,C(y))=-\frac{1}{4}y^2\lambda^4 + \frac{1}{2}\ii y\lambda^3 \\
+\frac{1}{4}\left(1-y^2A_{12}(w)A_{21}(w)-2y^2+\ii y^3(A_{12}(w)B_{21}(w)+A_{21}(w)B_{12}(w))\right)\lambda^2 
+\frac{1}{2}\ii y\lambda -\frac{1}{4}y^2.
\label{eq:Poly-rewrite}
\end{multline}
We note that the coefficient of $\lambda^2$ here is actually independent of $w$, since according to \eqref{eq:dotV-ODE} it is given by $C=C(y)$, but the above expression is more useful in the context of the present discussion.\smallskip

On the other hand, one may observe that the matrix $\mathbf{F}(\lambda;w):=\dot{\mathbf{O}}_n^{\mathrm{out}}(\lambda;y,w,m)\ee^{\ii w\varphi(\lambda)\sigma_3/2}$ satisfies jump conditions that are independent of $w\in\mathbb{C}$, and therefore $\mathbf{F}_w\mathbf{F}^{-1}$ is a function of $\lambda$ analytic except possibly at $\lambda=0$ where $\mathbf{F}$ has essential singularities.  By expansion for large and small $\lambda$ and Liouville's theorem, it follows that $\mathbf{F}_w\mathbf{F}^{-1}$ is a Laurent polynomial:
\begin{equation}
\frac{\partial\mathbf{F}}{\partial w}(\lambda;w)\mathbf{F}(\lambda;w)^{-1}=
\frac{1}{2}\ii\sigma_3\lambda +\frac{1}{2}\ii\left[\dot{\mathbf{O}}^\infty_{n,1}(y,w,m),\sigma_3\right]-\frac{1}{2}\ii\dot{\mathbf{O}}^0_{n,0}(y,w,m)\sigma_3\dot{\mathbf{O}}^0_{n,0}(y,w,m)^{-1}\lambda^{-1}.
\end{equation}  
Therefore, the outer parametrix $\mathbf{O}_n^{\mathrm{out}}(\lambda;y,w,m)$ itself satisfies the differential equation
\begin{multline}
\frac{\partial\dot{\mathbf{O}}_n^{\mathrm{out}}}{\partial w}(\lambda;y,w,m)=\frac{1}{2}\ii\left[\sigma_3,\dot{\mathbf{O}}_n^{\mathrm{out}}(\lambda;y,w,m)\right]\lambda + \frac{1}{2}\ii
\left[\dot{\mathbf{O}}^\infty_{n,1}(y,w,m),\sigma_3\right]\dot{\mathbf{O}}_n^{\mathrm{out}}(\lambda;y,w,m) \\{}+\frac{1}{2}\ii\left(\dot{\mathbf{O}}_n^{\mathrm{out}}(\lambda;y,w,m)\sigma_3-
\dot{\mathbf{O}}^0_{n,0}(y,w,m)\sigma_3\dot{\mathbf{O}}^0_{n,0}(y,w,m)^{-1}\dot{\mathbf{O}}_n^{\mathrm{out}}(\lambda;y,w,m)\right)\lambda^{-1}.
\label{eq:outer-parametrix-ODE}
\end{multline}
Substituting the large-$\lambda$ expansion of $\dot{\mathbf{O}}_n^{\mathrm{out}}(\lambda;y,w,m)$ yields an infinite hierarchy of differential equations on the expansion coefficient matrices, the first 
member of which is
\begin{equation}
\frac{\dd\dot{\mathbf{O}}^\infty_{n,1}}{\dd w}=\frac{1}{2}\ii\left[\sigma_3,\dot{\mathbf{O}}^\infty_{n,2}\right] +
\frac{1}{2}\ii\left[\dot{\mathbf{O}}^\infty_{n,1},\sigma_3\right]\dot{\mathbf{O}}^\infty_{n,1} +\frac{1}{2}\ii\sigma_3 -\frac{1}{2}\ii\dot{\mathbf{O}}^0_{n,0}\sigma_3(\dot{\mathbf{O}}^0_{n,0})^{-1}.
\label{eq:DE-infty-1}
\end{equation}
Using the off-diagonal part of the identity \eqref{eq:G0-eqn} we can eliminate the commutator $[\sigma_3,\dot{\mathbf{O}}^\infty_{n,2}]$, and therefore \eqref{eq:DE-infty-1} implies that
\begin{equation}
\frac{\dd\dot{\mathbf{O}}^\infty_{n,1}}{\dd w}=\frac{1}{2y}\mathbf{A}(w)+\frac{1}{2}\ii\left(\mathbf{A}(w)\dot{\mathbf{O}}_{n,1}^\infty\right)^\mathrm{D}+\frac{1}{2}\ii\sigma_3+\frac{1}{2}\ii(\mathbf{B}(w))^\mathrm{D}-\ii\mathbf{B}(w),
\end{equation}
where $(\cdot)^\mathrm{D}$ denotes the diagonal part of a matrix.  Taking the commutator of this equation with $\sigma_3$ then yields
\begin{equation}
\frac{\dd\mathbf{A}}{\dd w} = \frac{1}{2y}[\mathbf{A},\sigma_3] -\ii[\mathbf{B},\sigma_3].
\label{eq:ODE-1}
\end{equation}
Similarly, substituting into \eqref{eq:outer-parametrix-ODE} the small-$\lambda$ expansion of $\dot{\mathbf{O}}_n^{\mathrm{out}}(\lambda;y,w,m)$ and taking just the leading (constant) term gives the differential equation
\begin{equation}
\frac{\dd\dot{\mathbf{O}}^0_{n,0}}{\dd w}=\frac{1}{2}\ii[\dot{\mathbf{O}}^\infty_{n,1},\sigma_3]\dot{\mathbf{O}}^0_{n,0} +\frac{1}{2}\ii
\dot{\mathbf{O}}_{n,0}^{0}[(\dot{\mathbf{O}}_{n,0}^{0})^{-1}\dot{\mathbf{O}}^0_{n,1},\sigma_3].
\end{equation}
Multiplying the identity \eqref{eq:G1-eqn} on the right by $\dot{\mathbf{O}}_{n,0}^{0}$ allows $\dot{\mathbf{O}}_{n,1}^{0}$ to be eliminated from the right-hand side of the above differential equation, leading to
\begin{equation}
\frac{\dd\dot{\mathbf{O}}_{n,0}^0}{\dd w}=\ii\mathbf{A}(w)\dot{\mathbf{O}}_{n,0}^0 +\frac{1}{2y}[\sigma_3,\dot{\mathbf{O}}_{n,0}^0].
\end{equation}
This identity allows us to compute the derivative of $\mathbf{B}(w)$.  Using also $\mathbf{B}(w)^2=\mathbb{I}$ yields the differential equation
\begin{equation}
\frac{\dd\mathbf{B}}{\dd w}=\ii[\mathbf{A},\mathbf{B}]-\frac{1}{2y}[\mathbf{B},\sigma_3].
\label{eq:ODE-2}
\end{equation}
The differential equations \eqref{eq:ODE-1} and \eqref{eq:ODE-2} obviously form a closed system on the matrices $\mathbf{A}(w)$ and $\mathbf{B}(w)$.  

From \eqref{eq:elliptic-dot-u-def}, we can express $\dot{p}(w):=-\ii\dot{u}_n(y,w;m)$ in terms of the elements of $\mathbf{A}(w)$ and $\mathbf{B}(w)$ simply as
\begin{equation}
\dot{p}(w)=-\ii\dot{u}_n(y,w;m)=-\frac{A_{12}(w)}{B_{12}(w)}.
\label{eq:dot-p-AB}
\end{equation}
Now we use \eqref{eq:ODE-1} and \eqref{eq:ODE-2} to differentiate $\dot{p}(w)$:
\begin{equation}
\frac{\dd\dot{p}}{\dd w}=-2\ii\beta(w)\frac{A_{12}(w)^2}{B_{12}(w)^2}+\frac{2}{y}\frac{A_{12}(w)}{B_{12}(w)}-2\ii.
\end{equation}
Therefore, using \eqref{eq:beta-identity} to eliminate $\beta(w)^2$, we find that
\begin{multline}
\frac{y^2}{16}\left(\frac{\dd \dot{p}}{\dd w}\right)^2=-\frac{1}{4}y^2 -\frac{1}{2}\ii y\frac{A_{12}(w)}{B_{12}(w)} +
\left(\frac{1}{4}-\frac{1}{2}y^2\beta(w)\right)\frac{A_{12}(w)^2}{B_{12}(w)^2} \\
{}-\frac{1}{2}\ii y\beta(w)\frac{A_{12}(w)^3}{B_{12}(w)^3}+\frac{1}{4}y^2\left(B_{12}(w)B_{21}(w)-1\right)\frac{A_{12}(w)^4}{B_{12}(w)^4}.
\end{multline}
Substituting $\lambda=\dot{p}$ with \eqref{eq:dot-p-AB} into \eqref{eq:Poly-rewrite} gives
\begin{multline}
P(\dot{p};y,C(y))=-\frac{1}{4}y^2-\frac{1}{2}\ii y\frac{A_{12}(w)}{B_{12}(w)}\\
{}+\frac{1}{4}\left(1-y^2A_{12}(w)A_{21}(w)-2y^2+\ii y^3(A_{12}(w)B_{21}(w)+A_{21}(w)B_{12}(w))\right)\frac{A_{12}(w)^2}{B_{12}(w)^2}\\
{}-\frac{1}{2}\ii y\frac{A_{12}(w)^3}{B_{12}(w)^3} -\frac{1}{4}y^2\frac{A_{12}(w)^4}{B_{12}(w)^4}.
\end{multline}
Subtracting these two identities yields
\begin{multline}
\frac{y^2}{16}\left(\frac{\dd\dot{p}}{\dd w}\right)^2-P(\dot{p};y,C(y))=\\
\left(-\frac{1}{2}y^2(\beta(w)-1)+\frac{1}{4}y^2A_{12}(w)A_{21}(w)-\frac{1}{4}\ii y^3(A_{12}(w)B_{21}(w)+A_{21}(w)B_{12}(w))\right)
\frac{A_{12}(w)^2}{B_{12}(w)^2}\\
-\frac{1}{2}\ii y(\beta(w)-1)\frac{A_{12}(w)^3}{B_{12}(w)^3}+\frac{1}{4}y^2B_{12}(w)B_{21}(w)\frac{A_{12}(w)^4}{B_{12}(w)^4}.
\end{multline}
Finally, eliminating $\beta(w)$ using \eqref{eq:lambda-1-identity} yields the differential equation \eqref{eq:dotV-ODE}.  Together with the fact that the four roots of $P(\lambda;y,C(y))$ are distinct by choice of $C(y)$ satisfying the Boutroux conditions \eqref{eq:Boutroux} on $E_\mathrm{R}$, this proves the first statement of Theorem~\ref{theorem:eye}.

\subsection{Airy-type parametrices}
\label{sec:Airy}
Local parametrices for the matrix $\mathbf{O}_n(\lambda;y,w,m)$ are needed in neighborhoods of each of the four roots of $P(\lambda;y,C(y))$, $\lambda=\lambda_0,\lambda_1,\lambda_1^{-1},\lambda_0^{-1}$, where we recall that by definition $\lambda_0$ is adjacent to $\infty$ and $\lambda_1$ is adjacent to $\lambda_0$ on the Stokes graph of $y\in E_\mathrm{R}\setminus\mathbb{R}$.  Centering a disk of sufficiently small radius independent of $n$ at each of these points, a conformal map $W=W(\lambda)$ can be defined in each disk as indicated in 
Table~\ref{tab:ConfMapR}.
\begin{table}[h]
\centering
\caption{Conformal map data for $y\in E_\mathrm{R}\setminus\mathbb{R}$.}
\label{tab:ConfMapR}
\begin{tabular}{@{}|l|l|l|l|l|l|@{}}
\hline
\multirow{2}{*}{\makecell[l]{Center\\$W=0$}}&\multirow{2}{*}{Conformal map $W$}&\multicolumn{4}{c|}{Ray Preimages}\\
\cline{3-6}
                        &                                                                                                                                                                    & $W>0$                       & $\vphantom{\Big[}\arg(W)=\tfrac{2}{3}\pi$                                    & $\arg(W)=-\tfrac{2}{3}\pi$                                   & $W<0$          \\ \hline\hline
$\lambda_0$                          & \makecell[l]{\rule{0cm}{0.01cm}\\$(V-2g)^{2/3}$,\\ continued from $L^{\infty,1}_\squareurblack$\\\rule{0cm}{0.01cm}}                                                                                                                  & $L^{\infty,1}_\squareurblack$             & $\partial\Lambda^-_\squareurblack$                                   & $\partial\Lambda^+_\squareurblack$                                   & $L^{\infty,2}_\squareurblack$ \\ \hline
$\lambda_1$                          & \makecell[l]{\rule{0cm}{0.01cm}\\$(V-g_+-g_-)^{2/3}$,\\ continued from $L^{\infty,3}_\squareurblack$\\\rule{0cm}{0.01cm}}                                                                                                             & $L^{\infty,3}_\squareurblack$             & $\partial\Lambda^+_\squareurblack$                                   & $\partial\Lambda^-_\squareurblack$                                   & $L^{\infty,2}_\squareurblack$ \\ \hline
$\lambda_1^{-1}$                     & \makecell[l]{\rule{0cm}{0.01cm}\\$(2g-V-2g(\lambda_1^{-1})+V(\lambda_1^{-1}))^{2/3}$,\\ continued from $L^{\infty,1}_\squarellblack$\\\rule{0cm}{0.01cm}}                                                                             & $L^{\infty,1}_\squarellblack$             & $\partial\Lambda^-_\squarellblack$                                   & $\partial\Lambda^+_\squarellblack$                                   & $L^{\infty,2}_\squarellblack$ \\ \hline
\multirow{3}{*}{$\lambda_0^{-1}$} & \multirow{2}{*}{\makecell[l]{$(2g-V-2g(\lambda_0^{-1})+V(\lambda_0^{-1}))^{2/3}$,\\ continued from $L^0_\squarellblack$;\\ $g(\lambda_0^{-1})$ defined by limit along $L^0_\squarellblack$}}                      & \multirow{3}{*}{$L^0_\squarellblack$}                      & \makecell[l]{$\partial\Lambda^+_\squarellblack$ and $L^0_\squareurblack$,\\if $\mathrm{Im}(y)>0$}          & \makecell[l]{$\partial\Lambda^-_\squarellblack$ and $L^{\infty,3}_\squareurblack$,\\ if $\mathrm{Im}(y)>0$} & \multirow{3}{*}{$L^{\infty,2}_\squarellblack$} \\ \cline{4-5}
& 
&  & \makecell[l]{$\partial\Lambda^+_\squarellblack$ and $L^{\infty,3}_\squareurblack$,\\ if $\mathrm{Im}(y)<0$} & \makecell[l]{$\partial\Lambda^-_\squarellblack$ and $L^0_\squareurblack$,\\ if $\mathrm{Im}(y)<0$}          &  \\ \hline
\end{tabular}
\end{table}
As indicated in this table, we assume that certain contours near $\lambda_0^{-1}$
are fused together within the corresponding disk, and that all contours are locally arranged to lie along straight rays in the $W$-plane emanating from the origin.  Locally, the jump contours divide the $W$-plane into four sectors:  
\begin{equation}
S_\mathrm{I}:\; 0<\arg(W)<\frac{2}{3}\pi; \;\;
S_\mathrm{II}:\; \frac{2}{3}\pi<\arg(W)<\pi;\;\; 
S_\mathrm{III}:\; -\pi<\arg(W)<-\frac{2}{3}\pi;\;\;
S_\mathrm{IV}:\; -\frac{2}{3}\pi<\arg(W)<0.
\label{eq:Airy-sectors}
\end{equation}
In each case, the jump conditions satisfied by $\mathbf{O}_n(\lambda;y,w,m)$ can then be cast into a universal form by means of a substitution 
\begin{equation}
\mathbf{P}(\lambda):=\mathbf{O}_n(\lambda;y,w,m)\ee^{\ii w\varphi(\lambda)\sigma_3/2}\ee^{-L(\lambda;y,m)\sigma_3}\OurPower{\lambda}{-(m+1)\sigma_3/2}\mathbf{T}(\lambda),
\end{equation}
where $\mathbf{T}(\lambda)$ is a piecewise-constant matrix defined in the four sectors of each disk as indicated in 
Table~\ref{tab:T-in-ER}.  Note that the Boutroux conditions $K_j\in\mathbb{R}$, $j=1,2$, imply that $\mathbf{T}(\lambda)$ is uniformly bounded on compact sets with respect to $m\in\mathbb{C}$ and for arbitrary $n\in\mathbb{Z}_{\ge 0}$.  
\begin{table}[h]
\centering
\caption{The transformation $\mathbf{T}(\lambda)$ defined in the four sectors of the $W$-plane in each of the four disks for $y\in E_\mathrm{R}\setminus\mathbb{R}$.}
\label{tab:T-in-ER}
\begin{tabular}{@{}|l|l|l|l|l|@{}}
\hline
\multirow{2}{*}{\makecell{Center\\$W=0$}} & \multicolumn{4}{c|}{Transformation $\mathbf{T}(\lambda)$} \\
\cline{2-5}
&In Sector $S_\mathrm{I}$ & In Sector $S_\mathrm{II}$ & In Sector $S_\mathrm{III}$ & In Sector $S_\mathrm{IV}$  \\
\hline\hline
\multirow{2}{*}{$\lambda_0$} & $\vphantom{\Big[}c^{\sigma_3}$ &  $c^{\sigma_3}$ & $c^{\sigma_3}$ & $c^{\sigma_3}$ \\
\cline{2-5}
&\multicolumn{4}{c|}{$\vphantom{\Big[}c:=\ii (2\pi)^{1/4}\Gamma(\tfrac{1}{2}-m)^{-1/2}$}\\
\hline
\multirow{2}{*}{$\lambda_1$} & $\vphantom{\Big[}(c\ee^{\ii n K_1/2})^{\sigma_3}$ & $(c\ee^{\ii n K_1/2})^{\sigma_3}$ & 
$(c\ee^{-\ii n K_1/2})^{\sigma_3}$ & $(c\ee^{-\ii n K_1/2})^{\sigma_3}$ \\
\cline{2-5}
&\multicolumn{4}{c|}{$\vphantom{\Big[}c:=(2\pi)^{1/4}\Gamma(\tfrac{1}{2}-m)^{-1/2}$}\\
\hline
\multirow{2}{*}{$\lambda_1^{-1}$} & $\vphantom{\Big[}(-c\ee^{\ii\pi m/2})^{\sigma_3}\ii\sigma_1$ & $(-c\ee^{\ii\pi m/2})^{\sigma_3}\ii\sigma_1$ & $(c\ee^{-\ii\pi m/2})^{\sigma_3}\ii\sigma_1$ & $(c\ee^{-\ii\pi m/2})^{\sigma_3}\ii\sigma_1$ \\
\cline{2-5}
&\multicolumn{4}{c|}{$\vphantom{\Big[}c:=(2\pi)^{-1/4}\Gamma(\tfrac{1}{2}+m)^{1/2}\ee^{\ii nK_2/2}$}\\
\hline
\multirow{2}{*}{\makecell{$\lambda_0^{-1}$,\\$\mathrm{Im}(y)>0$}} & 
$\vphantom{\Big[}(c\ee^{-\ii\pi m/2})^{\sigma_3}\ii\sigma_1$ & 
$(c\ee^{-\ii\pi m/2})^{\sigma_3}\ii\sigma_1$ & 
$(c\ee^{\ii\pi m/2}\ee^{\ii n K_1})^{\sigma_3}\ii\sigma_1$ & 
$(c\ee^{\ii\pi m/2})^{\sigma_3}\ii\sigma_1$ \\
\cline{2-5}
&\multicolumn{4}{c|}{$\vphantom{\Big[}c:=(2\pi)^{-1/4}\Gamma(\tfrac{1}{2}+m)^{1/2}\ee^{\ii n(K_2-K_1)/2}$}\\
\hline
\multirow{2}{*}{\makecell{$\lambda_0^{-1}$,\\$\mathrm{Im}(y)<0$}} & 
$\vphantom{\Big[}(c\ee^{-\ii\pi m/2})^{\sigma_3}\ii\sigma_1$ & 
$(c\ee^{-\ii\pi m/2}\ee^{-\ii n K_1})^{\sigma_3}\ii\sigma_1$ & 
$(c\ee^{\ii\pi m/2})^{\sigma_3}\ii\sigma_1$ & 
$(c\ee^{\ii\pi m/2})^{\sigma_3}\ii\sigma_1$ \\
\cline{2-5}
&\multicolumn{4}{c|}{$\vphantom{\Big[}c:=(2\pi)^{-1/4}\Gamma(\tfrac{1}{2}+m)^{1/2}\ee^{\ii n(K_2+K_1)/2}$}\\
\hline
\end{tabular}
\end{table}
The jump conditions satisfied by $\mathbf{P}(\lambda)$ in each case are most conveniently written in terms of the rescaled variable $\zeta = n^{2/3}W(\lambda)$:
\begin{equation}
\begin{split}
\mathbf{P}_+(\lambda)&=\mathbf{P}_-(\lambda)\begin{bmatrix}1 & \ee^{-\zeta^{3/2}}\\0 & 1\end{bmatrix},\quad \arg(\zeta)=0,\\
\mathbf{P}_+(\lambda)&=\mathbf{P}_-(\lambda)\begin{bmatrix}1 & 0\\\ee^{\zeta^{3/2}} & 1\end{bmatrix},\quad\arg(\zeta)=\pm\frac{2}{3}\pi,\\
\mathbf{P}_+(\lambda)&=\mathbf{P}_-(\lambda)\begin{bmatrix}0 & 1\\-1 & 0\end{bmatrix},\quad
\arg(-\zeta)=0,
\end{split}
\label{eq:Airy-jumps}
\end{equation}
where in each case the boundary values of $\mathbf{P}$ are defined with respect to orientation in the direction of increasing real part of $\zeta$, and where all powers of $\zeta$ are principal branches. We may make a similar transformation of the outer parametrix, noting that in each disk the matrix
\begin{equation}
\dot{\mathbf{P}}^\mathrm{out}(\lambda):=\dot{\mathbf{O}}_n^{\mathrm{out}}(\lambda;y,m)\ee^{\ii w\varphi(\lambda)\sigma_3/2}\ee^{-L(\lambda;y,m)\sigma_3}\OurPower{\lambda}{-(m+1)\sigma_3/2}\mathbf{T}(\lambda),
\label{eq:O-out-local-transform}
\end{equation}
is analytic except for $\arg(-\zeta)=0$ where it satisfies exactly the same jump condition as does $\mathbf{P}(\lambda)$.  This fact, along with the fact that the matrix elements of $\dot{\mathbf{P}}^\mathrm{out}(\lambda)$ blow up at $W(\lambda)=0$ as negative one-fourth powers, implies that $\dot{\mathbf{P}}^\mathrm{out}(\lambda)$ can be written in the form
\begin{equation}
\dot{\mathbf{P}}^\mathrm{out}(\lambda)=
\mathbf{H}_n(\lambda;y,w,m)W(\lambda)^{\sigma_3/4}\mathbf{V}=\mathbf{H}_n(\lambda;y,w,m)n^{-\sigma_3/6}\zeta^{\sigma_3/4}\mathbf{V},\quad \mathbf{V}:=\frac{1}{\sqrt{2}}\begin{bmatrix}1 & -\ii\\-\ii & 1\end{bmatrix},
\label{eq:local-H-define}
\end{equation}
where $\mathbf{H}_n(\lambda;y,w,m)$ is a function of $\lambda$ that is analytic in the disk in question and uniformly bounded with respect to $m$ in compact subsets of $\mathbb{C}\setminus(\mathbb{Z}+\tfrac{1}{2})$ and $n\in\mathbb{Z}_{\ge 0}$, provided $|w|\le K$ for some $K>0$ and $y$ satisfy conditions such as enumerated in Lemma~\ref{lem:Outer-Bounded}.  Noting that the boundary of each disk corresponds to $\zeta$ proportional to $n^{2/3}$, we wish to model the matrix function $\mathbf{P}(\lambda)$ by something that satisfies the jump conditions \eqref{eq:Airy-jumps} exactly and that matches with the terms $\zeta^{\sigma_3/4}\mathbf{V}$ coming from the outer parametrix when $\zeta$ is large.  We are thus led to the the following model Riemann-Hilbert problem.
\begin{rhp}
\label{rhp:Airy}
Find a $2\times 2$ matrix function $\zeta\mapsto\mathbf{A}(\zeta)$ with the following properties:
\begin{itemize}
\item[1.] \textbf{Analyticity:}  $\zeta\mapsto\mathbf{A}(\zeta)$ is analytic in the sectors $S_\mathrm{I}$, $S_\mathrm{II}$, $S_\mathrm{III}$, and $S_\mathrm{IV}$ of the complex $\zeta$-plane (see \eqref{eq:Airy-sectors}), and takes continuous boundary values from each sector.
\item[2.] \textbf{Jump conditions:} The boundary values $\mathbf{A}_\pm(\zeta)$ are related on each ray of the jump contour by the following formul\ae\ (cf., \eqref{eq:Airy-jumps}),
\begin{equation}
\begin{split}
\mathbf{A}_+(\zeta)&=\mathbf{A}_-(\zeta)\begin{bmatrix}1 & \ee^{-\zeta^{3/2}}\\0 & 1\end{bmatrix},\quad\arg(\zeta)=0,\\
\mathbf{A}_+(\zeta)&=\mathbf{A}_-(\zeta)\begin{bmatrix}1 & 0\\\ee^{\zeta^{3/2}} & 1\end{bmatrix},\quad\arg(\zeta)=\pm\frac{2}{3}\pi,\\
\mathbf{A}_+(\zeta)&=\mathbf{A}_-(\zeta)\begin{bmatrix}0 & 1\\-1 & 0\end{bmatrix},\quad\arg(-\zeta)=0.
\end{split}
\label{eq:Airy-jumps-A}
\end{equation}
\item[3.]\textbf{Asymptotics:} $\mathbf{A}(\zeta)\mathbf{V}^{-1}\zeta^{-\sigma_3/4}\to\mathbb{I}$ as $\zeta\to\infty$.
\end{itemize}
\end{rhp}
This problem will be solved in all details in Appendix~\ref{app:Airy}, where it will be shown that $\mathbf{A}(\zeta)\mathbf{V}^{-1}\zeta^{-\sigma_3/4}$ has a complete asymptotic expansion in descending integer powers of $\zeta$ as $\zeta\to\infty$, with the dominant terms being given by 
\begin{equation}
\mathbf{A}(\zeta)\mathbf{V}^{-1}\zeta^{-\sigma_3/4}=\mathbb{I}+\begin{bmatrix}\mathcal{O}(\zeta^{-3}) & \mathcal{O}(\zeta^{-1})\\ \mathcal{O}(\zeta^{-2}) & \mathcal{O}(\zeta^{-3})\end{bmatrix},\quad\zeta\to\infty.
\label{eq:Airy-norm-better}
\end{equation}
In each disk we then build a local approximation of $\mathbf{O}_n(\lambda;y,w,m)$ by multiplying on the left by the holomorphic prefactor $\mathbf{H}_n(\lambda;y,w,m)n^{-\sigma_3/6}$ and on the right by the piecewise-analytic substitution relating $\mathbf{O}_n(\lambda;y,w,m)$ and $\mathbf{P}(\lambda)$:
\begin{equation}
\dot{\mathbf{O}}_n^{\mathrm{in}}(\lambda;y,w,m):=\mathbf{H}_n(\lambda;y,w,m)n^{-\sigma_3/6}\mathbf{A}(n^{2/3}W(\lambda))\mathbf{T}(\lambda)^{-1}\OurPower{\lambda}{(m+1)\sigma_3/2}\ee^{L(\lambda;y,m)\sigma_3}\ee^{-\ii w\varphi(\lambda)\sigma_3/2},
\end{equation}
where $W(\lambda)$ is the conformal map associated with the disk via 
Table~\ref{tab:ConfMapR}, $\mathbf{T}(\lambda)$ is the unimodular transformation matrix given in 
Table~\ref{tab:T-in-ER}, and $\mathbf{H}_n(\lambda;y,w,m)$ is associated with the outer parametrix and the disk in question via \eqref{eq:O-out-local-transform}--\eqref{eq:local-H-define}.

\subsection{Error analysis and proof of Theorem~\ref{theorem:eye}}
\label{sec:Error-elliptic}
Let $\Sigma_\mathbf{O}$ denote the jump contour for the matrix function $\mathbf{O}_n(\lambda;y,w,m)$, which consists of the contour $L$ augmented with the lens boundaries $\partial\Lambda^\pm_\squareurblack$ and $\partial\Lambda^\pm_\squarellblack$.
The \emph{global parametrix} denoted $\dot{\mathbf{O}}_n(\lambda;y,w,m)$ is defined as $\dot{\mathbf{O}}_n^{\mathrm{out}}(\lambda;y,w,m)$ when $\lambda$ lies outside of all four disks, but instead as $\dot{\mathbf{O}}_n^{\mathrm{in}}(\lambda;y,w,m)$ within each disk (the precise definition is different in each disk as explained in Section~\ref{sec:Airy}).  We wish to compare the global parametrix with the (unknown) matrix function $\mathbf{O}_n(\lambda;y,w,m)$, so we introduce the error matrix $\mathbf{E}_n(\lambda;y,w,m)$ defined by $\mathbf{E}_n(\lambda;y,w,m):=\mathbf{O}_n(\lambda;y,w,m)\dot{\mathbf{O}}_n(\lambda;y,w,m)^{-1}$.  The maximal domain of analyticity of $\mathbf{E}_n(\lambda;y,w,m)$ is determined from those of the two factors; therefore $\mathbf{E}_n(\lambda;y,w,m)$ is analytic in $\lambda$ except along a jump contour consisting of (i) the part of $\Sigma_\mathbf{O}$ lying outside of all four disks and (ii) the boundaries of all four disks.  That $\mathbf{E}_n(\lambda;y,w,m)$ can be taken to be an analytic function in the interior of each disk follows from the fact that the inner parametrices $\dot{\mathbf{O}}_n^{\mathrm{in}}(\lambda;y,w,m)$ satisfy exactly the same jump conditions locally as does $\mathbf{O}_n(\lambda;y,w,m)$ and an argument based on Morera's theorem.  The jump contour for $\mathbf{E}_n(\lambda;y,w,m)$ corresponding to the Stokes graph shown in Figure~\ref{fig:StokesGraph} is shown in Figure~\ref{fig:InsideError}.
\begin{figure}[h]
\begin{center}
\includegraphics{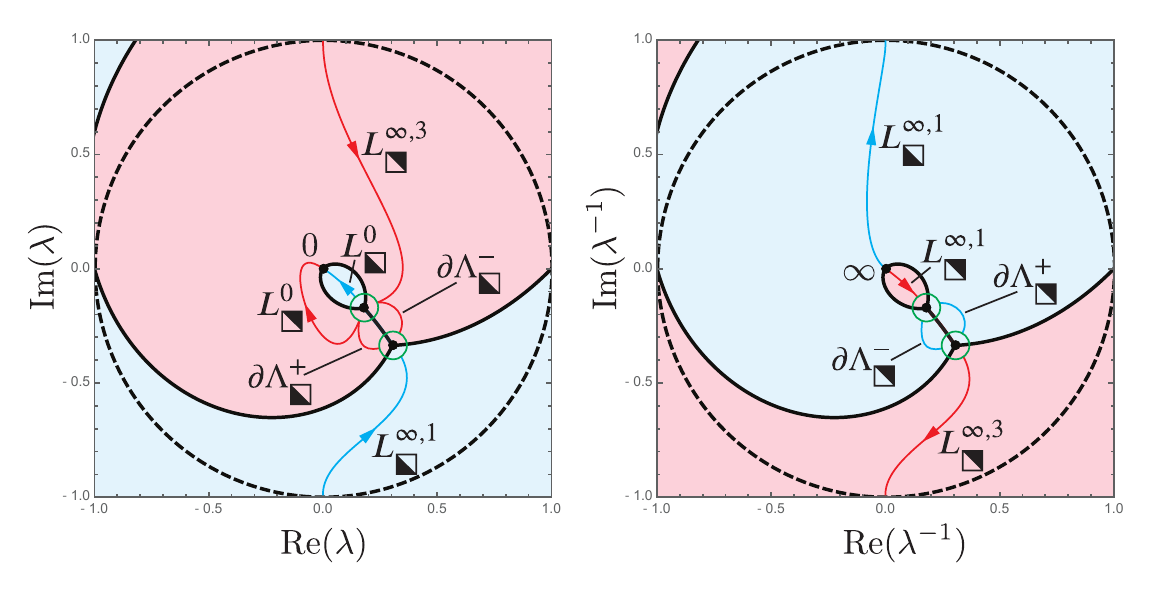}
\end{center}
\caption{The jump contour for $\mathbf{E}_n(\lambda;y,w,m)$ for $y=0.2+0.25\ii\in E_\mathrm{R}$ (cf., Figure~\ref{fig:StokesGraph}).
The jump contour consists of the red and cyan arcs (arcs of $L\setminus(\LInftyBlueTwo\cup\LInftyRedTwo)$ and of the lens boundaries $\partial\Lambda_\squarellblack^\pm$ and $\partial\Lambda_\squareurblack^\pm$ restricted to the exterior of the four disks) and the green circles (the boundaries of the four disks, each of which is taken to have clockwise orientation for the purposes of defining the boundary values of $\mathbf{E}_n(\lambda;y,w,m)$).  Note that the arcs of $\partial\Lambda_\squarellblack^\pm$ in the left-hand panel are oriented toward the upper left, while the arcs of $\partial\Lambda_\squareurblack^\pm$ in the right-hand panel are oriented toward the lower right.}
\label{fig:InsideError}
\end{figure}

Let us assume that, given $m\in\mathbb{C}\setminus (\mathbb{Z}+\tfrac{1}{2})$ and any fixed constants $K>0$ and $\epsilon>0$, $w\in\mathbb{C}$ and $y\in E_\mathrm{R}$ are restricted according to $|w|\le K$ and the inequalities \eqref{eq:Swiss-Cheese}. 
Then by Lemma~\ref{lem:Outer-Bounded}, $\dot{\mathbf{O}}_n(\lambda;y,w,m)$ is uniformly bounded whenever $\lambda$ lies outside all four disks (which both gives $\dot{\mathbf{O}}_n(\lambda;y,w,m)=\dot{\mathbf{O}}_n^{\mathrm{out}}(\lambda;y,w,m)$ and guarantees the condition \eqref{eq:lambda-cheese}).  Since $\mathbf{O}_{n+}(\lambda;y,w,m)=\mathbf{O}_{n-}(\lambda;y,w,m)(\mathbb{I}+\text{exponentially small})$ holds on all arcs of $\Sigma_\mathbf{O}$ lying outside the disks and on which $\dot{\mathbf{O}}_n^{\mathrm{out}}(\lambda;y,w,m)$ is analytic, and since $\mathbf{O}_n(\lambda;y,w,m)$ and $\mathbf{O}_n^{\mathrm{out}}(\lambda;y,w,m)$ satisfy exactly the same jump conditions across all remaining arcs of $\Sigma_\mathbf{O}$ outside of all disks, it follows from Lemma~\ref{lem:Outer-Bounded} that $\mathbf{E}_{n+}(\lambda;y,w,m)=\mathbf{E}_{n-}(\lambda;y,w,m)(\mathbb{I}+\text{exponentially small})$ holds on all jump arcs for $\mathbf{E}_n(\lambda;y,w,m)$ with the exception of the four disk boundaries.  Let the boundary of each disk be oriented in the clockwise direction.  Then a computation shows that on each disk boundary, the matrix $\mathbf{E}_n(\lambda;y,w,m)$ satisfies the jump condition $\mathbf{E}_{n+}(\lambda;y,w,m)=\mathbf{E}_{n-}(\lambda;y,w,m)\mathbf{H}_n(\lambda;y,w,m)n^{-\sigma_3/6}\mathbf{A}(\zeta)\mathbf{V}^{-1}\zeta^{-\sigma_3/4}n^{\sigma_3/6}\mathbf{H}_n(\lambda;y,w,m)^{-1}$ where $\zeta=n^{2/3}W(\lambda)$, $W(\lambda)$ is the relevant conformal mapping from Table~\ref{tab:ConfMapR},
$\mathbf{A}(\zeta)$ is the solution of Riemann-Hilbert Problem~\ref{rhp:Airy}, and $\mathbf{H}_n(\lambda;y,w,m)$ is a bounded function with unit determinant.  Applying the condition \eqref{eq:Airy-norm-better} and using the fact that $W(\lambda)$ is bounded away from zero on the disk boundary then yields the estimate $\mathbf{E}_{n+}(\lambda;y,w,m)=\mathbf{E}_{n-}(\lambda;y,w,m)(\mathbb{I}+\mathcal{O}(n^{-1}))$ as $n\to+\infty$. Since $\mathbf{E}_n(\lambda;y,w,m)\to\mathbb{I}$ as $\lambda\to\infty$, it then follows that this matrix satisfies the conditions of a small-norm Riemann-Hilbert problem.  This implies that (under the conditions of Lemma~\ref{lem:Outer-Bounded}) $\mathbf{E}_n(\lambda;y,w,m)$ exists for sufficiently large $n\in\mathbb{Z}_{>0}$ and satisfies:
\begin{equation}
\mathbf{E}_n(0;y,w,m)=\mathbb{I}+\mathcal{O}(n^{-1})\quad\text{and}\quad
\lim_{\lambda\to\infty}\lambda(\mathbf{E}_n(\lambda;y,w,m)-\mathbb{I})=\mathcal{O}(n^{-1}),\quad n\to +\infty.
\end{equation}
From this result, it follows that $u_n(ny+w;m)=\dot{u}_n(y,w;m)+\mathcal{O}(n^{-1})$ under the conditions \eqref{eq:cheese-carve-outs} which serve to bound the four factors in the fraction in \eqref{eq:udot-elliptic} away from zero.  Combining this result valid for $y\in E_\mathrm{R}$ with the symmetry \eqref{eq:u-n-exact-symmetry} then concludes the proof of Theorem~\ref{theorem:eye}.

\subsection{Detailed properties of the approximation $\dot{u}_n(y,w;m)$.  Proofs of Corollary~\ref{corollary:eye-zeros-and-poles:better} and Theorem~\ref{theorem:density}}
\label{sec:properties-of-udot-elliptic}
Given $n$ and $m$, the zeros of $\dot{u}_n(y,w;m)$ are the roots of the theta function factors in the numerator, namely the pairs $(y,w)$ for which
\begin{equation}
\text{Zeros of $\dot{u}_n(y,w;m)$:\quad}A(\infty;y)\pm A(\kappa(y);y)\mp s_n(y,w,m)\in 2\pi\ii\mathbb{Z}+B(y)\mathbb{Z}.
\label{eq:zeros-of-udot}
\end{equation}
Note that the zeros of $\dot{u}_n(y,w;m)$ corresponding to taking the top sign in \eqref{eq:zeros-of-udot} are also points of the Malgrange divisor \eqref{eq:outer-pole-divisor}, i.e., points at which the solution of Riemann-Hilbert Problem~\ref{rhp:outer-elliptic} fails to exist.  On the other hand, the singularities of $\dot{u}_n(y,w;m)$ are the pairs $(y,w)$ that produce zeros of the denominator in \eqref{eq:udot-elliptic},
\begin{equation}
\text{Singularities of $\dot{u}_n(y,w;m)$:\quad}A(0;y)\pm A(\kappa(y);y)\mp s_n(y,w,m)\in 2\pi\ii\mathbb{Z}+B(y)\mathbb{Z}.
\label{eq:poles-of-udot}
\end{equation}
We hesitate to call these singularities ``poles'' for reasons to be explained in Section~\ref{sec:DE-w} below.

\begin{lem}
Given $n\in\mathbb{Z}_{\ge 0}$, $m\in\mathbb{C}\setminus (\mathbb{Z}+\tfrac{1}{2})$, $y\in E_\mathrm{R}$, and $w\in\mathbb{C}$, at most one of the four conditions in \eqref{eq:zeros-of-udot}--\eqref{eq:poles-of-udot} holds, or equivalently at most one of the four factors $\mathcal{Z}_n^\bullet(y,w;m)$, $\mathcal{Z}_n^\circ(y,w;m)$, $\mathcal{P}_n^\bullet(y,w;m)$, or $\mathcal{P}_n^\circ(y,w;m)$ appearing in the formula \eqref{eq:udot-elliptic} vanishes.
\label{lem:disjoint-lattices}
\end{lem}
\begin{proof}
It suffices to show that none of $2A(\infty;y)$, $2A(0;y)$, $A(\infty;y)-A(0;y)$, nor $A(\infty;y)+A(0;y)$ lies in the lattice $\Lambda:=2\pi\ii\mathbb{Z}+B(y)\mathbb{Z}$.  As in the proof of Lemma~\ref{lemma:one-divisor}, we introduce the Riemann surface $\Gamma=\Gamma(y)$ and its Abel mapping $\mathcal{A}:\Gamma\to\mathrm{Jac}(\Gamma)=\mathbb{C}/\Lambda$, extended to divisors in the usual way.  Given that the base point is a branch point, it is equivalent to show that none of $\mathcal{A}(Q^+(\infty)-Q^-(\infty))$, $\mathcal{A}(Q^+(0)-Q^-(0))$, $\mathcal{A}(Q^+(\infty)-Q^+(0))$, nor $\mathcal{A}(Q^+(\infty)-Q^-(0))$ is mapped to $0\in\mathrm{Jac}(\Gamma)$.  But by the Abel-Jacobi theorem, we just need to rule out the existence of a nonzero meromorphic function on $\Gamma$ having any of the divisors $\mathscr{D}=Q^+(\infty)-Q^-(\infty)$, $\mathscr{D}=Q^+(0)-Q^-(0)$, $\mathscr{D}=Q^+(\infty)-Q^+(0)$, or $\mathscr{D}=Q^+(\infty)-Q^-(0)$.  However, in each case there is only one simple pole whose residue necessarily vanishes, so the desired meromorphic function would in fact be holomorphic with a zero somewhere on $\Gamma$, hence identically zero.
\end{proof}

\begin{proof}[Proof of Corollary~\ref{corollary:eye-zeros-and-poles:better}]
Each of the zero/singularity conditions \eqref{eq:zeros-of-udot}--\eqref{eq:poles-of-udot} defines a regular lattice of points in the $w$-plane, and the minimum distance between points of the four lattices is exactly
\begin{multline}
\delta(y):=\frac{1}{|\nu(y)|}\min\big\{\mathrm{dist}(2A(\infty,y),\Lambda(y)),\mathrm{dist}(2A(0,y),\Lambda(y)),\\
\mathrm{dist}(A(\infty,y)-A(0,y),\Lambda(y)),\mathrm{dist}(A(\infty,y)+A(0,y),\Lambda(y))\big\}>0,\quad y\in E_\mathrm{R},
\end{multline}
where $\Lambda(y):=2\pi\ii\mathbb{Z}+B(y)\mathbb{Z}$ where $\nu(y)$ is the function on $E_\mathrm{R}$ defined by \eqref{eq:nu-defining-equation}.  Note that $|\nu(y)|$ is continuous on $E_\mathrm{R}$ and hence the inequality above follows from Lemma~\ref{lem:disjoint-lattices}.  Moreover, the minimum factor in $\delta(y)$ is continuous on $E_\mathrm{R}$, so it follows from compactness of $K_y\subset E_\mathrm{R}$ in the statement of Corollary~\ref{corollary:eye-zeros-and-poles:better} that
\begin{equation}
\delta:=\inf_{n\ge N}\delta(y_n)\ge\inf_{y\in K_y}\delta(y) = \min_{y\in K_y}\delta(y)>0.
\end{equation}
By definition of the sequence $\{y_n\}_{n=N}^\infty$, taking $y=y_n$ makes one of the four factors in the fraction on the right-hand side of \eqref{eq:udot-elliptic} vanish at $w=0$ for all $n=N,N+1,\dots$, while the roots of the other three factors are bounded away from $w=0$ in the $w$-plane by the distance $\delta>0$.  Now choose $\epsilon=\tfrac{1}{3}\delta$ as the parameter in Theorem~\ref{theorem:eye}; the circle $|w|=\epsilon$ then lies on the boundary of the closed domain characterized by the inequalities \eqref{eq:cheese-carve-outs}.  Letting $n\to+\infty$, Theorem~\ref{theorem:eye} implies that the winding numbers (indices) about the circle $|w|=\epsilon$ of the rational function $f(w):=u_n(ny_n+w;m)$ and the meromorphic function $g(w):=\dot{u}_n(y_n,w;m)$ necessarily agree for sufficiently large $n$.  But since $\epsilon<\delta$ the index of $g(w)=\dot{u}_n(y_n,w;m)$ is $1$ ($-1$) if the sequence $\{y_n\}_{n=N}^\infty$ corresponds to roots of a factor of the numerator (denominator) in \eqref{eq:udot-elliptic}, and this common value of the index is precisely the net number of zeros less poles of the rational function $f(w)=u_n(ny_n+w;m)$ within the circle $|w|=\epsilon$.
This establishes the third statement of Theorem~\ref{theorem:eye} and completes the proof. 
\end{proof}

Since $\mathrm{Re}(B(y))<0$ holds for $y\in E_\mathrm{R}$ and therefore $B(y)$ and $2\pi\ii$ are necessarily linearly independent over the real numbers, we can resolve the left-hand sides of \eqref{eq:zeros-of-udot}--\eqref{eq:poles-of-udot} into
real multiples of the lattice periods:
\begin{equation}
\begin{split}
A(\infty;y)\pm A(\kappa(y);y)\mp s_n(y,w,m)&= 2\pi\ii\alpha_n^{0,\pm}(y,w,m) + B(y)\beta_n^{0,\pm}(y,w,m),\quad\text{where}\\
\alpha_n^{0,\pm}(y,w,m)&:=\frac{\mathrm{Im}(B(y)^*(A(\infty;y)\pm A(\kappa(y);y)\mp s_n(y,w,m)))}{2\pi\mathrm{Re}(B(y))}\\
\beta_n^{0,\pm}(y,w,m)&:=\frac{\mathrm{Re}(A(\infty;y)\pm A(\kappa(y);y)\mp s_n(y,w,m))}{\mathrm{Re}(B(y))},
\end{split}
\label{eq:Zero-Quantum}
\end{equation}
and
\begin{equation}
\begin{split}
A(0;y)\pm A(\kappa(y);y)\mp s_n(y,w,m)&= 2\pi\ii\alpha_n^{\infty,\pm}(y,w,m) + B(y)\beta_n^{\infty,\pm}(y,w,m),\quad\text{where}\\
\alpha_n^{\infty,\pm}(y,w,m)&:=\frac{\mathrm{Im}(B(y)^*(A(0;y)\pm A(\kappa(y);y)\mp s_n(y,w,m)))}{2\pi\mathrm{Re}(B(y))}\\
\beta_n^{\infty,\pm}(y,w,m)&:=\frac{\mathrm{Re}(A(0;y)\pm A(\kappa(y);y)\mp s_n(y,w,m))}{\mathrm{Re}(B(y))}.
\end{split}
\label{eq:Pole-Quantum}
\end{equation}
Thus, the zeros of $\dot{u}_n(y,w;m)$ satisfy the \emph{quantization conditions}
\begin{equation}
\text{Zeros of $\dot{u}_n(y,w;m)$:\quad}
\alpha_n^{0,\pm}(y,w,m)\in\mathbb{Z}\quad\text{and}\quad\beta_n^{0,\pm}(y,w,m)\in\mathbb{Z},
\label{eq:zeros-quantization}
\end{equation}
and similarly 
\begin{equation}
\text{Singularities of $\dot{u}_n(y,w;m)$:\quad}
\alpha_n^{\infty,\pm}(y,w,m)\in\mathbb{Z}\quad\text{and}\quad\beta_n^{\infty,\pm}(y,w,m)\in\mathbb{Z}.
\label{eq:poles-quantization}
\end{equation}
One way to parametrize points within the domain $E_\mathrm{R}$ is by choosing to set $w=0$ and thus $x=ny$ where $y$ ranges over $E_\mathrm{R}$.
Using this parametrization, we can give the following.
\begin{proof}[Proof of Theorem~\ref{theorem:density}]
Given $n$ and $m$, for each choice of sign $\pm$, the conditions \eqref{eq:zeros-quantization} (resp., \eqref{eq:poles-quantization}) for $w=0$ define a network of two families of curves whose common intersections locate the zeros (resp., singularities) of $\dot{u}_n(y,0;m)$ on $E_\mathrm{L}\cup E_\mathrm{R}$.  Given $m\in\mathbb{C}\setminus (\mathbb{Z}+\tfrac{1}{2})$ it is particularly interesting to consider how the curves depend on $n\in\mathbb{Z}$ large and positive.  For this, we observe that the only dependence on $n$ enters through $s_n(y,0,m)$; substituting from \eqref{eq:eta-again} and \eqref{eq:sn} gives, for $y\in E_\mathrm{R}$,
\begin{equation}
\alpha_n^{0,\pm}(y,0,m)=\alpha_0^{0,\pm}(y,0,m)\mp \frac{nK_2(y)}{2\pi},\quad\alpha_0^{0,\pm}(y,0,m)=
\frac{\mathrm{Im}(B(y)^*(A(\infty;y)\pm A(\kappa(y);y)\pm\delta(y,m)))}{2\pi\mathrm{Re}(B(y))},
\label{eq:alpha-zero-n}
\end{equation}
\begin{equation}
\beta_n^{0,\pm}(y,0,m)=\beta_0^{0,\pm}(y,0,m)\mp\frac{nK_1(y)}{2\pi},\quad\beta_0^{0,\pm}(y,0,m)=
\frac{\mathrm{Re}(A(\infty;y)\pm A(\kappa(y);y)\pm\delta(y,m))}{\mathrm{Re}(B(y))},
\end{equation}
\begin{equation}
\alpha_n^{\infty,\pm}(y,0,m)=\alpha_0^{\infty,\pm}(y,0,m)\mp \frac{nK_2(y)}{2\pi},\quad\alpha_0^{\infty,\pm}(y,0,m)=
\frac{\mathrm{Im}(B(y)^*(A(0;y)\pm A(\kappa(y);y)\pm\delta(y,m)))}{2\pi\mathrm{Re}(B(y))},
\end{equation}
\begin{equation}
\beta_n^{\infty,\pm}(y,0,m)=\beta_0^{\infty,\pm}(y,0,m)\mp\frac{nK_1(y)}{2\pi},\quad\beta_0^{\infty,\pm}(y,0,m)=
\frac{\mathrm{Re}(A(0;y)\pm A(\kappa(y);y)\pm\delta(y,m))}{\mathrm{Re}(B(y))}.
\label{eq:beta-infinity-n}
\end{equation}

The simplified formul\ae\ \eqref{eq:alpha-zero-n}--\eqref{eq:beta-infinity-n} show that when $n$ is large, the quantization conditions \eqref{eq:zeros-quantization}--\eqref{eq:poles-quantization} determine a locally (with respect to $y$) uniform tiling of the $y$-plane by parallelograms each of which has area (measured in the $y$-coordinate) $A_\lozenge(y)(1+o(1))$ as $n\to+\infty$, where
\begin{equation}
A_\lozenge(y)=\frac{4\pi^2}{n^2|\nabla K_1(y)\times\nabla K_2(y)|},\quad y\in E_\mathrm{R},
\end{equation}
see Figure~\ref{fig:parallelogram}. By working in the $w$-plane rather than the $y$-plane, one can see that the area $A_\lozenge(y)$ is also proportional by a factor of $n^2$ to the Jacobian determinant \eqref{eq:Jacobian}.
\begin{figure}[h]
\begin{center}
\hspace{1.5cm}\includegraphics{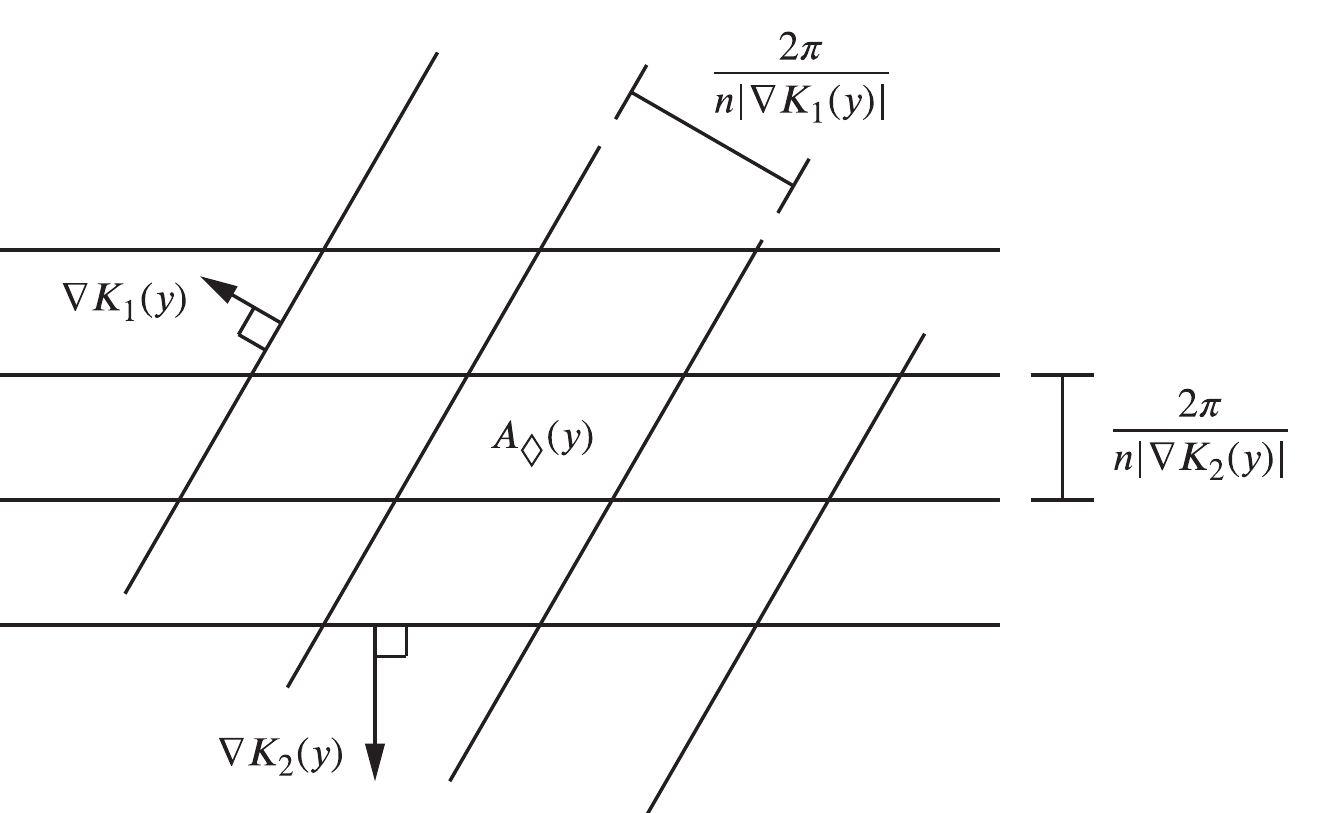}
\end{center}
\caption{The local tiling of the $y$-plane by parallelograms of area $A_\lozenge(y)$.}
\label{fig:parallelogram}
\end{figure}
For each choice of sign $\pm$, one associates via \eqref{eq:zeros-quantization} (resp., \eqref{eq:poles-quantization}) exactly one zero (resp., pole) of $\dot{u}_n(y,0;m)$ with each parallelogram.  Hence the densities (per unit $y$-area) of zeros and poles are exactly the same and are given by $n^2\rho(y)(1+o(1))$ as $n\to+\infty$, where
\begin{equation}
\rho(y):=\frac{2}{n^2A_\lozenge(y)}=\frac{1}{2\pi^2}|\nabla K_1(y)\times\nabla K_2(y)|,\quad y\in E_\mathrm{R}.
\label{eq:zero-pole-density}
\end{equation}
Note that since $K_1(y)$ and $K_2(y)$ are functions independent of $m\in\mathbb{C}\setminus(\mathbb{Z}+\tfrac{1}{2})$, the same is true of $\rho(y)$.  
The density $\rho(y)$ is a smooth nonnegative function on $E_\mathrm{R}$, but it vanishes on $\partial E_\mathrm{R}\setminus\{0\}$ and blows up as $y\to 0$. 
To prove the former, we may use the fact that $\rho(y)$ is inversely-proportional to the Jacobian determinant \eqref{eq:Jacobian}, which blows up as $y\to\partial E_\mathrm{R}\setminus\{0\}$ as mentioned in Section~\ref{sec:Boutroux}.  To prove the blowup of $\rho(y)$ at the origin, we first express the gradients in polar coordinates $y=r\ee^{\ii\theta}$, and thus
\begin{equation}
|\nabla K_1(y)\times\nabla K_2(y)| = \frac{1}{r}\left|\frac{\partial K_1}{\partial r}\frac{\partial K_2}{\partial\theta} -\frac{\partial K_1}{\partial \theta}\frac{\partial K_2}{\partial r}\right|,\quad y=r\ee^{\ii\theta}.
\label{eq:grad-cross-product-polar}
\end{equation}
Next, we compute the partial derivatives near $r=0$.  For this purpose, we recall the scalings introduced in Section~\ref{sec:Boutroux-y-small} to construct the solution of the Boutroux equations for small $r=|y|$.  Thus, $C=y\tilde{C}(r,\theta)$, where for $|\theta|<\pi/2$, 
\begin{equation}
\tilde{C}^0(\theta):=\lim_{r\downarrow 0}\tilde{C}(r,\theta),\quad
\tilde{C}^0_{r}(\theta):=\lim_{r\downarrow 0}\tilde{C}_r(r,\theta),\quad\text{and}\quad
\tilde{C}^0_{\theta}(\theta):=\lim_{r\downarrow 0}\tilde{C}_\theta(r,\theta)=\tilde{C}^{0\prime}(\theta)
\end{equation}
all exist (the subscripts $r$ and $\theta$ denote partial derivatives).  For each such $\theta$, in the limit $r\downarrow 0$, $\lambda_0\to\infty$ while $\lambda_1$ converges to a nonzero limit.  Since the integrands in the definitions of $K_j(y)$, $j=1,2$, are singular at $\lambda=0,\infty$, we first use the Cauchy theorem to rewrite $K_j(y)$ as contour integrals over contours that we may take to be independent of $r$ as $r\downarrow 0$.  Since $\lambda_0\to\infty$ and $\lambda_0^{-1}\to 0$ as $r\downarrow 0$, it is necessary to account for some residues at $\lambda=0,\infty$, but from \eqref{eq:gprimeminushalfVprime-asymp} and \eqref{eq:gprime-E} it follows that these contributions are independent of $y$, so they will not play any role upon taking the required derivatives.  Thus,
\begin{equation}
K_1(y)
= -\pi - \ii\int_{C_1}\frac{R(\lambda;y)}{\lambda^2}\,\dd\lambda,
\end{equation}
where the original contour of integration (a counterclockwise-oriented path just enclosing the arc $\LInftyRedTwo$ with endpoints $\lambda_0\to\infty$ and $\lambda_1$) has been replaced with  $C_1$, a counterclockwise-oriented closed path enclosing the arc $\LInftyBlueTwo$ with endpoints $\lambda_0^{-1}\to 0$ and $\lambda_1^{-1}$ as well as the limit point $\lambda=0$.  Likewise,
\begin{equation}
K_2(y)
= \pi -\ii\int_{C_2}\frac{R(\lambda;y)}{\lambda^2}\,\dd\lambda,
\end{equation}
where $C_2$ is a contour consisting of two arcs joining $\lambda_1$ with $\lambda_1^{-1}$ such that $\lambda_0^{-1}$ and $\lambda=0$ are contained in the region between the two arcs, while $\lambda_0$ is excluded.  With the help of a small additional contour deformation near $\lambda_1$ and $\lambda_1^{-1}$ in the case of $K_2$, both new contours may be taken to be locally independent of $y$ and hence derivatives may be computed by differentiation under the integral sign.  Thus
\begin{equation}
\frac{\partial K_j}{\partial r,\theta}(r\ee^{\ii\theta})=-\frac{\ii}{2}\int_{C_j}\frac{\partial P}{\partial r,\theta}(\lambda;r\ee^{\ii\theta},r\ee^{\ii\theta}\tilde{C}(r,\theta))\frac{\dd\lambda}{\lambda^2R(\lambda;r\ee^{\ii\theta})},\quad j=1,2.
\end{equation}
With the scaling of Section~\ref{sec:Boutroux-y-small}, $P(\lambda;r\ee^{\ii\theta},r\ee^{\ii\theta}\tilde{C}(r,\theta))=-\tfrac{1}{4}r^2\ee^{2\ii\theta}\lambda^4 +\tfrac{1}{2}\ii r\ee^{\ii\theta}\lambda^3 + r\ee^{\ii\theta}\tilde{C}(r,\theta)\lambda^2 +\tfrac{1}{2}\ii r\ee^{\ii\theta}\lambda-\tfrac{1}{4}r^2\ee^{2\ii\theta}$, and therefore
\begin{equation}
\frac{\partial K_j}{\partial r}(r\ee^{\ii\theta})=-\frac{\ii}{4}\int_{C_j}\frac{-r\ee^{2\ii\theta}\lambda^2 +\ii\ee^{\ii\theta}\lambda + 2\ee^{\ii\theta}\tilde{C}(r,\theta) + 2r\ee^{\ii\theta}\tilde{C}_r(r,\theta) +\ii\ee^{\ii\theta}\lambda^{-1}-r\ee^{2\ii\theta}\lambda^{-2}}{R(\lambda;r\ee^{\ii\theta})}\,\dd\lambda,\quad j=1,2,
\label{eq:Kj-r}
\end{equation}
and
\begin{equation}
\frac{\partial K_j}{\partial \theta}(r\ee^{\ii\theta})=-\frac{\ii}{4}\int_{C_j}\frac{-\ii r^2\ee^{2\ii\theta}\lambda^2 -r\ee^{\ii\theta}\lambda + 2\ii r\ee^{\ii\theta}\tilde{C}(r,\theta)+2r\ee^{\ii\theta}\tilde{C}_\theta(r,\theta) - r\ee^{\ii\theta}\lambda^{-1}-\ii r^2\ee^{2\ii\theta}\lambda^{-2}}{R(\lambda;r\ee^{\ii\theta})}\,\dd\lambda,\quad j=1,2.
\label{eq:Kj-theta}
\end{equation}
Now, given $\theta$, the following limit exists uniformly on $C_j$ (again after suitable small deformation near $\lambda_1$ and $\lambda_1^{-1}$ in the case of $C_2$):
\begin{equation}
\lim_{r\downarrow 0} r^{-1/2}R(\lambda;r\ee^{\ii\theta})=\tilde{R}(\lambda;\theta),\quad\tilde{R}(\lambda;\theta)^2 = \frac{1}{2}\ii\ee^{\ii\theta}\lambda^3+\ee^{\ii\theta}\tilde{C}^0(\theta)\lambda^2 + \frac{1}{2}\ii\ee^{\ii\theta}\lambda,
\end{equation}
where $\tilde{R}(\lambda;\theta)$ is analytic except on the limiting Stokes graph arcs $\Sigma^\mathrm{out}(y)$ and is well-defined by choosing the branch globally based on the above limit at any generic point $\lambda$.  Similar uniform limits for the numerators in the integrands of \eqref{eq:Kj-r}--\eqref{eq:Kj-theta} then show that
\begin{equation}
\lim_{r\downarrow 0} r^{1/2}\frac{\partial K_j}{\partial r}(r\ee^{\ii\theta})=\frac{\ee^{\ii\theta}}{4}\int_{C_j}\frac{\lambda+\lambda^{-1}}{\tilde{R}(\lambda;\theta)}\,\dd\lambda-\frac{\ii\ee^{\ii\theta}\tilde{C}^0(\theta)}{2}\int_{C_j}
\frac{\dd\lambda}{\tilde{R}(\lambda;\theta)},\quad j=1,2
\end{equation}
and
\begin{equation}
\lim_{r\downarrow 0}r^{-1/2}\frac{\partial K_j}{\partial\theta}(r\ee^{\ii\theta})=\frac{\ii\ee^{\ii\theta}}{4}\int_{C_j}
\frac{\lambda+\lambda^{-1}}{\tilde{R}(\lambda;\theta)}\,\dd\lambda -\frac{\ii\ee^{\ii\theta}}{2}(\ii\tilde{C}^0(\theta)+\tilde{C}^0_\theta(\theta))\int_{C_j}\frac{\dd\lambda}{\tilde{R}(\lambda;\theta)},\quad j=1,2.
\end{equation}
Therefore, the following limit exists:
\begin{equation}
\frac{1}{2\pi^2}\lim_{r\downarrow 0}\left[\frac{\partial K_1}{\partial r}\frac{\partial K_2}{\partial\theta}-\frac{\partial K_1}{\partial\theta}\frac{\partial K_2}{\partial r}\right]=-\frac{\ii\ee^{2\ii\theta}\tilde{C}^0_\theta(\theta)}{16\pi^2}
\det\begin{bmatrix}\displaystyle\int_{C_1}\frac{\lambda + \lambda^{-1}}{\tilde{R}(\lambda;\theta)}\,\dd\lambda & \displaystyle \int_{C_1}\frac{\dd\lambda}{\tilde{R}(\lambda;\theta)}\\
\displaystyle \int_{C_2}\frac{\lambda + \lambda^{-1}}{\tilde{R}(\lambda;\theta)}\,\dd\lambda & \displaystyle \int_{C_2}\frac{\dd\lambda}{\tilde{R}(\lambda;\theta)}
\end{bmatrix}.
\end{equation}
The absolute value of this limit is the quantity $h(\theta)$ referred to in the statement of Theorem~\ref{theorem:density}.
\end{proof}
%


%

\section{The Special Case of $m\in\mathbb{Z}+\tfrac{1}{2}$}
\label{sec:m-half}
As in Section~\ref{sec:outside}, in this section we study Riemann-Hilbert Problem~\ref{rhp:renormalized} under the substitution $x=ny$, i.e., we set $w=0$.
\subsection{Asymptotic behavior of $u_n(ny;m)$ for $y$ away from the distinguished eyebrow.  Proof of Theorem~\ref{theorem:closed-eye-equilibrium}}
\label{sec:Half-Integer-Away-From-Edge}
If $m\in\mathbb{Z}+\tfrac{1}{2}$, then the jump matrix for $\mathbf{Y}(\lambda)$ simplifies dramatically.  Indeed, if $m=\pm(\tfrac{1}{2}+k)$, $k=0,1,2,3,\dots$, then $\Gamma(\tfrac{1}{2}\mp m)^{-1}=0$.  Therefore, if $m$ is a positive half-integer, in place of \eqref{eq:Yjump-1}--\eqref{eq:Yjump-2} we have simply $\mathbf{Y}_+(\lambda)=\mathbf{Y}_-(\lambda)$ for $\lambda\in \LZeroRed\cup\LInftyRed$, while if $m$ is a negative half-integer, in place of \eqref{eq:Yjump-3}--\eqref{eq:Yjump-4} we have $\mathbf{Y}_+(\lambda)=\mathbf{Y}_-(\lambda)$ for $\lambda\in\LZeroBlue\cup\LInftyBlue$.  Now, we observe that the arc $\pEInftyRed\subset \partial E$ in the right half-plane is the locus of values of $y$ for which the inequality $\mathrm{Re}(V(\lambda;y))>0$ necessarily breaks down at some point of the contour $\LInftyRed$, and that this inequality is \emph{only needed to control the generically nonzero off-diagonal element of the corresponding jump matrix for $\mathbf{Y}$}.
Since this off-diagonal element vanishes for $m=\tfrac{1}{2}+k$, $k=0,1,2,3,\dots$, we see that in this case \emph{the open arc $\pEInftyRed\setminus\{\pm\tfrac{1}{2}\ii\}$ is no obstruction to the continuation of the asymptotic expansions \eqref{eq:u-outside} and \eqref{eq:u-prime-outside} into the domain $E$.}  Likewise, for $m=-(\tfrac{1}{2}+k)$, $k=0,1,2,3,\dots$, the open arc $\pEZeroBlue\setminus\{\pm\tfrac{1}{2}\ii\}$ is no obstruction to the continuation \eqref{eq:u-outside} and \eqref{eq:u-prime-outside} into $E$.

The function $\lNaught(y)$ may be continued through its branch cut $I$ connecting $\pm\tfrac{1}{2}\ii$ from the right.  This continuation can be written in terms of the principal branch of the square root by the formula
\begin{equation}
\lNaught(y)=\frac{\ii}{2y}-\frac{\ii}{y}\left(y-\tfrac{1}{2}\ii\right)^{1/2}\left(y+\tfrac{1}{2}\ii\right)^{1/2},\quad \text{$\mathrm{Re}(y)>0$ or $|\mathrm{Im}(y)|>\tfrac{1}{2}$}.
\label{eq:lambda-0-continue}
\end{equation}
Here the branch cuts of the two square-root factors emanate to the left from the corresponding roots $\pm\tfrac{1}{2}\ii$, so the right-hand side is analytic in the interior domain $E$ with the possible exception of a simple pole at $y=0$.  This particular continuation into $E$ through the open arc $\pEInftyRed\setminus\{\pm\tfrac{1}{2}\ii\}$ is precisely the function $\lNaughtZeroBlue(y)$, a function that has the arc $\pEZeroBlue\subset \partial E$ as its branch cut.  Since $(\pm\tfrac{1}{2}\ii)^{1/2}=2^{-1/2}\ee^{\pm\ii\pi/4}$, it is easy to check that $\mathop{\mathrm{Res}}_{y=0}\lNaughtZeroBlue(y)=0$, so $\lNaughtZeroBlue(y)$ is analytic throughout the interior of $E$.  We conclude that if $m=\tfrac{1}{2}+k$, $k=0,1,2,3,\dots$, the asymptotic formula \eqref{eq:u-outside} in which $\lNaught(y)$ is simply replaced by $\lNaughtZeroBlue(y)$,  is valid both for $y\in\mathbb{C}\setminus E$ as well as throughout the maximal domain of analyticity for $\lNaughtZeroBlue(y)$, namely $y\in\mathbb{C}\setminus\pEZeroBlue$.  Likewise, the continuation of $\lNaught(y)$ through its branch cut from the left can be written as
\begin{equation}
p(y)=\frac{\ii}{2y}+\frac{\ii}{y}\left(-\left(y-\tfrac{1}{2}\ii\right)\right)^{1/2}\left(-\left(y+\tfrac{1}{2}\ii\right)\right)^{1/2},\quad\text{$\mathrm{Re}(y)<0$ or $|\mathrm{Im}(y)|>\tfrac{1}{2}$,}
\end{equation}
which is precisely the branch $\lNaughtInftyRed(y)$ defined as a meromorphic function on the maximal domain $y\in\mathbb{C}\setminus\pEInftyRed$, the only singularity of which is a simple pole at the origin $y=0$.  For $y$ in the interior of the eye $E$, both $\lNaughtZeroBlue(y)$ and $\lNaughtInftyRed(y)$ are well-defined, and we have the identity $\lNaughtInftyRed(y)=\lNaughtZeroBlue(y)^{-1}$ (and of course for $\lambda\in\mathbb{C}\setminus E$ the identity $\lNaughtInftyRed(y)=\lNaughtZeroBlue(y)=\lNaught(y)$ holds).  Due to the pole at the origin, if $m=-(\tfrac{1}{2}+k)$, $k=0,1,2,3,\dots$, the formula \eqref{eq:u-outside} should be replaced with $u_n(ny;m)^{-1}=(\ii\lNaughtInftyRed(y))^{-1}+\mathcal{O}(n^{-1})$ as $n\to+\infty$ which is valid uniformly for $y$ in compact subsets of $\mathbb{C}\setminus\pEInftyRed$.  
This completes the proof of Theorem~\ref{theorem:closed-eye-equilibrium}. \qed

\subsection{Asymptotic behavior of $u_n(ny;m)$ for $y$ near the distinguished eyebrow.  Proof of Theorem~\ref{thm:edge-formulae}}
While to describe the asymptotic behavior of $u_n(ny;m)$ for $m=-(\tfrac{1}{2}+k)$ with $k\in\mathbb{Z}_{\ge 0}$ and $y$ bounded away from the eyebrow $\pEInftyRed$ it was useful to introduce the analytic continuation 
$\lNaughtInftyRed(y)$ of $\lNaught(y)$ from a neighborhood of $y=\infty$ to the maximal domain $\mathbb{C}\setminus\pEInftyRed$, for $y$ near $\pEInftyRed$ it is better to denote the two critical points of $V(\lambda;y)$ as $\lNaught(y)$ and $\lNaught(y)^{-1}$, both of which are analytic functions on all proper sub-arcs of the eyebrow $\pEInftyRed$.  We consider the matrix $\mathbf{M}_n(\lambda;y,m)$ with the simplest choice of $g$-function, namely $g(\lambda)\equiv 0$, which will treat the two critical points more symmetrically, as turns out to be appropriate for $y$ near the eyebrow $\pEInftyRed$.  It is then convenient to reformulate the Riemann-Hilbert conditions on $\mathbf{M}_n(\lambda;y,m)$ in the special case that $m=-(\tfrac{1}{2}+k)$ for $k\in\mathbb{Z}_{\ge 0}$.  Since the jump on $\LInftyBlue\cup\LZeroBlue$ reduces to the identity in this case (see \eqref{eq:Yjump-3}--\eqref{eq:Yjump-4}), the jump contour for $\mathbf{M}_n^{(k)}(\lambda;y):=\mathbf{M}_n(\lambda;y,-(\tfrac{1}{2}+k))=\mathbf{Y}_n(\lambda;ny,-(\tfrac{1}{2}+k))$ is simply $L=\LInftyRed\cup\LZeroRed$, and along the latter contour the factor $\OurPower{\lambda}{-(m+1)}=\OurPower{\lambda}{k-1/2}$ appearing in the jump conditions \eqref{eq:Yjump-1}--\eqref{eq:Yjump-2} changes sign at the junction point between $\LInftyRed$ and $\LZeroRed$.  Therefore, if we define a branch $\lambda_\infty^{k-1/2}$ analytic along $L$ and such that $\lambda_\infty^{k-1/2}=\OurPower{\lambda}{k-1/2}$ holds when $\lambda\in \LInftyRed$, the Riemann-Hilbert problem for $\mathbf{M}_n^{(k)}(\lambda;y)$ can be written as follows.
\begin{rhp}[Eyebrow problem for $m=-(\tfrac{1}{2}+k)$]
Given parameters $n,k\in\mathbb{Z}_{\ge 0}$ as well as $y$ in a tubular neighborhood $T$ of $\pEInftyRed$ as defined in \eqref{eq:T-define}, seek a $2\times 2$ matrix function $\lambda\mapsto\mathbf{M}_n^{(k)}(\lambda;y)$ with the following properties:
\begin{itemize}
\item[1.]\textbf{Analyticity:} $\lambda\mapsto\mathbf{M}_n^{(k)}(\lambda;y)$ is analytic in the domain $\lambda\in\mathbb{C}\setminus L$, $L:=\LInftyRed\cup\LZeroRed$.  It takes continuous boundary values on $L\setminus\{0\}$ from each maximal domain of analyticity.  
\item[2.]\textbf{Jump conditions:} The boundary values $\mathbf{M}^{(k)}_{n\pm}(\lambda;y)$ are related by
\begin{equation}
\mathbf{M}^{(k)}_{n+}(\lambda;y)=\mathbf{M}^{(k)}_{n-}(\lambda;y)\begin{bmatrix}1 & \displaystyle\frac{\sqrt{2\pi}}{k!}\lambda_\infty^{k-1/2}\ee^{-nV(\lambda;y)}\\0 & 1\end{bmatrix},\quad\lambda\in L.
\label{eq:M-edge-jump}
\end{equation}
\item[3.]\textbf{Asymptotics:}  $\mathbf{M}_n^{(k)}(\lambda;y)\to\mathbb{I}$ as $\lambda\to\infty$ and $\mathbf{M}_n^{(k)}(\lambda;y)\lambda^{k\sigma_3}$ has a well-defined limit as $\lambda\to 0$.
\end{itemize}
\label{rhp:edge}
\end{rhp}

\subsubsection{Motivation:  the special case of $k=0$}
When $k=0$, Riemann-Hilbert Problem~\ref{rhp:edge} reduces from a multiplicative matrix problem to an additive scalar problem for the $12$-entry, and the explicit solution is obtained from the Plemelj formula:
\begin{equation}
\mathbf{M}_n^{(0)}(\lambda;y)=\begin{bmatrix}1 & \displaystyle\frac{1}{\ii \sqrt{2\pi}}\int_L\frac{\mu_\infty^{-1/2}\ee^{-nV(\mu;y)}}{\mu-\lambda}\,\dd\mu\\
0 & 1\end{bmatrix}.
\end{equation}
Since $\mathbf{Y}_n(\lambda;ny,-\tfrac{1}{2})=\mathbf{M}_n^{(0)}(\lambda;y)$, applying \eqref{eq:u-n-from-Y-formula} gives the exact result
\begin{equation}
u_n(ny;-\tfrac{1}{2})=\ii\frac{\displaystyle\int_L\lambda_\infty^{-1/2}\ee^{-nV(\lambda;y)}\,\dd\lambda}{\displaystyle\int_L\lambda_\infty^{-3/2}\ee^{-nV(\lambda;y)}\,\dd\lambda}.
\label{eq:u-edge-k-zero}
\end{equation}
The large-$n$ asymptotic behavior of the rational solution $u_n(ny;-\tfrac{1}{2})$ is therefore reduced to the classical saddle-point expansion of two related contour integrals.  When $y$ is close to the eyebrow $\pEInftyRed$, $\mathrm{Re}(V(\lNaught(y);y))\approx 0$, so the landscape of $\mathrm{Re}(-V(\lambda;y))$ in the $\lambda$-plane is similar to that shown in the central panels of Figure~\ref{fig:y-real}, except in small neighborhoods of the two critical points $\lambda=\lNaught(y),\lNaught(y)^{-1}$.  In particular, for $\lambda$ bounded away from these two points, the contour $L=\LInftyRed\cup\LZeroRed$ lies entirely in the red-shaded domain and hence $\mathrm{Re}(-V(\lambda;y))<0$ holds.  This makes the corresponding contributions to the integrands in the numerator and denominator of \eqref{eq:u-edge-k-zero}  exponentially small by comparison with the contributions from neighborhoods of the two saddle points.  In a sense, this classical saddle point analysis can be embedded in a more general scheme that applies to Riemann-Hilbert Problem~\ref{rhp:edge} also for $k=1,2,3,\dots$.

\begin{rem} In our previous paper on the subject of rational solutions of the Painlev\'e-III equation \cite{BothnerMS18} we observed that when $m\in\mathbb{Z}+\tfrac{1}{2}$ it is possible to reduce Riemann-Hilbert Problem~\ref{rhp:renormalized} to a linear algebraic Hankel system of dimension independent of $n$ in which the coefficients are contour integrals amenable to the classical method of steepest descent when $n$ is large such as those just considered above.  We originally thought that these Hankel systems would provide the most efficient approach to the detailed analysis of $u_n(ny;m)$ for half-integral $m$, but it turns out that an approach based on more modern techniques of steepest descent for Riemann-Hilbert problems is more effective.  We develop this approach in the following paragraphs.  
\end{rem}

\subsubsection{Modified outer parametrix}
The same argument that focuses the contour integrals in \eqref{eq:u-edge-k-zero} on the critical points serves more generally to make the jump matrix in \eqref{eq:M-edge-jump} an exponentially small perturbation of the identity matrix except in neighborhoods of the critical points, which in turn suggests approximating $\mathbf{M}_n^{(k)}(\lambda;y)$ with a single-valued analytic function built to satisfy the required asymptotic conditions as $\lambda\to\infty$ and $\lambda\to 0$.  Thus,
given $k\in\mathbb{Z}_{\ge 0}$ and nonnegative integers $\alpha_1$ and $\alpha_2$ such that
\begin{equation}
\alpha_1+\alpha_2 = -\left(m+\frac{1}{2}\right)=k,
\label{eq:alpha-partition-general-red}
\end{equation}
we define an outer parametrix for $\mathbf{M}_n^{(k)}(\lambda;y)$ by the formula
\begin{equation}
\dot{\mathbf{M}}^{\mathrm{out},(\alpha_1,\alpha_2)}(\lambda;y)=\lambda^{-k\sigma_3}(\lambda-\lNaught(y)^{-1})^{\alpha_1\sigma_3}(\lambda-\lNaught(y))^{\alpha_2\sigma_3}.
\label{eq:edge-outer}
\end{equation}
This function
is analytic for $\lambda\in\mathbb{C}\setminus\{0,\lNaught(y),\lNaught(y)^{-1}\}$, and satisfies the required asymptotic conditions in the sense that $\dot{\mathbf{M}}^{\mathrm{out},(\alpha_1,\alpha_2)}(\lambda;y)\to\mathbb{I}$ as $\lambda\to\infty$ and that $\dot{\mathbf{M}}^{\mathrm{out},(\alpha_1,\alpha_2)}(\lambda;y)\lambda^{k\sigma_3}$ is analytic at $\lambda=0$.  The singularities in the outer parametrix at the critical points $\lambda=\lNaught(y),\lNaught(y)^{-1}$ are needed to balance the local behavior of $\mathbf{M}_n^{(k)}(\lambda;y)$ which we turn to approximating next.

\subsubsection{Inner parametrices based on Hermite polynomials}
As the tubular neighborhood $T$ containing $y$ excludes the branch points $\lambda=\pm\tfrac{1}{2}\ii$, the two critical points remain distinct and hence simple, and both are analytic and nonvanishing functions of $y\in T$.  To set up a uniform treatment of the two critical points, we may also refer to the critical points as $\lambda_1(y):=\lNaught(y)^{-1}$ and $\lambda_2(y):=\lNaught(y)$, which indicates the order in which neighborhoods of these points are visited as $\lambda$ traverses $L=\LInftyRed\cup\LZeroRed$.  
Since $\lambda_j(y)$ are analytic and nonvanishing functions on $T$, to define $\lambda_\infty^{k-1/2}=\lambda^k\lambda_\infty^{-1/2}$ for $\lambda=\lambda_j(y)$ it suffices by analytic continuation to determine the value when $y=0.331372$ corresponding to the real midpoint of the eyebrow $\pEInftyRed$.  Thus, from the central panels in Figure~\ref{fig:y-real} we get that when $\lambda=\lambda_j(0.331372)$, $\lambda_\infty^{-1/2}$ lies in the right half-plane for both $j=1,2$.\smallskip  

Let $D_j$ be simply-connected neighborhoods of $\lambda_j(y)$, $j=1,2$, respectively, and assume that these neighborhoods are sufficiently small but independent of $n$.  Exploiting the fact that both critical points of $V$ are simple, we conformally map $D_j$ to a neighborhood of the origin via a conformal mapping $\lambda\mapsto W_j(\lambda;y)$, where
\begin{equation}
V(\lambda;y)-V(\lambda_j(y);y)=W_j(\lambda;y)^2,\quad\lambda\in D_j,\quad j=1,2.
\end{equation}
In this equation we make sure to choose branches of $\log(\lambda)$ in $V$ so that the left-hand side is a well-defined analytic function of $\lambda$ that vanishes to second order at the critical point $\lambda=\lambda_j(y)$.  For small enough $D_j$, this relation defines $W_j(\lambda;y)$ as a conformal mapping up to a sign, which we select such that (possibly after some local adjustment of $L$ near the critical points) the image of the oriented arc $L\cap D_j$ is a real interval containing $W_j=0$ traversed in the direction of increasing $W_j$.  We will need the precise value of $W_j'(\lambda_j(y);y)$, and by implicit differentiation one finds that $\tfrac{1}{2}V''(\lambda_j(y);y)=W_j'(\lambda_j(y);y)^2$.  Now for $j=1,2$, $\tfrac{1}{2}V''(\lambda_j(y);y)$ is an analytic and non-vanishing function of $y$ on the tubular neighborhood $T$ in question, so to determine the $\lambda$-derivative $W_j'(\lambda_j(y);y)$ as an analytic function, it suffices to determine its value at any one point, say $y=0.331372$ where the eyebrow $\pEInftyRed$ intersects the positive real $y$-axis.  Here, from the central panels in Figure~\ref{fig:y-real} one can use the geometric interpretation of $W'_j(\lambda_j(y);y)^{-1}$ as the phase factor of the directed tangent to $L$ to deduce that $W_1'(\lambda_1(y);y)$ is positive imaginary while $W_2'(\lambda_2(y);y)$ is negative real when $y=0.331372$.\smallskip

Given the conformal maps $W_j:D_j\to \mathbb{C}$, $j=1,2$, we define corresponding analytic and non-vanishing functions of $\lambda\in D_j$ denoted $f_j^{(\alpha_1,\alpha_2)}(\lambda;y)$ such that
\begin{equation}
\dot{\mathbf{M}}^{\mathrm{out},(\alpha_1,\alpha_2)}(\lambda;y)=f_j^{(\alpha_1,\alpha_2)}(\lambda;y)^{\sigma_3}W_j(\lambda;y)^{\alpha_j\sigma_3},\quad\lambda\in D_j, \quad j=1,2.
\label{eq:F-j-define-red}
\end{equation}
Likewise, the function $\lambda\mapsto \sqrt{2\pi}\lambda_\infty^{k-1/2}/k!$ admits analytic continuation from $L\cap D_j$ to all of $D_j$, and this function is non-vanishing on $D_j$ (taken sufficiently small but independent of $n$).  Therefore, it has an analytic and non-vanishing square root, which we denote by $d_j(\lambda)$, $j=1,2$.  Then, using \eqref{eq:M-edge-jump} the jump condition for the modified matrix $\mathbf{N}_n^{(k,j)}(\lambda;y):=\mathbf{M}_n^{(k)}(\lambda;y)\ee^{-nV(\lambda_j(y);y)\sigma_3/2}d_j(\lambda)^{\sigma_3}$ satisfies the local jump conditions
\begin{equation}
\mathbf{N}^{(k,j)}_{n+}(\lambda;y)=\mathbf{N}^{(k,j)}_{n-}(\lambda;y)\begin{bmatrix}1 & \ee^{-nW_j(\lambda;y)^2}\\0 & 1\end{bmatrix},\quad\lambda\in L\cap D_j,\quad j=1,2.
\end{equation}
To define appropriate solutions of these jump conditions within the neighborhoods $D_j$ yielding inner parametrices matching well onto the outer parametrix when $\lambda\in\partial D_j$, we need to take into account the final factor on the right-hand side of \eqref{eq:F-j-define-red}.  Thus writing $\zeta=n^{1/2}W_j(\lambda;y)$, we arrive at the following model Riemann-Hilbert problem.
\begin{rhp}
Given $\alpha\in\mathbb{Z}_{\ge 0}$, seek a $2\times 2$ matrix function $\zeta\mapsto\mathbf{H}^{(\alpha)}(\zeta)$ with the following properties:
\begin{itemize}
\item[1.]\textbf{Analyticity:}  $\zeta\mapsto\mathbf{H}^{(\alpha)}(\zeta)$ is analytic for $\mathrm{Im}(\zeta)\neq 0$, taking continuous boundary values on the real axis oriented left-to-right.
\item[2.]\textbf{Jump conditions:}  The boundary values $\mathbf{H}^{(\alpha)}_\pm(\zeta)$ taken on the real axis satisfy the following jump condition:
\begin{equation}
\mathbf{H}^{(\alpha)}_+(\zeta)=\mathbf{H}^{(\alpha)}_-(\zeta)\begin{bmatrix}1 & \ee^{-\zeta^2}\\0 & 1\end{bmatrix},\quad \zeta\in\mathbb{R}.
\end{equation}
\item[3.]\textbf{Asymptotics:}
$\mathbf{H}^{(\alpha)}(\zeta)$ is required to satisfy the normalization condition 
\begin{equation}
\lim_{\zeta\to\infty}\mathbf{H}^{(\alpha)}(\zeta)\zeta^{-\alpha\sigma_3}=\mathbb{I}.
\label{eq:FIK-Hermite-normalize-red}
\end{equation}
\end{itemize}
\label{rhp:FIK-Hermite-red}
\end{rhp}
This problem is well-known \cite{FokasIK91} to be solvable explicitly in terms of Hermite polynomials $\{H_j(\zeta)\}_{j=0}^\infty$ defined by the positivity of the leading coefficient, 
$H_j(\zeta)=h_j\zeta^j+\cdots$ for $h_j>0$, and the orthogonality conditions
\begin{equation}
\int_\mathbb{R}H_j(\zeta)H_{j'}(\zeta)\ee^{-\zeta^2}\,\dd\zeta=\delta_{jj'}.
\end{equation}
Indeed the solution for $\alpha=0$ is explicitly 
\begin{equation}
\mathbf{H}^{(0)}(\zeta):=\begin{bmatrix}1&\displaystyle\frac{1}{2\pi\ii}\int_\mathbb{R}\frac{\ee^{-s^2}\,\dd s}{s-\zeta}\\0 & 1\end{bmatrix},
\end{equation}
and for positive degree,
\begin{equation}
\mathbf{H}^{(\alpha)}(\zeta):=\begin{bmatrix} \displaystyle\frac{1}{h_\alpha}H_\alpha(\zeta) & \displaystyle
\frac{1}{2\pi\ii h_\alpha}\int_\mathbb{R}\frac{H_\alpha(s)\ee^{-s^2}\,\dd s}{s-\zeta}\\
\displaystyle -2\pi\ii h_{\alpha-1}H_{\alpha-1}(\zeta) & \displaystyle -h_{\alpha-1}\int_\mathbb{R}\frac{H_{\alpha-1}(s)\ee^{-s^2}\,\dd s}{s-\zeta}\end{bmatrix},\quad \alpha\ge 1.
\end{equation}
From these formul\ae\ we see that the normalization condition \eqref{eq:FIK-Hermite-normalize-red} takes the more concrete form
\begin{equation}
\mathbf{H}^{(\alpha)}(\zeta)\zeta^{-\alpha\sigma_3}=
\begin{cases}
\begin{bmatrix}1 & (2\pi\ii h_0^2)^{-1}\zeta^{-1}+\mathcal{O}(\zeta^{-3})\\0 & 1\end{bmatrix},&\quad\alpha=0,\smallskip\\
\begin{bmatrix}1+O(\zeta^{-2}) & (2\pi\ii h_\alpha^2)^{-1}\zeta^{-1}+\mathcal{O}(\zeta^{-3})\\
-2\pi\ii h_{\alpha-1}^2\zeta^{-1}+\mathcal{O}(\zeta^{-3}) & 1+\mathcal{O}(\zeta^{-2})\end{bmatrix},&\quad\alpha\ge 1,
\end{cases}
\end{equation}
in the limit $\zeta\to\infty$, where the error terms on the diagonal (resp., off-diagonal) are full asymptotic series in descending even (resp., odd) powers of $\zeta$ (terminating after finitely-many terms in the first column), and where the leading coefficients are explicitly given by \cite[Chapter 18]{DLMF}
\begin{equation}
h_\alpha:=\frac{2^{\alpha/2}}{\pi^{1/4}\sqrt{\alpha!}},\quad\alpha=0,1,2,\dots,
\label{eq:Hermite-constants}
\end{equation}
and by convention we define $h_{-1}:=0$.
Now we define the inner parametrices by
\begin{multline}
\dot{\mathbf{M}}_n^{\mathrm{in},(\alpha_1,\alpha_2,j)}(\lambda;y):=\ee^{-nV(\lambda_j(y);y)\sigma_3/2}n^{-\alpha_j\sigma_3/2}d_j(\lambda)^{\sigma_3}f_j^{(\alpha_1,\alpha_2)}(\lambda;y)^{\sigma_3}\mathbf{H}^{(\alpha_j)}(n^{1/2}W_j(\lambda;y))d_j(\lambda)^{-\sigma_3}\ee^{nV(\lambda_j(y);y)\sigma_3/2},\\
\lambda\in D_j,\quad j=1,2.
\end{multline}
Since the factors to the left of $\mathbf{H}^{(\alpha_j)}(\cdot)$ are analytic within $D_j$, it is easy to see that $\dot{\mathbf{M}}_n^{\mathrm{in},(\alpha_1,\alpha_2,j)}(\lambda;y)$ is analytic within $D_j$ except along $L\cap D_j$, where it exactly satisfies the jump condition \eqref{eq:M-edge-jump}.  We also see easily that
\begin{multline}
\dot{\mathbf{M}}_n^{\mathrm{in},(\alpha_1,\alpha_2,j)}(\lambda;y)\dot{\mathbf{M}}^{\mathrm{out},(\alpha_1,\alpha_2)}(\lambda;y)^{-1}=
\ee^{-nV(\lambda_j(y);y)\sigma_3/2}n^{-\alpha_j\sigma_3/2}d_j(\lambda)^{\sigma_3}f_j^{(\alpha_1,\alpha_2)}(\lambda;y)^{\sigma_3}\\
{}\cdot\mathbf{H}^{(\alpha_j)}(\zeta)\zeta^{-\alpha_j\sigma_3}\cdot 
f_j^{(\alpha_1,\alpha_2)}(\lambda;y)^{-\sigma_3}d_j(\lambda)^{-\sigma_3}n^{\alpha_j\sigma_3/2}\ee^{nV(\lambda_j(y);y)\sigma_3/2},
\label{eq:MinMoutEdge}
\end{multline}
where $\zeta:=n^{1/2}W_j(\lambda;y)$.  Now, set
\begin{equation}
A_j^{(\alpha_1,\alpha_2)}(\lambda;y):=\frac{d_j(\lambda)^2f_j^{(\alpha_1,\alpha_2)}(\lambda;y)^2}{2\pi\ii h_{\alpha_j}^2W_j(\lambda;y)}
\quad\text{and}\quad
B_j^{(\alpha_1,\alpha_2)}(\lambda;y):= -\frac{2\pi\ii h_{\alpha_j-1}^2}{d_j(\lambda)^2f_j^{(\alpha_1,\alpha_2)}(\lambda;y)^2W_j(\lambda;y)},\quad j=1,2.
\end{equation}
These are meromorphic functions of $\lambda\in D_j$ with simple poles at $\lambda_j(y)$, and they are independent of $n$.
Since $W_j(\lambda;y)$ is bounded away from zero when $\lambda\in\partial D_j$, restriction \eqref{eq:MinMoutEdge} to the boundaries $\partial D_j$, $j=1,2$, gives
\begin{multline}
\dot{\mathbf{M}}_n^{\mathrm{in},(\alpha_1,\alpha_2,j)}(\lambda;y)\dot{\mathbf{M}}^{\mathrm{out},(\alpha_1,\alpha_2)}(\lambda;y)^{-1}=\ee^{-nV(\lambda_j(y);y)\sigma_3/2}n^{-\alpha_j\sigma_3/2}\\
\cdot\begin{bmatrix}1 + \mathcal{O}(n^{-1}) & A_j^{(\alpha_1,\alpha_2)}(\lambda;y)n^{-1/2}+\mathcal{O}(n^{-3/2})\\
B_j^{(\alpha_1,\alpha_2)}(\lambda;y)n^{-1/2}+\mathcal{O}(n^{-3/2}) & 1+\mathcal{O}(n^{-1})\end{bmatrix}\\
\cdot n^{\alpha_j\sigma_3/2}\ee^{nV(\lambda_j(y);y)\sigma_3/2},\quad\lambda\in\partial D_j,\quad j=1,2
\label{eq:red-eyebrow-disk-boundaries}
\end{multline}
if $\alpha_j\ge 0$, while in the special case $\alpha_j=0$ we may also write
\begin{multline}
\dot{\mathbf{M}}_n^{\mathrm{in},(\alpha_1,\alpha_2,j)}(\lambda;y)\dot{\mathbf{M}}^{\mathrm{out},(\alpha_1,\alpha_2)}(\lambda;y)^{-1} = \ee^{-nV(\lambda_j(y);y)\sigma_3/2}
\begin{bmatrix}1 & A_j^{(\alpha_1,\alpha_2)}(\lambda;y)n^{-1/2}+\mathcal{O}(n^{-3/2})\\0 & 1\end{bmatrix}\ee^{nV(\lambda_j(y);y)\sigma_3/2},\\
\lambda\in \partial D_j,\quad\alpha_j=0.
\label{eq:red-eyebrow-disk-boundaries-alpha-0}
\end{multline}

\subsubsection{Initial global parametrix construction and comparison matrices}
\label{sec:edge-comparison-matrices}
Given non-negative integers $\alpha_1,\alpha_2$ satisfying \eqref{eq:alpha-partition-general-red}, the global parametrix for $\mathbf{M}_n^{(k)}(\lambda;y)$ is then defined by:
\begin{equation}
\dot{\mathbf{M}}_n^{(\alpha_1,\alpha_2)}(\lambda;y):=\begin{cases}
\dot{\mathbf{M}}_n^{\mathrm{in},(\alpha_1,\alpha_2,1)}(\lambda;y),&\quad\lambda\in D_1\\
\dot{\mathbf{M}}_n^{\mathrm{in},(\alpha_1,\alpha_2,2)}(\lambda;y),&\quad\lambda\in D_2\\
\dot{\mathbf{M}}^{\mathrm{out},(\alpha_1,\alpha_2)}(\lambda;y),&\quad\lambda\in\mathbb{C}\setminus\overline{D_1\cup D_2}.
\end{cases}
\end{equation}

For later purposes, we will need to record the residues of $A_j^{(\alpha_1,\alpha_2)}(\lambda;y)$ and $B_j^{(\alpha_1,\alpha_2)}(\lambda;y)$ at $\lambda=\lambda_j$, where $W_j(\lambda;y)$ vanishes to first order.  Thus,
\begin{equation}
\mathop{\mathrm{Res}}_{\lambda=\lambda_j}A_j^{(\alpha_1,\alpha_2)}(\lambda;y)=\frac{d_j(\lambda_j)^2f_j^{(\alpha_1,\alpha_2)}(\lambda_j;y)^2}{2\pi\ii h_{\alpha_j}^2W_j'(\lambda_j;y)}\quad\text{and}\quad
\mathop{\mathrm{Res}}_{\lambda=\lambda_j}B_j^{(\alpha_1,\alpha_2)}(\lambda;y)=
-\frac{2\pi\ii h_{\alpha_j-1}^2}{d_j(\lambda_j)^2f_j^{(\alpha_1,\alpha_2)}(\lambda_j;y)^2W_j'(\lambda_j;y)},
\end{equation}
and combining \eqref{eq:edge-outer} with \eqref{eq:F-j-define-red} and l'H\^opital's rule gives
\begin{equation}
f_j^{(\alpha_1,\alpha_2)}(\lambda_j;y)=\lambda_j^{-k}(\lambda_j-\lambda_{3-j})^{\alpha_{3-j}}W_j'(\lambda_j;y)^{-\alpha_j},\quad j=1,2.
\end{equation}
Recalling the definition of $d_j(\lambda)$ as a square root of $\sqrt{2\pi}\lambda_\infty^{k-1/2}/k!$, 
we have
\begin{equation}
\begin{split}
d_j(\lambda_j)^2f_j^{(\alpha_1,\alpha_2)}(\lambda_j;y)^2 &= \frac{\sqrt{2\pi}}{k!}\lambda_j^{-k}\lambda_{j,\infty}^{-1/2}(\lambda_2-\lambda_1)^{2\alpha_{3-j}}W_j'(\lambda_j;y)^{-2\alpha_j}\\ &=
\frac{2^{\alpha_j}\sqrt{2\pi}}{k!}\lambda_j^{-k}\lambda_{j,\infty}^{-1/2}(\lambda_2-\lambda_1)^{2\alpha_{3-j}}V''(\lambda_j;y)^{-\alpha_j}
\end{split}
\end{equation}
where $\lambda_{j,\infty}^{-1/2}$ refers to the branch of the square root that lies in the right half-plane when $y=0.331372$ (the real point of $\pEInftyRed$), continued analytically to all $y\in T$, and we used the identity $W_j'(\lambda_j;y)^2=\tfrac{1}{2}V''(\lambda_j;y)$ and the fact that $2\alpha_{3-j}$ is even.  Combining this with \eqref{eq:Hermite-constants} finally gives
\begin{equation}
\begin{split}
\mathop{\mathrm{Res}}_{\lambda=\lambda_j}A^{(\alpha_1,\alpha_2)}_j(\lambda;y)&=
\frac{\alpha_j!\lambda_j^{-k}\lambda_{j,\infty}^{-1/2}(\lambda_2-\lambda_1)^{2\alpha_{3-j}}V''(\lambda_j;y)^{-\alpha_j}}{\ii k! V''(\lambda_j;y)^{1/2}}\\
\mathop{\mathrm{Res}}_{\lambda=\lambda_j}B^{(\alpha_1,\alpha_2)}_j(\lambda;y)&=
\frac{k!\lambda_j^k\lambda_{j,\infty}^{1/2}(\lambda_2-\lambda_1)^{-2\alpha_{3-j}}V''(\lambda_j;y)^{\alpha_j}}{\ii (\alpha_j-1)!V''(\lambda_j;y)^{1/2}},
\end{split}
\label{eq:edge-A-B-residues}
\end{equation}
where the analytic functions $V''(\lambda_1(y);y)^{1/2}$ and $V''(\lambda_2(y);y)^{1/2}$ are respectively positive imaginary and negative real when $y=0.331372$.\bigskip

To compare $\mathbf{M}_n^{(k)}(\lambda;y)$ with its parametrix, we define two types of comparison matrices by
\begin{equation}
\mathbf{F}_n^{(\alpha_1,\alpha_2,j)}(\lambda;y):=\ee^{nV(\lambda_j(y);y)\sigma_3/2}n^{(\alpha_j+1/2)\sigma_3/2}\mathbf{M}_n^{(k)}(\lambda;y)\dot{\mathbf{M}}_n^{(\alpha_1,\alpha_2)}(\lambda;y)^{-1}n^{-(\alpha_j+1/2)\sigma_3/2}\ee^{-nV(\lambda_j(y);y)\sigma_3/2},\quad j=1,2.
\end{equation}
Both types (i.e., for $j=1,2$) of comparison matrix have the property that they are analytic functions of $\lambda$ in both domains $D_1$ and $D_2$ (because the continuous boundary values of $\mathbf{M}_n^{(k)}(\lambda;y)$ and $\dot{\mathbf{M}}_n^{(\alpha_1,\alpha_2)}(\lambda;y)$ satisfy the same jump conditions there) and in the exterior domain except on the original jump contour $L=\LInftyRed\cup\LZeroRed$.  Moreover, it is easy to check that $\mathbf{F}_n^{(\alpha_1,\alpha_2,j)}(\lambda;y)\to\mathbb{I}$ as $\lambda\to\infty$.  The comparison matrices therefore satisfy the conditions of a Riemann-Hilbert problem specified by jump conditions across the part of $L$ exterior to the domains $D_1$ and $D_2$ and across the boundaries $\partial D_1$ and $\partial D_2$ of these domains.  
Jump contours for the comparison matrices are illustrated in Figures~\ref{fig:EdgeError-within} and \ref{fig:EdgeError-without} (cf., Figure~\ref{fig:y-real}) for two different values of $y$ on either side of the eyebrow $\pEInftyRed$.  
\begin{figure}[h]
\begin{center}
\includegraphics{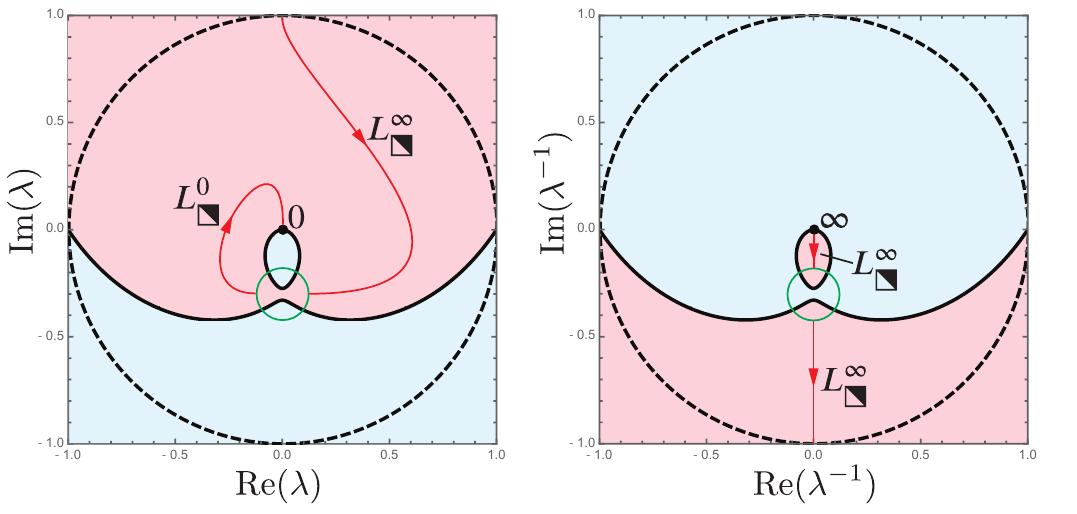}
\end{center}
\caption{The jump contour for the comparison matrices $\mathbf{F}_n^{(\alpha_1,\alpha_2,j)}(\lambda;y)$, and also for the final error matrix $\mathbf{E}_n^{(\alpha_1,\alpha_2,j)}(\lambda;y)$, for $y=0.33$, a point just to the left of the eyebrow $\pEInftyRed$. The jump contour consists of the arcs of $L=\LInftyRed\cup\LZeroRed$ outside the disks $D_1$ and $D_2$ (red), as well as the boundaries of the latter disks (green) that are oriented in the clockwise direction for the purposes of defining the boundary values taken thereon.  The background is a contour plot of $\mathrm{Re}(V(\lambda;y))$, with pink shading for $\mathrm{Re}(V(\lambda;y))>0$ and blue shading for $\mathrm{Re}(V(\lambda;y))<0$.}
\label{fig:EdgeError-within}
\end{figure}
\begin{figure}[h]
\begin{center}
\includegraphics{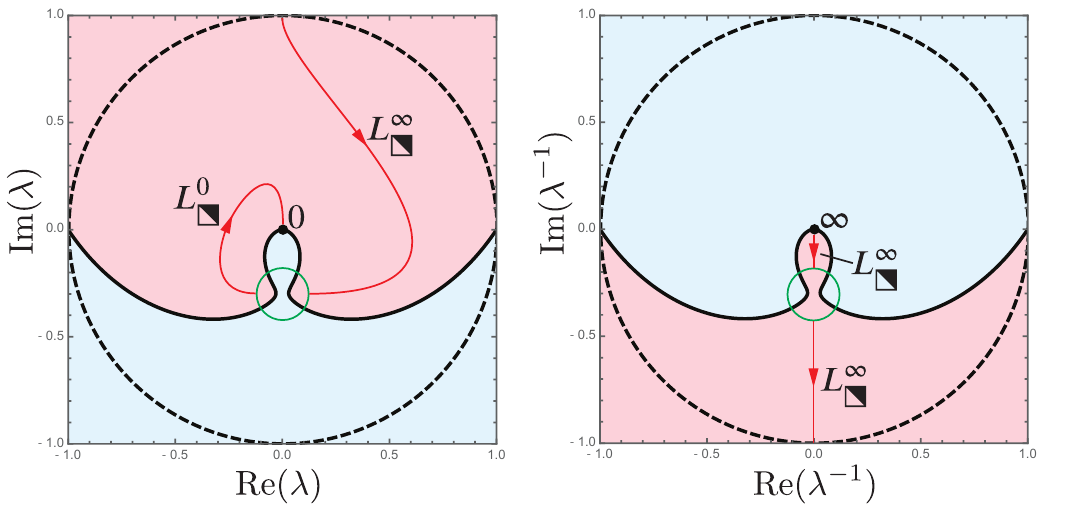}
\end{center}
\caption{As in Figure~\ref{fig:EdgeError-within}, but for $y=0.333$, a point just to the right of the eyebrow $\pEInftyRed$.}  
\label{fig:EdgeError-without}
\end{figure}
Observe that the landscape of $\mathrm{Re}(V(\lambda;y))$ shown in these plots resembles, at least for $\lambda$ not too close to $\lambda_1$ or $\lambda_2$, that illustrated in the central panels of Figure~\ref{fig:y-real}.  Therefore, when $y$ is close to the eyebrow $\pEInftyRed$, if the tubular neighborhood $T$ is taken to be sufficiently thin (by choosing $\delta_2$ sufficiently small in \eqref{eq:T-define}) given the domains $D_1$ and $D_2$, the jump condition satisfied by $\mathbf{F}_n^{(\alpha_1,\alpha_2,j)}(\lambda;y)$ on the arcs of $L$ exterior to the latter domains has the form (because all red contours lie strictly within the pink-shaded region)
\begin{equation}
\mathbf{F}^{(\alpha_1,\alpha_2,j)}_{n+}(\lambda;y)= \mathbf{F}^{(\alpha_1,\alpha_2,j)}_{n-}(\lambda;y)(\mathbb{I}+\text{exponentially small}),\quad n\to+\infty,\quad \lambda\in L\setminus(D_1\cup D_2),
\end{equation}
with the convergence being in the $L^p$ sense for every $p$ and holding uniformly for $y\in T$.  Therefore, the essential jump conditions for $\mathbf{F}_n^{(\alpha_1,\alpha_2,j)}(\lambda;y)$ are those across the domain boundaries $\partial D_1$ and $\partial D_2$.  Taking these to be oriented in the clockwise sense, using \eqref{eq:red-eyebrow-disk-boundaries}--\eqref{eq:red-eyebrow-disk-boundaries-alpha-0} gives
\begin{equation}
\mathbf{F}^{(\alpha_1,\alpha_2,j)}_{n+}(\lambda;y)=\mathbf{F}^{(\alpha_1,\alpha_2,j)}_{n-}(\lambda;y)
\begin{bmatrix}1+\mathcal{O}(n^{-1}) & A_j^{(\alpha_1,\alpha_2)}(\lambda;y)+\mathcal{O}(n^{-1})\\ \mathcal{O}(n^{-1}) & 1+\mathcal{O}(n^{-1})
\end{bmatrix},\quad\lambda\in\partial D_j
\end{equation}
(in the special case that $\alpha_j=0$ the $\mathcal{O}(n^{-1})$ error terms in all but the $12$ entry of the jump matrix vanish identically), while
\begin{multline}
\mathbf{F}^{(\alpha_1,\alpha_2,j)}_{n+}(\lambda;y)=\\
\mathbf{F}^{(\alpha_1,\alpha_2,j)}_{n-}(\lambda;y)
\begin{bmatrix}1+\mathcal{O}(n^{-1}) & a_n^{(\alpha_1,\alpha_2,j)}(y)[A_{3-j}^{(\alpha_1,\alpha_2)}(\lambda;y)+\mathcal{O}(n^{-1})]\\
b_n^{(\alpha_1,\alpha_2,j)}(y)[B_{3-j}^{(\alpha_1,\alpha_2)}(\lambda;y)+\mathcal{O}(n^{-1})] & 1+\mathcal{O}(n^{-1})\end{bmatrix},\\
\lambda\in\partial D_{3-j}
\label{eq:Edge-F-jump-other-boundary}
\end{multline}
where
\begin{equation}
a_n^{(\alpha_1,\alpha_2,j)}(y):=\ee^{n(V(\lambda_j(y);y)-V(\lambda_{3-j}(y);y))}n^{\alpha_j-\alpha_{3-j}}
\quad\text{and}\quad
b_n^{(\alpha_1,\alpha_2,j)}(y):=\ee^{n(V(\lambda_{3-j}(y);y)-V(\lambda_j(y);y))}n^{\alpha_{3-j}-\alpha_j-1}
\end{equation}
(in the special case that $\alpha_{3-j}=0$ the $O(n^{-1})$ error terms in all but the $12$ entry of the jump matrix vanish identically, and in addition $B_{3-j}^{(\alpha_1,\alpha_2)}(\lambda;y)\equiv 0$ because by convention $h_{-1}=0$).  Recalling that $\mathrm{Re}(V(\lambda_2(y);y))+\mathrm{Re}(V(\lambda_1(y);y))=0$ and $\alpha_1+\alpha_2=k$, upon suitable conditions on $y\in T$ the jump condition \eqref{eq:Edge-F-jump-other-boundary} reduces to one of the following forms:  
\begin{itemize}
\item Case I:  If $\alpha_{3-j}=0$ (so also $\alpha_j=k$) and the inequality
\begin{equation}
\mathrm{Re}(V(\lambda_j(y);y))\le-\frac{1}{2}\alpha_j\frac{\ln(n)}{n}=-\frac{1}{2}k\frac{\ln(n)}{n}
\label{eq:edge-inequality-3}
\end{equation}
holds, then \eqref{eq:Edge-F-jump-other-boundary} becomes
\begin{equation}
\mathbf{F}^{(\alpha_1,\alpha_2,j)}_{n+}(\lambda;y)=\mathbf{F}^{(\alpha_1,\alpha_2,j)}_{n-}(\lambda;y)
\begin{bmatrix}1 & a_n^{(\alpha_1,\alpha_2,j)}(y)[A_{3-j}^{(\alpha_1,\alpha_2)}(\lambda;y)+\mathcal{O}(n^{-1})]\\0 & 1\end{bmatrix},\quad\lambda\in\partial D_{3-j}
\label{eq:edge-jump-reduce-3}
\end{equation}
in which $a_n^{(\alpha_1,\alpha_2,j)}(y)=\mathcal{O}(1)$ as $n\to+\infty$.
\item Case II$_\mathrm{a}$:  If $\alpha_{3-j}>0$ (so also $\alpha_j<k$) and the inequalities
\begin{equation}
\frac{1}{2}(k-2\alpha_j-\tfrac{1}{2})\frac{\ln(n)}{n}\le\mathrm{Re}(V(\lambda_j(y);y))\le\frac{1}{2}(k-2\alpha_j)\frac{\ln(n)}{n}
\label{eq:edge-inequality-1}
\end{equation}
hold, then \eqref{eq:Edge-F-jump-other-boundary} becomes
\begin{equation}
\mathbf{F}_{n+}^{(\alpha_1,\alpha_2,j)}(\lambda;y)=\mathbf{F}_{n-}^{(\alpha_1,\alpha_2,j)}(\lambda;y)
\begin{bmatrix}1+\mathcal{O}(n^{-1}) & a_n^{(\alpha_1,\alpha_2,j)}(y)[A_{3-j}^{(\alpha_1,\alpha_2)}(\lambda;y)+\mathcal{O}(n^{-1})]\\ \mathcal{O}(n^{-1/2}) & 1+\mathcal{O}(n^{-1})\end{bmatrix},\quad\lambda\in\partial D_{3-j}
\label{eq:edge-jump-reduce-1}
\end{equation}
in which $a_n^{(\alpha_1,\alpha_2,j)}(y)=\mathcal{O}(1)$ as $n\to+\infty$.
\item Case II$_\mathrm{b}$:  If $\alpha_{3-j}>0$ (so also $\alpha_j<k$) and the inequalities 
\begin{equation}
\frac{1}{2}(k-2\alpha_j-1)\frac{\ln(n)}{n}\le\mathrm{Re}(V(\lambda_j(y);y))\le \frac{1}{2}(k-2\alpha_j-\tfrac{1}{2})\frac{\ln(n)}{n}
\label{eq:edge-inequality-2}
\end{equation}
hold, then \eqref{eq:Edge-F-jump-other-boundary} becomes
\begin{equation}
\mathbf{F}^{(\alpha_1,\alpha_2,j)}_{n+}(\lambda;y)=\mathbf{F}^{(\alpha_1,\alpha_2,j)}_{n-}(\lambda;y)
\begin{bmatrix}1+\mathcal{O}(n^{-1}) & \mathcal{O}(n^{-1/2})\\ b_n^{(\alpha_1,\alpha_2,j)}(y)[B_{3-j}^{(\alpha_1,\alpha_2)}(\lambda;y)+\mathcal{O}(n^{-1})] & 1+\mathcal{O}(n^{-1})\end{bmatrix},\quad\lambda\in\partial D_{3-j}
\label{eq:edge-jump-reduce-2}
\end{equation}
in which $b_n^{(\alpha_1,\alpha_2,j)}(y)=\mathcal{O}(1)$ as $n\to+\infty$.
\end{itemize}
By varying the index $j=1,2$ as well as the choice of non-negative integers $\alpha_1+\alpha_2=k$, the above inequalities \eqref{eq:edge-inequality-1}, \eqref{eq:edge-inequality-2}, and \eqref{eq:edge-inequality-3} on $y\in T$ actually cover the whole tubular neighborhood $T$.  Indeed, given $k\ge 0$, begin by taking $\alpha_1=k$ and $\alpha_2=0$, and consider the comparison matrix $\mathbf{F}_n^{(k,0,1)}(\lambda;y)$.  Assuming that $\mathrm{Re}(V(\lambda_1(y);y))\le -\tfrac{1}{2}kn^{-1}\ln(n)$, the inequality \eqref{eq:edge-inequality-3} of Case I guarantees that \eqref{eq:edge-jump-reduce-3} governs the jump condition on $\partial D_2$.  Then, for $\ell=1,\dots,k$,
\begin{itemize}
\item Take $\alpha_1=k-\ell+1>0$ and $\alpha_2=\ell-1<k$, and consider the comparison matrix $\mathbf{F}_n^{(k-\ell+1,\ell-1,2)}(\lambda;y)$.  Assuming that $-\tfrac{1}{2}(k-2\ell+2)n^{-1}\ln(n)\le\mathrm{Re}(V(\lambda_1(y);y))\le -\tfrac{1}{2}(k-2\ell+\tfrac{3}{2})n^{-1}\ln(n)$, the inequalities \eqref{eq:edge-inequality-1} of Case II$_\mathrm{a}$ imply that \eqref{eq:edge-jump-reduce-1} governs the jump condition on $\partial D_1$.
Assuming that $-\tfrac{1}{2}(k-2\ell+\tfrac{3}{2})n^{-1}\ln(n)\le\mathrm{Re}(V(\lambda_1(y);y))\le -\tfrac{1}{2}(k-2\ell+1)$, the inequalities \eqref{eq:edge-inequality-2} of Case II$_\mathrm{b}$ imply that \eqref{eq:edge-jump-reduce-2} governs the jump condition on $\partial D_1$.
\item Now take $\alpha_1=k-\ell<k$ and $\alpha_2=\ell>0$, and consider the comparison matrix $\mathbf{F}_n^{(k-\ell,\ell,1)}(\lambda;y)$.  Assuming that $-\tfrac{1}{2}(k-2\ell+1)n^{-1}\ln(n)\le\mathrm{Re}(V(\lambda_1(y);y))\le -\tfrac{1}{2}(k-2\ell+\tfrac{1}{2})n^{-1}\ln(n)$, the inequalities \eqref{eq:edge-inequality-2} of Case II$_\mathrm{b}$ imply that \eqref{eq:edge-jump-reduce-2} governs the jump condition on $\partial D_2$.
Assuming that $-\tfrac{1}{2}(k-2\ell+\tfrac{1}{2})n^{-1}\ln(n)\le\mathrm{Re}(V(\lambda_1(y);y))\le-\tfrac{1}{2}(k-2\ell)n^{-1}\ln(n)$, the inequalities \eqref{eq:edge-inequality-1} of Case II$_\mathrm{a}$ imply that \eqref{eq:edge-jump-reduce-1} governs the jump condition on $\partial D_2$.
\end{itemize}
Finally, take $\alpha_1=0$ and $\alpha_2=k$, and consider the comparison matrix $\mathbf{F}_n^{(0,k,2)}(\lambda;y)$.  Assuming that $\mathrm{Re}(V(\lambda_1(y);y))\ge \tfrac{1}{2}kn^{-1}\ln(n)$, the inequality \eqref{eq:edge-inequality-3} of Case I then guarantees that \eqref{eq:edge-jump-reduce-3} governs the jump condition on $\partial D_1$.

\subsubsection{Modeling of comparison matrices}
To determine the asymptotic behavior as $n\to+\infty$ of the various types of comparison matrices, it now becomes necessary to model the leading terms of the jump matrices, which generally do not decay to the identity on the domain boundaries $\partial D_1$ and $\partial D_2$, but that are guaranteed to be bounded by association of $y\in T$ with the appropriate indices $\alpha_1$, $\alpha_2$, and $j$ as described above.  In Cases I and II$_\mathrm{a}$, the dominant terms in the jump matrices on $\partial D_1$ and $\partial D_2$ are both upper triangular matrices, while in Case II$_\mathrm{b}$ one jump matrix is upper triangular and the other is lower triangular.  This situation requires two different types of parametrices, which we formulate as Riemann-Hilbert problems here.

\begin{rhp}[Upper-upper; Cases I and II$_\mathrm{a}$]
Let distinct points $\lambda_1\neq\lambda_2$, $\lambda_j\neq 0$, $j=1,2$, be given with corresponding simply-connected neighborhoods $D_1$ and $D_2$ with $D_1\cap D_2=\emptyset$.  For $j=1,2$, let $\phi_j$ be meromorphic on $D_j$ and continuous up to the boundary $\partial D_j$ with a simple pole at $\lambda_j$ as the only singularity in $D_j$.  Seek a $2\times 2$ matrix function $\lambda\mapsto\dot{\mathbf{F}}(\lambda)$ with the following properties:
\begin{itemize}
\item[1.] \textbf{Analyticity:}  $\lambda\mapsto\dot{\mathbf{F}}(\lambda)$ is analytic for $\lambda\in\mathbb{C}\setminus (\partial D_1\cup\partial D_2)$ and takes continuous boundary values from each side on $\partial D_1$ and $\partial D_2$.  
\item[2.] \textbf{Jump conditions:}  The boundary values are related by the following jump conditions.  Assuming clockwise orientation of $\partial D_j$, $j=1,2$, 
\begin{equation}
\dot{\mathbf{F}}_+(\lambda)=\dot{\mathbf{F}}_-(\lambda)\begin{bmatrix}1 & \phi_j(\lambda)\\
0 & 1\end{bmatrix},\quad\lambda\in \partial D_j,\quad j=1,2.
\end{equation}
\item[3.] \textbf{Asymptotics:}  $\dot{\mathbf{F}}(\lambda)\to\mathbb{I}$ as $\lambda\to\infty$.
\end{itemize}
\label{rhp:upper-upper}
\end{rhp}
This problem always has a unique solution, which may be sought in the form 
\begin{equation}
\dot{\mathbf{F}}(\lambda)=\begin{bmatrix}1 & \dot{f}(\lambda)\\0 & 1\end{bmatrix}.
\end{equation}
The conditions of Riemann-Hilbert Problem~\ref{rhp:upper-upper} descend to the conditions that the scalar function $\dot{f}(\lambda)$ be analytic for $\lambda\in\mathbb{C}\setminus(\partial D_1\cup\partial D_2)$ with $\dot{f}(\lambda)\to 0$ as $\lambda\to\infty$, and the jump conditions now take the additive form:  $\dot{f}_+(\lambda)=\dot{f}_-(\lambda)+\phi_j(\lambda)$ holds for $\lambda\in\partial D_j$, $j=1,2$.  It follows that $\dot{f}(\lambda)$ is given by the Plemelj formula
\begin{equation}
\dot{f}(\lambda)=\frac{1}{2\pi\ii}\oint_{\partial D_1}\frac{\phi_1(\mu)\,\dd\mu}{\mu-\lambda} +
\frac{1}{2\pi\ii}\oint_{\partial D_2}\frac{\phi_2(\mu)\,\dd\mu}{\mu-\lambda},\quad\lambda\in\mathbb{C}\setminus(\partial D_1\cup\partial D_2).
\end{equation}
The integrals may be evaluated explicitly, by residues.  In the special case that $\lambda$ is exterior to both domains $D_j$, $j=1,2$, we therefore find
\begin{equation}
\dot{\mathbf{F}}(\lambda)=\begin{bmatrix}1 & \displaystyle\frac{\Phi_1}{\lambda-\lambda_1} +\frac{\Phi_2}{\lambda-\lambda_2}\\0 & 1\end{bmatrix},\quad\lambda\in\mathbb{C}\setminus\overline{D_1\cup D_2},
\end{equation}
where $\Phi_j$ denotes the residue of $\phi_j$ at $\lambda_j$, $j=1,2$.  In particular, we see that
\begin{equation}
\dot{\mathbf{F}}(\lambda)=\mathbb{I} + \lambda^{-1}\dot{\mathbf{F}}^\infty_1 + \mathcal{O}(\lambda^{-2}),\quad\lambda\to\infty\quad\text{and}\quad
\dot{\mathbf{F}}(\lambda)=\dot{\mathbf{F}}^0_0 + \mathcal{O}(\lambda),\quad\lambda\to 0,
\label{eq:dotF-expand}
\end{equation}
in which we have
\begin{equation}
\dot{F}^\infty_{1,12}=\Phi_1+\Phi_2,\quad\dot{F}^0_{0,11}=1,\quad\dot{F}^0_{0,12}=-\frac{\Phi_1}{\lambda_1}-\frac{\Phi_2}{\lambda_2}.
\label{eq:dot-F-elements-upper-upper}
\end{equation}

\begin{rhp}[Upper-lower; Case II$_\mathrm{b}$]
Let distinct nonzero points $\lambda_\mathrm{U}\neq\lambda_\mathrm{L}$ be given with corresponding simply-connected neighborhoods $D_\mathrm{U}$ and $D_\mathrm{L}$ with $D_\mathrm{U}\cap D_\mathrm{L}=\emptyset$.
Let $\phi_\mathrm{U}$ and $\phi_\mathrm{L}$ be meromorphic on $D_\mathrm{U}$ and $D_\mathrm{L}$ respectively and continuous up to the corresponding boundary curves, with simple poles only at $\lambda_\mathrm{U}$ and $\lambda_\mathrm{L}$ respectively.  
Seek a $2\times 2$ matrix function $\lambda\mapsto\dot{\mathbf{F}}(\lambda)$ with the following properties:
\begin{itemize}
\item[1.]\textbf{Analyticity:} $\lambda\mapsto\dot{\mathbf{F}}(\lambda)$ is analytic for $\lambda\in\mathbb{C}\setminus(\partial D_\mathrm{U}\cup\partial D_\mathrm{L})$ and takes continuous boundary values from each side on $\partial D_\mathrm{U}$ and $\partial D_\mathrm{L}$.
\item[2.]\textbf{Jump conditions:}  The boundary values are related by the following jump conditions.  Assuming clockwise orientation of $\partial D_\mathrm{U}$,
\begin{equation}
\dot{\mathbf{F}}_+(\lambda)=\dot{\mathbf{F}}_-(\lambda)\begin{bmatrix}1 & \phi_\mathrm{U}(\lambda)\\0 & 1\end{bmatrix},\quad\lambda\in\partial D_\mathrm{U},
\label{eq:dot-F-jump-U}
\end{equation}
and assuming clockwise orientation of $\partial D_\mathrm{L}$,
\begin{equation}
\dot{\mathbf{F}}_+(\lambda)=\dot{\mathbf{F}}_-(\lambda)\begin{bmatrix}1 & 0\\\phi_\mathrm{L}(\lambda) & 1\end{bmatrix},\quad\lambda\in\partial D_\mathrm{L}.
\end{equation}
\item[3.]\textbf{Asymptotics:}  $\dot{\mathbf{F}}(\lambda)\to\mathbb{I}$ as $\lambda\to\infty$.
\end{itemize}
\label{rhp:F-dot-eyebrow}
\end{rhp}
This problem is a generalization of the one that characterizes the soliton solutions of AKNS systems \cite{BealsC84}. Unlike Riemann-Hilbert Problem~\ref{rhp:upper-upper} this problem is only conditionally solvable.
Letting $\Phi_\mathrm{U}$ denote the residue of $\phi_\mathrm{U}$ at $\lambda_\mathrm{U}$, and $\Phi_\mathrm{L}$ that of $\phi_\mathrm{L}$ at $\lambda_\mathrm{L}$, this problem has a unique solution if and only if 
\begin{equation}
\Delta:=\Phi_\mathrm{U}\Phi_\mathrm{L}+(\lambda_\mathrm{U}-\lambda_\mathrm{L})^2\neq 0.
\label{eq:edge-denominator}
\end{equation}
The solution is a rational function in the domain exterior to $D_\mathrm{U}\cup D_\mathrm{L}$:
\begin{equation}
\dot{\mathbf{F}}(\lambda)=\mathbb{I}+\frac{1}{\lambda-\lambda_\mathrm{U}}\frac{(\lambda_\mathrm{U}-\lambda_\mathrm{L})\Phi_\mathrm{U}}{\Delta}\begin{bmatrix}
0 & \lambda_\mathrm{U}-\lambda_\mathrm{L}\\0 & \Phi_\mathrm{L}\end{bmatrix}
+\frac{1}{\lambda-\lambda_\mathrm{L}}\frac{(\lambda_\mathrm{L}-\lambda_\mathrm{U})\Phi_\mathrm{L}}{\Delta}
\begin{bmatrix}\Phi_\mathrm{U} & 0\\\lambda_\mathrm{L}-\lambda_\mathrm{U} & 0\end{bmatrix},\quad
\lambda\in\mathbb{C}\setminus\overline{D_\mathrm{U}\cup D_\mathrm{L}}.
\label{eq:dotF-edge-outside}
\end{equation}
This formula determines $\dot{\mathbf{F}}(\lambda)$ in the domains $D_\mathrm{U}$ and $D_\mathrm{L}$ by the jump conditions; Laurent expansion of the interior boundary value $\dot{\mathbf{F}}_-(\lambda)$ shows in each case that its only singularity is removable.  Moreover, \eqref{eq:dotF-edge-outside} shows that expansions of the form \eqref{eq:dotF-expand} again hold whenever the solution exists, in which
\begin{equation}
\dot{F}^\infty_{1,12}= \frac{(\lambda_\mathrm{U}-\lambda_\mathrm{L})^2\Phi_\mathrm{U}}{\Delta},\quad
\dot{F}^0_{0,11}=1-\frac{(\lambda_\mathrm{L}-\lambda_\mathrm{U})\Phi_\mathrm{L}\Phi_\mathrm{U}}{\lambda_\mathrm{L}\Delta},\quad
\text{and}\quad
\dot{F}^0_{0,12}=-\frac{(\lambda_\mathrm{U}-\lambda_\mathrm{L})^2\Phi_\mathrm{U}}{\lambda_\mathrm{U}\Delta}.
\label{eq:dot-F-elements-upper-lower}
\end{equation}

\subsubsection{Final error analysis and asymptotic formul\ae\ for $u_n(ny;-(\tfrac{1}{2}+k))$}
\label{sec:Edge-error-analysis}
Suppose that $y$ is such that, after proper association of the data of Riemann-Hilbert Problem~\ref{rhp:upper-upper} or \ref{rhp:F-dot-eyebrow} with the leading terms of the jump matrices for the comparison matrix $\mathbf{F}_n^{(\alpha_1,\alpha_2,j)}(\lambda;y)$, $\dot{\mathbf{F}}(\lambda)$ exists and is bounded as $n\to+\infty$ (no condition in the case of Riemann-Hilbert Problem~\ref{rhp:upper-upper}).  Then it is easy to check that the error matrix $\mathbf{E}_n^{(\alpha_1,\alpha_2,j)}(\lambda;y):=\mathbf{F}_n^{(\alpha_1,\alpha_2,j)}(\lambda;y)\dot{\mathbf{F}}(\lambda)^{-1}$ satisfies the conditions of a small-norm Riemann-Hilbert problem formulated relative to a jump contour such as shown in Figures~\ref{fig:EdgeError-within} and \ref{fig:EdgeError-without}, with the result that 
\begin{equation}
\mathbf{E}_n^{(\alpha_1,\alpha_2,j)}(\lambda;y)=\mathbb{I}+\lambda^{-1}\mathbf{E}^\infty_{n,1}(y) + \mathcal{O}(\lambda^{-2}),\quad\lambda\to\infty\quad\text{and}\quad \mathbf{E}_n(\lambda;y)=\mathbf{E}^0_{n,0}(y) + \mathcal{O}(\lambda),\quad\lambda\to 0,
\end{equation}
where $\mathbf{E}^\infty_{n,1}(y)=\mathcal{O}(n^{-1/2})$ and $\mathbf{E}^0_{n,0}(y)=\mathbb{I}+\mathcal{O}(n^{-1/2})$ as $n\to+\infty$ uniformly for $y\in T$ (if Case I holds, we may replace $\mathcal{O}(n^{-1/2})$ in both estimates with $\mathcal{O}(n^{-1})$).  Note that for $\lambda$ in the exterior of the domain $D_1\cup D_2$, we have the exact identity
\begin{equation}
\begin{split}
\mathbf{Y}_n(\lambda;ny,-(\tfrac{1}{2}+k))&=\mathbf{M}_n^{(k)}(\lambda;y)\\
&=\delta_n(y)^{\sigma_3}\mathbf{F}_n^{(\alpha_1,\alpha_2,j)}(\lambda;y)\delta_n(y)^{-\sigma_3}\dot{\mathbf{M}}^{\mathrm{out},(\alpha_1,\alpha_2)}(\lambda;y)\\&=\delta_n(y)^{\sigma_3}\mathbf{F}_n^{(\alpha_1,\alpha_2,j)}(\lambda;y)\dot{\mathbf{M}}^{\mathrm{out},(\alpha_1,\alpha_2)}(\lambda;y)\delta_n(y)^{-\sigma_3}\\
&=\delta_n(y)^{\sigma_3}\mathbf{E}_n^{(\alpha_1,\alpha_2,j)}(\lambda;y)\dot{\mathbf{F}}(\lambda)\dot{\mathbf{M}}^{\mathrm{out},(\alpha_1,\alpha_2)}(\lambda;y)\delta_n(y)^{-\sigma_3},\quad\lambda\in\mathbb{C}\setminus\overline{D_1\cup D_2},
\end{split}
\end{equation}
where $\delta_n(y)\neq 0$ is independent of $\lambda$ and takes a different form in different parts of the tubular neighborhood $T$ containing $y$, the hypothesis on $\lambda$ ensures that $\dot{\mathbf{M}}_n^{(\alpha_1,\alpha_2)}(\lambda;y)=\dot{\mathbf{M}}^{\mathrm{out},(\alpha_1,\alpha_2)}(\lambda;y)$, and we used the fact that the outer parametrix is diagonal.
Consequently, from \eqref{eq:u-n-from-Y-formula}, we arrive at the following approximate formula for $u_n(ny;-(\tfrac{1}{2}+k))$ valid for large $n$:
\begin{equation}
u_n(ny;-(\tfrac{1}{2}+k))=\frac{-\ii\dot{F}^\infty_{1,12}+\mathcal{O}(n^{-1/2})}{\dot{F}^0_{0,11}\dot{F}^0_{0,12}+\mathcal{O}(n^{-1/2})},\quad n\to+\infty,
\end{equation}
which holds uniformly for $y\in T$ for which $\dot{\mathbf{F}}(\lambda)$ exists and is bounded (in Case II$_\mathrm{b}$ the denominator $\Delta$ should be bounded away from zero).  In Case I, the error terms can be replaced with $\mathcal{O}(n^{-1})$.  Therefore, if $y\in T$ is such that Case I or II$_\mathrm{a}$ holds, from \eqref{eq:dot-F-elements-upper-upper} we have
\begin{equation}
u_n(ny;-(\tfrac{1}{2}+k))=\ii\frac{\lambda_1\lambda_2(\Phi_1+\Phi_2)+\mathcal{O}(n^{-1})}{\lambda_2\Phi_1+\lambda_1\Phi_2 + \mathcal{O}(n^{-1})},\quad n\to+\infty,\quad \text{if Case I holds,}
\label{eq:u-upper-upper-CaseI}
\end{equation}
\begin{equation}
u_n(ny;-(\tfrac{1}{2}+k))=\ii\frac{\lambda_1\lambda_2(\Phi_1+\Phi_2)+\mathcal{O}(n^{-1/2})}{\lambda_2\Phi_1+\lambda_1\Phi_2 + \mathcal{O}(n^{-1/2})},\quad n\to+\infty,\quad \text{if Case II$_\mathrm{a}$ holds,}
\label{eq:u-upper-upper}
\end{equation}
while if instead Case II$_\mathrm{b}$ holds and $\Delta$ is bounded away from zero, from \eqref{eq:dot-F-elements-upper-lower} and \eqref{eq:edge-denominator} we have, for any given $\delta>0$ independent of $n$,
\begin{equation}
u_n(ny;-(\tfrac{1}{2}+k))=\ii\frac{\lambda_\mathrm{U}\lambda_\mathrm{L}((\lambda_\mathrm{L}-\lambda_\mathrm{U})^2+\Phi_\mathrm{U}\Phi_\mathrm{L})+\mathcal{O}(n^{-1/2})}{\lambda_\mathrm{L}(\lambda_\mathrm{L}-\lambda_\mathrm{U})^2+\lambda_\mathrm{U}\Phi_\mathrm{U}\Phi_\mathrm{L}+ \mathcal{O}(n^{-1/2})},\quad n\to+\infty,\quad\text{if Case II$_\mathrm{b}$ holds and $|\Delta|\geq \delta>0$.}
\label{eq:u-upper-lower}
\end{equation}
Due to the inequalities that characterize Cases I, II$_\mathrm{a}$, and II$_\mathrm{b}$, the leading terms in the numerator and denominator are bounded as $n\to+\infty$ in each case.  If in addition the leading terms in the denominator are bounded away from zero, one can extract a leading approximation $\dot{u}_n$ of $u_n(ny;-(\tfrac{1}{2}+k))$ with a small absolute error:
\begin{equation}
u_n(ny;-(\tfrac{1}{2}+k))=\dot{u}_n+\begin{cases}\mathcal{O}(n^{-1}),\quad& \text{(Case I)}\\
\mathcal{O}(n^{-1/2}),\quad&\text{(Case II$_\mathrm{a}$)},\end{cases}\quad
\dot{u}_n:=\ii\frac{\Phi_1+\Phi_2}{\lambda_2\Phi_1+\lambda_1\Phi_2},
\label{eq:dot-u-CaseI-IIa}
\end{equation}
if $|\lambda_2\Phi_1+\lambda_1\Phi_2|\ge\delta>0$
and for Case II$_\mathrm{b}$,
\begin{equation}
u_n(ny;-(\tfrac{1}{2}+k))=\dot{u}_n+\mathcal{O}(n^{-1/2}),\quad \dot{u}_n:=
\ii\frac{(\lambda_\mathrm{L}-\lambda_{\mathrm{U}})^2+\Phi_\mathrm{U}\Phi_\mathrm{L}}{\lambda_\mathrm{L}(\lambda_\mathrm{L}-\lambda_\mathrm{U})^2+\lambda_\mathrm{U}\Phi_\mathrm{U}\Phi_\mathrm{L}},
\label{eq:dot-u-CaseIIb}
\end{equation}
if $|\Delta|\ge\delta>0$ and also $|\lambda_\mathrm{L}(\lambda_\mathrm{L}-\lambda_\mathrm{U})^2+\lambda_\mathrm{U}\Phi_\mathrm{U}\Phi_\mathrm{L}|\ge\delta>0$.  Note that in this case the zeros of $\Delta$ (where Riemann-Hilbert Problem~\ref{rhp:F-dot-eyebrow} fails to be solvable) correspond to zeros of the approximation $\dot{u}_n$.  In deriving the above formul\ae\ for $\dot{u}_n$ we used the fact that $\lambda_1\lambda_2=\lambda_\mathrm{U}\lambda_\mathrm{L}=1$.

\subsubsection{Concrete formul\ae\ for $\dot{u}_n$ in domains covering the tubular neighborhood $T$}
\label{sec:edge-formulae}
Now recall that $\lambda_1(y)=\lNaught(y)^{-1}$ and $\lambda_2(y)=\lNaught(y)$.  We construct $\dot{u}_n$ for each $y\in T$ according to the scheme described at the end of Section~\ref{sec:edge-comparison-matrices}.\smallskip

Suppose first that $\mathrm{Re}(V(\lNaught(y)^{-1};y))\le-\tfrac{1}{2}kn^{-1}\ln(n)$.  Then we are to consider the comparison matrix $\mathbf{F}_n^{(k,0,1)}(\lambda;y)$ which corresponds to Case I and hence the formula \eqref{eq:dot-u-CaseI-IIa} with 
\begin{equation}
\Phi_1:=\mathop{\mathrm{Res}}_{\lambda=\lNaught(y)^{-1}}A_1^{(k,0)}(\lambda;y) = \frac{\lNaught(y)^k\lNaught(y)_\infty^{1/2}V''(\lNaught(y)^{-1};y)^{-k}}{\ii V''(\lNaught(y)^{-1};y)^{1/2}}
\end{equation}
and
\begin{equation}
\Phi_2:=a_n^{(k,0,1)}(y)\mathop{\mathrm{Res}}_{\lambda=\lNaught(y)}A_2^{(k,0)}(\lambda;y)=
\ee^{2nV(\lNaught(y)^{-1};y)}n^k\frac{\lNaught(y)^{-k}\lNaught(y)_\infty^{-1/2}(\lNaught(y)^{-1}-\lNaught(y))^{2k}}{\ii k! V''(\lNaught(y);y)^{1/2}}
\end{equation}
where we used the identity $V(\lNaught(y);y)=-V(\lNaught(y)^{-1};y)\pmod{2\pi\ii}$.  Since only the ratio of the residues $\Phi_1$ and $\Phi_2$ enters into the formula \eqref{eq:dot-u-CaseI-IIa}, it is possible to remove all ambiguity of branches of square roots as follows.  It is straightforward to check that whenever $\lambda$ is such that $V'(\lambda;y)=0$ we have $V''(\lambda;y)=\lambda^{-2}(\lambda^2-1)(\lambda^2+1)^{-1}$ and therefore the identity $V''(\lNaught(y)^{-1};y)=-\lNaught(y)^4V''(\lNaught(y);y)$ holds for $y\in T$.  Taking square roots and carefully determining the sign to choose for consistency when $y=0.331372\in\pEInftyRed$, this identity implies that $V''(\lNaught(y)^{-1};y)^{1/2}=\ii\lNaught(y)^2V''(\lNaught(y);y)^{1/2}$.  Furthermore, $\lNaught(y)_\infty^{1/2}/\lNaught(y)_\infty^{-1/2}=\lNaught(y)$, so all details of the ``$\infty$'' branch of the power functions disappears from the ratio of residues.  Using these facts, we arrive at a simple formula for $\dot{u}_n$ valid for $\mathrm{Re}(V(\lNaught(y)^{-1};y))\le -\tfrac{1}{2}kn^{-1}\ln(n)$:
\begin{equation}
\dot{u}_n=\ii\lNaught(y)\frac{\ee^{2nV(\lNaught(y)^{-1};y)}-\ii n^{-k}k!\lNaught(y)^{-1}(\lNaught(y)^{-1}+\lNaught(y))^k(\lNaught(y)^{-1}-\lNaught(y))^{-3k}}{\ee^{2nV(\lNaught(y)^{-1};y)}-\ii n^{-k}k!\lNaught(y)(\lNaught(y)^{-1}+\lNaught(y))^k(\lNaught(y)^{-1}-\lNaught(y))^{-3k}},\quad
\mathrm{Re}(V(\lNaught(y)^{-1};y))\le-\frac{1}{2}k\frac{\ln(n)}{n}.
\label{eq:dot-u-UU-left}
\end{equation}
We observe that as $y$ moves away from $\pEInftyRed$ into the interior of $E$, the exponentials tend to zero and $\dot{u}_n\approx\ii\lNaught(y)^{-1}$ consistent with Theorem~\ref{theorem:closed-eye-equilibrium}.\smallskip

Now let $\ell$ be an integer varying from $\ell=1$ to $\ell=k$; we must now analyze four corresponding sub-cases depending on $y\in T$.  First we are to consider the comparison matrix $\mathbf{F}_n^{(k-\ell+1,\ell-1,2)}(\lambda;y)$.  Assuming the inequalities $-\tfrac{1}{2}(k-2\ell+2)n^{-1}\ln(n)\le\mathrm{Re}(V(\lNaught(y)^{-1};y))\le-\tfrac{1}{2}(k-2\ell+\tfrac{3}{2})n^{-1}\ln(n)$, we are in Case II$_\mathrm{a}$ with 
\begin{equation}
\begin{split}
\Phi_1&=a_n^{(k-\ell+1,\ell-1,2)}(y)\mathop{\mathrm{Res}}_{\lambda=\lNaught(y)^{-1}}A_1^{(k-\ell+1,\ell-1)}(\lambda;y) \\ &= \ee^{-2nV(\lNaught(y)^{-1};y)}n^{2\ell-k-2}
\frac{(k-\ell+1)!\lNaught(y)^k\lNaught(y)_\infty^{1/2}(\lNaught(y)^{-1}-\lNaught(y))^{2\ell-2}V''(\lNaught(y)^{-1};y)^{\ell-k-1}}{\ii k!V''(\lNaught(y)^{-1};y)^{1/2}}
\end{split}
\end{equation}
and
\begin{equation}
\Phi_2=\mathop{\mathrm{Res}}_{\lambda=\lNaught(y)}A_2^{(k-\ell+1,\ell-1)}(\lambda;y) =
\frac{(\ell-1)!\lNaught(y)^{-k}\lNaught(y)_\infty^{-1/2}(\lNaught(y)^{-1}-\lNaught(y))^{2k-2\ell+2}
V''(\lNaught(y);y)^{-\ell+1}}{\ii k!V''(\lNaught(y);y)^{1/2}}.
\end{equation}
Applying similar arguments to express the ratio of residues appearing in \eqref{eq:dot-u-CaseI-IIa} in terms of integer powers of $\lNaught(y)$ gives
\begin{multline}
\dot{u}_n=\ii\lNaught(y)\frac{\ee^{2nV(\lNaught(y)^{-1};y)}-\ii(-1)^{\ell-1}n^{2\ell-k-2}\frac{(k-\ell+1)!}{(\ell-1)!}
\lNaught(y)^{-1}(\lNaught(y)^{-1}+\lNaught(y))^{k-2\ell+2}(\lNaught(y)^{-1}-\lNaught(y))^{6\ell-3k-6}}{\ee^{2nV(\lNaught(y)^{-1};y)}-\ii(-1)^{\ell-1}n^{2\ell-k-2}\frac{(k-\ell+1)!}{(\ell-1)!}
\lNaught(y)(\lNaught(y)^{-1}+\lNaught(y))^{k-2\ell+2}(\lNaught(y)^{-1}-\lNaught(y))^{6\ell-3k-6}},\\
-\frac{1}{2}(k-2\ell+2)\frac{\ln(n)}{n}\le\mathrm{Re}(V(\lNaught(y)^{-1};y))\le
-\frac{1}{2}(k-2\ell+\tfrac{3}{2})\frac{\ln(n)}{n}.
\label{eq:dot-u-UU-first}
\end{multline}
Comparing \eqref{eq:dot-u-UU-left} and \eqref{eq:dot-u-UU-first} in the case $\ell=1$ shows that the same formula for $\dot{u}_n$ holds over the whole range of values $\mathrm{Re}(V(\lNaught(y)^{-1};y))\le-\tfrac{1}{2}(k-\tfrac{1}{2})n^{-1}\ln(n)$ although different comparison matrices are involved in the derivation.
Continuing with studying the same comparison matrix but now assuming that $-\tfrac{1}{2}(k-2\ell+\tfrac{3}{2})n^{-1}\ln(n)\le\mathrm{Re}(V(\lNaught(y)^{-1};y))\le-\tfrac{1}{2}(k-2\ell+1)n^{-1}\ln(n)$, we are in Case II$_\mathrm{b}$ with $\lambda_\mathrm{U}=\lambda_2=\lNaught(y)$ and $\lambda_\mathrm{L}=\lambda_1=\lNaught(y)^{-1}$, with corresponding residues
\begin{equation}
\Phi_\mathrm{U}=\mathop{\mathrm{Res}}_{\lambda=\lNaught(y)}A_2^{(k-\ell+1,\ell-1)}(\lambda;y) =
\frac{(\ell-1)!\lNaught(y)^{-k}\lNaught(y)_\infty^{-1/2}(\lNaught(y)^{-1}-\lNaught(y))^{2k-2\ell+2}
V''(\lNaught(y);y)^{-\ell+1}}{\ii k!V''(\lNaught(y);y)^{1/2}},
\end{equation}
and
\begin{equation}
\begin{split}
\Phi_\mathrm{L}&=b_n^{(k-\ell+1,\ell-1,2)}\mathop{\mathrm{Res}}_{\lambda=\lNaught(y)^{-1}}B_1^{(k-\ell+1,\ell-1)}(\lambda;y)\\
&=\ee^{2nV(\lNaught(y)^{-1};y)}n^{k-2\ell+1}\frac{k!\lNaught(y)^{-k}\lNaught(y)_\infty^{-1/2}(\lNaught(y)^{-1}-\lNaught(y))^{2-2\ell}V''(\lNaught(y)^{-1};y)^{k-\ell+1}}{\ii (k-\ell)!V''(\lNaught(y)^{-1};y)^{1/2}}.
\end{split}
\end{equation}
After some simplification of the product $\Phi_\mathrm{U}\Phi_\mathrm{L}$ of the residues along the lines indicated above, the applicable formula \eqref{eq:dot-u-CaseIIb} for $\dot{u}_n$ becomes
\begin{multline}
\dot{u}_n=\ii\lNaught(y)^{-1}\frac{\ee^{2nV(\lNaught(y)^{-1};y)}-\ii(-1)^\ell n^{2\ell-k-1}\frac{(k-\ell)!}{(\ell-1)!}
\lNaught(y)(\lNaught(y)^{-1}+\lNaught(y))^{k-2\ell+1}(\lNaught(y)^{-1}-\lNaught(y))^{6\ell-3k-3}}
{\ee^{2nV(\lNaught(y)^{-1};y)}-\ii(-1)^\ell n^{2\ell-k-1}\frac{(k-\ell)!}{(\ell-1)!}
\lNaught(y)^{-1}(\lNaught(y)^{-1}+\lNaught(y))^{k-2\ell+1}(\lNaught(y)^{-1}-\lNaught(y))^{6\ell-3k-3}},\\
-\frac{1}{2}(k-2\ell+\tfrac{3}{2})\frac{\ln(n)}{n}\le\mathrm{Re}(V(\lNaught(y)^{-1};y))\le -\frac{1}{2}(k-2\ell+1)\frac{\ln(n)}{n}.
\label{eq:dot-u-UL-1}
\end{multline}
Switching now to the comparison matrix $\mathbf{F}_n^{(k-\ell,\ell,1)}(\lambda;y)$, we assume that the inequalities $-\tfrac{1}{2}(k-2\ell-1)n^{-1}\ln(n)\le\mathrm{Re}(V(\lNaught(y)^{-1};y))\le -\tfrac{1}{2}(k-2\ell-\tfrac{1}{2})n^{-1}\ln(n)$ hold, which again imply Case II$_\mathrm{b}$ but now with $\lambda_\mathrm{U}=\lambda_1=\lNaught(y)^{-1}$ and $\lambda_\mathrm{L}=\lambda_2=\lNaught(y)$, and corresponding residues
\begin{equation}
\Phi_\mathrm{U}=\mathop{\mathrm{Res}}_{\lambda=\lNaught(y)^{-1}} A_1^{(k-\ell,\ell)}(\lambda;y)=
\frac{(k-\ell)!\lNaught(y)^k\lNaught(y)_\infty^{1/2}(\lNaught(y)^{-1}-\lNaught(y))^{2\ell}V''(\lNaught(y)^{-1};y)^{\ell-k}}{\ii k!V''(\lNaught(y)^{-1};y)^{1/2}}
\end{equation}
and
\begin{equation}
\begin{split}
\Phi_\mathrm{L}&=b_n^{(k-\ell,\ell,1)}\mathop{\mathrm{Res}}_{\lambda=\lNaught(y)}B_2^{(k-\ell,\ell)}(\lambda;y)\\
&=\ee^{-2nV(\lNaught(y)^{-1};y)}n^{2\ell-k-1}\frac{k!\lNaught(y)^k\lNaught(y)_\infty^{1/2}(\lNaught(y)^{-1}-\lNaught(y))^{2\ell-2k}V''(\lNaught(y);y)^{\ell}}{\ii (\ell-1)!V''(\lNaught(y);y)^{1/2}}.
\end{split}
\end{equation}
Thus \eqref{eq:dot-u-CaseIIb} becomes
\begin{multline}
\dot{u}_n=\ii\lNaught(y)^{-1}\frac{\ee^{2nV(\lNaught(y)^{-1};y)}-\ii(-1)^\ell n^{2\ell-k-1}\frac{(k-\ell)!}{(\ell-1)!}
\lNaught(y)(\lNaught(y)^{-1}+\lNaught(y))^{k-2\ell+1}(\lNaught(y)^{-1}-\lNaught(y))^{6\ell-3k-3}}
{\ee^{2nV(\lNaught(y)^{-1};y)}-\ii(-1)^\ell n^{2\ell-k-1}\frac{(k-\ell)!}{(\ell-1)!}
\lNaught(y)^{-1}(\lNaught(y)^{-1}+\lNaught(y))^{k-2\ell+1}(\lNaught(y)^{-1}-\lNaught(y))^{6\ell-3k-3}},\\
-\frac{1}{2}(k-2\ell+1)\frac{\ln(n)}{n}\le\mathrm{Re}(V(\lNaught(y)^{-1};y))\le -\frac{1}{2}(k-2\ell+\tfrac{1}{2})\frac{\ln(n)}{n}.
\label{eq:dot-u-UL-2}
\end{multline}
Comparing \eqref{eq:dot-u-UL-1} and \eqref{eq:dot-u-UL-2}, we observe that the approximate formula $\dot{u}_n$ is the same over the whole range $-\tfrac{1}{2}(k-2\ell+\tfrac{3}{2})n^{-1}\ln(n)\le\mathrm{Re}(V(\lNaught(y)^{-1};y))\le-\tfrac{1}{2}(k-2\ell+\tfrac{1}{2})n^{-1}\ln(n)$ over which Case II$_\mathrm{b}$ applies with different comparison matrices.  Continuing with the same comparison matrix we now assume the inequalities $-\tfrac{1}{2}(k-2\ell+\tfrac{1}{2})n^{-1}\ln(n)\le\mathrm{Re}(V(\lNaught(y)^{-1};y))\le-\tfrac{1}{2}(k-2\ell)n^{-1}\ln(n)$ and find that Case II$_a$ applies once again with residues given by
\begin{equation}
\Phi_1=\mathop{\mathrm{Res}}_{\lambda=\lNaught(y)^{-1}}A_1^{(k-\ell,\ell)}(\lambda;y)=
\frac{(k-\ell)!\lNaught(y)^k\lNaught(y)_\infty^{1/2}(\lNaught(y)^{-1}-\lNaught(y))^{2\ell}V''(\lNaught(y)^{-1};y)^{\ell-k}}{\ii k!V''(\lNaught(y)^{-1};y)^{1/2}}
\end{equation}
and
\begin{equation}
\begin{split}
\Phi_2&=a_n^{(k-\ell,\ell,1)}(y)\mathop{\mathrm{Res}}_{\lambda=\lNaught(y)}A_2^{(k-\ell,\ell)}(\lambda;y)\\
&=\ee^{2nV(\lNaught(y)^{-1};y)}n^{k-2\ell}\frac{\ell!\lNaught(y)^{-k}\lNaught(y)_\infty^{-1/2}(\lNaught(y)^{-1}-\lNaught(y))^{2k-2\ell}V''(\lNaught(y);y)^{-\ell}}{\ii k!V''(\lNaught(y);y)^{1/2}}.
\end{split}
\end{equation}
Hence from \eqref{eq:dot-u-CaseI-IIa} we get
\begin{multline}
\dot{u}_n=\ii\lNaught(y)\frac{\ee^{2nV(\lNaught(y)^{-1};y)}-\ii(-1)^\ell n^{2\ell-k}\frac{(k-\ell)!}{\ell!}\lNaught(y)^{-1}(\lNaught(y)^{-1}+\lNaught(y))^{k-2\ell}(\lNaught(y)^{-1}-\lNaught(y))^{6\ell-3k}}
{\ee^{2nV(\lNaught(y)^{-1};y)}-\ii(-1)^\ell n^{2\ell-k}\frac{(k-\ell)!}{\ell!}\lNaught(y)(\lNaught(y)^{-1}+\lNaught(y))^{k-2\ell}(\lNaught(y)^{-1}-\lNaught(y))^{6\ell-3k}},\\
-\frac{1}{2}(k-2\ell+\tfrac{1}{2})\frac{\ln(n)}{n}\le\mathrm{Re}(V(\lNaught(y)^{-1};y))\le
-\frac{1}{2}(k-2\ell)\frac{\ln(n)}{n}.
\label{eq:dot-u-UU-last}
\end{multline}
We observe that \eqref{eq:dot-u-UU-first} and \eqref{eq:dot-u-UU-last} agree upon replacing $\ell$ with $\ell+1$ in the former.\bigskip

Having completed the above four cases with $\ell=k$, it remains only to turn to the comparison matrix $\mathbf{F}_n^{(0,k,2)}(\lambda;y)$ and assume the inequality $\mathrm{Re}(V(\lNaught(y)^{-1};y))\ge \tfrac{1}{2}kn^{-1}\ln(n)$.  This corresponds to Case I with residues
\begin{equation}
\Phi_1=a_n^{(0,k,2)}(y)\mathop{\mathrm{Res}}_{\lambda=\lNaught(y)^{-1}}A_1^{(0,k)}(\lambda;y)
=\ee^{-2nV(\lNaught(y)^{-1};y)}n^k\frac{\lNaught(y)^k\lNaught(y)_\infty^{1/2}(\lNaught(y)^{-1}-\lNaught(y))^{2k}}{\ii k!V''(\lNaught(y)^{-1};y)^{1/2}}
\end{equation}
and
\begin{equation}
\Phi_2=\mathop{\mathrm{Res}}_{\lambda=\lNaught(y)}A_2^{(0,k)}(\lambda;y)=
\frac{\lNaught(y)^{-k}\lNaught(y)_\infty^{-1/2}V''(\lNaught(y);y)^{-k}}{\ii V''(\lNaught(y);y)^{1/2}}.
\end{equation}
Using these in \eqref{eq:dot-u-CaseI-IIa} gives
\begin{equation}
\dot{u}_n=\ii\lNaught(y)\frac{\ee^{2nV(\lNaught(y)^{-1};y)}-\ii (-1)^kn^k\frac{1}{k!}\lNaught(y)^{-1}
(\lNaught(y)^{-1}+\lNaught(y))^{-k}(\lNaught(y)^{-1}-\lNaught(y))^{3k}}{\ee^{2nV(\lNaught(y)^{-1};y)}-\ii (-1)^kn^k\frac{1}{k!}\lNaught(y)
(\lNaught(y)^{-1}+\lNaught(y))^{-k}(\lNaught(y)^{-1}-\lNaught(y))^{3k}},\quad
\mathrm{Re}(V(\lNaught(y)^{-1};y))\ge\frac{1}{2}k\frac{\ln(n)}{n}.
\label{eq:dot-u-UU-right}
\end{equation}
If $k=0$, the formul\ae\ \eqref{eq:dot-u-UU-left} and \eqref{eq:dot-u-UU-right} agree and define $\dot{u}_n$ by the same formula for all $y\in T$.  However if $k>0$, then \eqref{eq:dot-u-UU-right} agrees with \eqref{eq:dot-u-UU-last} with $\ell=k$, showing that the latter formula defines the approximation $\dot{u}_n$ over the whole range $\mathrm{Re}(V(\lNaught(y)^{-1};y))\ge\tfrac{1}{2}(k-\tfrac{1}{2})n^{-1}\ln(n)$.
Also, as $y$ moves out of $T$ into the exterior of $E$, we have $\mathrm{Re}(V(\lNaught(y)^{-1};y))>0$, so as $n\to+\infty$ we get $\dot{u}_n\approx \ii\lNaught(y)$, which is again consistent with Theorem~\ref{theorem:closed-eye-equilibrium}.  Combining these formul\ae\ with the convergence results described in Section~\ref{sec:Edge-error-analysis} and the exact symmetry \eqref{eq:u-n-exact-symmetry} to extend the results to $m=\tfrac{1}{2}+k$, $k\in\mathbb{Z}_{\ge 0}$ completes the proof of Theorem~\ref{thm:edge-formulae}.

\subsubsection{Detailed asymptotics for poles and zeros.  Proofs of Corollary~\ref{corollary:eyebrow-zeros-and-poles} and Theorem~\ref{theorem:eyebrow-curves}}
\label{sec:eyebrow-zeros-and-poles}
\begin{proof}[Proof of Corollary~\ref{corollary:eyebrow-zeros-and-poles}]
Each of the formul\ae\ for $\dot{u}_n$ described in Section~\ref{sec:edge-formulae} is a different meromorphic function of $y$ whose accuracy as an approximation of $u_n(ny;-(\tfrac{1}{2}+k))$ holds in an absolute sense for $y$ in a certain curvilinear strip roughly parallel to the eyebrow $\pEInftyRed$ and of width proportional to $n^{-1}\ln(n)$.  The absolute accuracy of the approximation depends on the assumption that $y$ is bounded away from each pole and zero of $\dot{u}_n$ by a distance proportional to $n^{-1}$ by an arbitrarily small constant.  It is easy to see that this distance is an arbitrarily small fraction of the spacing between nearest poles or zeros of $\dot{u}_n$.  This allows one to compute the index (winding number) of $u_n(ny;-(\tfrac{1}{2}+k))$ about a small circle containing just one pole or zero of $\dot{u}_n$ and hence deduce that the index is $-1$ or $1$ respectively.  
\end{proof}

\begin{proof}[Proof of Theorem~\ref{theorem:eyebrow-curves}]
According to \eqref{eq:dot-u-UU-left} and the discussion following \eqref{eq:dot-u-UU-first}, the zeros and poles of $\dot{u}_n$ in the left-most sub-domain of $T$ given by the inequality $\mathrm{Re}(V(\lNaught(y)^{-1};y))\le-\tfrac{1}{2}(k-\tfrac{1}{2})n^{-1}\ln(n)$ lie exactly on the respective curves
\begin{equation}
\text{Zero curve of $\dot{u}_n$ for $\displaystyle\mathrm{Re}(V(\lNaught(y)^{-1};y))\le-\frac{1}{2}(k-\tfrac{1}{2})\frac{\ln(n)}{n}$:}\quad
\ee^{2n\mathrm{Re}(V(\lNaught(y)^{-1};y))}=\frac{k!|\lNaught(y)^{-1}+\lNaught(y)|^k}{n^{k}|\lNaught(y)^{-1}-\lNaught(y)|^{3k}}|\lNaught(y)|^{-1},
\label{eq:left-zero-curve}
\end{equation}
\begin{equation}
\text{Pole curve of $\dot{u}_n$ for $\displaystyle\mathrm{Re}(V(\lNaught(y)^{-1};y))\le-\frac{1}{2}(k-\tfrac{1}{2})\frac{\ln(n)}{n}$:}\quad
\ee^{2n\mathrm{Re}(V(\lNaught(y)^{-1};y))}=\frac{k!|\lNaught(y)^{-1}+\lNaught(y)|^k}{n^{k}|\lNaught(y)^{-1}-\lNaught(y)|^{3k}}|\lNaught(y)|.
\label{eq:left-pole-curve}
\end{equation}
Note that because $|\lNaught(y)|<1$ holds for all $y\in\pEInftyRed$, the zero curve lies to the right of the pole curve.
Then, for $\ell=1,\dots,k$, by \eqref{eq:dot-u-UL-1}  and \eqref{eq:dot-u-UL-2}, the zeros and poles of $\dot{u}_n$ in the domain $-\tfrac{1}{2}(k-2\ell+\tfrac{3}{2})n^{-1}\ln(n)\le\mathrm{Re}(V(\lNaught(y)^{-1};y))\le-\tfrac{1}{2}(k-2\ell+\tfrac{1}{2})n^{-1}\ln(n)$ lie exactly on the respective curves
\begin{multline}
\text{Zero curve of $\dot{u}_n$ for $\displaystyle-\frac{1}{2}(k-2\ell+\tfrac{3}{2})\frac{\ln(n)}{n}\le\mathrm{Re}(V(\lNaught(y)^{-1};y))\le-\frac{1}{2}(k-2\ell+\tfrac{1}{2})\frac{\ln(n)}{n}$:}\\
\ee^{2n\mathrm{Re}(V(\lNaught(y)^{-1};y))}=\frac{(k-\ell)!|\lNaught(y)^{-1}+\lNaught(y)|^{k-2\ell+1}}{(\ell-1)!n^{k-2\ell+1}|\lNaught(y)^{-1}-\lNaught(y)|^{3k-6\ell+3}}|\lNaught(y)|,
\label{eq:mid-UL-zero-curve}
\end{multline}
\begin{multline}
\text{Pole curve of $\dot{u}_n$ for $\displaystyle-\frac{1}{2}(k-2\ell+\tfrac{3}{2})\frac{\ln(n)}{n}\le\mathrm{Re}(V(\lNaught(y)^{-1};y))\le-\frac{1}{2}(k-2\ell+\tfrac{1}{2})\frac{\ln(n)}{n}$:}\\
\ee^{2n\mathrm{Re}(V(\lNaught(y)^{-1};y))}=\frac{(k-\ell)!|\lNaught(y)^{-1}+\lNaught(y)|^{k-2\ell+1}}{(\ell-1)!n^{k-2\ell+1}|\lNaught(y)^{-1}-\lNaught(y)|^{3k-6\ell+3}}|\lNaught(y)|^{-1},
\label{eq:mid-UL-pole-curve}
\end{multline}
(the zero curve lies to the left of the pole curve)
and from \eqref{eq:dot-u-UU-first} and \eqref{eq:dot-u-UU-last}, the zeros and poles of $\dot{u}_n$ in the adjacent domain 
$-\tfrac{1}{2}(k-2\ell+\tfrac{1}{2})n^{-1}\ln(n)\le\mathrm{Re}(V(\lNaught(y)^{-1};y))\le-\tfrac{1}{2}(k-2\ell-\tfrac{1}{2})n^{-1}\ln(n)$ lie exactly on the respective curves
\begin{multline}
\text{Zero curve of $\dot{u}_n$ for $\displaystyle-\frac{1}{2}(k-2\ell+\tfrac{1}{2})\frac{\ln(n)}{n}\le\mathrm{Re}(V(\lNaught(y)^{-1};y))\le-\frac{1}{2}(k-2\ell-\tfrac{1}{2})\frac{\ln(n)}{n}$:}\\
\ee^{2n\mathrm{Re}(V(\lNaught(y)^{-1};y))}=\frac{(k-\ell)!|\lNaught(y)^{-1}+\lNaught(y)|^{k-2\ell}}{\ell!n^{k-2\ell}|\lNaught(y)^{-1}-\lNaught(y)|^{3k-6\ell}}|\lNaught(y)|^{-1},
\label{eq:mid-UU-zero-curve}
\end{multline}
\begin{multline}
\text{Pole curve of $\dot{u}_n$ for $\displaystyle-\frac{1}{2}(k-2\ell+\tfrac{1}{2})\frac{\ln(n)}{n}\le\mathrm{Re}(V(\lNaught(y)^{-1};y))\le-\frac{1}{2}(k-2\ell-\tfrac{1}{2})\frac{\ln(n)}{n}$:}\\
\ee^{2n\mathrm{Re}(V(\lNaught(y)^{-1};y))}=\frac{(k-\ell)!|\lNaught(y)^{-1}+\lNaught(y)|^{k-2\ell}}{\ell!n^{k-2\ell}|\lNaught(y)^{-1}-\lNaught(y)|^{3k-6\ell}}|\lNaught(y)|
\label{eq:mid-UU-pole-curve}
\end{multline}
(again the zero curve lies to the right of the pole curve).
Finally, according to \eqref{eq:dot-u-UU-right}, the zeros and poles of $\dot{u}_n$ in the right-most sub-domain of $T$ given by the inequality $\mathrm{Re}(V(\lNaught(y)^{-1};y))\ge \tfrac{1}{2}(k+\tfrac{1}{2})n^{-1}\ln(n)$ lie along the latter curves in the terminal case of $\ell=k$.  
\end{proof}

\appendix
\section{Solution of Riemann-Hilbert Problem~\ref{rhp:ParabolicCylinder}}
\label{app:PC}
\subsection{Derivation of differential equation}
Suppose that $\mathbf{P}(\zeta;m)$ satisfies the conditions of Riemann-Hilbert Problem~\ref{rhp:ParabolicCylinder}.  It is easy to check that $\det(\mathbf{P}(\zeta;m))=1$ for all $\zeta$.  Then the related matrix $\mathbf{R}(\zeta;m):=\mathbf{P}(\zeta;m)\ee^{\zeta^2\sigma_3/2}$ has the same analyticity domain and is equally regular up to the jump contour $\mathrm{Re}(\zeta^2)=0$.  It satisfies jump conditions across the four rays of the jump contour that are direct analogues of \eqref{eq:N-zeta-jump-first}--\eqref{eq:N-zeta-jump-last}, except that the $\zeta$-dependent exponential factors $\ee^{\pm \zeta^2}$ have been removed from the off-diagonal elements of the jump matrices.  Thus, $\mathbf{R}(\zeta;m)$ satisfies jump conditions that are independent of $\zeta$ on each ray.
Differentiating these jump conditions with respect to $\zeta$ then shows that $\mathbf{R}'(\zeta;m)$ satisfies exactly the same jump conditions as does $\mathbf{R}(\zeta;m)$ and hence (using the fact that $\det(\mathbf{R}(\zeta;m))=1$ for $\mathrm{Re}(\zeta^2)\neq 0$) $\mathbf{R}'(\zeta;m)\mathbf{R}(\zeta;m)^{-1}$ can be continued to the jump contour unambiguously, defining an entire function of $\zeta$.  Supposing for the moment that the normalization condition \eqref{eq:N-zeta-normalize} holds in the stronger sense that 
\begin{equation}
\mathbf{P}(\zeta;m)\sim(\mathbb{I} + \zeta^{-1}\mathbf{P}^{\infty}_1(m) + \cdots)\zeta^{-(m+\tfrac{1}{2})\sigma_3},\quad\zeta\to\infty
\end{equation}
with $\mathbf{P}^{\infty}_1(m)$ being the same for all four sectors and with the indicated asymptotic series being differentiable term-by-term, the entire function $\mathbf{R}'(\zeta;m)\mathbf{R}(\zeta;m)^{-1}$ satisfies
\begin{equation}
\mathbf{R}'(\zeta;m)\mathbf{R}(\zeta;m)^{-1}=\sigma_3\zeta + [\mathbf{P}^{\infty}_1(m),\sigma_3] + \mathcal{O}(\zeta^{-1}),\quad\zeta\to\infty,
\end{equation}
and hence by Liouville's theorem, $\mathbf{R}'(\zeta;m)\mathbf{R}(\zeta;m)^{-1}=\sigma_3\zeta + [\mathbf{P}^{\infty}_1(m),\sigma_3]$ exactly.  In other words, $\mathbf{R}(\zeta;m)$ satisfies the following differential equation:
\begin{equation}
\frac{\dd\mathbf{R}}{\dd\zeta}=\begin{bmatrix}\zeta & \alpha\\\beta & -\zeta\end{bmatrix}\mathbf{R},\quad \alpha:=-2P^{\infty}_{1,12}(m),\quad\beta:=2P^{\infty}_{1,21}(m).
\label{eq:R-matrix-PC-ODE}
\end{equation}
Rescaling by $t:=\zeta\sqrt{2}$ and eliminating the second row shows that any element $R_{1k}$, $k=1,2$, of the first row satisfies the differential equation of parabolic cylinder functions (see \cite[Eq.~12.2.2]{DLMF})
\begin{equation}
\frac{\dd^2R_{1k}}{\dd t^2} - \left(\frac{1}{4}t^2 + a\right)R_{1k}=0,\quad a:=\frac{1}{2}(1+\alpha\beta), \quad k=1,2.
\end{equation}
According to \cite[\S12.2(i)]{DLMF}, we will take $R_{1k}$ as a linear combination of an appropriate ``numerically satisfactory'' pair of solutions in each of the four sectors:
\begin{equation}
R_{1k}(\zeta;m)=\alpha A_k^\mathrm{I}U(a,\sqrt{2}\zeta) + \alpha B_k^\mathrm{I}U(-a,-\ii\sqrt{2}\zeta),\quad 0\le\arg(\zeta)\le\frac{\pi}{2},
\label{eq:R1-I}
\end{equation}
\begin{equation}
R_{1k}(\zeta;m)=\alpha A_k^\mathrm{II}U(-a,-\ii\sqrt{2}\zeta) + \alpha B_k^\mathrm{II}U(a,-\sqrt{2}\zeta),\quad \frac{\pi}{2}\le\arg(\zeta)\le\pi,
\end{equation}
\begin{equation}
R_{1k}(\zeta;m)=\alpha A_k^\mathrm{III}U(a,-\sqrt{2}\zeta) + \alpha B_k^\mathrm{III}U(-a,\ii\sqrt{2}\zeta),\quad -\pi\le\arg(\zeta)\le -\frac{\pi}{2},
\end{equation}
and
\begin{equation}
R_{1k}(\zeta;m)=\alpha A_k^\mathrm{IV}U(-a,\ii\sqrt{2}\zeta)+\alpha B_k^\mathrm{IV}U(a,\sqrt{2}\zeta),\quad -\frac{\pi}{2}\le\arg(\zeta)\le 0.
\label{eq:R1-IV}
\end{equation}
From the first row of \eqref{eq:R-matrix-PC-ODE} we then find the corresponding matrix elements $R_{2k}(\zeta;m):=\alpha^{-1}(R_{1k}'(\zeta;m)-\zeta R_{1k}(\zeta;m))$.  Hence using \cite[Eqs.~12.8.2--12.8.3]{DLMF}, from \eqref{eq:R1-I}--\eqref{eq:R1-IV} we get
\begin{equation}
R_{2k}(\zeta;m)=-\sqrt{2}A_k^\mathrm{I}U(a-1,\sqrt{2}\zeta)-\ii\sqrt{2}(a-\tfrac{1}{2})B_k^\mathrm{I}U(1-a,-\ii\sqrt{2}\zeta),\quad 0\le\arg(\zeta)\le\frac{\pi}{2},
\label{eq:R2-I}
\end{equation}
\begin{equation}
R_{2k}(\zeta;m)=-\ii\sqrt{2}(a-\tfrac{1}{2})A_k^\mathrm{II}U(1-a,-\ii\sqrt{2}\zeta) +\sqrt{2}B_k^\mathrm{II}U(a-1,-\sqrt{2}\zeta),\quad \frac{\pi}{2}\le\arg(\zeta)\le\pi,
\end{equation}
\begin{equation}
R_{2k}(\zeta;m)=\sqrt{2}A_k^\mathrm{III}U(a-1,-\sqrt{2}\zeta)+\ii\sqrt{2}(a-\tfrac{1}{2})B_k^\mathrm{III}U(1-a,\ii\sqrt{2}\zeta),\quad -\pi\le\arg(\zeta)\le-\frac{\pi}{2},
\end{equation}
\begin{equation}
R_{2k}(\zeta;m)=\ii\sqrt{2}(a-\tfrac{1}{2})A_k^\mathrm{IV}U(1-a,\ii\sqrt{2}\zeta) -\sqrt{2}B_k^\mathrm{IV}U(a-1,\sqrt{2}\zeta),\quad -\frac{\pi}{2}\le\arg(\zeta)\le 0.
\label{eq:R2-IV}
\end{equation}
Note that in addition to the sixteen coefficients $A_k^S$ and $B_k^S$, $k=1,2$, $S=\mathrm{I},\mathrm{II},\mathrm{III},\mathrm{IV}$, it remains to determine also the parameters $\alpha$ and $\beta$.
\subsection{Selection of solutions and parameter determination}
Now we impose that the matrix $\mathbf{R}(\zeta;m)$ satisfy the leading-order normalization condition implied by \eqref{eq:N-zeta-normalize}.  For this purpose, given the choice of basis made above in each sector $S$, it is sufficient to use the large-$z$ asymptotic expansion for $U(a,z)$ given by \cite[Eq.~12.9.1]{DLMF}, which implies that $U(a,z)=e^{-z^2/4}z^{-a-1/2}(1+\mathcal{O}(z^{-2}))$ as $z\to\infty$ with $|\arg(z)|< 3\pi/4$.
Thus, in order to avoid unwanted exponential growth it is necessary to choose:
\begin{equation}
A_1^\mathrm{I}=B_2^\mathrm{I}=0,\quad A_2^\mathrm{II}=B_1^\mathrm{II}=0,\quad A_1^\mathrm{III}=B_2^\mathrm{III}=0,\quad A_2^\mathrm{IV}=B_1^\mathrm{IV}=0.
\end{equation}
With these choices, all four elements of $\mathbf{R}(\zeta;m)\ee^{-\zeta^2\sigma_3/2}\zeta^{(m+\tfrac{1}{2})\sigma_3}$ are bounded by a power of $\zeta$ as $\zeta\to\infty$.  Determining the parameter $a=\tfrac{1}{2}(1+\alpha\beta)$ in terms of $m$ explicitly by
\begin{equation}
a=-m
\end{equation}
is then necessary to ensure the existence of a finite limit as $\zeta\to\infty$.  By examination of the diagonal elements in the four sectors, we then deduce that $\mathbf{R}(\zeta;m)\ee^{-\zeta^2\sigma_3/2}\zeta^{(m+\tfrac{1}{2})\sigma_3}=\mathbb{I}+\mathcal{O}(\zeta^{-2})$ as $\zeta\to\infty$ in all directions, provided that the remaining eight nonzero coefficients are determined as follows:
\begin{equation}
\begin{gathered}
B_1^\mathrm{I}=A_1^\mathrm{II}=\alpha^{-1}2^{\tfrac{1}{4}-\tfrac{1}{2}a}\ee^{\ii\pi(\tfrac{1}{2}a-\tfrac{1}{4})},\quad
B_1^\mathrm{III}=A_1^\mathrm{IV}=\alpha^{-1}2^{\tfrac{1}{4}-\tfrac{1}{2}a}\ee^{-\ii\pi(\tfrac{1}{2}a-\tfrac{1}{4})},\\
A_2^\mathrm{I}=B_2^\mathrm{IV}=-2^{\tfrac{1}{2}a-\tfrac{3}{4}},\quad B_2^\mathrm{II}=2^{\tfrac{1}{2}a-\tfrac{3}{4}}\ee^{\ii\pi(\tfrac{1}{2}-a)},\quad A_2^\mathrm{III}=2^{\tfrac{1}{2}a-\tfrac{3}{4}}\ee^{-\ii\pi(\tfrac{1}{2}-a)}.
\end{gathered}
\end{equation}
The only remaining ambiguity concerns the precise values of $\alpha$ and $\beta$ such that $\tfrac{1}{2}(1+\alpha\beta)=a=-m$.  This ambiguity is resolved by resorting to the jump conditions in Riemann-Hilbert Problem~\ref{rhp:ParabolicCylinder}.  Here, we make use of the connection formula $U(a,z)=\pm\ii\ee^{\pm\ii\pi a}U(a,-z)+\sqrt{2\pi}\ee^{\pm\ii\pi(a-1/2)/2}U(-a,\pm\ii z)/\Gamma(a+1/2)$, see \cite[Eq.~12.2.19]{DLMF}.  With the help of this formula, it is straightforward to check that the jump conditions are satisfied by $\mathbf{P}(\zeta;m):=\mathbf{R}(\zeta;m)\ee^{-\zeta^2\sigma_3/2}$, provided that $\alpha$ is given by:
\begin{equation}
\alpha=\ii\ee^{\ii \pi m}2^{m+1}.
\end{equation}
Then from $a=\tfrac{1}{2}(1+\alpha\beta)=-m$ we get 
\begin{equation}
\beta=\ii\ee^{-\ii\pi m}2^{-m}(m+\tfrac{1}{2}).
\end{equation}

\subsection{Refined asymptotics of $\mathbf{P}(\zeta)$}
The matrix $\mathbf{P}(\zeta;m)=\mathbf{R}(\zeta;m)\ee^{-\zeta^2\sigma_3/2}$ clearly satisfies all of the conditions of Riemann-Hilbert Problem~\ref{rhp:ParabolicCylinder}, and it is easy to show that there can be only one solution.  Here we give the complete asymptotic expansion of the solution $\mathbf{P}(\zeta;m)$ in the large-$\zeta$ limit, giving information beyond the normalization condition \eqref{eq:N-zeta-normalize}.  Indeed, using \cite[Eq. 12.9.1]{DLMF}, we find that
\begin{equation}
\mathbf{P}(\zeta;m)\zeta^{(m+\tfrac{1}{2})\sigma_3}\sim\begin{bmatrix}
\displaystyle\sum_{j=0}^\infty\frac{(\tfrac{1}{2}+m)_{2j}}{4^jj!\zeta^{2j}} &
\displaystyle - \ii\ee^{\ii\pi m}2^m\zeta^{-1}\sum_{j=0}^\infty \frac{(-1)^j(\tfrac{1}{2}-m)_{2j}}{4^jj!\zeta^{2j}}\\
\displaystyle \ii\ee^{-\ii\pi m}2^{-m-1}(m+\tfrac{1}{2})\zeta^{-1}\sum_{j=0}^\infty \frac{(\tfrac{3}{2}+m)_{2j}}{4^jj!\zeta^{2j}} &
\displaystyle\sum_{j=0}^\infty\frac{(-1)^j(-\tfrac{1}{2}-m)_{2j}}{4^jj!\zeta^{2j}}
\end{bmatrix},\quad\zeta\to\infty
\label{eq:N-expansion}
\end{equation}
uniformly in all directions of the complex plane, including along the sector boundaries.  Notably, the expansion coefficients do not depend on
the sector in which the solution is analyzed, a fact that also follows from the exponential decay of the off-diagonal elements of the jump matrices in Riemann-Hilbert Problem~\ref{rhp:ParabolicCylinder}.

\section{Solution of Riemann-Hilbert Problem~\ref{rhp:Airy}}
\label{app:Airy}
\subsection{Derivation of differential equation}  
Reasoning as in Appendix~\ref{app:PC}, one checks that if $\mathbf{A}(\zeta)$ satisfies Riemann-Hilbert Problem~\ref{rhp:Airy}, then the related matrix $\mathbf{S}(\zeta):=\mathbf{A}(\zeta)\ee^{-\zeta^{3/2}\sigma_3/2}$ has unit determinant and satisfies jump conditions analogous to \eqref{eq:Airy-jumps-A} except that the exponential factors in the jump matrices for $\arg(\zeta)=0$ and $\arg(\zeta)=\pm \tfrac{2}{3}\pi$ have been cancelled.  It follows by differentiation of the resulting constant jump matrices with respect to $\zeta$ that $\mathbf{S}'(\zeta)\mathbf{S}(\zeta)^{-1}$ is an entire function of $\zeta$.  Assuming that $\mathbf{A}(\zeta)\mathbf{V}^{-1}\zeta^{-\sigma_3/4}\sim\mathbb{I} + \zeta^{-1}\mathbf{A}_1^\infty + \cdots$ as $\zeta\to\infty$ with the coefficient $\mathbf{A}_1^\infty$ being the same for each of the four sectors of analyticity and with the asymptotic series being differentiable term-by-term, it follows that 
\begin{equation}
\mathbf{S}'(\zeta)\mathbf{S}(\zeta)^{-1}=\frac{3\ii}{4}\begin{bmatrix}A^\infty_{1,21} & A^\infty_{1,22}-A^\infty_{1,11}-\zeta\\1 & -A^\infty_{1,21}\end{bmatrix}+\mathcal{O}(\zeta^{-1}),\quad\zeta\to\infty.
\end{equation}
By Liouville's theorem, we derive the first-order system of differential equations satisfied by $\mathbf{S}(\zeta)$:
\begin{equation}
\frac{\dd\mathbf{S}}{\dd\zeta}=\frac{3\ii}{4}\begin{bmatrix}A^\infty_{1,21} & A^\infty_{1,22}-A^\infty_{1,11}-\zeta\\1 & -A^\infty_{1,21}\end{bmatrix}\mathbf{S}.
\label{eq:Airy-system}
\end{equation}
Rescaling by $t=(\tfrac{3}{4})^{2/3}(\zeta-c)$, where $c:=(A^\infty_{1,21})^2+A^\infty_{1,22}-A^\infty_{1,11}$ and eliminating the first row of $\mathbf{S}$ it follows that the elements of the second row of $\mathbf{S}$ are solutions of Airy's equation \cite[Eqn.\@ 9.2.1]{DLMF}
\begin{equation}
\frac{\dd^2S_{2k}}{\dd t^2}-\zeta S_{2k}=0,\quad k=1,2.
\label{eq:Airy}
\end{equation}
We represent the second-row elements of $\mathbf{S}$ as linear combinations of ``numerically satisfactory'' solutions (see \cite[Table 9.2.1]{DLMF}) appropriate for each sector:
\begin{equation}
S_{2k}(\zeta)=A_k^\mathrm{I}\mathrm{Ai}(t) + B_k^\mathrm{I}\mathrm{Ai}(\ee^{-2\pi\ii/3}t),\quad
0\le\arg(\zeta)\le\frac{2}{3}\pi,
\end{equation}
\begin{equation}
S_{2k}(\zeta)=A_k^\mathrm{II}\mathrm{Ai}(\ee^{-2\pi\ii/3}t)+B_k^\mathrm{II}\mathrm{Ai}(\ee^{2\pi\ii/3}t),
\quad \frac{2}{3}\pi\le\arg(\zeta)\le\pi,
\end{equation}
\begin{equation}
S_{2k}(\zeta)=A_k^\mathrm{III}\mathrm{Ai}(\ee^{-2\pi\ii/3}t)+B_k^\mathrm{III}\mathrm{Ai}(\ee^{2\pi\ii/3}t),\quad -\pi\le\arg(\zeta)\le-\frac{2}{3}\pi,
\end{equation}
\begin{equation}
S_{2k}(\zeta)=A_k^\mathrm{IV}\mathrm{Ai}(t)+B_k^\mathrm{IV}\mathrm{Ai}(\ee^{2\pi\ii/3}t),\quad
-\frac{2}{3}\pi\le\arg(\zeta)\le 0.
\end{equation}
Once the sixteen coefficients $A_k^S$ and $B_k^S$, $k=1,2$, $S=\mathrm{I},\mathrm{II},\mathrm{III},\mathrm{IV}$ have been determined, the first row elements of $\mathbf{S}(\zeta)$ will then be determined from the first-order system \eqref{eq:Airy-system}.  The parameter $c$ involved in the relation connecting $\zeta$ with $t$ also needs to be determined.

\subsection{Selection of solutions}
We now consider the normalization condition on $\mathbf{A}(\zeta)$ which in terms of $\mathbf{S}(\zeta)$ reads $\mathbf{S}(\zeta)\ee^{\zeta^{3/2}\sigma_3/2}\mathbf{V}^{-1}\zeta^{-\sigma_3/4}=\mathbb{I}+\mathcal{O}(\zeta^{-1})$ as $\zeta\to\infty$.  Imposing this condition in each sector with the help of the well-known asymptotic formula \cite[Eqn.\@ 9.7.5]{DLMF} for $\mathrm{Ai}(t)$ valid for large $t$, we find that in order to avoid unwanted exponential growth it is necessary to choose
\begin{equation}
B_1^\mathrm{I}=A_2^\mathrm{I}=0,\quad A_1^\mathrm{II}=B_2^\mathrm{II}=0,\quad
B_1^\mathrm{III}=A_2^\mathrm{III}=0,\quad B_1^\mathrm{IV}=A_2^\mathrm{IV}=0.
\label{eq:Airy-zero-coeffs}
\end{equation}
While these conditions remove the terms exhibiting the most rapid exponential growth as $\zeta\to\infty$, there are still subdominant exponentially growing terms that, since $\det(\mathbf{S}(\zeta))=1$, can only be removed if the parameter $c$ is made to vanish:  $c=0$.  Assuming \eqref{eq:Airy-zero-coeffs} and $c=0$, the quantity $\mathbf{S}(\zeta)\ee^{\zeta^{3/2}\sigma_3/2}\mathbf{V}^{-1}\zeta^{-\sigma_3/4}$ tends to a finite limit as $\zeta\to\infty$ but with a leading error term proportional to $\zeta^{-1/2}$.  Removing this term and imposing the condition that the finite limit in question is $\mathbb{I}$ gives two additional conditions per sector that determine all eight remaining coefficients.\smallskip  

Thus, the second row of the matrix $\mathbf{S}(\zeta)$ is uniquely determined in each of the four sectors simply from the differential equation \eqref{eq:Airy} and by imposing the desired asymptotic behavior for large $\zeta$.  The first row is then determined from the second using \eqref{eq:Airy-system} and $c=0$, and finally once $\mathbf{S}(\zeta)$ is known in all four sectors, $\mathbf{A}(\zeta)=\mathbf{S}(\zeta)\ee^{-\zeta^{3/2}\sigma_3/2}$.  The resulting formul\ae\ are as follows, in which $t=(\tfrac{3}{4})^{2/3}\zeta$ because $c=0$:
\begin{equation}
\mathbf{A}(\zeta):=\sqrt{2\pi}\left(\frac{4}{3}\right)^{\sigma_3/6}\begin{bmatrix}-\mathrm{Ai}'(t) & \ee^{2\pi\ii/3}\mathrm{Ai}'(t\ee^{-2\pi\ii/3})\\
-\ii\mathrm{Ai}(t) & \ii\ee^{-2\pi\ii/3}\mathrm{Ai}(t\ee^{-2\pi\ii/3})\end{bmatrix}\ee^{2t^{3/2}\sigma_3/3},\quad 0<\arg(\zeta)<\frac{2}{3}\pi,
\end{equation}
\begin{equation}
\mathbf{A}(\zeta):=\sqrt{2\pi}\left(\frac{4}{3}\right)^{\sigma_3/6}\begin{bmatrix}
\ee^{-2\pi\ii/3}\mathrm{Ai}'(t\ee^{2\pi\ii/3}) & \ee^{2\pi\ii/3}\mathrm{Ai}'(t\ee^{-2\pi\ii/3})\\
\ii\ee^{2\pi\ii/3}\mathrm{Ai}(t\ee^{2\pi\ii/3}) & \ii\ee^{-2\pi\ii/3}\mathrm{Ai}(t\ee^{-2\pi\ii/3})
\end{bmatrix}\ee^{2t^{3/2}\sigma_3/3},\quad\frac{2}{3}\pi<\arg(\zeta)<\pi,
\end{equation}
\begin{equation}
\mathbf{A}(\zeta):=\sqrt{2\pi}\left(\frac{4}{3}\right)^{\sigma_3/6}\begin{bmatrix}
\ee^{2\pi\ii/3}\mathrm{Ai}'(t\ee^{-2\pi\ii/3}) & -\ee^{-2\pi\ii/3}\mathrm{Ai}'(t\ee^{2\pi\ii/3})\\
\ii\ee^{-2\pi\ii/3}\mathrm{Ai}(t\ee^{-2\pi\ii/3}) & -\ii\ee^{2\pi\ii/3}\mathrm{Ai}(t\ee^{2\pi\ii/3})
\end{bmatrix}\ee^{2t^{3/2}\sigma_3/3},\quad -\pi<\arg(\zeta)<-\frac{2}{3}\pi,
\end{equation}
\begin{equation}
\mathbf{A}(\zeta):=\sqrt{2\pi}\left(\frac{4}{3}\right)^{\sigma_3/6}\begin{bmatrix}
-\mathrm{Ai}'(t) & -\ee^{-2\pi\ii/3}\mathrm{Ai}'(t\ee^{2\pi\ii/3})\\
-\ii\mathrm{Ai}(t) & -\ii\ee^{2\pi\ii/3}\mathrm{Ai}(t\ee^{2\pi\ii/3})\end{bmatrix}
\ee^{2t^{3/2}\sigma_3/3},\quad -\frac{2}{3}\pi<\arg(\zeta)<0.
\end{equation}
The jump conditions relating the boundary values of $\mathbf{A}(\zeta)$ in Riemann-Hilbert Problem~\ref{rhp:Airy} are now seen to simply be a consequence of the connection formula $\mathrm{Ai}(t)+\ee^{-2\pi\ii/3}\mathrm{Ai}(t\ee^{-2\pi\ii/3})+\ee^{2\pi\ii/3}\mathrm{Ai}(t\ee^{2\pi\ii/3})=0$ (see \cite[Eqn.\@ 9.2.12]{DLMF}).

\subsection{Refined asymptotics of $\mathbf{A}(\zeta)$}
Using the known asymptotic expansions of $\mathrm{Ai}(t)$ and $\mathrm{Ai}'(t)$ (see \cite[Eqns.\@ 9.7.5--9.7.6]{DLMF}), it is easy to show that $\mathbf{A}(\zeta)\mathbf{V}^{-1}\zeta^{-\sigma_3/4}$ has a complete asymptotic expansion in integer powers of $\zeta$ as $\zeta\to\infty$, and that the leading error term is characterized by the formula \eqref{eq:Airy-norm-better}.


\begin{thebibliography}{99}
\bibitem{BassNRS10}  L. Bass, J. J. Nimmo, C. Rogers, W. K. Schief, ``Electrical structures of interfaces: a Painlev\'e II model,'' \textit{Proc.\@ R.\@ Soc.\@ A} \textbf{466}, 2117--2136, 2010.
\bibitem{BealsC84} R. Beals and R. R. Coifman, ``Scattering and inverse scattering for first order systems,'' \textit{Comm.\@ Pure Appl.\@ Math.\@} \textbf{37}, 39--90, 1984.
\bibitem{BertolaB15} M. Bertola and T. Bothner, ``Zeros of large degree Vorob'ev-Yablonski polynomials via a Hankel determinant identity,''
\textit{Int.\@ Math.\@ Research Notices} \textbf{2015}, 9330--9399, 2015.
\bibitem{BothnerMS18}
T. J. Bothner, P. D. Miller, and Y. Sheng, ``Rational solutions of the Painlev\'e-III equation,'' to appear in \textit{Stud.\@ Appl.\@ Math}, 2018.  DOI:  10.1111/sapm.12220.  \texttt{arXiv:1801.04360}.
\bibitem{Buckingham17}
R. J. Buckingham, ``Large-degree asymptotics of rational Painlev\'e-IV functions associated to generalized Hermite polynomials,'' to appear in \textit{Int.\@ Math.\@ Research Notices}, 2018.  \texttt{arXiv:1706.09005}.
\bibitem{BuckinghamM12} R. J. Buckingham and P. D. Miller, ``The sine-Gordon equation in the semiclassical limit: critical behavior near a separatrix,'' \textit{J.\@ Anal.\@ Math.\@} \textbf{118}, 397--492, 2012.
\bibitem{BuckinghamM14}
R. J. Buckingham and P. D. Miller, ``Large-degree asymptotics of rational Painlev\'e-II functions:  noncritical behaviour,'' \textit{Nonlinearity} \textbf{27}, 2489--2577, 2014.
\bibitem{BuckinghamM15} R. J. Buckingham and P. D. Miller, ``Large-degree asymptotics of rational Painlev\'e-II functions: critical behaviour,'' \textit{Nonlinearity} \textbf{28}, 1539--1596, 2015.
\bibitem{Clarkson03}
P. A. Clarkson, ``The third Painlev\'e equation and associated special polynomials,'' \textit{J. Phys. A:  Math. Gen.} \textbf{36}, 9507--9532, 2003.
\bibitem{Clarkson06}  P. A. Clarkson, ``Special polynomials associated with rational solutions of the Painlev\'e equations and applications to soliton equations,'' \textit{Comput.\@ Meth.\@ Funct.\@ Theory} \textbf{6}, 329--401, 2006.
\bibitem{Clarkson09} P. A. Clarkson, ``Vortices and polynomials,'' \textit{Stud.\@ Appl.\@ Math.\@} \textbf{123}, 37--62, 2009.
\bibitem{ClarksonLL16}
P. A. Clarkson, C.-K. Law, C.-H. Lin, ``An algebraic proof for the Umemura polynomials for the third Painlev\'e equation,'' \texttt{arXiv:1609.00495}, 2016.
\bibitem{Dubrovin81} B. A. Dubrovin, ``Theta functions and non-linear equations,'' \textit{Russ.\@ Math.\@ Surveys} \textbf{36}, 11--92, 1981.
\bibitem{FokasIK91} A. S. Fokas, A. R. Its, and A. V. Kitaev, ``Discrete Painlev\'e equations and their appearance in quantum gravity,'' \textit{Comm.\@ Math.\@ Phys.\@} \textbf{142}, 313--344, 1991.
\bibitem{Johnson06}  C. V. Johnson, ``String theory without branes,'' \texttt{arXiv:hep-th/0610223}, 2006.
\bibitem{MasoeroR18}
D. Masoero and P. Roffelsen, ``Poles of Painlev\'e IV rationals and their distribution,''
\textit{SIGMA} \textbf{14}, 002 (49 pages), 2018.
\bibitem{MillerS17}
P. D. Miller and Y. Sheng, ``Rational solutions of the Painlev\'e-II equation revisited,'' \textit{SIGMA} \textbf{13} 065 (29 pages), 2017.
\bibitem{DLMF}
F. W. J. Olver, A. B. Olde Daalhuis, D. W. Lozier, B. I. Schneider, R. F. Boisvert, C. W. Clark, B. R. Miller, and B. V. Saunders, eds., NIST Digital Library of Mathematical Functions, \texttt{http://dlmf.nist.gov/}, Release 1.0.14,  2016. 
\bibitem{ShapiroT14}   B. Shapiro and M. Tater, ``On spectral asymptotics of quasi-exactly solvable quartic and Yablonskii-Vorob'ev polynomials,'' \texttt{arXiv:1412.3026}, 2014.
\bibitem{Umemura99}
H. Umemura, ``100 years of the Painlev\'e equation,'' \textit{S\-ugaku} \textbf{51}, 395--420, 1999 (in Japanese).
\end{thebibliography}
\end{document}